\documentclass{amsart}

\usepackage{Preset}
\usepackage[makeroom]{cancel}
\usepackage{MnSymbol}
\usepackage{nccmath}

\usepackage{mathrsfs}
\usepackage[latin1]{inputenc}
\usepackage[T1]{fontenc}
\usepackage{tikz}
\usepackage[colorlinks=true,linkcolor=blue!50!black,anchorcolor=red,citecolor=blue,filecolor=black,menucolor=black,runcolor=black,urlcolor=black]{hyperref}

\usetikzlibrary{decorations.pathmorphing}
\usetikzlibrary{shapes,arrows,graphs}

\title[Hyperbolicity of renormalization]{Hyperbolicity of renormalization of critical quasicircle maps}
\author[Willie Rush Lim]{Willie Rush Lim}
\address{Dept. of Mathematics, Brown University, RI 02912}
\email{willie\_rush\_lim@brown.edu}
\date{}

\begin{document}

\begin{abstract}
    There is a well developed renormalization theory of real analytic critical circle maps by de Faria, de Melo, and Yampolsky. In this paper, we extend Yampolsky's result on the hyperbolicity of renormalization periodic points to a larger class of dynamical objects, namely critical quasicircle maps, i.e. analytic self-homeomorphisms of a quasicircle with a single critical point. Unlike critical circle maps, the inner and outer criticalities of critical quasicircle maps can be distinct. We develop a compact analytic renormalization operator called ``Corona Renormalization`` with a hyperbolic fixed point whose stable manifold has codimension one and consists of critical quasicircle maps of the same criticality and periodic type rotation number. Our proof is an adaptation of Pacman Renormalization Theory for Siegel disks as well as rigidity results on the escaping dynamics of transcendental entire functions.
\end{abstract}

% Title in French:
% Hyperbolicité de la renormalisation des applications critiques des quasicercles

% Abstract in French: 
% Il existe une théorie bien développée de la renormalisation des applications analytiques réelles critiques du cercle, due à de Faria, de Melo et Yampolsky. Dans cet article, nous étendons le résultat de Yampolsky sur l'hyperbolicité des points périodiques de renormalisation à une classe plus large d'objets dynamiques, à savoir les applications critiques des quasicercles, c'est-à-dire des homéomorphismes analytiques d'un quasicercle possédant un unique point critique. Contrairement aux applications critiques du cercle, les criticités intérieure et extérieure des applications critiques des quasicercles peuvent être distinctes. Nous développons un opérateur de renormalisation analytique compact appelé "Corona Renormalization", qui possède un point fixe hyperbolique dont la variété stable est de codimension un et est constituée d'applications critiques des quasicercles ayant la même criticité et un nombre de rotation de type périodique. Notre démonstration est une adaptation de la théorie de la renormalisation de Pacman pour les disques de Siegel ainsi que des résultats de rigidité sur la dynamique échappatoire des fonctions entières transcendantes.

\maketitle

\setcounter{tocdepth}{1}
\tableofcontents

\section{Introduction}
\label{sec:intro}

A Jordan curve $\Hq$ in the Riemann sphere $\RS = \C \cup \{\infty\}$ is called a quasicircle if it is the image of the unit circle under a quasiconformal homeomorphism of $\RS$. The key player in this paper is the following.

\begin{definition}
    A (\emph{uni}-)\emph{critical quasicircle map} is an orientation-preserving homeomorphism $f: \Hq \to \Hq$ of a quasicircle which extends to a holomorphic map on a neighborhood of $\Hq$ and has exactly one critical point on $\Hq$.
\end{definition}

Given a critical quasicircle map $f: \Hq \to \Hq$, the behaviour at the unique critical point on $\Hq$ can be encoded by two positive integers, namely the inner criticality $d_0$ and the outer criticality $d_\infty$. 
Roughly speaking, for any point near the critical value located inside or outside of $\Hq$, the number of preimages that are located inside or outside is $d_0$ or $d_\infty$ respectively.
The total local degree of $f$ at the critical point is $d_0+d_\infty-1$ and it is at least $2$. When the criticalities are specified, we call $f: \Hq \to \Hq$ a \emph{$(d_0,d_\infty)$-critical} quasicircle map. 
See Figure \ref{fig:cqc-comparison} for some examples.

\begin{figure}
    \centering
    \includegraphics[width=\linewidth]{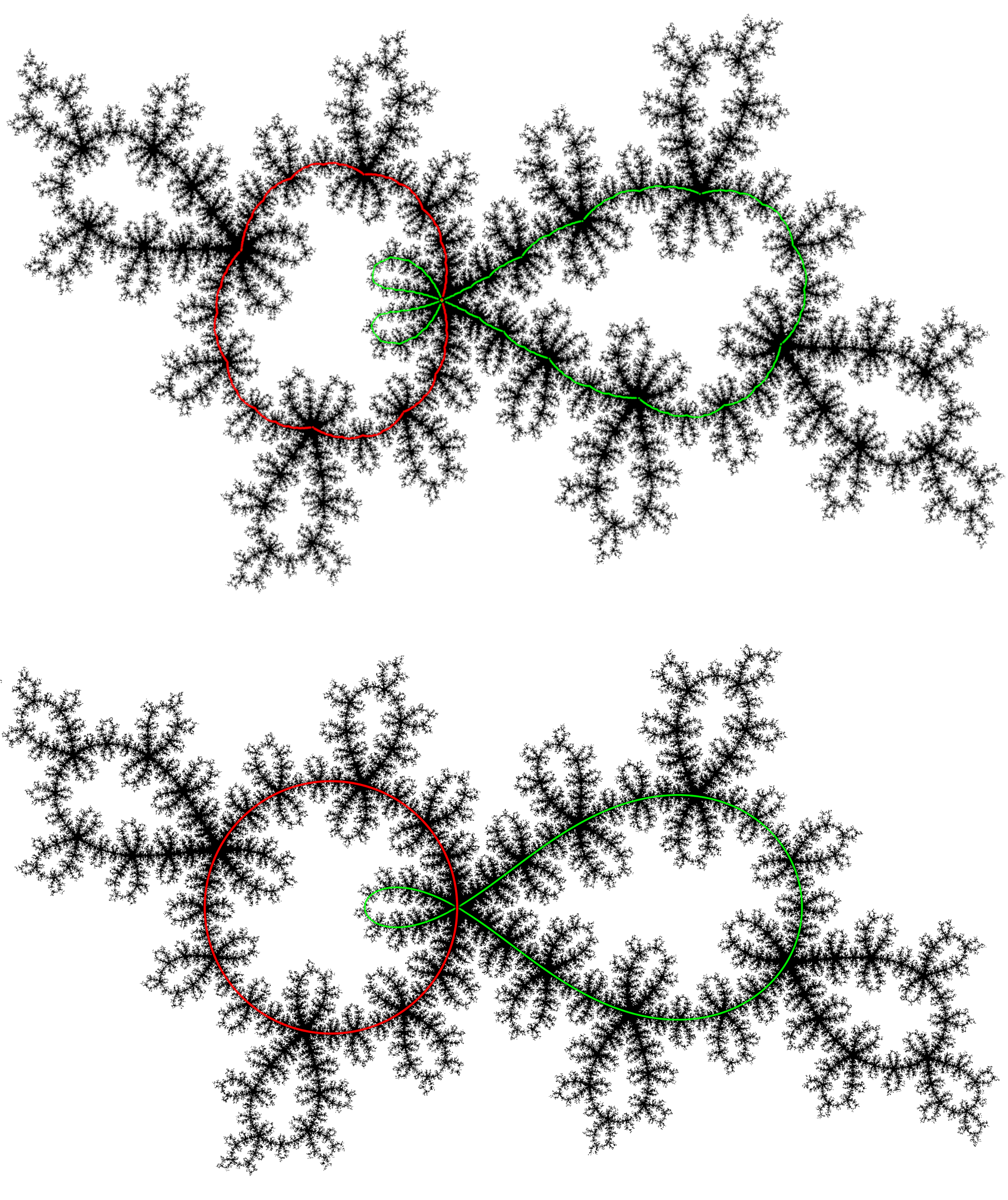}
    \caption{The Julia sets of
    \leavevmode\\
    \\
    \begin{minipage}{\linewidth} 
    \begin{align*}
        f_{3,2}(z) = bz^3\dfrac{4-z}{1-4z+6z^2} \quad \text{and} \quad f_{2,2}(z) = cz^2 \dfrac{z-3}{1-3z}
    \end{align*}
    at the top and the bottom respectively. The critical values $b\approx -1.144208-0.964454i$ and $c \approx -0.755700-0.654917i$ are picked such that $f_{3,2}:\Hq \to \Hq$ is a $(3,2)$-critical quasicircle map on some quasicircle $\Hq$, $f_{2,2}: \T \to \T$ is a $(2,2)$-critical circle map, and both have the golden mean rotation number $\theta = \frac{\sqrt{5}-1}{2}$. Both $\Hq$ and $\T$ are colored red, and their preimages are colored green.
  \end{minipage}
    }
    \label{fig:cqc-comparison}
\end{figure}

For trivial reasons, we are exclusively concerned with the case when the rotation number $\theta$ of $f$ is irrational. Given an irrational number $\theta \in (0,1)$ with continued fraction expansion 
\[
\theta = [0;a_1,a_2, a_3,\ldots] := \dfrac{1}{a_1 + \frac{1}{a_2 + \frac{1}{a_3 + \ldots}}},
\]
we say that $\theta$ is \emph{of bounded type} if $a_n$'s are uniformly bounded above, \emph{pre-periodic} if there are positive integers $m$ and $p$ such that $a_n = a_{n+p}$ for all $n \geq m$, and \emph{periodic} if additionally $m=1$. We will denote corresponding spaces by $\IrratBdd$, $\IrratPre$ and $\IrratPer$ respectively.

The term ``critical quasicircle map`` first appeared in the work of Petersen \cite{Pe04} who proved Denjoy distortion estimates and demonstrated that $f$ is quasisymmetrically conjugate to irrational rotation if and only if $\theta$ is of bounded type. 
Petersen's result was a generalization of the Herman-\'{S}wi\k{a}tek theorem \cite{H87,S88} for critical circle maps and they all apply to the general multicritical case.
While our primary focus will be on the bounded type case, it is worth noting that these estimates imply that for any irrational $\theta$, $f$ is topologically conjugate to an irrational rotation.

We will be working with critical quasicircle maps $f: \Hq \to \Hq$ where $\Hq$ is a \emph{Herman curve}, that is, $\Hq$ is not contained in the closure of any rotation domain of $f$. In the bounded type regime, this is equivalent to the assumption that both $d_0$ and $d_\infty$ are at least two.

Given any pair of integers $d_0, d_\infty \geq 2$, the problem of realization of $(d_0,d_\infty)$-critical quasicircle maps was solved in our previous work \cite{Lim23a} by studying \emph{a priori bounds} and degeneration of Herman rings of the simplest configuration. In \cite{Lim23b}, we initiated the study of renormalizations $\Rcp^n f$ of $f$ and proved $C^{1+\alpha}$ rigidity for bounded type critical quasicircle maps. In short, $\Rcp^n f$ is defined as follows. 

Let $\{p_n/q_n\}_{n\geq 1}$ denote the best rational approximations of $\theta$. 
The $n$\textsuperscript{th} pre-renormalization of $f$ is the commuting pair $p\Rcp^n f= ( f^{q_n}|_{I_{n+1}}, f^{q_{n+1}}|_{I_n})$, where $I_n \subset \Hq$ is the shortest interval between the critical point $c$ of $f$ and $f^{q_n}(c)$. 
This is the first return map of $f$ back to the interval $I_{n+1} \cup I_n$. Then, the $n$\textsuperscript{th} renormalization $\Rcp^n f$ is the normalized commuting pair obtained by rescaling $p\Rcp^n f$ to unit size. 
A major consequence of $C^{1+\alpha}$ rigidity is that $\Rcp^n f$ converges exponentially fast to a unique horseshoe attractor.

Before we state our main result, let us review two extensively studied special cases, namely critical circle maps and Siegel maps.

\subsection{Critical circle maps}
\label{ss:critical-circle-maps}

A \emph{critical circle map} is a critical quasicircle map $f: \Hq \to \Hq$ where $\Hq$ is simply the unit circle $\T = \R/\Z$. By symmetry, the inner criticality $d_0$ must coincide with the outer criticality $d_\infty$.

The renormalization theory of critical circle maps serves to justify the golden mean universality phenomena empirically observed in smooth families of critical circle maps by Feigenbaum et al. \cite{FKS} and \"Ostlund et al. \cite{ORSS}. Historically, this is one of the two main examples of universality in one-dimensional dynamics, the other being the Feigenbaum-Coullet-Tresser universality observed in unimodal maps \cite{F1,F2,TC,CT}. In \cite{FKS, ORSS}, the golden mean universality was translated into a conjecture on the hyperbolicity of the renormalization operator on the space of critical commuting pairs. The conjecture was later generalized by various authors, in particular Lanford \cite{Lan88} who introduced renormalization horseshoes to account for more complicated universalities. Below, we provide a brief historical summary of the development of the theory.
    
In \cite{dF99}, de Faria introduced the notion of \emph{holomorphic commuting pairs} and proved the universality of scaling ratios and the existence of renormalization horseshoe for critical circle maps with bounded type rotation number.
Shortly after, de Faria and de Melo \cite{dFdM2} proved exponential convergence of renormalization to the horseshoe as well as $C^{1+\alpha}$ conjugacy between any two critical circle maps under the bounded type assumption. 
By studying the parabolic limits, Yampolsky extended the horseshoe to all irrational rotation numbers in \cite{Y01}.
Two years later, Yampolsky \cite{Y03a,Y03b} introduced the cylinder renormalization operator $\Rcyl$ which, unlike $\Rcp$, acts on the space of critical circle maps directly. In this framework, he proved the global uniform hyperbolicity of the renormalization horseshoe, bringing Lanford's program to completion.
% Khmelev and Yampolsky \cite{KY06} proved $C^{1+\alpha}$ rigidity at the critical point for arbitrary irrational rotation number by studying parabolic bifurcations. Moreover, Yampolsky extended the horseshoe for all irrational rotation numbers in \cite{Y01}, and brought Lanford's program to completion in \cite{Y02,Y03} using \emph{cylinder renormalization}. 
% Unlike $\Rcp$, the cylinder renormalization operator $\Rcyl$ acts on the space of critical circle maps directly.

\begin{theorem}[Hyperbolicity of the renormalization horseshoe \cite{Y01,Y03a,Y03b}]
    The cylinder renormalization operator $\Rcyl$ is a compact real-analytic operator on the space of critical circle maps on a small neighborhood of $\mathbb{T}$. 
    It admits a uniformly hyperbolic horseshoe attractor $\mathcal{A}$ on which $\Rcyl: \mathcal{A}\to \mathcal{A}$ is conjugated to the two-sided shift. 
    The horseshoe $\mathcal{A}$ has exactly one unstable direction and critical circle maps with irrational rotation number are contained in the stable sets of $\mathcal{A}$.
\end{theorem}

\subsection{Siegel maps}

What happens to a critical quasicircle map $f: \Hq \to \Hq$ if either $d_0$ or $d_\infty$ is one? In the bounded type regime, this is equivalent to the statement that $\Hq$ is the boundary of a rotation domain. By Douady-Ghys surgery \cite{D87,G84}, $\Hq$ can then be assumed to be the boundary of a \emph{Siegel disk}, i.e. a simply connected rotation domain centered at a neutral fixed point $\alpha$, and $f$ is called a \emph{Siegel map}. 

Stirnemann \cite{St94} first gave a computer-assisted proof of the existence of a renormalization fixed point with a golden mean Siegel disk. McMullen \cite{McM98} applied a measurable deep point argument to prove the existence of renormalization horseshoe for bounded type rotation number. In \cite[\S4]{AL22}, Avila and Lyubich established quasiconformal rigidity of bounded type Siegel maps; via McMullen's method, the regularity can be improved to $C^{1+\alpha}$. Gaidashev and Yampolsky \cite{Y08, GY22} gave a computer-assisted proof of the golden mean hyperbolicity of renormalization of Siegel disks using the formalism of \emph{almost commuting pairs}. 

In \cite{DLS}, Dudko, Lyubich, and Selinger constructed for every periodic irrational $\theta$ a variant of $\Rcp$ called \emph{pacman renormalization operator} $\Rpac$ that acts on a space of \emph{pacmen}, i.e. holomorphic maps $f: (U,\alpha) \to (V,\alpha)$ between two nested disks admitting a fixed point at $\alpha$ and a unique critical point on $U$, and satisfying certain branched covering-like conditions. They showed that $\Rpac$ admits a hyperbolic fixed point whose stable manifold has codimension one and consists of Siegel maps with rotation number $\theta$. One remarkable feature is that every pacman on the unstable manifold admits a global transcendental analytic extension. Ideas from transcendental dynamics were successfully adapted in \cite{DL23} to study the escaping dynamics of pacmen on the unstable manifold, which ultimately led to a progress in MLC and new examples of positive area Julia sets.

The idea that infinite anti-renormalizations induce a global transcendental extension was first observed and established in the setting of the classical Feigenbaum renormalization fixed point in the 1990's \cite{Ep92,McM96,Bu97}. Recently, T. Alland \cite{A23} adapted the framework developed in \cite{DL23} to study the detailed structure of the transcendental dynamics and fully describe the group of quasisymmetries of the Feigenbaum Julia set.

\subsection{The main result}
\label{ss:corona-renormalization}

In this paper, we will continue our study of renormalization of critical quasicircle maps and prove hyperbolicity of renormalization periodic points. Our approach will follow the ideas behind Pacman Renormalization Theory. We design a renormalization operator acting on the space of \emph{coronas}, a doubly-connected version of pacmen. 

A corona is a holomorphic map $f: U \to V$ between two nested annuli $U \Subset V$ such that $f: U\backslash \gamma_0 \to V \backslash \gamma_1$ is a unicritical branched covering map where $\gamma_0$ (resp. $\gamma_1$) is an arc connecting the two boundary components of $U$ (resp. $V$). The number of preimages of $\gamma_1$ on the boundary components of $U$ determines the inner and outer criticalities $d_0 \geq 2$ and $d_\infty\geq 2$ of $f$; the total degree of $f$ is equal to $d_0+d_\infty-1$. When the criticalities are specified, we call $f$ a $(d_0,d_\infty)$-critical corona. See Figure \ref{fig:corona} for reference. Within the space of unicritical holomorphic maps on an annulus, the space of maps admitting a corona structure is open.

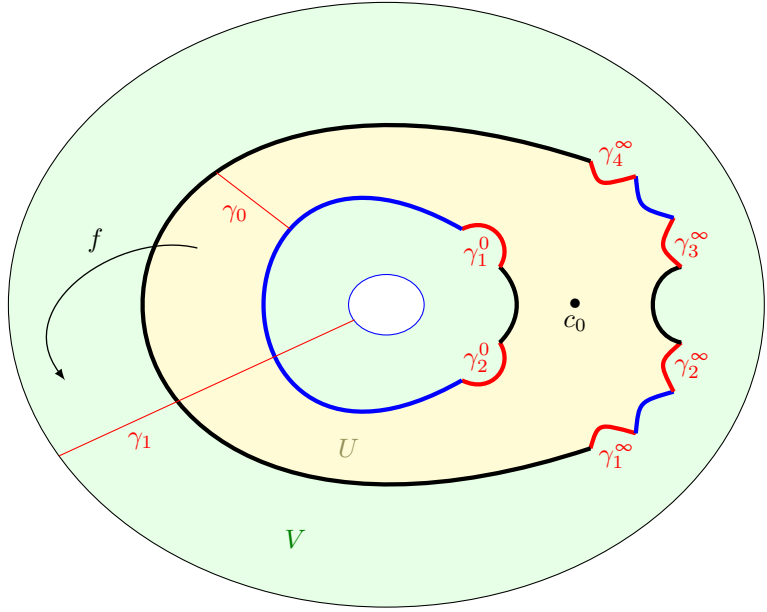
\begin{figure}
\begin{tikzpicture}[scale=1]
\coordinate (v4) at (1,1) {};
\coordinate (v3) at (1.5,0.5) {};
\coordinate (v2) at (1.5,-0.5) {};
\coordinate (v1) at (1,-1) {};
\coordinate (w1) at (2.7,-1.9) {};
\coordinate (w2) at (3.3,-1.7) {};
\coordinate (w3) at (3.8,-1.15) {};
\coordinate (w4) at (3.9,-0.5) {};
\coordinate (w5) at (3.9,0.5) {};
\coordinate (w6) at (3.8,1.15) {};
\coordinate (w7) at (3.3,1.7) {};
\coordinate (w8) at (2.7,1.9) {};

% fill
\draw[fill=green!10!white] (0,0) ellipse (5 and 4);
\draw[fill=yellow!20!white] (w1) .. controls (2.8,-1.56) .. (w2) .. controls (3.34,-1.28) .. (w3) .. controls (3.6,-0.76) .. (w4) .. controls (3.4,-0.4) and (3.4,0.4) .. (w5) .. controls (3.6,0.76) .. (w6) .. controls (3.34,1.28) .. (w7) .. controls (2.8,1.56) .. (w8) .. controls (-5.2,4.5) and (-5.2,-4.5) .. (w1); 
\draw[fill=green!10!white] (v1) .. controls (1.4,-1.2) and (1.7,-0.8) .. (v2) .. controls (1.8,-0.2) and (1.8,0.2) .. (v3) .. controls (1.7,0.8) and (1.4,1.2) .. (v4) .. controls (-2.5,3) and (-2.5,-3) .. (v1);
\draw[blue, fill=white] (0,0) ellipse (0.5 and 0.4);

% edges
\draw[ultra thick, red] (v1) .. controls (1.4,-1.2) and (1.7,-0.8) .. (v2); 
\draw[ultra thick, black] (v2) .. controls (1.8,-0.2) and (1.8,0.2) .. (v3); 
\draw[ultra thick, red] (v3) .. controls (1.7,0.8) and (1.4,1.2) .. (v4); 
\draw[ultra thick, blue] (v4) .. controls (-2.5,3) and (-2.5,-3) .. (v1);
\draw[ultra thick, red] (w1) .. controls (2.8,-1.56) .. (w2);
\draw[ultra thick, blue] (w2) .. controls (3.34,-1.28) .. (w3); 
\draw[ultra thick, red] (w3) .. controls (3.6,-0.76) .. (w4); 
\draw[ultra thick, black] (w4) .. controls (3.4,-0.4) and (3.4,0.4) .. (w5); 
\draw[ultra thick, red] (w5) .. controls (3.6,0.76) .. (w6); 
\draw[ultra thick, blue] (w6) .. controls (3.34,1.28) .. (w7); 
\draw[ultra thick, red] (w7) .. controls (2.8,1.56) .. (w8);
\draw[ultra thick, black] (w8) .. controls (-5.2,4.5) and (-5.2,-4.5) .. (w1); 

% labels
\node[yellow!50!black] at (-0.5,-1.9) {$U$};
\node[green!50!black] at (-1.2,-3.1) {$V$};
\node[red] at (1.2,-0.7) {$\gamma^0_2$};
\node[red] at (1.2,0.7) {$\gamma^0_1$};
\node[red] at (3.05,-2) {$\gamma^\infty_1$};
\node[red] at (4.05,-0.8) {$\gamma^\infty_2$};
\node[red] at (4.05,0.8) {$\gamma^\infty_3$};
\node[red] at (3.05,2) {$\gamma^\infty_4$};
\draw[red] (-1.27,1) -- (-2.25,1.75);
\draw[red] (-0.42,-0.2) -- (-4.33,-2);
\node[red] at (-2,1.2) {$\gamma_0$};
\node[red] at (-3.25,-1.8) {$\gamma_1$};
\node at (2.5,0) {$\bullet$};
\node at (2.5,-0.25) {$c_0$};
\node at (-3.85,0.85) {$f$};
\draw[line width=0.5pt,-latex] (-2.5,0.75) .. controls (-3.5,1) and (-5,0) .. (-4.25,-1);
\end{tikzpicture}

\caption{A $(d_0,d_\infty)$-critical corona $f: U \to V$ with $d_0=2$ and $d_\infty=3$. 
% The map $f: U \backslash \gamma_0 \to V \backslash \gamma_1$ is a degree $4$ covering map branched at a single critical point $c_0$.
There are $2d_0-2$ preimages of $\gamma_1$ on the inner boundary of $U$, and $2d_\infty-2$ preimages of $\gamma_1$ on the outer boundary of $U$.}
\label{fig:corona}
\end{figure}

We define corona renormalization operator as follows. 
First, consider the quadrilateral that is the connected component of $V$ minus $\gamma_1$ and $f(\gamma_1 \cap U)$ that contains the critical value of $f$.
The first return map onto this quadrilateral will be called a \emph{pre-corona}. Gluing opposite sides of this quadrilateral projects the pre-corona to a new corona on a new annulus, which is called the \emph{prime corona renormalization} $\Rprm f$ of $f$. In general, a corona renormalization operator is any iterate of the prime corona renormalization. The precise definition is supplied in \S\ref{ss:renormalization-operator}.

We say that a $(d_0,d_\infty)$-critical corona $f$ is \emph{rotational} with rotation number $\theta$ if it admits an invariant quasicircle $\Hq$ such that $f:\Hq \to \Hq$ defines a $(d_0,d_\infty)$-critical quasicircle map with rotation number $\theta$. The prime renormalization of a $(d_0,d_\infty)$-critical rotational corona is again a $(d_0,d_\infty)$-critical rotational corona, and the induced action on the rotation number is governed by
\[
r_{\textnormal{prm}}(\theta) = \begin{cases}
    \dfrac{\theta}{1-\theta}, & \text{ if } 0 \leq \theta < \cfrac{1}{2},\\
    2-\dfrac{1}{\theta}, & \text{ if } \cfrac{1}{2} \leq \theta < 1.
\end{cases}
\]
Every number in $\IrratPer$ is a periodic point of $r_{\textnormal{prm}}$. In general, every bounded type critical quasicircle map $f$ is corona renormalizable: every sufficiently high pre-renormalization $p\Rcp^n f$ can be projected under some gluing map to a rotational corona with the same criticality. 

Here is our main theorem.

\begin{thmx}[Hyperbolicity of renormalization]
\label{main-theorem}
    For any integers $d_0, d_\infty \geq 2$ and any $\theta \in \IrratPer$, there exists a corona renormalization operator 
    \[
    \Rstar: (\Ustar,f_*) \to (\Bstar,f_*)
    \]
    with the following properties.
    \begin{enumerate}[label=\textnormal{(\arabic*)}]
        \item\label{main-1} $\Ustar$ is an open subset of a Banach analytic manifold $\Bstar$ consisting of $(d_0,d_\infty)$-critical coronas.
        \item\label{main-2} $\Rstar$ is a compact analytic operator with a unique fixed point $f_*$ which is hyperbolic.
        \item\label{main-3} The local stable manifold $\mani^s_{\textnormal{loc}}$ of $f_*$ corresponds to the space of rotational coronas with rotation number $\theta$ in $\Ustar$.
        \item\label{main-4} The local unstable manifold $\unstloc$ is one-dimensional.
    \end{enumerate}
\end{thmx}

One immediate application of this theorem is the following. Given a critical quasicircle map $f: \Hq \to \Hq$, we define a \emph{Banach neighborhood} $\mathcal{N}$ of $f$ to be a Banach ball of the form $\mathcal{N}_U(f,\varepsilon)$ defined as follows.
Let $\varepsilon>0$ be a small number, and let $U$ be a skinny annular neighborhood of $\Hq$ such that $f$ is holomorphic on a neighborhood of $U$. 
Then, $\mathcal{N}_U(f,\varepsilon)$ is the space of unicritical holomorphic maps $g:U \to \C$ such that $g$ extends continuously to the boundary of $U$ and $\sup_{z \in U}|f(z)-g(z)| < \varepsilon$, equipped with the sup norm. 
It is natural to ask whether a generic perturbation of $f$ in a Banach neighborhood of $f$ still admits an invariant quasicircle.

\begin{corx}[Structure of conjugacy classes]
\label{cor:continuity-submanifold}
    Consider any sufficiently small Banach neighborhood $\mathcal{N}$ of a $(d_0,d_\infty)$-critical quasicircle map $f$ with pre-periodic rotation number $\theta$. The space $\mathcal{S}$ of maps in $\mathcal{N}$ that restrict to a $(d_0,d_\infty)$-critical quasicircle map with rotation number $\theta$ forms an analytic submanifold of $\mathcal{N}$ of codimension at most one. The invariant quasicircle of $g \in \mathcal{S}$ moves holomorphically in $g$.
\end{corx}

At this stage, we cannot yet conclude that the codimension of $\mathcal{S}$ in $\mathcal{N}$ is always exactly one.
This issue, together with other conjectures and further future applications, is discussed at the end of this paper.

\subsection{Passing to transcendental dynamics}

The most difficult step in the proof of Theorem \ref{main-theorem} is item \ref{main-4}, which will be accomplished via transcendental dynamics. 
Given a corona $f$ on the local unstable manifold $\unstloc$, its anti-renormalizations can be projected to a single dynamical plane and admit a maximal transcendental extension $\Fbold$ called a \emph{cascade} associated to $f$. 
A cascade can be described as a collection $\left\{\Fbold^P\right\}_{P \in \Tbold}$ of $\sigma$-proper maps parametrized by a dense semigroup $\Tbold \subset (\R_{\geq 0},+)$ such that $\Fbold^{P} \circ \Fbold^Q = \Fbold^{P+Q}$.
The maps $\Fbold^P$ share a common critical value, which will be normalized to be at $0$. 
The second half of this paper is dedicated to the study of the dynamics of $\Fbold$. 

A central conjecture in rational dynamics is the absence of invariant line fields, which implies the famous density of hyperbolicity conjecture \cite{McM94,McS98}. In our framework, we can define an invariant line field of a cascade $\Fbold$ to be a measurable Beltrami differential $\mu(z) \frac{d \bar{z}}{dz}$ on $\C$ such that $\left(\Fbold^P\right)^*\mu=\mu$ almost everywhere for all $P \in \Tbold$, $|\mu|=1$ on a positive measure set, and $\mu=0$ elsewhere. The existence of an invariant line field $\mu$ indicates the existence of a non-trivial deformation space for $\Fbold$ associated to the support of $\mu$. To justify item \ref{main-4}, we prove a rigidity theorem for cascades $\Fbold$. 

\begin{thmx}[Rigidity of escaping dynamics in $\unstloc$]
\label{main-theorem-rigidity}
    Consider a cascade $\mathbf{F}$ associated to a corona $f$ in $\unstloc$. The full escaping set
    \[
        \Esc(\Fbold) := \left\{ z \in \C \: : \: \text{ either } z \not\in \bigcap_{P} \Dom\left(\Fbold^P\right) \text{ or } \Fbold^P(z) \to \infty \text{ as } P \to \infty \right\}
    \]
    supports no invariant line field and moves conformally within the subspace $\{ f \in \unstloc \: | \: 0 \not\in \Esc(\Fbold)\}$.
    Furthermore, if $\Fbold$ has an attracting cycle, then the Julia set of $\Fbold$ supports no invariant line field.
\end{thmx}

Here, a \emph{conformal motion} of a set $E \subset \C$ refers to a holomorphic motion of $E$ that is conformal almost everywhere on $E$. (See Definition \ref{def:conformal-motion}.) One may compare this theorem to Rempe's result \cite{R09} on the rigidity of the escaping set of transcendental entire functions. Our methods also allow for an analog of Theorem \ref{main-theorem-rigidity} in other settings, such as pacman and period-doubling renormalization fixed points. Ultimately, 
\[
    \textnormal{Theorem \ref{main-theorem-rigidity}} \enspace \Longrightarrow \enspace \dime{\unstloc} \leq \text{number of critical orbits} \leq 1 \enspace \Longrightarrow \enspace \textnormal{Theorem \ref{main-theorem}\ref{main-4}}.
\]

Let us highlight three key differences between our case and the pacmen case. A more comprehensive summary can be found in \S\ref{ss:outline}.
    
Firstly, the existence of a non-attracting transverse direction for pacman renormalization is much more straightforward. Unlike coronas, every pacman is designed to admit a natural fixed point $\alpha$ associated to it. For a Siegel pacman, the $\alpha$-fixed point is the center of its Siegel disk. The multiplier $\lambda=f'(\alpha)$ of $\alpha$ induces a non-attracting eigenvalue at the pacman renormalization fixed point, whereas in the corona case, a rigidity result from \cite[Theorem B]{Lim23b} needs to be applied. 

Secondly, the proof of item \ref{main-4} for pacmen does not require such a rigidity theorem. After obtaining the transcendental structure, it is immediate that the multiplier $\lambda$ induces a natural foliation $\{W^u(\lambda)\}_\lambda$ of the unstable manifold of the pacman renormalization fixed point. By applying the $\lambda$-lemma along parabolic leaves, Dudko, Lyubich, and Selinger showed that
\[
    \dime{W^u\left(p/q\right)} \leq \textnormal{number of free critical orbits in } W^u\left(p/q\right) = 0
\]
where $p/q \in \Q$ is sufficiently close to $\theta \in \IrratPer$.
    
Thirdly, the original aim of the study of the finite-time escaping set associated to the transcendental extension of pre-pacmen in \cite{DL23} was to attain a puzzle structure, which was ultimately applied to understand the dynamics of maps on the unstable manifold and transfer the results to the quadratic family $\{z^2+c\}_c$. In our case, the full escaping set $\Esc(\Fbold)$ is of great interest because, together with the postcritical set, it is the measure-theoretic attractor of $\Fbold$ on the Julia set.

\subsection{Outline}
\label{ss:outline}

This paper is divided into five parts. Part 1 covers a few preliminary concepts. Section \ref{sec:sector-renorm} is a review of sector renormalization from \cite{DLS,DL23}, which is a toy model of the induced action of $\mathcal{R}$ on the invariant quasicircle. In Section \ref{sec:small-orbits}, we prove a generalization of Lyubich's Small Orbits Theorem \cite[\S2]{L99}, which will later be applied to deduce the absence of neutral eigenvalues for the renormalization operator of interest. (In the pacman case, the foliation induced by the multiplier of the $\alpha$-fixed point removes the need of such a generalization.)

A lot of Parts 2, 3, and 4 are inspired by the original work on pacman renormalization in \cite{DLS} and the detailed study of transcendental dynamics on the unstable manifold in \cite{DL23}. As previously mentioned, the main difference lies in the proof that the renormalization operator has exactly one unstable direction. Once we prove that our renormalization fixed point is hyperbolic, we treat the local unstable manifold $\unstloc$ as a holomorphic family of unicritical transcendental maps of unknown dimension. By adapting some ideas from \cite{R09}, we deduce the rigidity of escaping dynamics and claim that the deformation space of hyperbolic coronas on the unstable manifold must be supported on the Fatou set, the domain of stability. This implies that $\unstloc$ is one-dimensional.

Part 2 begins with the introduction of \emph{coronas} and \emph{pre-coronas}. We define the corona renormalization operator and show that for any renormalizable corona $f$, we can define a compact analytic operator $\renorm$ on a Banach neighborhood of $f$. In Section \ref{sec:rotational-coronas}, we analyze the structure of a rotational corona $f$. We prove that any critical quasicircle map can be renormalized to a rotational corona and, by applying \cite[Theorem F]{Lim23b}, we show that rotational coronas are quasiconformally rigid.

In Section \ref{sec:stab-mani}, we construct a compact analytic corona renormalization operator 
\[
\Rstar: (\Ustar, f_*) \to (\Bstar,f_*)
\]
whose fixed point $f_*$ is a rotational corona with periodic rotation number. In Theorem \ref{thm:weak-hyperbolicity}, we show that $\Rstar$ and $f_*$ satisfy items \ref{main-2} and \ref{main-3} in Theorem \ref{main-theorem}, and that the dimension of the local unstable manifold $\unstloc$ is finite and positive. The proof relies on a number of ingredients. 
\begin{enumerate}[label=(\roman*)]
    \item For any corona $[f: U \to V] \in \Ustar$ that is many times renormalizable, we construct a renormalization tiling on the domain $U$ which approximates the Herman quasicircle $\Hq_*$ of $f_*$ by lifting the domain of a high renormalization of $f$. This tiling is robust under perturbation, and we use them to show in Corollary \ref{cor:rotational-criterion} that any infinitely renormalizable rotational corona that stays close to $f_*$ must be a rotational corona. 
    \item By \cite[Theorem K]{Lim23b}, renormalizations $\Rstar^n f$ of a rotational corona near $f_*$ must converge exponentially fast to $f_*$.
    \item We apply Small Orbits Theorem \ref{thm:small-orbits}, which works in the presence of both attracting and repelling eigenvalues. 
\end{enumerate}
These three ingredients imply that $(D\Rstar)_{f_*}$ has no neutral eigenvalues. To show that a repelling direction exists, we apply \cite[Theorem B]{Lim23b}, a result on combinatorial rigidity of unicritical Herman quasicircles of a nice class of rational maps.

The rest of the paper is dedicated to proving that $(D\Rstar)_{f_*}$ has exactly one repelling eigenvalue. In Part 3, we begin to look at coronas $f$ on the local unstable manifold $\unstloc$. We show the maximal analytic extension of the pre-corona associated to $f$ is a commuting pair of $\sigma$-proper maps $\Fbold = (\fbold_\pm : \Xbold_\pm \to \C)$. This allows us to identify $\unstloc$ with $\manibold_{\textnormal{loc}}$, the holomorphic family of pairs $\Fbold$. Given $\Fbold = (\fbold_\pm) \in \manibold_{\textnormal{loc}}$ and $n \leq 0$, we consider the $\sigma$-proper commuting pair $\Fbold_n$ associated to $\Rstar^n f$ and denote by $\Fbold_n^\# = (\fbold_{n,\pm}^\#)$ the rescaled version of $\Fbold_n$ such that $\fbold_\pm$ are iterates of $\fbold_{n,\pm}^\#$. We identify $\Fbold$ as a cascade, that is, the semigroup $\left(\Fbold^{\geq 0}, \circ\right)$ generated by $\fbold_{n,\pm}^\#$ for all $n \leq 0$. It is isomorphic to a dense sub-semigroup $(\Tbold,+)$ of $\R_{\geq 0}$ and its elements can be written as $\Fbold^P: \Dom\left(\Fbold^P\right) \to \C$ for $P \in \Tbold$. 

In Part 4 of this paper, we study in depth the escaping dynamics of $\Fbold$ as well as the properties of the Fatou set $\Fcas(\Fbold)$ and the Julia set $\Jcas(\Fbold)$ of $\Fbold$. Most of Section \ref{sec:external} mirrors \cite[\S5]{DL23} (aside from a number of combinatorial differences), in which we study the structure of the finite-time escaping set of the cascade $\Fbold_*$ associated to the renormalization fixed point. We construct external rays and deduce its tree structure via their branch points, which are called \emph{alpha-points}. These rays induce dynamical wakes partitioning the whole dynamical plane.

In Section \ref{sec:escaping-sets}, we apply the external structure of $\Fbold_*$ to prove the first half of Theorem \ref{main-theorem-rigidity}. We first prove that for $\Fbold \in \manibold_{\textnormal{loc}}$, the \emph{finite-time escaping set} 
\[
\Esc_{<\infty}(\Fbold) := \bigcup_{P \in \T} \C \backslash \Dom\left(\Fbold^P\right)
\]
carries no invariant line field and locally moves holomorphically unless it contains a pre-critical point. We then observe that appropriately truncated wakes of $\Fbold_*$ survive under perturbation; these truncated wakes act as an analog of logarithmic tracts for transcendental functions. By adapting the ideas from \cite{R09}, we study the motion of points $z$ in the Julia set whose orbit $\{\Fbold^P(z)\}_{P\in \Tbold}$ remains close to $\infty$ for all $P$. We then show that the \emph{infinite-time escaping set}
    \[
    \Esc_\infty(\Fbold) := \{z \in \C \backslash \Esc_{<\infty}(\Fbold) \: : \: \Fbold^P(z) \to \infty \}
    \] 
also carries no invariant line field and locally moves holomorphically unless it contains a pre-critical point. 

In Section \ref{sec:fatou-julia}, we prove fundamental results on $\Fcas(\Fbold)$ and $\Jcas(\Fbold)$. This includes the density of periodic points, expansion and measure-theoretic properties, and the absence of no wandering domains --- all of which are classical for rational maps. Section \ref{sec:hyperbolic} is entirely dedicated to hyperbolic cascades. In it, we show that when $\Fbold$ is hyperbolic, the Julia set of $\Fbold$ is the union of the full escaping set 
\[
\Esc(\Fbold):= \Esc_{<\infty}(\Fbold) \cup \Esc_\infty(\Fbold)
\]
and a zero measure set of non-escaping points. We also prove the existence and rigidity of hyperbolic cascades in $\manibold_{\textnormal{loc}}$, including the second part of Theorem \ref{main-theorem-rigidity}. 

Part 5 wraps up the proof. Since $\Fbold$ is unicritical, we conclude that hyperbolic components in $\manibold_{\textnormal{loc}}$, and thus $\manibold_{\textnormal{loc}}$ itself, is one-dimensional. We then formulate three conjectures on potential applications of hyperbolicity and future directions.

%%%%%%%%%%%%%%%%%%%%%%%%%%%%%%%%%%%%%%%%%%%%%%%%%%%%%%%%%%%%%%%%%%%%%%%%%%%%%%%%%%%%%%%%%%%%%%%%%%%%%%%%%%%%%%%%%%%%%%%%%%%%%%%%%%%%%%%%%%%%%%%%%%%%%%%%%%%%%%%%%%%%%%%%%%%%%%%%%%%%

\part{Preliminaries}

\section{Sector renormalization}
\label{sec:sector-renorm}

In this section, we describe a toy model for our renormalization operator, which is sector renormalization acting on rigid rotations on the circle. We also describe cascades associated to sector renormalization fixed points, following \cite{DL23}.

\subsection{Renormalization of rotations and translations}
\label{ss:renorm-rotation}
Let us equip the unit circle $\T \subset \C$ with the induced intrinsic metric. Given two points $x$ and $y$ on $\T$, we denote by $[x,y] \subset \T$ the shortest closed interval with endpoints $x$ and $y$. Consider the rigid rotation 
\[
R_\theta : \T \to \T, z \mapsto e^{2\pi i \theta} z
\]
for some fixed $\theta \in \R/\Z$. Let us fix a point $x \in \T$ and consider 
\[
    X_-:=[R_\theta^{-1}(x),x], \quad Y:=[x,R_\theta(x)], \quad X_+ := \overline{\T \backslash (Y \cup X_-)}. 
\]
The first return map on $X_- \cup X_+$ is precisely the commuting pair
\[
    (R_\theta|_{X_+},\: R_\theta^2 |_{X_-}),
\]
Let us assume that $1 \not\in Y$ and denote by $\omega$ the length of $X_-\cup X_+$. Then, the power map $z \mapsto z^{1/\omega}$ sends $X_-\cup X_+$ onto $\T$ by gluing its endpoints, and it projects the commuting pair above to a new rigid rotation $R_{r_{\textnormal{prm}}(\theta)}$ on $\T$.
We will call $R_{r_{\textnormal{prm}}(\theta)}$ the \emph{prime renormalization} of $R_\theta$; it is independent of the initial choice of $x$.

\begin{lemma}[{\cite[Lemma A.1]{DLS}}]
    We have
    \[
    r_{\textnormal{prm}}(\theta) = \begin{cases}
        \dfrac{\theta}{1-\theta}, & \text{ if } 0 \leq \theta < \dfrac{1}{2},\\[0.1in]
        2-\dfrac{1}{\theta}, & \text{ if } \dfrac{1}{2} \leq \theta < 1.
    \end{cases}
    \]
\end{lemma}

In general, a \emph{sector renormalization} $\Rsec(R_\theta)$ of $R_\theta$ is a rotation $R_\mu$ defined as follows. First, consider a pair of intervals $X_-$ and $X_+$ on $\T$ satisfying $X_- \cap X_+ = \{1\}$. Suppose the first return map on $X := X_- \cup X_+$, which we call a sector pre-renormalization, is a pair of the form
    \begin{equation}
    \label{eqn:pair-01}
        (R_\theta^\abold|_{X_-}, \: R_\theta^\bbold|_{X_+}),
    \end{equation}
for some positive integers $\abold$ and $\bbold$ called the \emph{renormalization return times} of $\Rsec$. The map $z \mapsto z^{1/\omega}$, where $\omega$ is the length of $X$, glues the endpoints of $X$ together and projects the pair (\ref{eqn:pair-01}) to a new rotation $R_\mu = \Rsec(R_\theta)$.

\begin{example}
\label{eg:standard-renormalization}
    Recall that the Gauss map $G$ sends an irrational $x=[0;a_1,a_2,\ldots]$ to another irrational $G(x)=[0;a_2,a_3,\ldots]$. Often, we are interested in the $n$\textsuperscript{th} \emph{standard} renormalization operator sending $R_\theta$ to $R_{G^n\theta}$. This can be constructed as follows. Let $\{p_n/q_n\}_{n\geq 1}$ denote the best rational approximations of $\theta$. For $n \geq 1$, set 
    \[
    X_{n,-}=[R_\theta^{q_{n-1}+q_n}(1),1] \quad \text{and} \quad X_{n,+}=[1,R_\theta^{q_n}(1)],
    \]
    then $(R_\theta^{q_n}|_{X_{n,-}}, \: R_\theta^{q_{n-1}+q_n}|_{X_{n,+}})$ projects onto the rotation $R_{G^n(\theta)}$.
\end{example}

Let $\IrratPer$ denote the set of irrationals $\theta \in (0,1)$ of periodic type, that is, $\theta = G^l(\theta)$ for some $l\geq 1$.

\begin{lemma}[{\cite[Lemma A.2]{DLS}}]
\label{lem:prime-to-sector}
    Every sector renormalization is an iteration of the prime renormalization. In particular, $\mu = r_{\textnormal{prm}}^m(\theta)$ for some $m \geq 1$, and $R_\theta$ is a fixed point of some sector renormalization if and only if $\theta \in \IrratPer$.
\end{lemma}

Under the universal cover $\R \to \T, z \mapsto e^{-2\pi i z}$, the rotation $R_\theta$ can be lifted to the commuting pair of translations
\[
    T_{-\theta} : z \mapsto z-\theta, \qquad T_{1-\theta}: z \mapsto z+1-\theta.
\]
The deck transformation $\chi:= T_{1}$ is equal to $T_{1-\theta} \circ T_{-\theta}^{-1}$, and the original rotation $R_\theta$ can be recovered from the quotient map $T_{-\theta}/\langle \chi \rangle$.

Consider a general commuting pair of translations $(T_{-\ubold}, T_{\vbold})$ where $\ubold,\vbold \in \R_{\geq 0}$. The prime renormalization $\Rprm$ of $(T_{-\ubold}, T_{\vbold})$ is the new commuting pair $(T_{-\ubold_1}, T_{\vbold_1})$ where
\begin{equation}
\label{eqn:sector-renorm}
(T_{-\ubold_1}, T_{\vbold_1}):= \Bigg\{
\setlength\arraycolsep{0.5pt}\def\arraystretch{1.2}
\begin{array}{cccccc}
     (& T_{-\ubold} \circ T_{\vbold} & \text{, } & T_{\vbold} &) & \quad \text{ if } \ubold \geq \vbold,\\
     (& T_{-\ubold} & \text{, } & T_{-\ubold} \circ T_{\vbold}&) & \quad \text{ if }\ubold <\vbold.
\end{array}
\end{equation}
Set $\chi := T_{\vbold} \circ T_{-\ubold}^{-1}$ and $\chi_1 =  T_{\vbold_1} \circ T_{-\ubold_1}^{-1}$. The prime renormalization of pairs of translations is equivalent to that of rotations in the following sense.

\begin{lemma}
\label{lem:extra}
    If $T_{-\ubold}/\langle \chi \rangle \equiv R_\theta$, then 
    \[
    \theta = \frac{\vbold}{\ubold + \vbold}
    \quad \text{and} \quad
    T_{-\ubold_1}/\langle \chi_1 \rangle \equiv R_{r_{\textnormal{prm}}(\theta)}.
    \]
\end{lemma}

\subsection{Cascade of translations}
\label{ss:cascade}
By writing $(-\ubold, \vbold)$ as a column vector, the transformation in (\ref{eqn:sector-renorm}) is represented by either $I^- := \begin{pmatrix}
    1 & 1 \\ 0 & 1
\end{pmatrix}$ if $\ubold \geq \vbold$ or $I^+ := \begin{pmatrix}
    1 & 0 \\ 1 & 1
\end{pmatrix}$ if $\ubold < \vbold$. Consider the region $\R_{\leq 0} \times \R_{\geq 0}$, which is split equally into two sectors by the diagonal line $\{x+y=0\}$. The lower sector is mapped by $I^-$ onto $\R_{\leq 0} \times \R_{\geq 0}$, whereas the upper sector is mapped by $I^+$ onto $\R_{\leq 0} \times \R_{\geq 0}$.

From now on, suppose $\theta$ is of periodic type, so there exists some $m >0$ such that $r_{\textnormal{prm}}^m(\theta) = \theta$. Set $\ubold =\theta$ and $\vbold = 1-\theta$. By (\ref{eqn:sector-renorm}), there is a unique $2\times 2$ matrix $\Mbold$ of the form $I_1 I_2 \ldots I_m$, where $m \in \N$ is some integer and $I_i \in \{I^+, I^-\}$ for all $i$, such that the $m$\textsuperscript{th} prime renormalization of $(T_{-\ubold},T_{\vbold})$ is equal to
\[
    \Rprm^m(T_{-\ubold},T_{\vbold}) = (T_{-\ubold_1},T_{\vbold_1})
    \quad \text{ where } \quad
    \begin{pmatrix}
        -\ubold_1 \\ \vbold_1
    \end{pmatrix} = \Mbold \begin{pmatrix}
        -\ubold \\ \vbold
    \end{pmatrix}.
\]
The matrix $\Mbold$ is an element of the modular group $\text{SL}_2(\Z)$ mapping a sector in $\R_{\leq 0} \times \R_{\geq 0}$ onto $\R_{\leq 0} \times \R_{\geq 0}$. The condition $r^m_{\textnormal{prm}}(\theta) = \theta$ implies that $\begin{pmatrix}
        -\ubold_1 \\ \vbold_1
\end{pmatrix}$ is a scalar multiple of $\begin{pmatrix}
        -\ubold \\ \vbold
    \end{pmatrix}$. We conclude that $\Mbold$ has two eigenvalues $\tbold >1$ and $1/\tbold$, and that
\[
    \begin{pmatrix}
        -\ubold_1 \\ \vbold_1
    \end{pmatrix} = \frac{1}{\tbold} \begin{pmatrix}
        -\ubold \\ \vbold
    \end{pmatrix}.
\]
We call $\Mbold$ the \emph{anti-renormalization matrix} associated with $\theta$.

Observe that $\Mbold$ has to be a matrix of positive integers and its leading eigenvalue $\tbold$ must be irrational. For $n \in \N$, we write 
\[
    \ubold_n := \tbold^{-n} \ubold \quad \text{and} \quad
    \vbold_n := \tbold^{-n} \vbold.
\]
We then obtain a full pre-renormalization tower $\{ (T_{-\ubold_n}, T_{\vbold_n})\}_{n\in \Z}$ where 
\[
    \Rprm^m (T_{-\ubold_n}, T_{\vbold_n}) = (T_{-\ubold_{n+1}}, T_{\vbold_{n+1}}).
\]
Given $(n,a,b) \in \Z \times \Z_{\geq 0} \times \Z_{\geq 0}$, let us write
\[
    T^{(n,a,b)} := T_{-\ubold_n}^a \circ T_{\vbold_n}^b = T_{\tbold^{-n}(b\vbold - a \ubold)}.
\]

\begin{lemma}
\label{lem:free-action}
    Given a pair of elements $(n,a,b)$ and $(n',a',b')$ of $\Z \times \Z_{\geq 0} \times \Z_{\geq 0}$, 
    \[
    T^{(n,a,b)} = T^{(n',a',b')} \quad \text{ if and only if } \quad (a \ b)\Mbold^{n} = (c \ d)\Mbold^{n'}.
    \]
\end{lemma}

\begin{definition}
    We define the space $\Tbold$ of \emph{power-triples} to be the quotient of the semigroup $\Z\times \Z_{\geq 0} \times \Z_{\geq 0}$ under the equivalence relation $\sim$ where $(n,a,b) \sim (n-1,a',b')$ if and only if $(a' \ b') = (a \ b) \Mbold$.
\end{definition}

We will equip $\Tbold$ with the binary operation $+$ defined by
\[
    (n,a,b) + (n,a',b') = (n, a+a', b+b').
\]
With respect to $+$, $\Tbold$ has a unique identity element $0 := (n,0,0)$. Thus, $(\Tbold,+)$ still has the structure of a semigroup. According to Lemma \ref{lem:free-action}, $\Tbold$ acts freely on $\R$ as a cascade of translations $\left(T^P\right)_{P \in \Tbold}$.

\begin{lemma}[{\cite[Lemma 2.2]{DL23}}]
\label{lem:power-triple-embedding}
    There is an embedding $\iota: \Tbold \to \R_{\geq 0}$ such that $\iota(n-1,a,b) = \tbold^{-1} \iota(n,a,b)$. Identifying $\Tbold$ with $\iota(\Tbold) \subset \R_{\geq 0}$ equips $\Tbold$ with 
    \begin{enumerate}[label=\textnormal{(\arabic*)}]
        \item a linear order $\geq$, which can be described as follows: $P\geq Q$ if and only if for any sufficiently large negative integer $n$, we can write $P=(n,a,b)$ and $Q=(n,a',b')$ where $a \geq a'$ and $b \geq b'$;
        \item subtraction, that is, if $P,T \in \Tbold$ and $P \geq T$, then $P-T \in \Tbold$;
        \item scalar multiplication by $\tbold$: $P=(n,a,b) \mapsto \tbold P = (n+1,a,b)$, which is an automorphism of $\Tbold$.
    \end{enumerate}
    Moreover, for $P \in \Tbold$, $n \in \Z$, and $x \in \R$,
    \[
        T^P(x) = \tbold^n \cdot T^{\tbold^n P} \left(\tbold^{-n} x\right).
    \]
\end{lemma}

If $T^P$ is a translation by $l >0$, then $T^{\tbold^n P}$ is a translation by $\tbold^{-n} l$. The following observation is immediate.

\begin{lemma}[Proper discontinuity]
\label{lem:proper-discontinuity-cascades}
    If $P \in \Tbold_{>0}$ is small, then $|T^P(0)|$ is large.
\end{lemma}

For all $P \in \Tbold$, let us denote $b_P := T^{-P}(0)$. We say that $b_P$ is \emph{dominant} if every $b_Q$ on $[0,b_P]$ satisfies $Q\geq P$. By proper discontinuity, we can enumerate all dominant points $\left\{ b_{P_n} \right\}_{n \in \Z}$ such that $P_n < P_{n+1}$ for all $n$. The following lemma will be applied in the proof of Proposition \ref{prop:zero-chains}.

\begin{lemma}[{\cite[Lemma 2.4]{DL23}}]
    \label{lem:dominant-push}
    For every $i \in \Z$, there exist some $Q_i \in \Tbold_{>0}$ and some integers $m,n$ such that $n<m\leq i$ and $T^{Q_i}$ maps $[b_{P_i}, b_{P_{i+1}}]$ to $[b_{P_n}, b_{P_m}]$.
\end{lemma}

\section{Small Orbits Theorem}
\label{sec:small-orbits}

Consider a complex Banach space $\mathcal{B}$. Given a linear operator $L: \mathcal{B} \to \mathcal{B}$, denote the corresponding set of eigenvalues by $\spec{L}$. We say that an eigenvalue $\lambda \in \spec{L}$ is \emph{attracting} if $|\lambda|<1$, \emph{neutral} if $|\lambda|=1$, and \emph{repelling} if $|\lambda|>1$.

\begin{theorem}[Small Orbits Theorem]
\label{thm:small-orbits}
    Let $\renorm : (\mathcal{U},0) \to (\mathcal{B},0)$ be a compact analytic operator on a neighborhood $\mathcal{U}$ of $0$ in a complex Banach space $\mathcal{B}$. If the differential $D\renorm_0: \mathcal{B} \to \mathcal{B}$ has a neutral eigenvalue, then $\renorm$ has slow small orbits: for any neighborhood $\mathcal{V}$ of $0$, there is a forward orbit $\{\renorm^n g\}_{n \in \N}$ in $\mathcal{V}$ such that 
    \[
    \lim_{n\to \infty} \frac{1}{n}\log\|\renorm^n g\| = 0.
    \]
\end{theorem}

In the absence of repelling eigenvalues of $D\renorm_0$, the theorem above was proven by Lyubich in \cite[\S2]{L99}. The original Small Orbits Theorem was applied to rule out neutral eigenvalues in the proof of hyperbolicity of quadratic-like renormalization horseshoe \cite{L99,L02} and more recently the proof of hyperbolicity of pacman renormalization fixed points \cite{DLS}. In the context of corona renormalization (see \S\ref{ss:weak-hyperbolicity}), Lyubich's original formulation is not enough to rule out neutral eigenvalues and we have to resort to the upgraded version, which is Theorem \ref{thm:small-orbits}.

Below we will generalize Lyubich's proof. The key addition is the application of two invariant cones, namely the center-stable cone $\mathcal{C}^{cs}$ and the center-unstable cone $\mathcal{C}^{cu}$.

\begin{proof}
Let $\renorm$ be as in the hypothesis. Denote the unit disk in $\C$ by $\D:= \{z \in \C \: : \: |z|<1\}$. We present the Banach space $\mathcal{B}$ as a direct sum 
\[
\mathcal{B} = E^s \oplus E^c \oplus E^u,
\]
where subspaces $E^s, E^c, E^u$ are invariant under $D\renorm_0$ and
\[
\spec{D\renorm_0|_{E^s}} \subset \D, \quad \spec{D\renorm_0|_{E^c}} \subset \partial \D, \quad \spec{D\renorm_0|_{E^u}} \subset \C \backslash \overline{\D}.
\]
Note that the spectrum can only accumulate at $0$ because $\renorm$ is a compact operator. In particular, the subspace $E^c \oplus E^u$ must be finite-dimensional. We will assume that each of the three subspaces have positive dimension. (Else, we are reduced to \cite[\S 2]{L99}.)

For $h \in \mathcal{B}$, we will write $h=h^s +h^c+h^u$, where for $a \in \{s,c,u\}$, $h^a$ is the projection of $h$ onto the subspace $E^a$. We will also denote $h^{cs}:= h^c + h^s$ and $h^{cu}:= h^c + h^u$. There exist an adapted norm $\| \cdot \|$ on $\mathcal{B}$ and some positive constants $\mu_s,\mu_{cs},\mu_{cu},\mu_u$ such that $\mu_s<1<\mu_u$, $\mu_s<\mu_{cu}$, $\mu_{cs}<\mu_u$, and
\begin{align*}
\|D\renorm_0 h\|&\leq \mu_s \| h \| &&\text{ for all } h \in E^s, \\
\|D\renorm_0 h\|&\leq \mu_{cs} \| h \| && \text{ for all } h \in E^{cs}, \\
\|D\renorm_0 h\| &\geq \mu_{cu} \| h \| && \text{ for all } h \in E^{cu},\\
\|D\renorm_0 h\| &\geq \mu_u \| h \| && \text{ for all } h \in E^u.
\end{align*}

The proof below will involve two fixed constants $\alpha>1$ and $\delta>0$ where $\delta$ is small. We consider a pair of cone fields $C^{cu}$ and $C^{cs}$ given by
\begin{equation}
\label{eqn:cone-fields}
C^{cu}_f = \{ h \in T_f \mathcal{U} \: : \: \alpha \| h^s\| \leq \|h^{cu}\|\} \quad \text{and} \quad C^{cs}_f = \{ h \in T_f \mathcal{U} \: : \: \alpha\| h^u\| \leq \|h^{cs}\|\}
\end{equation}
for every $f \in \mathcal{U}$. For $a \in \{s,c,u\}$, we denote by $D^a = D^a(\delta)$ the open ball of radius $\delta$ centered at $0$ in $E^a$. Let
\[
\mathcal{D} := D^s \times D^c \times D^u
\]
the corresponding open polydisk centered at $0$ in $\mathcal{B}$. The boundary of $\mathcal{D}$ can be decomposed as follows:
\[
\partial^s \mathcal{D} := \partial D^s \times D^c \times D^u, \qquad \partial^c \mathcal{D} := D^s \times \partial D^c \times D^u, \qquad \partial^u \mathcal{D} := D^s \times D^c \times \partial D^u.
\]

\begin{claim1}
    For sufficiently small $\delta>0$, the following properties hold.
    \begin{enumerate}
        \item If $f \in \overline{\mathcal{D}}$, then $\renorm f \not\in \partial^s \mathcal{D}$;
        \item If $f \in \partial^u \mathcal{D}$, then $\renorm f \not\in \overline{\mathcal{D}}$;
        \item The cone field $C^{cu}$ is forward invariant: if $f, \renorm f \in \mathcal{D}$, then 
        \[
        D\renorm_f(C^{cu}_f) \subset C^{cu}_{\renorm f};
        \]
        \item The cone field $C^{cs}$ is backward invariant: if $f, \renorm f \in \mathcal{D}$, then 
        \[
        (D\renorm_f)^{-1}(C^{cs}_{\renorm f}) \subset C^{cs}_f.
        \]
    \end{enumerate}
\end{claim1}

\begin{proof}
    Fix a small constant $\varepsilon>0$. We can assume that $\delta$ is sufficiently small depending on $\varepsilon$ such that the difference
    \[
        Gf := \renorm f - D\renorm_0 f
    \]
    has $C^1$ norm on $\overline{\mathcal{D}}$ bounded by $\varepsilon$, that is, for all $f \in \overline{\mathcal{D}}$ and $h \in T_f \mathcal{U}$,
    \[
    \|G f\| \leq \varepsilon \|f\|, \quad \text{ and } \|DG_f h \| \leq \varepsilon \|h\|.
    \]
    
    When $f$ lies in $\overline{\mathcal{D}}$,
    \[
    \| (\renorm f)^s \| \leq \|D\renorm_0|_{E^s} (f^s)\| + \|(Gf)^s\| \leq \mu_s \| f^s \| + \varepsilon \| f \|.
    \]
    Assuming $\mu_s+3\varepsilon<1$, we then have $\| (\renorm f)^s\| < \delta$. Additionally, when $\| f^u \| = \delta$,
    \[
        \| (\renorm f)^u \| \geq \|D\renorm_0|_{E^u} (f^u)\| - \|(Gf)^u\| \geq \mu_u \delta - \varepsilon \|f\|.
    \]
    Assuming $\mu_u-3\varepsilon>1$, we then have $\| (\renorm f)^u \| > \delta$. Hence, (1) and (2) hold.
    
    Suppose both $f$ and $\renorm f$ are in $\mathcal{D}$. For every $h \in C^{cu}_f$, we have
    \begin{align*}
        \|(D\renorm_f h)^{cu}\| &= \| D\renorm_0|_{E^c\oplus E^u}(h^{cu}) + (DG_f(h))^{cu} \| \\
        &\geq \mu_{cu} \| h^{cu}\| - \varepsilon \|h\| \\
        &\geq \left( \mu_{cu} -\varepsilon\left(1+ \frac{1}{\alpha}\right)\right) \|h^{cu}\|,
    \end{align*}
    and
    \begin{align*}
        \alpha \|(D\renorm_f h)^s\| &= \alpha \| D\renorm_0|_{E^s}(h^s) + (DG_f(h))^s\| \\
        &\leq \alpha \left( \mu_s \|h^s\| + \varepsilon \|h\| \right) \\
        &\leq \left(\mu_s+ (\alpha + 1)\varepsilon\right) \|h^{cu}\|.
    \end{align*}
    We can take $\varepsilon$ to be small enough such that $\alpha \|(D\renorm_f h)^s\| \leq \|(D\renorm_f h)^{cu}\|$ and thus $D\renorm_f h \in C^{cu}_{\renorm f}$. The proof that the cone field $C^{cs}$ is backward invariant works in a similar way.
\end{proof}

Let us consider the perturbation $\renorm_\lambda := \lambda\cdot \renorm$ for $0<\lambda<1$. When $\lambda$ is sufficiently close to $1$, $\renorm_\lambda$ still satisfies all the properties listed in Claim 1. The following claim is a consequence of Lemma \ref{lem:main-for-small-orbits}, which we will elaborate later separately.

\begin{claim2}
    For any $0<\lambda<1$ sufficiently close to $1$ and for any sufficiently small $\delta>0$, there exists some point $f_\lambda \in \partial^c \mathcal{D}$ such that the orbit $\{\renorm^n_\lambda f_\lambda\}_{n\in\N}$ lies entirely inside of $\overline{\mathcal{D}}$ and $\renorm^n_\lambda f_\lambda \to 0$. 
\end{claim2}

Since $\renorm$ is compact, there exist an increasing sequence $\{\lambda_n\}_{n\in\N}$ of positive numbers and some $g \in \overline{\mathcal{D}}$ such that as $n\to \infty$, $\lambda_n \to 1$ and $\renorm_{\lambda_n} f_{\lambda_n} \to g$. Clearly, for all $n \in \N$, the $n^{\text{th}}$ iterate $g_n:= \renorm^n g$ lies in $\overline{\mathcal{D}}$.

Consider the cone $\hat{C}^{cu} = \{ f \in \mathcal{B} \: : \: \|f^s\| \leq \|f^{cu}\|\}$.
Similar to the proof of Claim 1 (3), $\hat{C}^{cu}$ is forward invariant under $\renorm_\lambda$ in the sense that if both $f$ and $\renorm_\lambda f$ are in $\overline{\mathcal{D}}$, then $f \in \hat{C}^{cu}$ implies $\renorm_{\lambda} f \in \hat{C}^{cu}$. 
Since $f_\lambda$ is in $ \partial^c \mathcal{D}$, then $f_\lambda $ is also inside of $\hat{C}^{cu}$. 
Hence, each $g_n$ is also in $\hat{C}^{cu}$. This implies that for every $n \in \N$,
\begin{equation}
    \label{ineq:c,u-estimate}
    g_{n+1}^{cu} = D\renorm_0|_{E^c \oplus E^u} (g_n^{cu}) + O\left(\|g_n^{cu}\|^2\right). 
\end{equation}

At last, we will show that the orbit of $g$ is a slow small orbit. Indeed, suppose for a contradiction that 
\begin{equation}
\label{ineq:last}
\liminf_{n\to\infty} \frac{1}{n} \log \|g_n\| < -c_0
\end{equation}
for some constant $c_0 >0$. Note that this property holds for every norm that is equivalent to $\| \cdot \|$. Pick some $c_1 \in [0,c_0)$. There exists an adapted norm $\| \cdot \|$ equivalent to the original one such that the operator norm of $D\renorm_0|_{E^c \oplus E^u}^{-1}$ is at most $e^{c_1}$. By (\ref{ineq:c,u-estimate}), for sufficiently small $\delta>0$, there is some $c_2 \in (0,c_0)$ such that 
\[
    \|g^{cu}_{n+1}\| \geq e^{-c_2} \|g^{cu}_n\| \qquad \text{for all } n \in \N.
\]
This contradicts (\ref{ineq:last}).
\end{proof}

It remains to prove Claim 2, which will follow directly from the lemma below. Again, we suppose $\mathcal{B}$ can be decomposed into $E^s \oplus E^c \oplus E^u$ and consider the cone fields $C^{cu}$ and $C^{cs}$ defined in (\ref{eqn:cone-fields}). We consider a small neighborhood $\mathcal{U} \subset \mathcal{B}$ of some polydisk $\mathcal{D}$ centered at $0$. For any $r>0$, we denote the open disk $\{ z \in \C \: : \: |z| < r \}$ by $\D_r$.

\begin{lemma}
\label{lem:main-for-small-orbits}
     Let $\renorm : (\mathcal{U},0) \to (\mathcal{B},0)$ be a compact analytic operator such that the differential $D\renorm_0$ preserves the decomposition $\mathcal{B} = E^s \oplus E^c \oplus E^u$ and satisfies the following properties.
    \begin{enumerate}[label=\textnormal{(\arabic*)}]
        \item \label{H1} Hyperbolicity: There exists some $0<r<1$ such that
        \[
        \spec{D\renorm_0|_{E^s}} \subset \D_r, \quad  \spec{D\renorm_0|_{E^c}} \subset \D \backslash \D_r, \quad \spec{D\renorm_0|_{E^u}} \subset \C \backslash \overline{\D}.
        \]
        \item \label{H2} Boundary behaviour: If $f \in \overline{\mathcal{D}}$, then $\renorm f \not\in \partial^s \mathcal{D}$. If $f \in \partial^u \mathcal{D}$, then $\renorm f \not\in \overline{\mathcal{D}}$.
        \item \label{H3} Invariant cone fields: Whenever $f,\renorm f \in \overline{\mathcal{D}}$,
        \[
        D\renorm_f(C^{cu}_f) \subset C^{cu}_{\renorm f}, \qquad (D\renorm_f)^{-1}(C^{cs}_{\renorm f}) \subset C^{cs}_f.
        \]
    \end{enumerate}
    Then, there exists some $f \in \partial^c \mathcal{D}$ such that $\{ \renorm^n f \}_{n \in \N} \subset \overline{\mathcal{D}}$ and $\|\renorm^n f\| \to 0$ as $n \to \infty$.
\end{lemma}

\begin{proof}
    By the compactness of $\renorm$, the subspace $E^c \oplus E^u$ has a finite dimension. Let $d_c:= \dime{E^c}$ and $d_u := \dime{E^u}$. By \ref{H1}, the stable manifold
    \[
    \mathcal{A} = \{ f \in \overline{\mathcal{D}} \: : \: \{ \renorm^n f \}_{n \in \N} \subset \overline{\mathcal{D}} \text{ and } \|\renorm^n f \| \to 0 \}
    \]
    exists and is a forward invariant analytic submanifold of codimension $d_u$. 
    
    Let us assume for a contradiction that $\mathcal{A}$ is disjoint from $\partial^c \mathcal{D}$.

    \begin{claim1}
        The set $\mathcal{A}^o:= \mathcal{A} \cap \mathcal{D}$ is a forward invariant open submanifold of $\mathcal{A}$.
    \end{claim1}

    \begin{proof}
        The only non-trivial property to prove here is forward invariance. Suppose $f \in \mathcal{A}^o$. As $f \in \mathcal{A}$, then $\renorm^n f \in \overline{\mathcal{D}}$ for all $n \geq 1$. By \ref{H2}, $\renorm f$ cannot lie in $\partial^s \mathcal{D} \cup \partial^u \mathcal{D}$. By the assumption, $\renorm f$ cannot lie in $\partial^c \mathcal{D}$ either. Thus, $\renorm f \in \mathcal{D}$.
    \end{proof}

    \begin{claim2}
        The set $\partial^c \mathcal{A} := \overline{\mathcal{A}} \backslash (\mathcal{A}^o \cup \partial^s \mathcal{D})$ is also forward invariant.
    \end{claim2}

    \begin{proof}
        Suppose for a contradiction that there is some $f \in \partial^c \mathcal{A}$ such that $\renorm f \in \mathcal{A}^o \cup \partial^s \mathcal{D}$. By \ref{H2}, $\renorm f$ must lie in $\mathcal{A}^o$, which implies that $f \in \mathcal{A} \cap (\partial^c \mathcal{D} \cup \partial^u \mathcal{D})$. However, this is impossible because $f$ does not lie in $\partial^c \mathcal{D}$ by our main assumption, nor in $\partial^u \mathcal{D}$ due to \ref{H2}. 
    \end{proof}

    \begin{claim3}
        The tangent space $T_f \mathcal{A}^o$ at every point $f$ in $\mathcal{A}^o$ is contained in $C^{cs}_f$.
    \end{claim3}

    \begin{proof}
        Let $f \in \mathcal{A}^o$. As $\mathcal{A}^o$ is tangent to the subspace $E^s \oplus E^c$ at $0$, for all sufficiently high $n$, $\renorm^n f$ is sufficiently close to $0$ and so the tangent space $T_{\renorm^nf}\mathcal{A}^o$ lies within $C^{cs}_{\renorm^nf}$. By backward invariance of $C^{cs}$ in \ref{H3}, $T_f \mathcal{A}^o$ also lies within $C^{cs}_f$.
    \end{proof}

    Let us consider the family $\mathcal{G}$ of all immersed analytic $d_c$-dimensional submanifolds $\Gamma$ of $\mathcal{A}^o$ containing $0$ with the following properties.
    \begin{enumerate}[label=(\alph*), start=1]
        \item \label{G1} The tangent space $T_f\Gamma$ at every point $f \in \Gamma$ lies in the cone $C_f^{cu}$;
        \item \label{G2} The accumulation set $\overline{\Gamma} \backslash \Gamma$ lies in $\partial^c \mathcal{A}$.
    \end{enumerate}
    Note that $\mathcal{G}$ is non-empty: it contains $\mathcal{A}^o \cap (E^c \oplus E^u)$ because, by Claim 3, the intersection between $\mathcal{A}^o$ and the subspace $E^c \oplus E^u$ is transversal. Another consequence of Claim 3 is the following claim.

    \begin{claim4}
        For every $\Gamma \in \mathcal{G}$ and $h \in T_f\Gamma$, $\| h^c \| \asymp \| h \|$. In particular, the projection $P: \Gamma \to D^c$ is non-singular.
    \end{claim4}

    \begin{proof}
        Let $h \in T_f\Gamma$. By Property \ref{G1} and Claim 3, $\alpha \|h^s\| \leq \| h^{cu} \|$ and $\alpha \|h^u\| \leq \| h^{cs}\|$. By triangle inequality, these imply that $(\alpha -1 ) \max\{ \| h^s \|, \| h^u\| \} \leq \| h^c \|$ and consequently $\| h^c\| \leq \|h\| \leq \frac{\alpha+1}{\alpha-1}\|h^c\|$.
    \end{proof}

    Recall that the Kobayashi norm of a tangent vector $v \in T_f \Gamma$ at a point $f$ on a complex manifold $\Gamma$ is defined as
    \[
    \| h \|_\Gamma := \inf \left\{ \|w\|_\D \: : \: D\phi_f(w)=h \text{ for some holomorphic map } \phi:(\D,0) \to (\Gamma,f) \right\}
    \]
    where $\|w \|_\D$ denotes the Poincar\'e metric of $w \in T_0\D$ on the unit disk $\D$. We will supply every $\Gamma \in \mathcal{G}$ with the Kobayashi metric.
    
    \begin{claim5}
        There is some $K >0$ such that for every $\Gamma \in \mathcal{G}$ and $h \in T_0\Gamma$, $\|h\|_\Gamma \leq K \|h\|$.
    \end{claim5}

    \begin{proof}
        By Claim 4, there is some $\delta>0$ such that for every $\Gamma \in \mathcal{G}$, the component $\Gamma(\delta)$ of $\Gamma \cap D(\delta)$ containing $0$ is a graph of an analytic map $D^c(\delta) \to D^s \times D^u$. Therefore, for any $h \in T_0\Gamma$,
        \[
            \| h \|_\Gamma \leq \|h\|_{\Gamma(\delta)} = \| h^c\|_{D^c(\delta)}.
        \]
        Clearly, $\| h^c \|_{D^c(\delta)} \asymp \| h^c \|$ (with bounds depending only on $\delta$). By Claim 4, this yields the desired inequality $\|h\|_\Gamma \leq K \| h \|$ for some $K$ independent of $\Gamma$.
    \end{proof}
    
    By Property \ref{H3} and Claim 2, the map $\renorm$ induces a well-defined graph transform
    \[
        \renorm* : \mathcal{G} \to \mathcal{G}, \quad \Gamma \mapsto \renorm\Gamma.
    \]
    Note that $\renorm: \Gamma \to \renorm\Gamma$ is a proper non-singular map, hence a holomorphic covering map. Therefore, for every $\Gamma \in \mathcal{G}$, $n \in \N$, and non-zero tangent vector $h \in T_0 \Gamma$,
    \[
        \|h\|_\Gamma = \| (D\renorm^n)_0 (h)\|_{(\renorm*)^n(\Gamma)}.
    \]
    By Claim 5,
    \[
        \| h \|_\Gamma \leq K \| (D\renorm^n)_0 (h) \|.
    \]
    However, by \ref{H1}, $\| (D\renorm^n)_0 (h) \|$ tends to $0$ as $n\to \infty$. This yields a contradiction.
\end{proof}

\begin{remark}
    By considering the perturbation $\renorm_\lambda$ for $\lambda>1$ and proving a backward version of Lemma \ref{lem:main-for-small-orbits}, one can also show that $\renorm$ has backward small orbits, that is, a sequence $\{g_{-n}\}_{n \in \N}$ that is all contained in a neighborhood of $0$ such that $\renorm g_{-n-1} = g_{-n}$ for all $n \in \N$ and 
    \[
    \lim_{n\to \infty} \frac{1}{n} \log \| g_{-n} \| = 0.
    \]
    It is unlikely that this method can be easily modified to produce bi-infinite small orbits, an analog of P\'{e}rez-Marco's hedgehogs \cite{PM97}. In several complex variables, some progress has been achieved by \cite{FLRT,LRT} when there is only one neutral eigenvalue. 
\end{remark}

%%%%%%%%%%%%%%%%%%%%%%%%%%%%%%%%%%%%%%%%%%%%%%%%%%%%%%%%%%%%%%%%%%%%%%%%%%%%%%%%%%%%%%%%%%%%%%%%%%%%%%%%%%%%%%%%%%%%%%%%%%%%%%%%%%%%%%%%%%%%%%%%%%%%%%%%%%%%%%%%%%%%%%%%%%%%%%%%%%%%

%%%%%%%%%%%%%%%%%%%%%%%%%%%%%%%%%%%%%%%%%%%%%%%%%%%%%%%%%%%%%%%%%%%%%%%%%%%%%%%%%%%%%%%%%%%%%%%%%%%%%%%%%%%%%%%%%%%%%%%%%%%%%%%%%%%%%%%%%%%%%%%%%%%%%%%%%%%%%%%%%%%%%%%%%%%%%%%%%%%%

\part{Coronas}

\section{Corona renormalization operator}
\label{sec:corona-renorm}

From now on, we fix a pair of positive integers $d_0,d_\infty \geq 2$ and set $d:=d_0+d_\infty-1$. 

\subsection{$(d_0,d_\infty)$-critical coronas}

For any open annulus $A$ compactly contained in $\C$, we label the boundary components of $A$ by $\partial^0 A$ and $\partial^\infty A$, and make the convention that $\partial^\infty A$ is the outer boundary, i.e. the one that is closer to $\infty$. We also say that another annulus $A'$ is \emph{essentially} contained in $A$ if $A'$ is a deformation retract of $A$.

\begin{definition}
\label{def:corona}
    A \emph{$(d_0,d_\infty)$-critical corona}\footnote{The shape of the domain $U$ in Figure \ref{fig:corona} resembles a crown or a wreath, which is what \emph{corona} means in Latin.} is a map $f: U \to V$ between two bounded open annuli in $\C$ with the following properties.
    \begin{enumerate}
        \item The boundary components of both $U$ and $V$ are Jordan curves, and $U$ is compactly and essentially contained in $V$.
        \item There is a proper arc $\gamma_1 \subset V$ connecting $\partial^{0} V$ and $\partial^{\infty} V$ such that the preimage $f^{-1}(\gamma_1)$ is disjoint from $\gamma_1$ and is a union of $2d-1$ pairwise disjoint arcs 
        \[
        \gamma_0 \subset U, \quad \gamma^0_1, \ldots, \gamma^0_{2(d_0-1)} \subset \partial^{0} U, \quad \gamma^\infty_1, \ldots, \gamma^\infty_{2(d_\infty-1)} \subset \partial^{\infty} U.
        \]
        \item $f: U \to V$ is holomorphic and $f: U \backslash \gamma_0 \to V \backslash \gamma_1$ is a degree $d$ branched covering map.
        \item $f$ has a unique critical point $c_0$.
    \end{enumerate}
    The arc $\gamma_1$ is called the \emph{critical arc} of $f$. See Figure \ref{fig:corona} for an illustration.
\end{definition}

Let $f: U \to V$ be a $(d_0,d_\infty)$-critical corona. For any $\bullet \in \{0,\infty\}$, we divide the boundary component $\partial^\bullet U$ into
\[
    \partial^{\bullet}_\lgt U := \partial^\bullet U \cap f^{-1}(\partial^\bullet V) \quad \text{ and } \quad \partial^{\bullet}_\frb U := \partial^\bullet U \backslash f^{-1}(\partial^\bullet V)
\]
according to whether or not it is mapped to the same side the annulus. Each of the above consists of $d_\bullet-1$ components. Set
\[
    \partial_{\lgt} U := \partial^{0}_\lgt U \cup \partial^{\infty}_\lgt U \quad \text{ and } \quad \partial_\frb U := \partial^{0}_\frb U \cup \partial^{\infty}_\frb U.
\]
We call $\partial_{\lgt} U$ the \emph{legitimate boundary} of $U$ and $\partial_\frb U$ the \emph{forbidden boundary} of $U$.

\subsection{Corona renormalization}
\label{ss:renormalization-operator}

\begin{definition}
\label{def:pre-corona}
    A \emph{$(d_0,d_\infty)$-critical pre-corona} is a pair of holomorphic maps 
    \[
    F=(f_- : U_- \to S, \: f_+: U_+\to S)
    \]
    satisfying the following properties.
    \begin{enumerate}
        \item $S$ is a topological rectangle with vertical sides $\beta_-$ and $\beta_+$.
        \item $\beta_0$ is a vertical arc in $S$ dividing $S$ into subrectangles $T_-$ and $T_+$, where $\beta_\pm \subset \partial T_\pm$ and $U_\pm$ is a subrectangle of $T_\pm$ with vertical sides contained in $\beta_\pm$ and $\beta_0$.
        \item There is a gluing map $\psi: \overline{S} \to \overline{V}$ such that $\psi(\beta_-)=\psi(\beta_+)$, $\psi$ is conformal on a neighborhood of $S$ and injective on $S \backslash (\beta_-\cup\beta_+)$, and $\psi$ projects $F$ into a $(d_0,d_\infty)$-critical corona with critical arc $\psi(\beta_\pm)$.
    \end{enumerate}
    The gluing map $\psi$ will also be called the \emph{renormalization change of variables} of $F$. It glues together $f_+(x) \in \beta_-$ and $f_-(x) \in \beta_+$ for every $x$ in $\beta_0 \cap \partial U_{\pm}$. See Figure \ref{fig:pre-corona}.
\end{definition}

\begin{figure}
\begin{tikzpicture}
\draw[white, fill=green!10!white] (-5,-3) -- (5,-3) -- (5,3) -- (-5,3) -- (-5,-3);
\draw[white, fill=yellow!15!white] (-5,-1.5) -- (-2,-1.5) -- (-2,1.5) -- (-5,1.5) -- (-5,-1.5);
\draw[white, fill=yellow!15!white] (-2,-1.5) .. controls (-0.5,-1.5) .. (0,-1.8) .. controls (0.15,-1.7) .. (0.3,-1.9) .. controls (0.5,-1.8) .. (0.7,-1.9) .. controls (0.85,-1.7) .. (1,-1.8) .. controls (1.3,-1.5) and (1.7,-1.5) .. (2,-1.8) .. controls (2.15,-1.7) .. (2.3,-1.9) .. controls (2.5,-1.8) .. (2.7,-1.9) .. controls (2.85,-1.7) .. (3,-1.8) .. controls (3.5,-1.5) .. (5,-1.5) -- (5,1.5) .. controls (3,1.5) .. (2.5,1.8) .. controls (2.3,1.7) .. (2,2) .. controls (1.5,1.7) .. (1,2) .. controls (0.7,1.7) .. (0.5,1.8) .. controls (0,1.5) .. (-2,1.5);

% edges
\draw[ultra thick] (-5,1.5) -- (-2,1.5);
\draw[ultra thick] (-5,-1.5) -- (-2,-1.5);
\draw[ultra thick, black] (-2,1.5) .. controls (0,1.5) .. (0.5,1.8);
\draw[ultra thick, red] (0.5,1.8) .. controls (0.7,1.7) .. (1,2); 
\draw[ultra thick, blue] (1,2) .. controls (1.5,1.7) .. (2,2); 
\draw[ultra thick, red] (2,2) .. controls (2.3,1.7) .. (2.5,1.8); 
\draw[ultra thick, black] (2.5,1.8) .. controls (3,1.5) .. (5,1.5);
\draw[ultra thick, black] (-2,-1.5) .. controls (-0.5,-1.5) .. (0,-1.8);
\draw[ultra thick, red] (0,-1.8) .. controls (0.15,-1.7) .. (0.3,-1.9); 
\draw[ultra thick, blue] (0.3,-1.9) .. controls (0.5,-1.8) .. (0.7,-1.9); 
\draw[ultra thick, red] (0.7,-1.9) .. controls (0.85,-1.7) .. (1,-1.8); 
\draw[ultra thick, black] (1,-1.8) .. controls (1.3,-1.5) and (1.7,-1.5) .. (2,-1.8);
\draw[ultra thick, red] (2,-1.8) .. controls (2.15,-1.7) .. (2.3,-1.9); 
\draw[ultra thick, blue] (2.3,-1.9) .. controls (2.5,-1.8) .. (2.7,-1.9); 
\draw[ultra thick, red] (2.7,-1.9) .. controls (2.85,-1.7) .. (3,-1.8); 
\draw[ultra thick, black] (3,-1.8) .. controls (3.5,-1.5) .. (5,-1.5);

% labels
\node[green!50!black] at (-3,2.2) {$T_-$};
\node[green!50!black] at (-0.5,2.2) {$T_+$};
\node[yellow!50!black] at (-3.5,0) {$U_-$};
\node[yellow!50!black] at (1.5,0) {$U_+$};
\node at (4.5,-3.25) {$S$};
\draw (-5,3) -- (5,3);
\draw (-5,-3) -- (5,-3);
\draw[red] (-2,3) -- (-2,-3);
\draw[red] (-5,3) -- (-5,-3);
\draw[red] (5,3) -- (5,-3);
\node[red] at (-2.25,0) {$\beta_0$};
\node[red] at (5.25,0) {$\beta_+$};
\node[red] at (-5.25,0) {$\beta_-$};
\node at (1,-3.6) {$f_-$};
\draw[line width=0.5pt,-latex] (-3,-1) .. controls (-1,-3) and (1,-4) .. (3,-2.5);
\node at (-2.3,-3.6) {$f_+$};
\draw[line width=0.5pt,-latex] (0,-1) .. controls (-1,-3) and (-2.5,-4) .. (-3.6,-2.5);
\end{tikzpicture}

\caption{A (2,3)-critical pre-corona. It projects to the corona in Figure \ref{fig:corona} after gluing $\beta_+$ and $\beta_-$ }
\label{fig:pre-corona}
\end{figure}
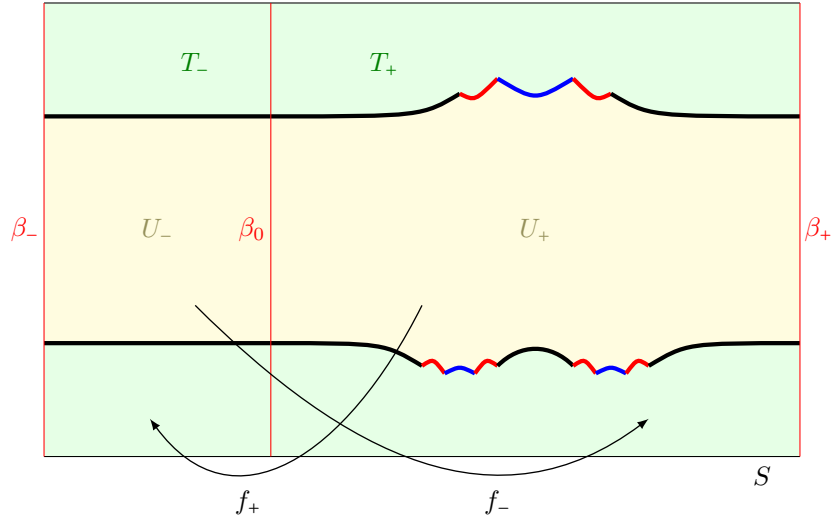
    
\begin{definition}
\label{def:renormalization}
    A corona $f: U \to V$ is \emph{renormalizable} if there exists a pre-corona
    \[
    F = (f^{k_-} : U_- \to S, \: f^{k_+} : U_+ \to S)
    \]
    on a rectangle $S \subset V$ such that $f^{k_-}$ and $f^{k_+}$ are the first return maps back to $S$ and
    \[
    \mathbf{\Delta}_F = \bigcup_{i=0}^{k_--1} \overline{f^i(U_-)} \cup \bigcup_{j=0}^{k_+-1} \overline{f^j(U_+)}
    \]
    is a closed annulus essentially contained in $U$. We call $F$ the \emph{pre-renormalization} of $f$, $k_-$ and $k_+$ the \emph{return times} of $F$, and $\mathbf{\Delta}_F$ the \emph{renormalization tiling} of $F$. The corona obtained by projecting $F$ under its gluing map is called the \emph{renormalization} of $f$.
\end{definition}

The renormalization change of variables projects a pre-corona to a corona on a non-canonical abstract Riemann surface. We will bypass this issue later by making a canonical choice for a corona renormalization fixed point and all nearby coronas.

\begin{example}[Prime renormalization]
    We say that the renormalization of a corona $f: U \to V$ is \emph{prime} if $k_-+k_+ = 3$. Below is an example of a prime corona renormalization. 
    
    Assume that the arcs $\gamma_0$, $\gamma_1$, and $\gamma_2 := f(\gamma_1)$ are pairwise disjoint. Denote by $S_1$ the unique open quadrilateral obtained by cutting $V$ along $\gamma_1 \cup \gamma_2$ such that $S_1$ is disjoint from $\gamma_0$. Let us assume further that $S_1$ does not contain the critical value nor the forbidden boundary of $U$.

    Let us remove $S_1$ from the dynamical plane. We define $\hat{V}$ to be the Riemann surface with boundary obtained from $\overline{V \backslash S_1}$ by gluing $\gamma'_1 := f^{-1}(\gamma_2) \cap \gamma_1$ and its image $\gamma_2$ along $f$. In other words, there is a quotient map $\psi: \overline{V \backslash S_1} \to \hat{V}$ that is conformal on the interior and $\psi(z)=\psi(f(z))$ for all $z \in \gamma'_1$. We embed the abstract Riemann surface $\hat{V}$ into the plane.

    The prime renormalization of $f$ is defined by the induced first return map of $f$ on $\hat{V}$. More precisely, consider the lift $S_0$ of $S_1$ under $f$ attached to $\gamma_1$. The piecewise holomorphic map
    \[
    \begin{cases}
        f(z), & \text{ if } z \in U \backslash \left(S_1 \cup f^{-1}(S_1)\right), \\
        f^2(z), & \text{ if } z \in S_0 \cap f^{-1}(U).
    \end{cases}
    \]
    descends via $\psi$ into a corona $\hat{f} : \hat{U} \to \hat{V}$ with critical ray $\hat{\gamma}_1 = \psi(\gamma'_1)$.
\end{example}

\subsection{Banach neighborhood}

In what follows, every unicritical holomorphic map $f: U \to V$ under consideration will be assumed to admit a slightly larger domain $\tilde{U}$ with piecewise smooth boundary such that $\tilde{U}$ compactly contains $U$ and $f$ extends to a unicritical holomorphic map on $\tilde{U}$ extending continuously to $\partial \tilde{U}$. We define a \emph{Banach neighborhood} of $f$ to be a neighborhood of $f$ of the form $N_{\tilde{U}}(f,\varepsilon)$, which we define to be the space of holomorphic maps $g: \tilde{U} \to \C$ that extend continuously to $\partial \tilde{U}$, admit a single critical point in $c_0(g)$, and
\[
    \sup_{z \in \tilde{U}}|f(z)-g(z)| <\varepsilon.
\]
We equip $N_{\tilde{U}}(f,\varepsilon)$ with the sup norm over $\tilde{U}$.

\begin{lemma}[Stability of the corona structure]
\label{lem:stability-of-coronas}
    Let $f: U\to V$ be a $(d_0,d_\infty)$-critical corona. For sufficiently small $\varepsilon>0$, there is a holomorphic motion $\partial U_g$ of $\partial U$ over $g \in N_{\tilde{U}}(f,\varepsilon)$ such that $g: U_g \to V$ is a $(d_0,d_\infty)$-critical corona with the same codomain $V$ and critical arc $\gamma_1$.
\end{lemma}

\begin{proof}
    Let $A_\delta$ be the $\delta$-neighborhood of $\partial U$, where $\delta>0$ is picked small enough such that $A_\delta$ contains no critical points of $f$. For sufficiently small $\varepsilon$, the derivative of $g \in N_{\tilde{U}}(f,\varepsilon)$ is uniformly bounded and non-vanishing on $A_\delta$, and so $g$ has no critical points in $A_\delta$. Thus, we have a well-defined map $\tau_g : \partial U \to A_\delta$ such that $\tau_f = \Id$ and $f = g\circ \tau_g$ on $\partial U$. Since $f$ has no critical value along $\partial U$, $\tau_g(z)$ is injective in $z$ and holomorphic in $g$. Therefore, we have a holomorphic motion of $\partial U$, and $\tau_g(\partial U)$ bounds an open annulus $U_g$ on which $g :U_g \to V$ is a well-defined $(d_0,d_\infty)$-critical corona with the same critical arc.
\end{proof}

The following theorem is inspired by Yampolsky's holomorphic motions argument \cite[\S7]{Y03a}. See also \cite[Proposition 2.11]{Y08} and \cite[\S2]{DL23}.

\begin{theorem}
\label{thm:renormalization}
    Suppose a unicritical holomorphic map $f: U \to V$ admits a pre-corona which projects to a corona $\hat{f}: \hat{U} \to \hat{V}$ via a quotient map $\psi_f: S_f \to \hat{V}$. For sufficiently small $\varepsilon >0$, there is a compact analytic renormalization operator $\renorm$ on a Banach neighborhood $N_{\tilde{U}}(f,\varepsilon)$ such that $\renorm f = \hat{f}$ and for each $g \in N_{\tilde{U}}(f,\varepsilon)$,
    \begin{enumerate}[label=\textnormal{(\arabic*)}]
        \item $g$ admits a pre-corona which projects to the corona $\renorm g: \hat{U}_g \to \hat{V}$, and
        \item the boundary $\partial \hat{U}_g$ and the associated gluing map $\psi_g$ depend holomorphically on $g$.
    \end{enumerate}
\end{theorem}

\begin{proof}
    There exists a pre-corona $F= \left( f^{k_\pm} : U_\pm \to S \right)$ of $f$ and a quotient map $\psi_f$ projecting $F$ to $\hat{f}$. Recall the arcs $\beta_\pm$ and $\beta_0$ associated to $F$. For $g \in N_{\tilde{U}}(f,\varepsilon)$, consider the map $\tau_g : \beta_0 \cup \beta_\pm \to \C$ defined by setting $\tau_g$ to be the identity map on $\beta_0$ and the composition $g^{k_\mp}\circ f^{-k_\mp}$ on $\beta_\pm$; this is an equivariant holomorphic motion of $\beta_0 \cup \beta_\pm$ for sufficiently small $\varepsilon>0$. By $\lambda$-lemma \cite{BR86,ST86}, $\tau_g$ extends to a holomorphic motion of $S$ over a neighborhood of $f$.

    Let $\mu_g$ be the Beltrami differential of $\tau_g$. Define a global Beltrami differential $\nu_g$ by setting $\nu_g = (\psi_f)_* \mu_g$ on $\hat{V}$ and $\nu_g \equiv 0$ outside of $\hat{V}$. Integrate $\nu_g$ to obtain a unique quasiconformal map $\phi_g$ fixing $\infty$, the critical point of $f$, and the critical value of $f$. Then, $\psi_g := \phi_g \circ \psi_f \circ \tau_g^{-1}$ is a conformal map on $S_g := \tau_g(S_f)$ depending holomorphically on $g$.

    The gluing map $\psi_g$ projects the pair $\left( g^{k_-}, g^{k_+} \right)$ on $S_g$ to a map $\hat{g}$ close to $\hat{f}$. By Lemma \ref{lem:stability-of-coronas}, $\hat{g}$ restricts to a corona that has the same range as $\hat{f}$ and depends analytically on $g$. This yields an analytic operator $ g \mapsto \hat{g}$. To make this operator compact, we modify it as follows. Pick another annulus $U'$ where $U \Subset U' \Subset \tilde{U}$. We define $\renorm$ on $N_{\tilde{U}}(f,\varepsilon)$ to be the renormalization of the restriction of $g$ to $U'$.
\end{proof}

%%%%%%%%%%%%%%%%%%%%%%%%%%%%%%%%%%%%%%%%%%%%%%%%%%%%%%%%%%%%%%%%%%%%%%%%%%%%%%%%%%%%%%%%%%%%%%%%%%%%%%%%%%%%%%%%%%%%%%%%%%%%%%%%%%%%%%%%%%%%%%%%%%%%%%%%%%%%%%%%%%%%%%%%%%%%%%%%%%%%

\section{Rotational coronas}
\label{sec:rotational-coronas}

Throughout this section, we fix a bounded type irrational number $\theta \in \IrratBdd$.

\subsection{Definition of rotational coronas}

\begin{definition}[Inner and outer criticalities]
    Consider a quasicircle $\Hq \subset \C$ and denote the bounded and unbounded components of $\RS \backslash \Hq$ by $Y^0$ and $Y^\infty$ respectively. We say that $f: \Hq \to \Hq$ is a \emph{$(d_0,d_\infty)$-critical quasicircle map} if it is a critical quasicircle map where for any $\bullet \in \{0,\infty\}$ and any point $z \in Y^\bullet$ close to the critical value of $f$, there are exactly $d_\bullet$ preimages of $z$ in $Y^\bullet$ that are close to the critical point of $f$.
\end{definition}

When a holomorphic map $f$ is given, we also say that an invariant quasicircle $\Hq \subset \C$ is a \emph{$(d_0,d_\infty)$-critical Herman quasicircle} if $f: \Hq \to \Hq$ is a $(d_0,d_\infty)$-critical quasicircle map. The term \emph{Herman quasicircle} originates from \cite{Lim23a} and is meant to acknowledge that its first non-trivial examples arise from degeneration of Herman rings.

Now, let us consider a $(d_0,d_\infty)$-critical corona $f: U \to V$.
We say that $f$ is \emph{laminated} if it admits a collection of laminations $\{\mathscr{L}^\bullet_i\}_{\bullet \in \{0,\infty\}, i \geq 0}$ satisfying the following properties for each $\bullet \in \{0,\infty\}$.
\begin{enumerate}[label = \textnormal{(\roman*)}]
    \item The support of each $\mathscr{L}^\bullet_i$ is a union of closed topological rectangles with pairwise disjoint interiors, and every leaf of $\mathscr{L}^\bullet_i$ is vertical, i.e. it joins the top and bottom horizontal sides of a rectangle. For $i=0$, there are $d_\bullet-1$ rectangles.
    \item $\mathscr{L}^\bullet_0$ lives in $\overline{V} \backslash U$. The union $B^\bullet$ of the bottom horizontal sides of the rectangles of $\mathscr{L}^\bullet_0$ is equal to the legitimate boundary $\partial^\bullet_\lgt U$. The union $T^\bullet$ of their top horizontal sides is a subset of $\partial^\bullet V$.
    \item For every $i \geq 0$, the top horizontal side of every rectangle of $\mathscr{L}^\bullet_{i+1}$ is contained in the bottom horizontal side of a rectangles of $\mathscr{L}^\bullet_i$. 
    \item The closure of $\gamma_1$ contains a unique leaf $l^\bullet_i$ of $\mathscr{L}^\bullet_i$ for every $i$. The bottom endpoint of each $l^\bullet_i$ is the top endpoint of $l^\bullet_{i+1}$.
    \item For every $i \geq 0$, $\mathscr{L}^\bullet_{i+1}$ is the collection of the preimages of leaves of $\mathscr{L}^\bullet_i$ under $f: \overline{U} \to \overline{V}$ such that condition (iii) is satisfied. Hence, $f: \mathscr{L}^\bullet_{i+1} \to \mathscr{L}^\bullet_i$ is $d_\bullet$-to-$1$ away from $f^{-1}(\gamma_1)$.
\end{enumerate}
The leaves of each $\mathscr{L}^\bullet_i$ will be called \emph{external ray segments}. 
For $\bullet \in \{0,\infty\}$, an infinite concatenation of external ray segments $e_i \in \mathcal{R}^\bullet_i$, $i \geq 0$ is called an \emph{external ray} of $f$; it is called an \emph{inner} external ray if $\bullet = 0$, an \emph{outer} external ray if $\bullet \in \infty$.
For example, the arc $\gamma_1$ contains a unique inner external ray $l^0_0 \cup l^0_1 \cup \ldots$ and a unique outer external ray $l^\infty_0 \cup l^\infty_1 \cup \ldots$
For any external ray $\gamma$ of $f$, we denote the image by
\[
f(\gamma) := f(\gamma \cap \overline{U})
\]
which is also an external ray of $f$ by definition. 

Consider a laminated corona $f: U \to V$. For each $\bullet \in \{0,\infty\}$, define the map $\pi_\bullet : B^\bullet \to T^\bullet$ sending the endpoint of each leaf of $\mathscr{L}^\bullet_0$ to the starting point. Consider the partially defined $d_\bullet$-to-$1$ self map $\phi_\bullet := \pi_\bullet^{-1} \circ f$ of $B^\bullet$. Denote by $\mathcal{A}^\bullet$ the set of points in $B^\bullet$ which are invariant under $\phi_\bullet$. Let us identify the unit circle $\T$ with the quotient $\R/\Z$. There is a semiconjugacy $s_\bullet: \mathcal{A}^\bullet \to \T$
between $\phi_\bullet : \mathcal{A}^\bullet \to \mathcal{A}^\bullet$ and the multiplication map $\T \to \T$, $x \mapsto d_\bullet x$ (mod $1$), which is unique up to conjugation with addition by multiples of $\frac{1}{d_\bullet-1}$.
For an external ray $\gamma$, we define the \emph{external angle} of $\gamma$ to be the angle $s_\bullet(z)$ where $z$ is the unique point of intersection of $\gamma$ and $B^\bullet$ for some $\bullet \in \{0,\infty\}$.

\begin{definition}
    A corona $f:U \to V$ is a \emph{rotational corona} if 
    \begin{enumerate}
        \item $U$ essentially contains a Herman quasicircle $\Hq$ that passes through the unique critical point of $f$;
        \item the critical arc $\gamma_1$ intersects $\Hq$ precisely at one point $m(f)$, which we call the \emph{marked point} $m(f)$ of $f$;
        \item $f$ is laminated and $m(f)$ splits $\overline{\gamma_1}$ into an inner external ray $R^0$ and an outer external ray $R^\infty$.
    \end{enumerate}
    A pre-corona is called \emph{rotational} if it projects to a rotational corona under its renormalization change of variables.
\end{definition}

By design, if a $(d_0,d_\infty)$-critical corona is rotational, then it admits a $(d_0,d_\infty)$-critical Herman quasicircle. In this section, we will construct rotational coronas out of critical quasicircle maps and discuss a rigidity property for rotational coronas.

\subsection{Realization of rotational coronas}

 Consider the family of degree $d$ rational maps $\{F_c\}_{c \in \C^*}$ where
\begin{equation}
\label{eqn:rat-map-formula}
    F_{c}(z) := - c \,\dfrac{ \displaystyle\sum_{j=d_0}^{d} \binom{d}{j} \cdot (-z)^j}{ \displaystyle\sum_{j=0}^{d_0-1} \binom{d}{j} \cdot (-z)^j}.    
\end{equation}
By \cite[Proposition 10.1]{Lim23a}, this family is characterized by the property that $F_c$ has critical points at $0$, $\infty$, and $1$ with local degrees $d_0$, $d_\infty$, and $d$ respectively, and that $F_c(0)=0$, $F_c(\infty)=\infty$, and $F_c(1)=c$. 

\begin{theorem}[{\cite{Lim23a, Lim23b}}]
\label{thm:comb-rigidity}
    There exists a unique parameter $c=c(\theta) \in \C^*$ such that $F_c$ admits a $(d_0,d_\infty)$-critical Herman quasicircle $\Hq$ with rotation number $\theta$ which passes through $1$. 
\end{theorem}

Consider $f:= F_{c}$ and $\Hq$ from the theorem above. The map $f$ will serve as our model map. Using external rays, we will show in this subsection that $f$ is corona renormalizable.

Recall that the \emph{Fatou set} $F(R)$ of a rational map $R: \RS \to \RS$ is the set of points in $\RS$ near which the iterates of $R$ are equicontinuous, and the \emph{Julia set} $J(R)$ of $R$ is equal to $\RS \backslash F(R)$. For the rational map $f$, the Fatou set $F(f)$ is known to be the union of iterated preimages of the immediate attracting basin $A_0$ of the superattracting fixed point $0$ and the immediate attracting basin $A_\infty$ of the superattracting fixed point $\infty$. The Julia set $J(f)$ is equal to the closure of the union of the iterated preimages of the Herman quasicircle $\Hq$.

For any $n\geq 1$, we refer to the closure of a component of $f^{-n}(\Hq) \backslash f^{-(n-1)}(\Hq)$ as a \emph{bubble} of \emph{generation} $n$. Every bubble $B$ of generation $n$ is a quasicircle admitting a unique point, which we will call the \emph{root} of $B$, that lies on the pre-critical set $ f^{-(n-1)}(1)$. We call a bubble $B$ of generation $n$ an \emph{outer bubble} (resp. \emph{inner bubble}) if the bubbles $B$, $f(B),\ldots$, $f^{n-1}(B)$ all lie in the connected component of $\RS \backslash \Hq$ containing $\infty$ (resp. $0$).

A \emph{limb} of generation one is the closure of a connected component of $J(f) \backslash \{1\}$ that is disjoint from $\Hq$. In general, a limb $L$ of generation $n\geq 1$ is the connected component of the preimage under $f^{n-1}$ of a limb of generation one. A \emph{filled limb} $\hat{L}$ of generation $n$ is the hull of a limb $L$ of generation $n$, that is, $\RS \backslash \hat{L}$ is the unbounded connected component of $\RS \backslash L$. 

Every limb $L$ of generation $n$ contains a unique bubble $B_L$ of generation $n$. The \emph{root} of $L$ is the root of $B_L$. We call $L$ an \emph{outer/inner limb} if $B_L$ is an outer/inner bubble. 

\begin{lemma}
\label{lem:small-limbs}
    Consider the immediate basins $A_0$ and $A_\infty$ of $0$ and $\infty$ respectively.
    \begin{enumerate}[label=\textnormal{(\arabic*)}]
        \item The boundary of $A_\infty$ is equal to the closure of the union of $\Hq$ and all outer bubbles of $f$, whereas the boundary of $A_0$ is equal to the closure of the union of $\Hq$ and all inner bubbles of $f$. 
        \item Both $\partial A_0$ and $\partial A_\infty$ are locally connected. 
        \item For any $\varepsilon>0$, all but finitely many inner and outer limbs of $f$ have diameter at most $\varepsilon$.
    \end{enumerate}
\end{lemma}

\begin{proof}
    Denote by $Y^0$ and $Y^\infty$ the connected components of $\RS \backslash \Hq$ containing $0$ and $\infty$ respectively. Perform Douady-Ghys surgery \cite{G84,D87} (see also \cite[\S7.2]{BF14}) along $\Hq$ to replace the dynamics of $f$ in $Y^0$ with a rotation disk and obtain a degree $d_\infty$ unicritical polynomial $f_\infty$ whose critical point lies on the boundary of an invariant Siegel disk $Z_\infty$ of $f_\infty$. The maps $f|_{\overline{Y^\infty}}$ and $f_\infty|_{\RS \backslash Z_\infty}$ are quasiconformally conjugate, and this conjugacy sends $A_\infty$ onto the immediate basin of $\infty$ of $f_\infty$. In particular, the external boundary of the filled outer limbs of $f$ are quasiconformally equivalent to the limbs of $f_\infty$. The work of \cite{Pe96} (or more generally \cite{WYZZ}) guarantees that the Julia set of $f_\infty$ is locally connected, and so any infinite sequence of limbs of $f_\infty$ must shrink to a point. Therefore, for any $\varepsilon>0$, all but finitely many outer limbs of $f$ have diameter at most $\varepsilon$. By swapping the roles of $0$ and $\infty$, we obtain the same result for inner limbs.
\end{proof}

\begin{remark}
    In fact, the whole Julia set of $f$ is locally connected! In case $(d_0,d_\infty)=(2,2)$, this was proven by Petersen \cite[\S4]{Pe96}. For arbitrary criticalities $(d_0,d_\infty)$, the availability of complex bounds \cite[\S6.3]{Lim23b} facilitates a direct generalization of Petersen's proof.
\end{remark}

For $\bullet \in \{0,\infty\}$, consider the B\"ottcher conjugacy $b_\bullet : (A_\bullet,\bullet) \to (\D,\bullet)$ between $f$ and the power map $z \mapsto z^{d_\bullet}$. An \emph{external ray} in $A_\bullet$ of angle $t \in \R/\Z$ is defined by
\[
\left\{ b_\bullet^{-1}(r e^{2\pi i t}) \: : \: 0<r<1 \right\},
\]
and an \emph{equipotential} in $A_\bullet$ of level $\lambda>0$ is the analytic Jordan curve defined by
\[
\left\{ b_\bullet^{-1}(z) \: : \: |z|= e^{-\lambda} \right\}.
\]
External rays and equipotentials form a pair of $f$-invariant transverse foliations of $A_\bullet$. By Lemma \ref{lem:small-limbs}, every external ray in $A_\bullet$ lands at a point on $\partial A_\bullet$. Every point $x$ in $\partial A_\bullet$ is the landing point of exactly one external ray in $A_\bullet$, except when $x$ is a pre-critical point in which case it is the landing point of $d_\bullet$ external rays in $A_\bullet$.

Consider the operator $r_{\textnormal{prm}}$ from \S\ref{sec:sector-renorm}, which encodes how rotation number is transformed under sector renormalization.

\begin{lemma}
\label{lem:construction-of-pre-corona}
    For any point $x \in \Hq$ that is not a pre-critical point of $f$, any $\varepsilon>0$, and any sufficiently high $n \in \N$, there is a rotational pre-corona
    \[
    P = \left( f_- := f^{k_-}: U_- \to S, \: f_+ := f^{k_+}: U_+ \to S \right)
    \]
    around $x$ such that
    \begin{enumerate}[label=\textnormal{(\arabic*)}]
        \item $P$ has rotation number $r_{\textnormal{prm}}^n(\theta)$;
        \item every external ray segment of $P$ is contained in an external ray of $f$;
        \item the union $\bigcup_{\diamond \in \{-,+\}}\bigcup_{i=0}^{k_\diamond-1} f^i(U_\diamond)$ lies in the $\varepsilon$-neighborhood of $\Hq$.
    \end{enumerate}
\end{lemma}

\begin{proof}
    For every integer $i \in \Z$, let $x_i := (f|_\Hq)^i(x)$. By Lemma \ref{lem:prime-to-sector}, for all $n \geq 1$, there exist return times $\abold_n$, $\bbold_n \in \N$ such that the commuting pair 
    \[
\left(f^{\abold_n}|_{[x_{\bbold_n},x_0]},f^{\bbold_n}|_{[x_0,x_{\abold_n}]}\right)
    \]
    is a sector pre-renormalization of $f|_\Hq$ with rotation number $r_{\textnormal{prm}}^n(\theta)$. (In short, the pair above is the first return map of $f$ on the interval $[x_{\bbold_n},x_{\abold_n}] \subset \Hq$, and gluing the two ends of the interval via $f^{|\bbold_n-\abold_n|}$ projects the pair to a critical quasicircle map with rotation number $r_{\textnormal{prm}}^n(\theta)$. Refer to \S\ref{sec:sector-renorm} for more details.)
    
    Let $k_- := \abold_n$ and $k_+ := \bbold_n$, and let us pick a small constant $\lambda>0$. For $\bullet \in \{0,\infty\}$, denote by $E^\bullet$ the equipotential in $A_\bullet$ of level $\lambda$, and by $R^\bullet_+$, $R^\bullet$, and $R^\bullet_-$ the external rays in $A_\bullet$ that land at $x_{k_+}$, $x_{k_-+k_+}$, and $x_{k_-}$ respectively. Then, the union $\bigcup_{\bullet \in \{0,\infty\}} R^\bullet_\pm \cup R^\bullet \cup E^\bullet$ encloses a rectangle $S_\pm$ containing the interval $[x_{k_\pm}, x_{k_-+k_+}] \subset \Hq$.
    
    Let $I_- := [x_{k_+}, x_0]$ and $I_+ := [x_0, x_{k_-}]$. Precisely one of the two intervals, say $I_-$ without loss of generality, contains a critical point of $f^{k_-}$. The rectangle $S_\pm$ lifts under $f^{k_\pm}$ to a topological disk $\Upsilon_{\pm}$ containing $I_\pm$, where $f^{k_-}: \Upsilon_- \to S_-$ is a degree $d$ branched covering map and $f^{k_+}: \Upsilon_+ \to S_+$ is univalent. There are precisely $d-1$ disjoint disks $D_1,\ldots, D_{d-1}$ which are the lifts of $S_+$ under $f^{k_-}$ that are touching $\Upsilon_-$ on the boundary and are disjoint from $\Hq$. Set 
    \[
    U_+ := \Upsilon_+, \quad U_- := \Upsilon_- \cup \bigcup_{j=1}^{d-1} D_j, \quad S:= S_- \cup S_+.
    \]
    See Figure \ref{fig:corona-construction} for an illustration. Then, 
    \[
    \left(f^{k_-} : U_- \to S, \quad f^{k_+}: U_+ \to S\right)
    \]
    is a $(d_0, d_\infty)$-critical pre-corona with rotation number $r^n_{\textnormal{prm}}(\theta)$.

    Let us embed the restriction of external rays of $f$ in $S \backslash U$ where $U:= U_- \cup U_+$. Notice that the boundaries of $U_-$ and $U_+$ contain equipotential segments of different levels. Assume without loss of generality that the equipotential segments in $U_-$ have smaller level, or equivalently, $k_- > k_+$. In order to satisfy (2), let us truncate a pair of small topological triangles near two vertices of the rectangle $S_+$, one where $R^0_+$ meets $E^0$ and the other where $R^\infty_+$ meets $E^\infty$. Refer to Figure \ref{fig:corona-construction}. We will also truncate preimages of these triangles under $f^{k_-}$ in $U_-$. Replace $U$ and $S$ with the new truncated domains. Then, every point in the legitimate boundary of $U$ is now a landing point of an external ray segment, and (2) follows.

    We claim that (3) follows from taking $n$ to be sufficiently large and $\lambda$ to be sufficiently small. Indeed, if $z \in U_\pm$ intersects an external ray landing at a point $w \in J(f) \cap U_\pm$, then the orbits of $z$ and $w$ remain close under iteration $f^{i}$ for $i =1,\ldots,k_\pm$. Suppose $z \in U_\pm$ is outside of $A_0 \cup A_\infty$. Then, it must lie within some filled limb $\hat{L}$ rooted at some pre-critical point $c_{-j}:= \left(f|_\Hq\right)^{-j}(1)$ for some $j\geq 0$. If $c_{-j}$ is not the unique critical point of $f^{k_-}$, then the forward images $\hat{L}$, $f(\hat{L})$, $\ldots, f^{k_\pm}(\hat{L})$ must remain small due to Lemma \ref{lem:small-limbs}. If $c_{-j}$ is the critical point of $f^{k_-}$ in $U_-$, then we must have $0<j<k_-$. In the latter case, $f^j(U_-)$ must remain in a small neighborhood of the critical point $c_0=1$ as we take $\lambda$ to be small and $n$ to be large. Therefore, the forward orbit $z$, $f(z)$, \ldots, $f^j(z)$ must be close to $\Hq$.
\end{proof}

\begin{figure}
    \centering
    \begin{tikzpicture}[scale=1.08]
        % for the upper domains
        \filldraw[black, ultra thin, fill=green!10!white] (4.15, -1.25) .. controls (4.16,-0.7) .. (4.25,0.05) .. controls (4.34,0.7) .. (4.35, 1.2) .. controls (5,1.15) and (6,1.15) .. (6.65, 1.2) .. controls (6.64,0.7) .. (6.8,-0.32) .. controls (6.94,-0.8) .. (6.95, -1.3) .. controls (6,-1.4) and (5,-1.25) .. (4.15, -1.25);

        \filldraw[black, ultra thin, fill=yellow!10!white] (4.16,-0.7) -- (4.25,0.05) -- (4.34,0.7) 
        .. controls (4,0.7) and (2.7,0.35) .. 
        (2.8,0.9) -- (2.47,1.1) 
        .. controls (2.36,0.95) and (2.1,1) ..
        (2.05,1.22) -- (1.7,1.3)
        .. controls (1.6,1.1) and (1.4,1.1) ..
        (1.3,1.3) -- (0.95,1.22)
        .. controls (0.9,1) and (0.62,0.95) ..
        (0.53,1.1) -- (0.2,0.9)
        .. controls (0.3,0.35) and (0,0.7) ..
        (-0.75, 0.7) .. controls (-0.73,0.5) .. (-0.6,-0.12) .. controls (-0.47,-0.5) .. (-0.45, -0.75) 
        .. controls (0,-0.75) and (0.75,-0.5) ..
        (0.8,-1.22) -- (1.2,-1.3)
        .. controls (1.35,-1.1) and (1.65,-1.1) ..
        (1.8,-1.3) -- (2.2,-1.2)
        .. controls (2.25,-0.5) and (3.5,-0.7) .. (4.16,-0.7);

        % for the little guys
        \filldraw[black, ultra thin, fill=green!15!white] (2.8,0.9) -- (2.47,1.1) -- (2.7,1.5) -- (3.1,1.27) -- (2.8,0.9);
        \filldraw[black, ultra thin, fill=green!15!white] (1.3,1.3) -- (0.95,1.22) -- (0.75,1.65) -- (1.25,1.75) -- (1.3,1.3);
        \filldraw[black, ultra thin, fill=green!15!white] (0.8,-1.22) -- (1.2,-1.3) -- (1.15,-1.75) -- (0.7,-1.65) -- (0.8,-1.22);
    
        % for the lower domains
        \filldraw[black, ultra thin, fill=green!10!white] (-0.7, -4) .. controls (-0.75,-5) .. (-0.6,-6.12) .. controls (-0.45,-7) .. (-0.4, -8) .. controls (0.5,-7.9) and (1.5,-7.9) .. (1.9, -7.9) .. controls (1.9,-7) .. (2,-6.15) .. controls (2.1,-5) .. (2.1, -4.1) .. controls (1.5,-4.1) and (0.4,-4.1) ..  (-0.7, -4);
        \filldraw[black, ultra thin, fill=yellow!10!white] (2.1, -4.1) .. controls (3.5,-4.1) and (5,-4.1) .. (6.6, -4) .. controls (6.65,-5.2) .. (6.8,-6.32) .. controls (6.95,-7) .. (7, -8) .. controls (5,-7.9) and (3.5, -7.9) .. (1.9, -7.9)  .. controls (1.9,-7) .. (2,-6.15)  .. controls (2.1,-5) .. (2.1, -4.1);
        
        % for the truncation
        \draw[black, densely dotted] (-0.7,-4.8) -- (0.15,-4.05);
        \draw[black, densely dotted] (4.34,0.7) -- (5,1.15);
        \draw[black, densely dotted] (3,1.3) -- (2.97,1.15);
        \draw[black, densely dotted] (1.15,1.7) -- (1.26,1.55);
        \draw[black, densely dotted] (0.8,-1.67) -- (0.75,-1.5);
    
        % for Hq
        \draw[ultra thick] (-3,0.25) -- (-1.5,-0.25) -- (0,0) -- (0.75,-0.25) -- (1.5,0) -- (2.66,-0.33) -- (4,0) -- (6,-0.5) -- (8,0) -- (8.5,-0.125);
        \draw[thick] (0,1.5) -- (0.8,0.4) -- (1.5,0) -- (2.2,0.4) -- (3,1.5);
        \draw[thick] (0.9,2) -- (1.2,1) -- (1.5,0) -- (1.8,1) -- (2.1,2);
        \draw[thick] (0.9,-1.9) -- (1.05,-1) -- (1.5,0) -- (1.95,-1) -- (2.1,-1.9);
        \node [black, font=\bfseries] at (-2,-0.5) {$\Hq$};

        \draw[ultra thick] (-3,-5.75) -- (-1.5,-6.25) -- (0,-6) -- (0.75,-6.25) -- (1.5,-6) -- (2.66,-6.33) -- (4,-6) -- (6,-6.5) -- (8,-6) -- (8.5,-6.125);
        \node [black, font=\bfseries] at (-2,-6.5) {$\Hq$};

        % for the labels
        \node [red, font=\bfseries] at (-0.6,-0.12) {\large $\bullet$};
        \node [red, font=\bfseries] at (-0.69,-0.4) {\small $x_{k_+}$};
        \node [red, font=\bfseries] at (4.25,-0.07) {\large $\bullet$};
        \node [red, font=\bfseries] at (4.39,-0.35) {\small $x_{0}$};
        \node [red, font=\bfseries] at (6.8,-0.32) {\large $\bullet$};
        \node [red, font=\bfseries] at (7.1,-0.55) {\small $x_{k_-}$};

        \node [red, font=\bfseries] at (-0.6,-6.12) {\large $\bullet$};
        \node [red, font=\bfseries] at (-0.73,-6.4) {\small $x_{k_+}$};
        \node [red, font=\bfseries] at (2,-6.15) {\large $\bullet$};
        \node [red, font=\bfseries] at (1.95,-6.4) {\small $x_{k_-+k_+}$};
        \node [red, font=\bfseries] at (6.8,-6.32) {\large $\bullet$};
        \node [red, font=\bfseries] at (7.1,-6.55) {\small $x_{k_-}$};

        \node [green!50!black, font=\bfseries] at (3.3,1.2) {\small $D_1$};
        \node [green!50!black, font=\bfseries] at (0.6,1.45) {\small $D_2$};
        \node [green!50!black, font=\bfseries] at (0.55,-1.8) {\small $D_3$};
        \node [green!50!black, font=\bfseries] at (5.5,0.5) {$\Upsilon_+$};
        \node [yellow!50!black, font=\bfseries] at (2.9,0.23) {$\Upsilon_-$};
        \node [green!50!black, font=\bfseries] at (0.75,-5) {$S_+$};
        \node [yellow!50!black, font=\bfseries] at (4.5,-5) {$S_-$};
                
        \draw[-latex] (2.2,-1.45) -- (4.5,-3.95);
        \node[black, font=\bfseries] at (4.2,-3) {\small $f^{k_-}$};
        \draw[-latex] (5,-1.45) -- (1,-3.95);
        \node[black, font=\bfseries] at (2,-3) {\small $f^{k_+}$};
        
\end{tikzpicture}

    \caption{The construction of the pre-corona in the proof of Lemma \ref{lem:construction-of-pre-corona} when $(d_0,d_\infty)=(3,2)$. The triangle defined by the dotted line in $S_+$ and its preimages are to be removed. }
    \label{fig:corona-construction}
\end{figure}
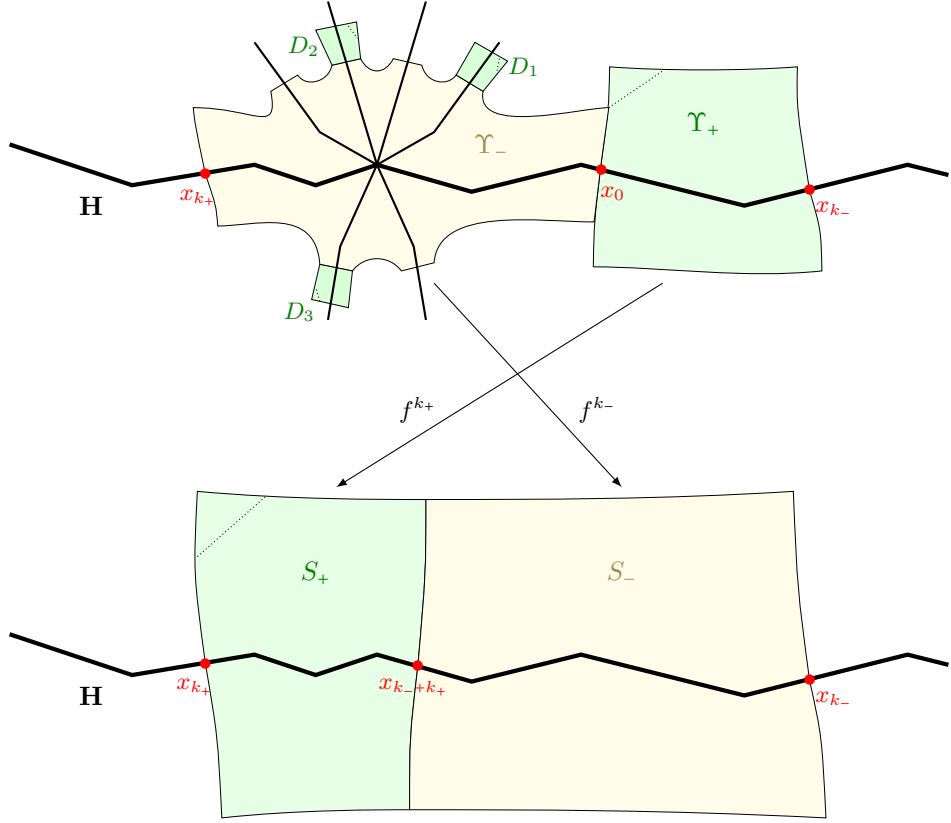

In our previous work, we proved a rigidity theorem for critical quasicircle maps.

\begin{theorem}[{\cite[Theorem F]{Lim23b}}]
\label{thm:naive-rigidity}
    Every two $(d_0,d_\infty)$-critical quasicircle maps of the same bounded type rotation number are quasiconformally conjugate on some neighborhood of their Herman curves.
\end{theorem}

Together with Lemma \ref{lem:construction-of-pre-corona}, we have the following result.

\begin{corollary}
\label{renormalizability-of-rotational}
    Any $(d_0,d_\infty)$-critical quasicircle map $g: \Hq_g \to \Hq_g$ with bounded type rotation number is corona renormalizable, that is, there is a $(d_0,d_\infty)$-critical rotational pre-corona which is an iterate of $g$ near $\Hq_g$.
\end{corollary}

\begin{proof}
    Given any $(d_0,d_\infty)$-critical quasicircle map $g$ of bounded type rotation number, Theorem \ref{thm:naive-rigidity} asserts that there is a global quasiconformal map $\phi$ conjugating $g$ on some neighborhood $W$ of $\Hq_g$ and $f:=F_{c}$ on a neighborhood $\phi(W)$ of its Herman quasicircle. By Lemma \ref{lem:construction-of-pre-corona}, $f$ admits a pre-corona $P$ with range contained within $\phi(W)$. Then, $g$ admits a $(d_0,d_\infty)$-critical pre-corona of the form $\phi^{-1}\circ P \circ \phi$.
\end{proof}

\subsection{Quasiconformal rigidity}
\label{ss:qc-rigidity-initial}

Given a critical quasicircle map $g: \Hq_g \to \Hq_g$, there is a unique conjugacy $h_g: \Hq_g \to \T$ between $g$ and a rigid rotation such that $h_g$ sends the critical point of $g$ to $1 \in \T$. We will equip $\Hq_g$ with the \emph{combinatorial metric}, which is the pullback of the normalized Euclidean metric of $\T$ under $h_g$ and thus the unique normalized $g$-invariant metric of $\Hq_g$. For any point $z \in \Hq_g$, the \emph{combinatorial position} of $z$ is the point $h_g(z)$ on the unit circle.

We say that two $(d_0,d_\infty)$-critical rotational coronas $g_1$ and $g_2$ are \emph{combinatorially equivalent} if 
\begin{enumerate}
    \item they have the same rotation number,
    \item their marked points $m(g_1)$ and $m(g_2)$ have the same combinatorial position, and
    \item for $\bullet \in \{0,\infty\}$, the external rays $R^\bullet(g_1)$ and $R^\bullet(g_2)$ have the same external angles.
\end{enumerate}

In this subsection, we will prove the following theorem.

\begin{theorem}[Quasiconformal rigidity]
\label{thm:qc-rigidity}
    Two combinatorially equivalent $(d_0,d_\infty)$-critical rotational coronas with bounded type rotation number are quasiconformally conjugate.
\end{theorem}

Theorem \ref{thm:naive-rigidity} tells us about the existence of a local quasiconformal conjugacy on a small neighborhood of the invariant Herman quasicircles.
In contrast, Theorem \ref{thm:qc-rigidity} tells us that there is indeed a more global quasiconformal conjugacy that preserves the whole corona structure.
The proof below is an application of the pullback argument. 

Let us make a couple of technical preparations. Recall the model rational map $f$ introduced in the previous subsection.

\begin{definition}
    A \emph{bubble chain} of $f$ of generation $n \geq 1$ is an infinite sequence of bubbles $\{B_j\}_{j\geq 1}$ of $f$ where $B_1$ has generation $n$ and for all $j\geq 1$, $B_j$ contains the root of $B_{j+1}$ and the generation of $B_j$ is strictly increasing in $j$. We say that a bubble chain $\{B_j\}_{j\geq 1}$ 
\begin{list}{$\rhd$}{}
    \item is an \emph{outer/inner} bubble chain if each $B_j$ is an outer/inner bubble,
    \item is \emph{periodic} of period $p$ if there exists some $k \geq 1$ such that $f^p(B_{j+k}) = B_j$ for all $j\geq k$, and
    \item \emph{lands} if the accumulation set $\bigcap_{j\geq 1} \overline{ \cup_{k\geq j} B_k}$ is a single point, which we call the \emph{landing point} of the bubble chain.
\end{list} 
We say that a repelling periodic point $z$ of $f$ is an \emph{outer} (resp. \emph{inner}) periodic point if its orbit is contained in the connected component of $\RS \backslash \Hq$ containing $\infty$ (resp. $0$).
\end{definition}

% \begin{lemma}
% \label{lem:transfer-periodic-bubble-chain}
%     Every periodic outer/inner bubble chain of $f$ lands at an outer/inner periodic point with the same period.
%     Conversely, every outer/inner periodic point of $f$ is a landing point of a unique maximal bubble chain of $f$, which is outer/inner and periodic.
% \end{lemma}

% \begin{proof}
%     Recall the unicritical polynomial $f_\infty$ from Lemma \ref{lem:small-limbs}. Via Douady-Ghys surgery, outer periodic points and outer bubble chains of $f$ correspond to repelling periodic points and bubble chains of $f_\infty$. Since the statement of the lemma holds for periodic bubbles chains of $f_\infty$ (e.g. from local connectivity), it also applies to $f$. The inner case is analogous.
% \end{proof}

Let us fix a rotational pre-corona 
\[
    P=(f_\pm : U_\pm \to S)
\]
of $f$ (which exists thanks to Lemma \ref{lem:construction-of-pre-corona}).
 
\begin{definition}
    Define the \emph{non-escaping set} $K(P)$ of $P$ to be the set of points whose orbit under $f_\pm$ never escapes $\overline{U_\pm}$. By spreading around $K(P)$, we define the \emph{local non-escaping set} of $f$ relative to $P$ by
\[
    K^{loc}(f) := \bigcup_{n\geq 0} f^n(K(P)).
\]
    The set $K^{loc}(f)$ is precisely the set of points which does not escape the tiling $\Deltabold_P$ associated to $P$.

    Let us also define a \emph{partial bubble $\check{B}$ in} $K^{loc}(f)$ of generation $n\geq 1$ to be a set of the form $B \cap K^{loc}(f)$ where $B$ is a bubble of $f$ of generation $n$.
    Such a set $\check{B}$ is a always a quasiarc and it contains the root of $B$, which we will refer to as the \emph{root} of $\check{B}$.
    In a similar way, we define \emph{partial bubble chain in} $K^{loc}(f)$ to be an infinite sequence $\{\check{B}_j\}_{j\geq 1}$ of partial bubbles in $K^{loc}(g)$ with increasing generation such that each $\check{B}_j$ contains the root of $\check{B}_{j+1}$.
\end{definition}

Let $g: U \to V$ be a $(d_0,d_\infty)$-rotational corona that has the same rotation number as $f$, and let us denote the Herman quasicircle of $g$ by $\Hq_g$. By Theorem \ref{thm:naive-rigidity}, there is a quasiconformal conjugacy $\phi$ between $g$ and $f$ on some neighborhood $W_g$ of $\Hq_g$ onto a neighborhood $W$ of $\Hq$. 

By Lemma \ref{lem:construction-of-pre-corona}, we can assume that the range $S$ of $P$ is contained in $W$. 
The corona $g$ also admits a pre-corona $P_g=(g_\pm : U_{g,\pm} \to S_g)$ contained in $W_g$ and it will be selected such that it is conjugate to $P$ via $\phi$.
As such, we can define the non-escaping set $K(P_g)$ of $P_g$ in a similar way and spread it around to obtain the \emph{local non-escaping set} $K^{loc}(g)$ of $g$ relative to $P_g$. 
The quasiconformal map $\phi$ induces a conjugacy between $g|_{K^{loc}(g)}$ and $f|_{K^{loc}(f)}$.
Via $\phi$, we can also define the notion of partial bubbles and partial bubble chains in $K^{loc}(g)$.

Let $x$ be the marked point of $g$, and let $R^\infty$ and $R^0$ be the outer and inner external rays of $g$ landing at $x$. These rays make up the arc $\gamma_1(g)$.
Let $c_0(g)$ denote the critical point of $g$ and for $n \in \Z$, denote $c_{n}(g):= (g|_{\Hq_g})^n(c_0(g))$.

\begin{lemma}
\label{lem:construction-of-periodic-bubble}
    For any pre-critical point $c_{-t}(g) \in \Hq_g$ of $g$, there exist an outer periodic point $y_t^\infty$ in $K^{loc}(g)$ such that 
    \begin{enumerate}[label=\textnormal{(\arabic*)}]
        \item $y_t^\infty$ is the landing point of a unique maximal partial bubble chain $\mathcal{B}^\infty_t$ in $K^{loc}(g)$ which is rooted at $c_{-t}(g)$;
        \item $y_t^\infty$ is the landing point of a unique outer external ray $R^\infty_t$.
    \end{enumerate}
    The word ``outer`` can be replaced with ``inner``, yielding $y_t^0$, $\mathcal{B}^0_t$, and $R^0_t$.
\end{lemma}

\begin{proof}
    We will first prove (1), which can be reduced to the dynamical plane of $f$. 
    Let us denote by $I_\varepsilon \subset \Hq$ the interval of combinatorial length $\varepsilon$ centered at $c_1$. We will pick $\varepsilon>0$ to be small enough such that the full preimage under $f$ of $I_\varepsilon$ is contained in the tiling $\Deltabold_{P}$. Let us pick the first $s \in \N$ such that $c_{-t-s}$ is contained in $I_\varepsilon$. Below, we will construct an outer periodic point $\tilde{y}^\infty_t$ of $f$ in $K^{loc}(f)$, which will have period $p:=s+t+1$. 
    
    First, let us pick a small closed interval neighborhood $J_0 \subset \Hq$ of $c_{-t}$. Let us arrange that the endpoints of $J_0$ are not in the grand orbit of the critical point of $f$, so there exists a unique pair of external rays $R_-$ and $R_+$ of $f$ in the basin $A_\infty$ that land on the pair of endpoints of $J_0$ respectively. Consider the open rectangle $D_0$ cut out by the union of $J_0 \cup R_- \cup R_+$ and an arc connecting $R_-$ and $R_+$ which is contained in the equipotential in $A_\infty$ of some small level $\lambda>0$.

    Consider the intervals $J_{-j} := (f|_\Hq)^{-j}(J_0)$ for $j\geq 1$. We assume that the combinatorial length of $J_0$ is small enough such that $J_{-j}$ does not contain $c_1$ for all $j \in \{0,1,\ldots,s\}$, and in particular $J_{-s}$ is contained in $I_\varepsilon \backslash \{c_1\}$. Let us pick an outer bubble $B$ of generation one. 
    (There are $d_\infty-1$ of such bubbles.) 
    Let $D'_0$ be the unique lift of $D_0$ under $f^{s+1}$ such that $\partial D'_0 \cap f^{-1}(I_\varepsilon) = B \cap f^{-1}(I_\varepsilon)$. 
    For sufficiently small $\varepsilon>0$ and $\lambda>0$, we can guarantee that the union of $D'_0$ and the subarc of $B$ connecting $c_0$ and $D'_0$ is contained in $\Deltabold_P$.
    Therefore, $\partial D'_0 \cap f^{-1}(I_\varepsilon)$ is contained in the corresponding partial bubble $\check{B} \subset B$.

    Next, consider the outer partial bubble $\check{B}_1$ rooted at $c_{-t}$ such that $f^t(\check{B}_1)=\check{B}$. Let $D_{-1}$ be the lift of $D'_0$ under $f^t$ that is attached to $\check{B}_1$. See Figure \ref{fig:outer-repelling-periodic-point} for a reference. Since $D'_0 \subset \Deltabold_P$, then $D_{-1} \subset \Deltabold_P$ too. Also, since $D_{-1}$ is compactly contained in $D_{0}$, then the univalent map $f^{p} : D_{-1} \to D_0$ is uniformly expanding with respect to the hyperbolic metric of $D_{-1}$. Hence, this map admits a unique repelling fixed point $\tilde{y}_t^\infty$ which is contained in $K^{loc}(f)$.

    Denote $z_1 := c_{-t}$. Let $z_2$ be the lift of $z_1$ under $f^p : \overline{D_{-1}} \to \overline{D_0}$. 
    Inductively, we have an infinite sequence of pre-critical points $z_1, z_2, z_3, \ldots$ such that for $n \geq 2$, $z_n$ is contained in $\overline{D_{-1}}$ and $f^p(z_n) = z_{n-1}$.
    By the uniform expansion, $z_n$ converges to $y_t^\infty$.
    Notice that for $n\geq 2$, each $z_n$ is in $K^{loc}(f)$ and is a root of an outer partial bubble $\check{B}_n$ with $f^p(\check{B}_n) = \check{B}_{n-1}$.
    Therefore, $\{ \check{B}_j\}_{j\geq 1}$ is a maximal partial bubble chain in $K^{loc}(f)$ that is landing at $\tilde{y}_t^\infty$ and is rooted at $c_{-t}$.

    Now, let us return to the dynamical plane of $g$ via $\phi$. We have the desired outer periodic point $y_t^\infty = \phi(\tilde{y}_t^\infty)$ and the partial bubble chain $\mathcal{B}^\infty_t = \{\acute{B}_j\}_{j\geq 1}$ satisfying (1), where $\acute{B}_j = \phi(\check{B}_j)$. Now, let us prove (2).

    Consider the outer external ray $R^\infty$ of $g$ which lands at the marked point $x$. Let us pick iterated preimages $R_l$ and $R_r$ of $R^\infty$ landing at points $x_{l,0}$ and $x_{r,0}$ on $\acute{B}_1$ that are located on just slightly on the left and right of the root of $\acute{B}_2$ respectively. 
    The union $\acute{B}_1 \cup R_l \cup R_r \cup \partial V$ bounds a topological rectangle $E_0$ that contains $y$ and is disjoint from $\Hq_g$. 
    Then, $D_0$ lifts under $g^p$ to a rectangle $D_{-1}$ containing $y_t^\infty$. 
    Since the vertical sides of $E_{-1}$ are external ray segments with a much smaller external angle difference compared to $E_0$, then $E_{-1}$ is compactly contained in $E_0$. 
    By Schwarz Lemma, the univalent map $g^p: E_{-1} \to E_0$ is uniformly expanding with respect to the hyperbolic metric of $E_0$ and $y_t^\infty$ is its unique repelling fixed point. 
    
    For every $n \in \N$, let $E_{-n}$ be the lift of $E_0$ under $g^{pn}$ containing $y$. Consider the lifts $R_{l,n}$ and $R_{r,n}$ of $R_l$ and $R_r$ under $g^{pn}$ which touch the boundary of $E_{-n}$; these are external rays landing at points $x_{l,n}$ and $x_{r,n}$ respectively, which are vertices of $E_{-n}$. 
    Since $g^p: E_{-1} \to E_0$ is uniformly expanding, as $n\to \infty$, $x_{l,n}$ and $x_{r,n}$ converge to $y_t^\infty$ and the external rays $R_{l,n}$ and $R_{r,n}$ converge to a limiting external ray $R^\infty_t$, which is a $p$-periodic outer external ray landing at $y_t^\infty$. This implies (2).
\end{proof}

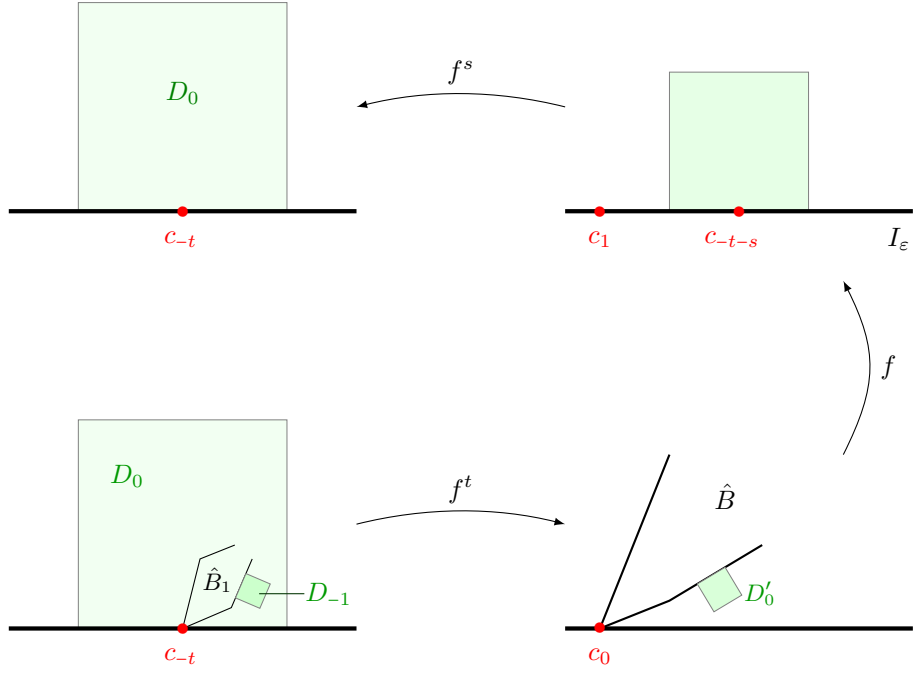
\begin{figure}
    \centering
    \begin{tikzpicture}[scale=0.92]
        % for the upper domains
        \filldraw[gray, ultra thin, fill=green!5!white] (-4.5,0) -- (-1.5,0) -- (-1.5,3) -- (-4.5,3) -- (-4.5,0);
        \draw[black,ultra thick] (-5.5,0) -- (-0.5,0);
        \node[green!50!black] at (-3,1.7) {$D_0$};
        \node[red] at (-3,-0.02) {\large $\bullet$};
        \node[red] at (-3.03,-0.4) {$c_{-t}$};
        
        \draw[-latex] (2.5,1.5) .. controls (1.5,1.75) and (0.5,1.75) .. (-0.5,1.5);
        \node[black] at (1,2) {$f^s$};
        
        \filldraw[gray, ultra thin, fill=green!10!white] (6,0) -- (4,0) -- (4,2) -- (6,2) -- (6,0);
        \draw[black,ultra thick] (7.5,0) -- (2.5,0);
        \node[black] at (7.3,-0.4) {$I_\varepsilon$};
        \node[red] at (5,-0.02) {\large $\bullet$};
        \node[red] at (4.9,-0.4) {$c_{-t-s}$};
        \node[red] at (3,-0.02) {\large $\bullet$};
        \node[red] at (3,-0.4) {$c_1$};

        \draw[-latex] (6.5,-3.5) .. controls (7,-2.5) and (7,-2) .. (6.5,-1);
        \node[black] at (7.15,-2.25) {$f$};

        \draw[black,ultra thick] (7.5,-6) -- (2.5,-6);
        \draw[black,thick] (5.33333,-4.8) -- (4,-5.6) -- (3,-6) -- (4,-3.5);
        \filldraw[gray, ultra thin, fill=green!15!white] (4.4,-5.36) -- (4.8,-5.12) -- (5.04,-5.52) -- (4.64,-5.76) -- (4.4,-5.36); 
        \node[green!60!black] at (5.3,-5.5) {\small $D'_0$};
        \node[black] at (4.8,-4.1) {$\hat{B}$};
        \node[red] at (3,-6) {\large $\bullet$};
        \node[red] at (3,-6.4) {$c_0$};
        
        \draw[-latex] (-0.5,-4.5) .. controls (0.5,-4.25) and (1.5,-4.25) .. (2.5,-4.5);
        \node[black] at (1,-4) {$f^t$};

        \filldraw[gray, ultra thin, fill=green!5!white] (-4.5,-6) -- (-1.5,-6) -- (-1.5,-3) -- (-4.5,-3) -- (-4.5,-6);
        \draw[black] (-2.25,-4.8) -- (-2.75,-5) -- (-3,-6) -- (-2.3,-5.7) -- (-2,-5);        
        \filldraw[gray, ultra thin, fill=green!20!white] (-2.24,-5.56) -- (-2.09,-5.21) -- (-1.74,-5.36) -- (-1.89,-5.71) -- (-2.24,-5.56);
        \draw[black,ultra thick] (-5.5,-6) -- (-0.5,-6);
        \node[black] at (-2.5,-5.3) {\small $\hat{B}_1$};
        \node[green!60!black] at (-3.8,-3.8) {$D_0$};
        \draw[green!20!black,thin] (-1.9,-5.5) -- (-1.25,-5.5);
        \node[green!60!black] at (-0.9,-5.5) {$D_{-1}$};
        \node[red] at (-3,-6.02) {\large $\bullet$};
        \node[red] at (-3.03,-6.4) {$c_{-t}$};
\end{tikzpicture}

    \caption{The construction of the map $f^p: D_{-1} \to D_0$ in the first part of the proof of Lemma \ref{lem:construction-of-periodic-bubble}.}
    \label{fig:outer-repelling-periodic-point}
\end{figure}

For each pre-critical point $c_{-t}(g)$ of $g$, consider the objects $\mathcal{B}_t^\infty$, $\mathcal{B}_t^0$, $R_t^\infty$, $R_t^0$ from the previous lemma. Denote
\begin{equation}
\label{eqn:separator}
    \mathcal{T}_{t} :=   \mathcal{B}_t \cup R_t \qquad \text{where} \qquad \mathcal{B}_t:= \mathcal{B}_t^\infty \cup \mathcal{B}_t^0 \text{  and  } R_t := R_t^\infty \cup R_t^0.
\end{equation}

\begin{lemma}[Rational approximation of $\gamma_1(g)$]
\label{lem:rational-approx-of-rays}
    For every $\varepsilon>0$, there exists a pair of pre-critical points $c_{-t_l}, c_{-t_r} \in \Hq_g$ located on the left and right of $x$ respectively such that $\mathcal{T}_{t_l}$ and $\mathcal{T}_{t_r}$ are both in the $\varepsilon$-neighborhood of $\gamma_1(g)$.
\end{lemma}

\begin{proof}
    Since pre-critical points are dense on $\Hq_g$, there exists a pair of pre-critical points $c_{-t_l}$ and $c_{-t_r}$ on the left and right of $x$, where the moments $t_l$ and $t_r$ grow as we require them to be arbitrarily close to $x$. By Lemma \ref{lem:small-limbs}, the partial bubble chains within $\mathcal{T}_{t_l}$ and $\mathcal{T}_{t_r}$ shrink as we get close to $x$. The outer (resp. inner) external rays within $\mathcal{T}_{t_l}$ and $\mathcal{T}_{t_r}$ are also close to the external ray $R^\infty$ (resp. $R^0$) of $g$ landing at $x$ because their external angles are close to that of $R^\infty$.
\end{proof}

We are now ready to run the pullback argument.

\begin{proof}[proof of Theorem \ref{thm:qc-rigidity}]
    Let $g_1: U_1 \to V_1$ and $g_2: U_2 \to V_2$ be two combinatorially equivalent $(d_0,d_\infty)$-critical rotational coronas with rotation number $\theta \in \IrratBdd$. Let $f$ be the rational map from Theorem \ref{thm:comb-rigidity} which admits a $(d_0,d_\infty)$-critical Herman quasicircle $\Hq$ with rotation number $\theta$. From the previous discussion, for $i \in \{1,2\}$, there is a quasiconformal conjugacy $\phi_i$ between $g_i$ and $f$ on some neighborhood $W_i$ of the Herman quasicircle $\Hq_i$ of $g_i$ onto a neighborhood $W$ of $\Hq$. 
    
    Fix a pre-corona $P=(f_\pm : U_\pm \to S)$ of $f$ where $S$ is contained in $W$, and for $i \in \{1,2\}$, let $P_i=(g_\pm : U_{i,\pm} \to S_i)$ be the corresponding pre-corona of $g_i$ conjugate to $P$ via $\phi_i$. Consider the local non-escaping set $K^{loc}(g_i)$ of $g_i$ relative to $P_i$. The quasiconformal map $\phi_2 \circ \phi_1^{-1}:W_1 \to W_2$ restricts to a conjugacy $h: K^{loc}(g_1) \to K^{loc}(g_2)$ between $g_1$ and $g_2$.
    
    For $i \in \{1,2\}$ and $t \in \{t_l,t_r\}$, consider the sets $\mathcal{T}_{t}(g_i) = \mathcal{B}_{t}(g_i) \cup R_t(g_i)$ from Lemma \ref{lem:rational-approx-of-rays} which approximate the critical arc $\gamma_1(g)$. By design, we can arrange such that for each $t \in \{t_l,t_r\}$, $\phi_2 \circ \phi_1^{-1}$ sends $\mathcal{B}_t(g_1)$ to $\mathcal{B}_t(g_2)$, and the outer/inner rays in $R_t(g_1)$ and $R_t(g_2)$ have the same external angles. For $i\in \{1,2\}$, consider the union 
    \[
        Z_i = K^{loc}(g_i) \cup \bigcup_{n\geq 0}g_i^n(R_{t_l}(g_i) \cup R_{t_r}(g_i)).
    \]
    Clearly, $Z_i$ is forward invariant and $V_i \backslash Z_i$ consists of finitely many connected components. Since $R_t(g_1)$ and $R_t(g_2)$ have the same external angles, $h$ extends to a quasiconformal map $h: V_1 \to V_2$ that is equivariant on $Z_i \cup \partial_{\lgt} U_1$.

    Let us define a new domain $\hat{U}_1$ out of $U_1$ by replacing the forbidden boundary $\partial_\frb U_1$ with some set $\partial_\frb \hat{U}_1$ of curves slightly outside of $\partial_\frb U_1$ such that the image $g_1(\partial_\frb \hat{U}_1)$ is now contained inside of $\Hq_1 \cup \mathcal{T}_{t_l}(g_1) \cup \mathcal{T}_{t_r}(g_1)$. In the same manner, we replace $U_2$ with a slightly larger disk $\hat{U}_2$ such that $h|_{Z_1}$ lifts to a conjugacy between $g_1|_{\partial \hat{U}_1}$ and $g_2|_{\partial \hat{U}_2}$. 
    
   Set $h_0 := h$. Inductively, construct quasiconformal maps $h_n : V_1 \to V_2$ such that
    \[
        h_n(z) = \begin{cases}
            h_{n-1}(z), & \text{ if } z \not\in \hat{U}_1,\\
            g_2^{-1} \circ h_{n-1} \circ g_1(z), & \text{ if } z\in \hat{U}_1.
        \end{cases}
    \]
    Each $h_n$ has the same dilatation as $h$. Since $K^{loc}(g_1)$ is nowhere dense, $h_n$ stabilizes and converges to a quasiconformal conjugacy between $g_1$ and $g_2$.
\end{proof}

%%%%%%%%%%%%%%%%%%%%%%%%%%%%%%%%%%%%%%%%%%%%%%%%%%%%%%%%%%%%%%%%%%%%%%%%%%%%%%%%%%%%%%%%%%%%%%%%%%%%%%%%%%%%%%%%%%%%%%%%%%%%%%%%%%%%%%%%%%%%%%%%%%%%%%%%%%%%%%%%%%%%%%%%%%%%%%%%%%%%

\section{Hyperbolic renormalization fixed point}
\label{sec:stab-mani}

From now on, let us fix a periodic type irrational $\theta \in \IrratPer$. We will now construct the desired corona renormalization fixed point $f_*$ and prove most of Theorem \ref{main-theorem}. 

\subsection{Renormalization of critical commuting pairs}
\label{ss:critical-commuting-pairs}

Let us consider a $(d_0,d_\infty)$-critical quasicircle map $f: \Hq \to \Hq$ with critical point $c$ and rotation number $\tau$. Let us denote by $\{p_n/q_n\}_{n\in \N}$ the best rational approximations of $\tau$. For every $n \geq 2$, denote by $I_{n}$ the shortest interval in $\Hq$ connecting $c$ and $f^{q_n}(c)$. The \emph{$n$\textsuperscript{th} pre-renormalization} of $f$ is the pair of maps
\[
(f^{q_n}|_{I_{n-1}}, f^{q_{n-1}}|_{I_n})
\]
and the \emph{$n$\textsuperscript{th}} (\emph{commuting pair}) \emph{renormalization} $\Rcp^n f$ of $f$ is the commuting pair obtained by rescaling of the $n$\textsuperscript{th} pre-renormalization by either the affine map if $n$ is even, or the anti-affine map if $n$ is odd, that sends $0$ to $c$ and $-1$ to $f^{q_{n-1}}(c)$. Each renormalization $\Rcp^n f$ is what we call a normalized $(d_0,d_\infty)$-critical commuting pair. Below, we will formally define what this terminology means.

Let $\mathbb{H}$ and $-\mathbb{H}$ denote the upper and lower half planes of $\C$ respectively.

\begin{definition}
\label{def:com-pair}
    Let $\mathbf{I} \Subset \C$ be a closed quasiarc containing $0$ on its interior. A \emph{commuting pair} $\zeta$ based on $\mathbf{I}$ is a pair of orientation preserving analytic homeomorphisms 
    \[
    \zeta = \left(f_-: I_- \to f_-(I_-), \: 
    f_+: I_+ \to f_+(I_+) \right)
    \]
    with the following properties.
    \begin{enumerate}[label = (P\textsubscript{\arabic*})]
        \item\label{P-1} $I_-$ and $I_+$ are closed subintervals of $\mathbf{I}$ of the form $[f_+(0),0]$ and $[0,f_-(0)]$ respectively such that $\mathbf{I} = I_- \cup I_+ = f_-(I_-) \cup f_+(I_+)$ and $I_- \cap I_+ = \{0\}$.
        \item\label{P-2} For all $x \in I_\pm \backslash \{0\}$, $f_\pm'(x) \neq 0$.
        \item\label{P-3} Both $f_-$ and $f_+$ admit holomorphic extensions to a neighborhood $B$ of $0$ on which $f_-$ commutes with $f_+$ and $f_- \circ f_+ (\mathbf{I} \cap B) \subset I_-$.
    \end{enumerate}
    Additionally, a commuting pair $\zeta$ is a \emph{critical commuting pair} if
    \begin{enumerate}[label = (P\textsubscript{\arabic*}), start=4]
        \item\label{P-4} $0$ is a critical point of both $f_-$ and $f_+$.
    \end{enumerate}
    We say that $\zeta$ is \emph{normalized} if $f_+(0) = -1$. A critical commuting pair $\zeta$ is called a \emph{$(d_0,d_\infty)$-critical commuting pair} if for any quasiconformal map $\phi$ mapping $I_-$ and $I_+$ to real intervals $[-1,0]$ and $[0,1]$ respectively and for any sufficiently small round disk $D$ centered at $\phi(f_+(f_-(0)))$, the number of connected components of $\phi (f_+ \circ f_-)^{-1} \phi^{-1}(D \cap -\mathbb{H})$ in $-\He$ is $d_\infty$, whereas the number of connected components of $\phi (f_+ \circ f_-)^{-1} \phi^{-1}(D \cap \He)$ in $\mathbb{H}$ is $d_0$.
\end{definition}

We say that a $(d_0,d_\infty)$-critical commuting pair $\zeta = (f_-,f_+)$ is \emph{renormalizable} if there exists a positive integer $\chi=\chi(\zeta)$ that corresponds to the first time $f_-^{\chi+1}\circ f_+(0)$ lies in the interior of $I_+$. If renormalizable, we call the $(d_\infty,d_0)$-critical commuting pair 
    \[
        p\Rcp \zeta := 
        \left(
        f_-^\chi \circ f_+|_{[0,f_-(0)]}, \:
        f_-|_{[f_-^\chi f_+(0),0]}
        \right)
    \]
    the \emph{pre-renormalization} of $\zeta$, and we call the normalized $(d_0,d_\infty)$-critical commuting pair obtained by conjugating $p\Rcp \zeta$ with the antilinear map $z \mapsto -f_-(0) \bar{z}$ the \emph{renormalization} of $\Rcp \zeta$ of $\zeta$. 
    
    If $\Rcp \zeta$ is again renormalizable, we call $\zeta$ twice renormalizable, and so on. If $\zeta$ is infinitely renormalizable, we define the \emph{rotation number} of $\zeta$ to be the irrational number
    \[
        \text{rot}(\zeta) := [0; \chi(\zeta), \chi(\Rcp \zeta), \chi(\Rcp^2 \zeta), \ldots].
    \]
The operator $\Rcp$ acts on the rotation number as follows. If $\zeta$ is $n$ times renormalizable,
\[
\text{rot}(\Rcp^n \zeta) = G^n(\text{rot}(\zeta)),
\]
where $G(\tau) = \left\{\frac{1}{\tau}\right\}$ is the Gauss map.
    
One can convert a $(d_0,d_\infty)$-critical commuting pair $\zeta = (f_-,f_+)$ into a $(d_0,d_\infty)$-critical quasicircle map by taking the dynamics modulo $f_+$.

\begin{proposition}
    \label{prop:gluing-crit-circ-maps}
    Let $\zeta= (f_-|_{I_-}, f_+|_{I_+})$ be a commuting pair. Let $G_\zeta$ be the conformal gluing map from a neighborhood of $[f_+f_-(0),f_-(0)]$ to an annulus in $\C$ that identifies $z$ with $f_+(z)$ for every point $z$ in a neighborhood of $f_-(0)$. Then, $G_\zeta$ projects the pair $(f_-|_{[f_+f_-(0),0]}, f_+f_-|_{[0,f_-(0)]})$ into a quasicircle map $f_\zeta: \Hq \to \Hq$ having the same rotation number as $\zeta$. If $\zeta$ is $(d_0,d_\infty)$-critical, then $f_\zeta: \Hq \to \Hq$ is a $(d_0,d_\infty)$-critical quasicircle map.
\end{proposition}

By studying the rigidity properties of critical commuting pairs, we have previously established the following result.

\begin{theorem}[{\cite[\S7.3-7.5]{Lim23b}}] 
\label{horseshoe}
    Let $p \geq 2$ be the smallest even period of $\theta$ under the Gauss map $G(\theta)=\left\{ \frac{1}{\theta} \right\}$.
    There is a unique normalized $(d_0,d_\infty)$-critical commuting pair $\zeta_*$ with rotation number $\theta$ with the following properties. 
    \begin{enumerate}[label=\textnormal{(\arabic*)}]
        \item Renormalization fixed point: There is a linear map $z \mapsto \mu z$, $|\mu|<1$, which conjugates $\zeta_*$ and the pre-renormalization $p\Rcp^p\zeta_*$.
        \item Exponential convergence: For any normalized $(d_0,d_\infty)$-critical commuting pair $\zeta$ of some rotation number $\tau \in \IrratPre$ where $G^k(\tau)=\theta$ for some $k \in \N$, the renormalizations $\Rcp^{k+np} \zeta$ converge exponentially to $\zeta_*$ as $n \to \infty$.
    \end{enumerate}
\end{theorem}

Let us clarify what we mean in the second item of the theorem. We say that a sequence of commuting pairs $\zeta_n = (f_{n,-}|_{I_{n,-}},f_{n,+}|_{I_{n,+}})$ converges exponentially to $\zeta_* = (f_{-}|_{I_{-}},f_{+}|_{I_{+}})$ if the following two properties are satisfied. Firstly, the Hausdorff distance between $I_{n,\pm}$ and $I_{*,\pm}$ tends to $0$ exponentially fast. Secondly, for sufficiently large $n$, both $f_{n,\pm}$ and $f_{\pm}$ extend holomorphically to a neighborhood of $I_{n,\pm} \cup I_{\pm}$ on which the sup norm of $f_{n,\pm}-f_{\pm}$ converges exponentially fast to $0$.

\subsection{Corona renormalization fixed point}

We say that a rotational corona is \emph{standard} if the arc $\gamma_0$ passes through the critical value. 
Similarly, we say that a rotational pre-corona is \emph{standard} if the corresponding gluing map projects it to a standard rotational corona.

\begin{theorem}
\label{thm:renorm-fixed-pt}
    There exists a standard $(d_0,d_\infty)$-critical rotational corona $f_* : U_* \to V_*$ with rotation number $\theta$ which admits a standard rotational pre-corona
    \[
    F_* = \left(f_*^a: U_- \to S_*, \: f_*^b: U_+ \to S_*\right)
    \]
    together with a gluing map $\psi_*: S_* \to \overline{V_*}$ projecting $F_*$ back to $f_*: U_* \to V_*$. Moreover, we have an improvement of domain: $\mathbf{\Delta}_{F_*} \Subset U_*$.
\end{theorem}

\begin{proof}
    Consider the unique normalized $(d_0,d_\infty)$-critical commuting pair 
    \[
    \zeta_* = (f_-: I_- \to I, f_+: I_+ \to I)
    \]
    with rotation number $\theta$ from Theorem \ref{horseshoe}. The pair is supported on the quasiarc $I = I_- \cup I_+ = [f_+(0),0] \cup [0,f_-(0)]$. There exists some $\mu \in \D$ such that for any $n \in \N$, there is a pre-renormalization $\zeta_n = (f_{n,-} : J_- \to J, f_{n,+}: J_+ \to J)$ of $\zeta_*$ on a subinterval $J \subset I$ that is conjugate to $\zeta_*$ via the linear map $L^n(z) = \mu^n z$. We will convert this renormalization fixed point in the category of commuting pairs to that in the category of critical quasicircle maps, and then project it to that in the category of rotational coronas.

    Consider the gluing map $\phi_1:= G_{\zeta_*}$ described in Proposition \ref{prop:gluing-crit-circ-maps}. Then, $\phi_1$ projects the modified commuting pair $\zeta'_* := \left( f_-|_{ [f_+f_-(0),0] }, f_+f_-|_{[0,f_-(0)]},  \right)$ into a $(d_0,d_\infty)$-critical quasicircle map $g: \Hq \to \Hq$ having the same rotation number $\theta$. 
    
    Denote by $c_0:= \phi_1(0)$ the critical point of $g$, and let $c_k := g^k(c_0)$ for all $k \in \N$. Consider the modification of $\zeta_n$, which is $\zeta'_*$ rescaled by $L^n$, and project it to the dynamical plane of $g$ via $\phi_1$ to obtain a commuting pair $g_n = \left( g^\abold|_{[c_\bbold,c_0]}, g^\bbold|_{[c_0,c_\abold]} \right)$ for some return times $\abold$ and $\bbold$. Then, $\psi_1 := \phi_1 L^n \phi_1^{-1}$ is the gluing map projecting $g_n$ back to $g$.

    To make the resulting corona standard, we will push $g_n$ forward under one iterate of $g$. More precisely, we set $\psi_2 := g \circ \psi_1 \circ g^{-1}$. It is well-defined because for every point $z$ close to $c_1$, the preimage $g^{-1}(z)$ is a set of $d_0+d_\infty-1$ points close to $c_0$ whose images under $\psi_1$ remain close to $c_0$ and get mapped to the same point $\psi_2(z)$ under $g$. The new gluing map $\psi_2$ sends a small neighborhood of $c_1$ to a neighborhood of $\Hq$. Moreover, $\psi_2$ fixes the critical value $c_1$ and projects $\tilde{g}_n = \left( g^\abold|_{[c_{\bbold+1},c_1]}, g^\bbold|_{[c_1,c_{\abold+1}]} \right)$ back to $g$.

    By Corollary \ref{renormalizability-of-rotational}, $g$ admits a standard pre-corona $P$ defined in a small neighborhood of $c_1$. The corresponding gluing map $\phi_2$ projects $P$ onto a $(d_0,d_\infty)$-critical rotational corona $f_* : U_* \to V_*$. Since $\theta$ is periodic, we can prescribe $f_*$ to have rotation number $\theta$. The corresponding Herman quasicircle $\Hq_*$ of $f_*$ is the image of (an interval in) $\Hq$ under $\phi_2$.
    
    Let us rescale the pre-corona $P$ by $\psi_2^{-1}$ to obtain yet another pre-corona $P'$ in the dynamical plane of $g$ that is much smaller than $P$. Project $P'$ via $\phi_2$ to obtain a pre-corona $F_*$ of $f_*$. The map $\psi_* := \phi_2 \circ \psi_2 \circ \phi_2^{-1}$ will project the pre-corona $F_*$ back to $f_*$. The improvement of domain property is satisfied once we take $n$ to be sufficiently high.
\end{proof}

\begin{corollary}
\label{cor:renorm-operator}
    Let $f_*$, $F_*$, and $\psi_*$ be from the previous theorem. There exist a pair of small Banach neighborhoods $\Ustar$ and $\Bstar$ of $f_*$ and a compact analytic corona renormalization operator $\Rstar : \Ustar \to \Bstar$ such that 
    \begin{enumerate}[label=\textnormal{(\arabic*)}]
        \item $\Rstar f_* = f_*$ and $F_*$ is the pre-renormalization of $\Rstar f_*$ with gluing map $\psi_*$;
        \item for any rotational corona $f$ in $\Ustar$ with the same rotation number $\theta$, $f$ is infinitely renormalizable and $\Rstar^n f$ converges exponentially fast to $f_*$ with respect to the ambient sup norm of $\Ustar$.
    \end{enumerate}
\end{corollary}

\begin{proof}
    The existence of $\Rstar: \Ustar \to \Bstar$ satisfying (1) follows from Theorems \ref{thm:renormalization} and \ref{thm:renorm-fixed-pt}. Item (2) is guaranteed by Theorem \ref{horseshoe} provided that $\Ustar$ is a sufficiently small neighborhood of $f_*$.
\end{proof}

From now on, the corona renormalization operator 
\[
    \Rstar : (\Ustar, f_*) \to (\Bstar, f_*)
\]
constructed above will be fixed. Sometimes, we may replace $\Ustar$ with an even smaller Banach neighborhood of $f_*$. Exactly once (Lemma \ref{lem:renormalization-tiling-01}), we may replace $\Rstar$ with an iterate $\Rstar^k$.

\begin{lemma}
\label{lem:renorm-of-any-critical rotational}
    For any $(d_0,d_\infty)$-critical quasicircle map $f$ of rotation number $\tau \in \IrratPre$ where $G^k(\tau) = \theta$ for some $k \in \N$, there is a compact analytic corona renormalization operator $\renorm: \mathcal{N} \to \Ustar$ from a Banach neighborhood $\mathcal{N}$ of $f$ to the Banach neighborhood $\Ustar$ of $f_*$.
\end{lemma}

\begin{proof}
    By Theorem \ref{horseshoe}, there is a high $m \in \N$ such that $\Rcp^m f$ is a critical commuting pair with rotation number $\theta$ that is arbitrarily close to the pair $\zeta_*$. By Theorem \ref{thm:qc-rigidity}, $f$ then admits a rotational pre-corona $F$ which projects to a rotational corona $g$ with rotation number $\theta$ that is close to the corona $f_*$ in Corollary \ref{cor:renorm-operator}. By Theorem \ref{thm:renormalization}, there is a compact analytic renormalization operator $\renorm$ on a small neighborhood of $f$ such that $\renorm (f) = g$.
\end{proof}

\subsection{Renormalization tiling}
\label{ss:renorm-tiling}

Recall from Theorem \ref{thm:renormalization} that every corona $f$ in $\Ustar$ has the same codomain $V$ and critical arc $\gamma_1$ as the renormalization fixed point $f_*$.

Let us pick a positive integer $n$ and a corona $f=f_0$ in the neighborhood
\[
    \mathcal{U}_n := \bigcap_{0\leq k \leq n} \Rstar^{-k}(\Ustar)
\]
of $f_*$. In particular, $f$ is $n$ times renormalizable. For $k \leq n$, denote by 
\begin{itemize}
    \item $f_k := \Rstar^k f=[ f_k : U_k \to V]$ the $k$\textsuperscript{th} renormalization of $f$, 
    \item $\psi_k : S_k \to V$ the renormalization change of variables for $f_{k-1}$, and
    \item $\phi_k := \psi_k^{-1}$.
\end{itemize}

For $k \in \{0,1,\ldots,n\}$, let us cut the dynamical plane of $f_k$ along the critical arc $\gamma_1$ and obtain a pre-corona
\[
F_k = (f_{k,\pm} : U_{k,\pm} \to V\backslash \gamma_1 ).
\]
The map
\[
\Phi_n := \phi_1 \circ \phi_2 \circ \ldots \phi_n
\]
is well defined on $V \backslash \gamma_1$ and projects $F_n$ to the dynamical plane of $f$ as the pre-corona
\[
    F^{(0)}_n = \left( f^{(0)}_{n,\pm} : U^{(0)}_{n,\pm} \to S^{(0)}_n \right) \quad \text{ where } \quad f^{(0)}_{n,-} = f^{\abold_n}_0 \text{ and } f^{(0)}_{n,+} = f^{\bbold_n}_0
\]
for some return times $\abold_n$ and $\bbold_n$. Let us also set $\Phi_0 := \text{Id}$ on $V\backslash \gamma_1$. 

Let us divide $\overline{U_0}$ along the arcs $\gamma_0$ and $ \gamma_1$ to obtain a tiling $\Deltabold_0$ of $\overline{U_0}$ consisting of two tiles $\Delta_0(0)$ and $\Delta_0(1)$. We make the convention that $\Delta_0(0)$, $\gamma_0$, and $\Delta_0(1)$ are in counterclockwise order. The tiling $\Deltabold_0$ is called the \emph{zeroth tiling} associated to $f_0$.

Next, define the \emph{$n$\textsuperscript{th} tiling} $\Deltabold_n$ associated to $f$ by spreading around $U^{(0)}_{n,\pm}$ via $f$. It consists of $f^i \left( U^{(0)}_{n,-} \right)$ for $i\in\{0,1\ldots,\abold_n-1\}$ and $f^j \left( U^{(0)}_{n,+} \right)$ for $j\in \{0,1\ldots,\bbold_n-1\}$. Let us denote by $\Delta_n(0)$ the image of the zeroth tile $\Delta_0(0,f_n)$ of $f_n$ under $\Phi_n$, label the rest of the tiles in $\Deltabold_n$ in counterclockwise order by $\Delta_n(i)$ for $i\in \{0,1,\ldots, \abold_n + \bbold_n-1\}$. See Figure \ref{fig:construction-of-tiling}.

If $f$ is rotational, then $\Deltabold_n$ always forms an annular neighborhood of the Herman quasicircle of $f$. In general, the map $f$ always acts almost like a rotation on the tiling $\Deltabold_n$. There exists $\pbold_n \in \N_{\geq 1}$ such that $f$ maps $\Delta_n(i)$ univalently onto $\Delta_n(i+\pbold_n)$ whenever $i \not\in \{-\pbold_n, -\pbold_n+1 \}$. Moreover, $f$ maps $\Delta_n(-\pbold_n) \cup \Delta_n(-\pbold_n+1)$ back to $S^{(0)}_n$ almost as a degree $d$ covering map branched at its critical point $c_0(f)$.

\begin{figure}
\hspace*{-0.1in}
\begin{tikzpicture}[scale=0.54]
\coordinate (v1) at (1,-0.45) {};
\coordinate (v2) at (1,0.45) {};
\coordinate (v3) at (-1,0.5) {};
\coordinate (v4) at (-1,-0.5) {};

\coordinate (w1) at (4.25,-1.5) {};
\coordinate (w2) at (4.5,-0.5) {};
\coordinate (w3) at (4.5,0.5) {};
\coordinate (w4) at (4.25,1.5) {};
\coordinate (w5) at (-3,3.5) {};
\coordinate (w6) at (-3,-3.5) {};

\coordinate (a1) at (1.6,-1.1) {};
\coordinate (a2) at (1.8,-0.3) {};
\coordinate (a3) at (1.8,0.3) {};
\coordinate (a4) at (1.55,1) {};
\coordinate (a5) at (0.8,1.5) {};
\coordinate (a6) at (0.2,1.6) {};
\coordinate (a7) at (-0.9,1.55) {};
\coordinate (a8) at (-2.1,0.2) {};
\coordinate (a9) at (-1.7,-1) {};
\coordinate (a10) at (-1.3,-1.3) {};
\coordinate (a11) at (-0.3,-1.8) {};

\coordinate (b1) at (2.7,-2) {};
\coordinate (b2) at (3.2,-1) {};
\coordinate (b3) at (3.3,-0.4) {};
\coordinate (b4) at (3.3,0.4) {};
\coordinate (b5) at (3.2,1) {};
\coordinate (b6) at (2.6,2) {};
\coordinate (b7) at (1.8,2.6) {};
\coordinate (b8) at (1.1,2.8) {};
\coordinate (b9) at (0.4,2.9) {};
\coordinate (b10) at (-0.2,3) {};
\coordinate (b11) at (-1.4,2.8) {};
\coordinate (b12) at (-3.5,0.4) {};
\coordinate (b13) at (-3.4,-0.9) {};
\coordinate (b14) at (-3,-1.4) {};
\coordinate (b15) at (-2.6,-2.2) {};
\coordinate (b16) at (-2,-2.6) {};
\coordinate (b17) at (-0.6,-3) {};

\filldraw[fill=blue!10!white, shift={(13 cm,0 cm)}] (4.25,-1.5) decorate [decoration={zigzag,segment length = 2mm,amplitude = 0.7mm}] { -- (4.5,-0.5)} .. controls (4,-0.4) and (4,0.4) .. (4.5,0.5) decorate [decoration={zigzag,segment length = 1.5mm,amplitude = 0.7mm}] { -- (4.25,1.5)} .. controls (3,4) and (-1,4.5) .. (-3,3.5) .. controls (-5.5,2) and (-5.5,-2) .. (-3,-3.5) .. controls (-1,-4.5) and (3,-4) .. (4.25,-1.5);
\filldraw[fill=white, shift={(13 cm,0 cm)}] (1,-0.45) decorate [decoration={zigzag,segment length = 1.5mm,amplitude = 0.5mm}] { -- (1,0.45)} .. controls (0.5,1) and (-0.5,1) .. (-1,0.5) .. controls (-1.2,0.25) and (-1.2,-0.25) .. (-1,-0.5) .. controls (-0.5,-1) and (0.5,-1) .. (1,-0.45);
\draw[red, shift={(13 cm,0 cm)}] (-0.15,-0.9) -- (-0.3,-1.8) -- (-0.6,-3) -- (-0.75,-4);
\draw[red, shift={(13 cm,0 cm)}] (-0.6,0.77) -- (-1,1.7) -- (-1.4,2.7) -- (-1.9,3.88);

% edges
\filldraw[fill=yellow!10!white] (w1) decorate [decoration={zigzag,segment length = 2mm,amplitude = 0.7mm}] { -- (w2)} .. controls (4,-0.4) and (4,0.4) .. (w3) decorate [decoration={zigzag,segment length = 1.5mm,amplitude = 0.7mm}] { -- (w4)} .. controls (3,4) and (-1,4.5) .. (w5) .. controls (-5.5,2) and (-5.5,-2) .. (w6) .. controls (-1,-4.5) and (3,-4) .. (w1);
\filldraw[green!50!black, fill=green!10!white] (b1) -- (b2) decorate [decoration={zigzag,segment length = 1.5mm,amplitude = 0.7mm}] { -- (b3)} -- (b4) decorate [decoration={zigzag,segment length = 1.5mm,amplitude = 0.7mm}] { -- (b5)} -- (b6) -- (b7) decorate [decoration={zigzag,segment length = 1.5mm,amplitude = 0.7mm}] { -- (b8)} -- (b9) decorate [decoration={zigzag,segment length = 1.5mm,amplitude = 0.7mm}] { -- (b10)} -- (b11) .. controls (-2.2,2.6) and (-3.4,1.5) .. (b12) -- (b13) decorate [decoration={zigzag,segment length = 1.5mm,amplitude = 0.7mm}] { -- (b14)} -- (b15) decorate [decoration={zigzag,segment length = 1.5mm,amplitude = 0.7mm}] { -- (b16)} -- (b17) .. controls (0.1,-3.2) and (2.3,-2.4) .. (b1);
\filldraw[green!50!black, fill=yellow!10!white] (a1) -- (a2) decorate [decoration={zigzag,segment length = 1.5mm,amplitude = 0.7mm}] { -- (a3)} -- (a4) -- (a5) decorate [decoration={zigzag,segment length = 1.5mm,amplitude = 0.7mm}] { -- (a6)} -- (a7) .. controls (-1.55,1.4) and (-2,0.75) .. (a8) -- (a9) decorate [decoration={zigzag,segment length = 1.5mm,amplitude = 0.7mm}] { -- (a10)} -- (a11) .. controls (0.2,-1.8) and (1.5,-1.2).. (a1);
\filldraw[fill=white] (v1) decorate [decoration={zigzag,segment length = 1.5mm,amplitude = 0.5mm}] { -- (v2)} .. controls (0.5,1) and (-0.5,1) .. (v3) .. controls (-1.2,0.25) and (-1.2,-0.25) .. (v4) .. controls (-0.5,-1) and (0.5,-1) .. (v1);

\filldraw[blue!50!black, thick, fill=blue!10!white] (b6) -- (b7) decorate [decoration={zigzag,segment length = 1.5mm,amplitude = 0.7mm}] { -- (b8)} -- (b9) decorate [decoration={zigzag,segment length = 1.5mm,amplitude = 0.7mm}] { -- (b10)} -- (b11) .. controls (-2.2,2.6) and (-3.4,1.5) .. (b12) -- (a8) .. controls (-2,0.75)  and (-1.55,1.4) .. (a7) -- (a6) decorate [decoration={zigzag,segment length = 1.5mm,amplitude = 0.7mm}] { -- (a5)} -- (a4) -- (b6);

\draw[green!50!black] (a1) -- (b1);
\draw[green!50!black] (a4) -- (b6);
\draw[green!50!black] (a7) -- (b11);
\draw[green!50!black] (a8) -- (b12);
\draw[red] (-0.15,-0.9) -- (a11) -- (b17) -- (-0.75,-4);
\draw[red] (-0.6,0.77) -- (a7) -- (b11) -- (-1.9,3.88);

% labels
\node[red] at (-2, 4.4) {\small $\gamma_0(f)$};
\node[red] at (-0.85,-4.4) {\small $\gamma_1$};
\node[blue!50!black] at (0.6,2.15) {\footnotesize $\Delta_1(0)$};
\node[blue!50!black] at (-2.2,1.5) {\footnotesize $\Delta_1(1)$};
\node[green!50!black] at (-2.1,-1.6) {\footnotesize $\Delta_1(2)$};
\node[green!50!black] at (0.9,-2.1) {\footnotesize $\Delta_1(3)$};
\node[green!50!black] at (2.6,0) {\footnotesize $\Delta_1(4)$};

\node[blue!50!black] at (10,-0.5) {\small $\Delta_0(1,f_1)$};
\node[blue!50!black] at (14.8,1.5) {\small $\Delta_0(0,f_1)$};
\node[red] at (11.3, 4.4) {\small $\gamma_0(f_1)$};
\node[red] at (12.2,-4.4) {\small $\gamma_1$};

\draw[-latex] (7.5,0) -- (5,0);
\node[black] at (6.3,-0.5) {\small $\phi_1$};
\draw[-latex] (2.5,-2.75) .. controls (2.55,-5.25) and (5.25,-3.75) .. (3.75,-2.75);
\node[black] at (4.3,-4.5) {\small $f$};
\draw[-latex] (15.5,-2.75) .. controls (15.55,-5.25) and (18.25,-3.75) .. (16.75,-2.75);
\node[black] at (17.2,-4.5) {\small $f_1$};
\end{tikzpicture}
\caption{The construction of the first tiling $\Deltabold_1$ of $f$ when $(\abold_n,\bbold_n) = (3,2)$.}
\label{fig:construction-of-tiling}
\end{figure}
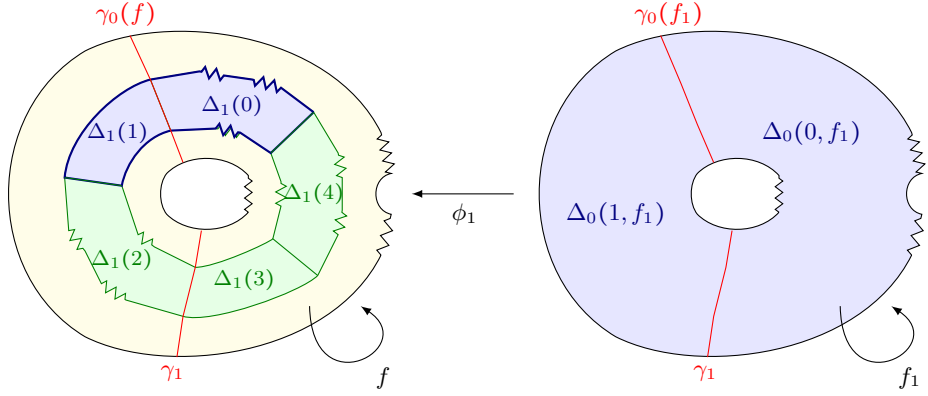

\begin{lemma}
\label{lem:renormalization-tiling-01}
    The operator $\Rstar: \Ustar \to \Bstar$ can be arranged such that the following holds. For $f \in \mathcal{U}_n$,
    \begin{enumerate}[label=\textnormal{(\arabic*)}]
        \item there is a holomorphic motion of $ \partial \Deltabold_0, \ldots,\partial \Deltabold_n$ over $f \in \mathcal{U}_n$ that is equivariant with respect to $f: \partial \Delta_n(i) \to \partial \Delta_n(i+\pbold_n)$ for $i \not\in \{-\pbold_n, -\pbold_n+1\}$;
        \item for every $f \in \mathcal{U}_n$ and $1\leq k \leq n$, $\Deltabold_{m} \cup f(\Deltabold_{m}) \Subset \Deltabold_{m-1}$;
        \item the tiling $\Deltabold_n(f)$ is close to the Herman quasicircle $\Hq_*$ of $f_*$ in Hausdorff topology.
    \end{enumerate}
\end{lemma}

\begin{proof}
    Let us first consider the case where $f=f_*$. By the improvement of domain property in Theorem \ref{thm:renorm-fixed-pt}, the diameters of the tiles in $\Deltabold_n(f_*)$ must shrink to $0$ as $n \to \infty$. Consider a tile $\Delta_1(i, f_*)$. There is some $t\geq 0$ and $j \in \{0,1\}$ such that $f_*^t$ sends $\Delta_1(j,f_*)$ onto $\Delta_1(i,f_*)$. By replacing $\Rstar$ with some high iterate $\Rstar^k$ if necessary, the map 
    \[
        \psi_* \circ f_*^{-t}: \Delta_1(i, f_*) \to \Delta_0(j, f_*)
    \]
    expands the Euclidean metric by some high factor $C>1$. Inductively, (2) and (3) hold for $f_*$.

    By design, it is clear that $\partial \Deltabold_0$ moves holomorphically over $f \in \Ustar$. For $1\leq k \leq n$, we push forward the holomorphic motion $\partial \Deltabold_0(f_k)$ via $\Phi_k$ and spread it around dynamically to obtain a holomorphic motion of $\partial \Deltabold_k(f)$ over $f \in \mathcal{U}_n$.

    By continuity, every $f \in \mathcal{U}_n$ also satisfies the following property. For any tile $\Delta_{n}(i, f)$ in $\Deltabold_{n}(f)$, there is some $t\geq 0$ and $j \in \{0,1\}$ such that $f^t$ sends $\Delta_{n}(j,f)$ onto $\Delta_{n}(i,f)$. We obtain a holomorphic motion of $\partial \Deltabold_n(f)$ by pulling back the holomorphic motion of $\partial \Deltabold_0(f_n)$ via maps of the form
    \begin{equation}
        \label{eqn:big-psi-tiling}
        \Psi_{n,i} := \Phi_n^{-1} \circ f^{-t} : \Delta_{n}(i, f) \to \Delta_{0}(j, f_{n})
    \end{equation}
    for each tile. This implies (1). Moreover, (2) follows from the observation that each $\Psi_{n,i}$ expands the Euclidean metric by a factor close to $C^n$. Moreover, (3) follows from (1) as well as the special case of (3) for $f=f_*$. 
\end{proof}

Let us extend the tiling $\Deltabold_n$ of a subset of $\overline{U_0}$ to a full tiling of $\overline{U_0}$ as follows. Consider
\[
\hat{\gamma}_0 := \gamma_0 \backslash f^{-1}(U_0) \quad \text{and} \quad \Gamma := \partial U_0 \cup \hat{\gamma}_0.
\]
Observe that $\hat{\gamma}_0$ is a disjoint union of two subarcs $\hat{\gamma}_0^0$ and $\hat{\gamma}_0^\infty$ of $\gamma_0$ where each $\hat{\gamma}_0^\bullet$ connects the boundary component $\partial^\bullet U_0$ to $f^{-1}(U_0)$. Consider the maps $\Psi_{n,i}$ from (\ref{eqn:big-psi-tiling}).

\begin{lemma}
\label{lem:renormalization-tiling-02}
    When $\Ustar$ is sufficiently small, the following holds for all $f \in \Ustar$.
    \begin{enumerate}[label=\textnormal{(\arabic*)}]
        \item $\Gamma(f_1)$ contains $\Psi_{1,i} \left( \partial \Deltabold_1(f) \cap \partial \Delta_1(i,f) \right)$ for every $i \in \{0,1,\ldots, \abold_n + \bbold_n-1\}$. There is some $i$ such that $\hat{\gamma}_0(f_1)$ is contained in $\Psi_{1,i} \left( \partial \Deltabold_1(f) \cap \partial \Delta_1(i,f) \right)$.
        \item $\Gamma(f)$ is disjoint from $\partial \Deltabold_1(f)$.
        \item For $\bullet \in \{0,\infty\}$, there is an arc $\xi^\bullet_0$ such that both $\xi^\bullet_0 \cup \hat{\gamma}_0^\bullet$ and $\xi^\bullet_1:= f(\xi^\bullet_0)$ connect $\partial^\bullet U_0$ and $\partial \Deltabold_1(f)$.  
    \end{enumerate}
    Moreover, $\xi_0 := \xi^0_0 \cup \xi^\infty_0 $ and $\xi_1 := \xi^0_1 \cup \xi^\infty_1$ can be chosen such that there is a holomorphic motion of 
    \[
        \Gamma \cup \xi_0 \cup \xi_1 \cup \Deltabold_1
    \]
    over $f \in \Ustar$ that is equivariant with respect to 
    \begin{itemize}
        \item $f: \xi_0(f) \to \xi_1(f)$,
        \item $f: \Delta_1(i,f) \to \Delta_1(i+\pbold_1, f)$ for $i \neq \{-\pbold_1, -\pbold_1+1\}$, and
        \item $\Psi_{1,i} : \partial \Deltabold_1(f) \cap \Delta_1(i,f) \to \Gamma(f_1)$ for all $i$.
    \end{itemize}
\end{lemma}

\begin{proof}
    Every tile $\Delta_1(i,f)$ is a rectangle. Clearly, each $\Psi_{1,i}$ maps the horizontal sides of $\Delta_1(i,f)$ into $\partial U_1$. Let us label the vertical sides of $\Delta_1(i,f)$ by $l(i)$ and $r(i)$ such that $l(i)$ intersects the side $r(i+1)$ of the next tile. Then, the intersection $\partial \Deltabold_1(f) \cap \partial \Delta_1(i,f)$ is the union of the horizontal sides of $\Delta_1(i,f)$ and the symmetric difference $l(i) \triangle r(i+1)$ between touching vertical sides across all $i$'s.

    It is clear that $l(i) \neq r(i+1)$ for at least one $i$. For such $i$, either $l(i)$ is the preimage of $\gamma_0(f_1)$ under $\Psi_{1,i}$ and $r(i+1)$ is the preimage of the arc $\gamma_1(f_1)$ under $\Psi_{1,i+1}$, or vice versa. In this case, $l(i) \triangle r(i+1)$ will be mapped by $\Phi_{1,i}$ or $\Phi_{1,i+1}$ onto $\hat{\gamma}_0(f_1)$. This implies (1).
    
    Item (2) follows directly from Lemma \ref{lem:renormalization-tiling-01}. Moreover, (2) allows us to find for each $\bullet \in \{0,\infty\}$ a proper arc $\xi^\bullet_0$ in $U_0 \backslash (\hat{\gamma}_0 \cup \Deltabold_1)$ in a small neighborhood of $\gamma_0$ that connects the tip of $\hat{\gamma}_0^\bullet$ to a point on $\partial\Delta_1(i,f)$ for some $i \neq \{-\pbold_1, -\pbold_1+1\}$. This yields (3).

    In Lemma \ref{lem:renormalization-tiling-01}, we already established the equivariant holomorphic motion of $\partial \Deltabold_0 \cup \partial \Deltabold_1$. By lifting via $\Phi_{1,i}$, this motion immediately extends to an equivariant motion of $\Gamma$. We then lift the motion of $\Deltabold_0(f_1)$ via $\Psi_{1,i}$ to obtain an equivariant motion of $\partial \Deltabold_1 \cup \Gamma$. Finally, by applying the $\lambda$-lemma, we extend this motion to $\Gamma \cup \xi_0 \cup \xi_1 \cup \Deltabold_1$.
\end{proof}

For $n \in \mathbb{N}$, we define the \emph{$n$\textsuperscript{th} full renormalization tiling} of $U_0$ to be the union of the tilings $\Deltabold_n$ and $\Abold_k$ for $k=0,1,\ldots,n-1$ where the latter is constructed as follows. Each $\Abold_k$ is a disjoint union of two tilings $\Abold_k^0$ and $\Abold_k^\infty$ where the former is closer to $\partial^0 U_0$ and the latter is closer to $\partial^\infty U_0$. For each $\bullet \in \{0,\infty\}$,
\begin{list}{$\rhd$}{}
    \item $\Abold_0^\bullet$ is the connected component of $\overline{\Deltabold_0 \backslash \Deltabold_1}$ that touches $\partial^\bullet U_0$ on the boundary, and it is split by $\hat{\gamma}_0^\bullet \cup \xi_0^\bullet \cup \xi_1^\bullet$ into two tiles $A_0^\bullet (0)$, $A_0^\bullet (1)$. Again, we make the convention that $A_0^\bullet(0)$, $\hat{\gamma}_0^\bullet \cup \xi_0^\bullet$, $A_0^\bullet(1)$ are in counterclockwise order.
    \item $\Abold_k^\bullet$ is the connected component of $\overline{\Deltabold_k \backslash \Deltabold_{k+1}}$ that touches $\partial^\bullet \Deltabold_k$ on the boundary, and it has tiles $\{A_k^\bullet(i)\}_{i=0,1,\ldots, \abold_k+\bbold_k-1}$ obtained by spreading via forward iterates of $f$ the tiles $A_k^\bullet(j, f) := \Phi_k(A_0^\bullet(j, f_k) )$ for $j \in \{0,1\}$ and labeled in counterclockwise order.
\end{list}

The first full renormalization tiling is illustrated in Figure \ref{fig:tilings}.

\begin{figure}
\begin{tikzpicture}
\coordinate (v1) at (1,-0.45) {};
\coordinate (v2) at (1,0.45) {};
\coordinate (v3) at (-1,0.5) {};
\coordinate (v4) at (-1,-0.5) {};

\coordinate (w1) at (4.25,-1.5) {};
\coordinate (w2) at (4.5,-0.5) {};
\coordinate (w3) at (4.5,0.5) {};
\coordinate (w4) at (4.25,1.5) {};
\coordinate (w5) at (-3,3.5) {};
\coordinate (w6) at (-3,-3.5) {};

\coordinate (a1) at (1.7,-1.1) {};
\coordinate (a2) at (2,-0.3) {};
\coordinate (a3) at (2,0.3) {};
\coordinate (a4) at (1.7,1.1) {};
\coordinate (a5) at (0.9,1.7) {};
\coordinate (a6) at (0.3,1.8) {};
\coordinate (a7) at (-1,1.7) {};
\coordinate (a8) at (-2.3,0.2) {};
\coordinate (a9) at (-1.7,-1) {};
\coordinate (a10) at (-1.3,-1.3) {};
\coordinate (a11) at (-0.3,-1.8) {};

\coordinate (b1) at (2.6,-2) {};
\coordinate (b2) at (3,-1) {};
\coordinate (b3) at (3.1,-0.4) {};
\coordinate (b4) at (3.1,0.4) {};
\coordinate (b5) at (3,1) {};
\coordinate (b6) at (2.6,2) {};
\coordinate (b7) at (1.8,2.5) {};
\coordinate (b8) at (1.1,2.7) {};
\coordinate (b9) at (0.4,2.8) {};
\coordinate (b10) at (-0.2,2.9) {};
\coordinate (b11) at (-1.4,2.7) {};
\coordinate (b12) at (-3.5,0.4) {};
\coordinate (b13) at (-3.2,-0.9) {};
\coordinate (b14) at (-2.9,-1.4) {};
\coordinate (b15) at (-2.5,-1.9) {};
\coordinate (b16) at (-2,-2.4) {};
\coordinate (b17) at (-0.6,-3) {};

% edges
\filldraw[fill=yellow!10!white] (w1) decorate [decoration=zigzag] { -- (w2)} .. controls (4,-0.4) and (4,0.4) .. (w3) decorate [decoration=zigzag] { -- (w4)} .. controls (3,4) and (-1,4.5) .. (w5) .. controls (-5.5,2) and (-5.5,-2) .. (w6) .. controls (-1,-4.5) and (3,-4) .. (w1);
\filldraw[green!50!black, fill=green!10!white] (b1) -- (b2) decorate [decoration=zigzag] { -- (b3)} -- (b4) decorate [decoration=zigzag] { -- (b5)} -- (b6) -- (b7) decorate [decoration=zigzag] { -- (b8)} -- (b9) decorate [decoration=zigzag] { -- (b10)} -- (b11) .. controls (-2.2,2.5) and (-3.3,1.5) .. (b12) -- (b13) decorate [decoration=zigzag] { -- (b14)} -- (b15) decorate [decoration=zigzag] { -- (b16)} -- (b17) .. controls (0,-3) and (2.3,-2.3) .. (b1);
\filldraw[green!50!black, fill=yellow!10!white] (a1) -- (a2) decorate [decoration=zigzag] { -- (a3)} -- (a4) -- (a5) decorate [decoration=zigzag] { -- (a6)} -- (a7) .. controls (-1.4,1.6) and (-2,0.8) .. (a8) -- (a9) decorate [decoration=zigzag] { -- (a10)} -- (a11) .. controls (0.2,-1.8) and (1.5,-1.2).. (a1);
\filldraw[fill=white] (v1) decorate [decoration=zigzag] { -- (v2)} .. controls (0.5,1) and (-0.5,1) .. (v3) .. controls (-1.2,0.25) and (-1.2,-0.25) .. (v4) .. controls (-0.5,-1) and (0.5,-1) .. (v1);

\draw[green!50!black] (a1) -- (b1);
\draw[green!50!black] (a4) -- (b6);
\draw[green!50!black] (a7) -- (b11);
\draw[green!50!black] (a8) -- (b12);
\draw[green!50!black] (a11) -- (b17);
\draw[thick, red] (b17) -- (-0.7,-4);
\draw[thick, red] (a11) -- (-0.15,-0.9);
\draw[thick, red] (b11) -- (-1.7,3.3);
\draw[thick, blue] (-1.7,3.3) -- (-2,3.85);
\draw[thick, red] (a7) -- (-0.83,1.2);
\draw[thick, blue] (-0.83,1.2) -- (-0.6,0.77);

% labels
\node[blue] at (-1, 0.9) {$\hat{\gamma}_0^0$};
\node[blue] at (-1.5,3.65) {$\hat{\gamma}_0^\infty$};
\node[red] at (-0.65, 1.5) {$\xi_0^0$};
\node[red] at (-1.77,2.9) {$\xi_0^\infty$};
\node[red] at (0.03,-1.3) {$\xi_1^0$};
\node[red] at (-0.85,-3.5) {$\xi_1^\infty$};
\node[yellow!50!black] at (3.5,-1.6) {$A^\infty_0(0)$};
\node[yellow!50!black] at (-3.9,-1) {$A^\infty_0(1)$};
\node[yellow!50!black] at (0.75,1.1) {$A^0_0(0)$};
\node[yellow!50!black] at (-1.65,0.05) {$A^0_0(1)$};
\node[green!50!black] at (0.8,2.25) {$\Delta_1(0)$};
\node[green!50!black] at (-2.2,1.5) {$\Delta_1(1)$};
\node[green!50!black] at (-2.2,-1.4) {$\Delta_1(2)$};
\node[green!50!black] at (0.9,-2) {$\Delta_1(3)$};
\node[green!50!black] at (2.6,0) {$\Delta_1(4)$};
\end{tikzpicture}

\caption{The first full renormalization tiling of $U_0$.}
\label{fig:tilings}
\end{figure}
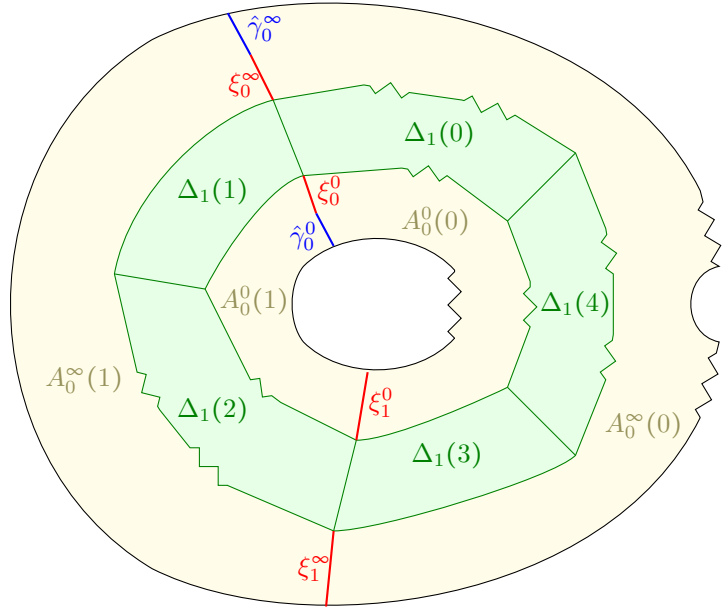

\begin{definition}
    A \emph{quasiconformal combinatorial pseudo-conjugacy of level $n$} between $f$ and $f_*$ is a quasiconformal map $h: \RS \to \RS$ that sends $\overline{U_0}$ to $\overline{U_\star}$ and preserves the $n$\textsuperscript{th} renormalization tiling as follows.
    \begin{enumerate}[label=(\arabic*)]
        \item The map $h$ sends $\Delta_n(i,f)$ to $\Delta_n(i,f_*)$ for all $i$, and is equivariant on $\Delta_n(i,f)$ for all $i \not\in \{-\pbold_n, -\pbold_n+1\}$;
        \item For all $\bullet \in \{0,\infty\}$ and $k \in \{0,1,\ldots n-1\}$, $h$ sends $A_k^\bullet(i,f)$ to $A_k^\bullet(i,f_*)$ for all $i$, and is equivariant on $A_k^\bullet(i,f)$ for all $i \not\in \{-\pbold_k, -\pbold_k+1\}$.
    \end{enumerate}
\end{definition}

\begin{theorem}[Combinatorial pseudo-conjugacy]
\label{thm:comb-pseudo-conj}
    Consider $f \in \mathcal{U}_n$ and let 
    \[
    D:= \max_{0\leq k \leq n} \dist(f_k, f_*).
    \] 
    There is a $K_D$-quasiconformal combinatorial pseudo-conjugacy $h$ of level $n$ between $f$ and $f_*$ such that 
    \[
    \sup_{z \in \Deltabold_n(f)} |h(z)-z| \leq M_D.
    \]
    Moreover, $K_D \to 1$ and $M_D \to 0$ as $D \to 0$.
\end{theorem}

\begin{proof}
    Recall that each tile $A_k^\bullet(i,f)$ admits some $t \in \N$ and $j \in \{0,1\}$ such that $\Psi_{k,i}:= \Phi_k^{-1} \circ f^{-t}$ univalently maps $A_k^\bullet(i,f)$ onto $A_0^\bullet(j,f_k)$. By Lemma \ref{lem:renormalization-tiling-02}, we have a holomorphic motion of the first full renormalization tiling over $\mathcal{U_1}$. Let us pull back this motion via maps of the form $\Psi_{n,i}$ to obtain a holomorphic motion of the full $n$\textsuperscript{th} renormalization tiling. By equivariance and $\lambda$-lemma, this holomorphic motion induces the desired quasiconformal map $h$. The dilatation $K_D$ of $h$ is bounded by the dilatation of the motion at $f_0,f_1,\ldots f_n$, which depends only on $D$, where $K_D \to 1$ as $D \to 0$. The upper bound $M_D$ follows from the continuity of the holomorphic motion and the compactness of quasiconformal maps.
\end{proof}

\begin{corollary}
\label{cor:rotational-criterion}
    There is some $\varepsilon>0$ such that the following holds. Suppose $f \in \Ustar$ is infinitely renormalizable and $\Rstar^n f$ is in the $\varepsilon$-neighborhood of $f_*$ for all $n \in \N$. Then, $f$ is a rotational corona with rotation number $\theta$.
\end{corollary}

\begin{proof}
    By Theorem \ref{thm:comb-pseudo-conj}, we have a $K(\varepsilon)$-quasiconformal combinatorial pseudo-conjugacy $h_n$ of level $n$ between $f$ and $f_*$ for all $n \in \N$. By the compactness of $K$-quasiconformal maps, $h_n$ converges in subsequence to a quasiconformal map $h: \RS \to \RS$, and $h^{-1}$ must be a conjugacy on the Herman quasicircle $\Hq_*$ of $f_*$. The image $h^{-1}(\Hq_*)$ is a Herman quasicircle of $f$ containing the critical point $c_0(f)$ and separating the boundaries of the domain of $f$. It follows that $f$ must be a rotational corona with rotation number $\theta$.
\end{proof}

\subsection{Towards hyperbolicity}
\label{ss:weak-hyperbolicity}

\begin{theorem}
\label{thm:weak-hyperbolicity}
    The renormalization operator $\Rstar: \Ustar \to \Bstar$ is hyperbolic at the fixed point $f_*$ and its local unstable manifold $\unstloc$ has a finite positive dimension. If $\Ustar$ is sufficiently small, the local stable manifold $\mani^s_{\textnormal{loc}}$ of $f_*$ consists of the set of $(d_0,d_\infty)$-critical rotational coronas in $\Ustar$ with rotation number $\theta$.
\end{theorem}

\begin{proof}  
    Consider a corona $f$ near $f_*$ lying on the local stable manifold $\mani^s_{\textnormal{loc}}$. For sufficiently small $\Ustar$, $\Rstar^n f$ is in the $\varepsilon$-neighborhood of $f_*$ for all $n \in \N$. By Corollary \ref{cor:rotational-criterion}, $f$ must be a rotational corona with rotation number $\theta$.
    
    Let us consider the derivative $(D\Rstar)_{f_*}$ of the renormalization operator at the fixed point $f_*$. By the compactness of $\Rstar$, the number of neutral and repelling eigenvalues is finite. We claim that neutral eigenvalues do not exist and repelling eigenvalues must exist.
    
    Suppose for a contradiction that there are neutral eigenvalues. By Small Orbits Theorem \ref{thm:small-orbits}, there exists an infinitely renormalizable corona $f$ such that its forward orbit lies entirely in the $\varepsilon$-neighborhood of $f_*$ and it satisfies
    \begin{equation}
    \label{eqn-slow-convergence}
        \lim_{n\to \infty} \frac{1}{n} \log \dist( \Rstar^n f, f_* ) = 0.
    \end{equation}
    By Corollary \ref{cor:rotational-criterion}, $f$ must be a rotational corona with the same rotation number $\theta$ as $f_*$. By Corollary \ref{cor:renorm-operator}, renormalizations $\Rstar^n f$ converge to $f_*$ exponentially fast, which contradicts (\ref{eqn-slow-convergence}). Hence, neutral eigenvalues do not exist.

    Consider the family of rational maps $F_c$ from (\ref{eqn:rat-map-formula}). By Theorem \ref{thm:comb-rigidity}, there is a unique parameter $c_*$ such that $F_{c_*}$ admits a Herman quasicircle with the same rotation number as $f_*$. By Lemma \ref{lem:renorm-of-any-critical rotational}, there is an analytic renormalization operator $\renorm$ on a neighborhood of $F_{c_*}$ such that $\renorm F_{c_*}$ is a rotational corona with rotation number $\theta$ that is sufficiently close to $f_*$. For any parameter $c \neq c_*$ sufficiently close to $c_*$, $\renorm F_c$ is also sufficiently close to $f_*$. By the uniqueness of $c_*$, the parameter $c$ can be picked such that $F_c$ is postcritically finite, and so $\renorm F_c$ is not a rotational corona. 
    
    Suppose for a contradiction that $(D\Rstar)_{f_*}$ has no repelling eigenvalues. With this assumption, $\mani^s_{\textnormal{loc}}$ would be an open neighborhood of $f_*$ and contains $\renorm F_c$. However, the non-rotationality of $\renorm F_c$ would contradict Corollary \ref{cor:rotational-criterion}.
\end{proof}

%%%%%%%%%%%%%%%%%%%%%%%%%%%%%%%%%%%%%%%%%%%%%%%%%%%%%%%%%%%%%%%%%%%%%%%%%%%%%%%%%%%%%%%%%%%%%%%%%%%%%%%%%%%%%%%%%%%%%%%%%%%%%%%%%%%%%%%%%%%%%%%%%%%%%%%%%%%%%%%%%%%%%%%%%%%%%%%%%%%%

\part{From coronas to cascades}

From now on, we will consider the corona renormalization operator $\Rstar: \Ustar \to \Bstar$ on a small Banach neighborhood $\Ustar$ of its hyperbolic fixed point $f_*: U_* \to V_*$ constructed in \S\ref{sec:stab-mani}. 
Theorem \ref{thm:weak-hyperbolicity} tells us that $\Rstar$ satisfies our main Theorem \ref{main-theorem} with the exception of item (4) on the dimension of the local unstable manifold. 
The rest of this paper is dedicated to proving item (4).

\section{Transcendental extension}
\label{sec:trans-ext}

\begin{definition}
    A map $g: A \to B$ is said to be \emph{$\sigma$-proper} if there exist exhaustions $A_n$, $B_n$ of $A$, $B$ respectively such that for all $n$, $g: A_n \to B_n$ is a proper map; equivalently, every connected component of the preimage of a compact set under $g$ is compact.
\end{definition}

In \cite{McM98}, McMullen proved the existence of maximal $\sigma$-proper extensions of holomorphic commuting pairs associated to renormalizations of quadratic Siegel disks. This is generalized in \cite[Theorem 5.5]{DLS} where pre-pacmen on the local unstable manifold are shown to admit maximal $\sigma$-proper extension. In this section, we will show that our case is no different. We will study coronas in the local unstable manifold $\unstloc$ of $f_*$, which we will identify as a parameter space (of unknown dimension) of transcendental holomorphic maps onto $\C$.

\subsection{Maximal $\sigma$-proper extension}
\label{ss:maximal-extension}

Consider a corona $f : U \to V$ lying in the local unstable manifold $\unstloc$ of $f_*$. Since $f$ is infinitely anti-renormalizable, it comes with a backward tower of corona renormalizations $\{f_k: U_k \to V\}_{k \leq 0}$, where each $f_k$ embeds into $U_{k-1}$ as a pre-corona $F_k = \left(f_{k,\pm} : U_{k,\pm} \to S_k\right)$ consisting of iterates of $f_{k-1}$. Let $\psi_{f_k} : S_k \to V$ be the renormalization change of variables realizing the renormalization of $f_{k-1}$ and let $\phi_{f_k} := \psi_{f_k}^{-1} : V \to S_k$.

Let us normalize our coronas such that they have a critical value at $0$. For each $k \leq 0$, consider the translation $T_{f_k}(z) = z- c_1(f_k)$ and denote
\[
U_k^\natural = T_{f_k}(U_k), \quad V_k^\natural = T_{f_k}(V), \quad U_{k,\pm}^\natural = T_{f_{k-1}}(U_{k,\pm}), \quad S_k^\natural = T_{f_{k-1}}(S_k).
\]
The translations $T_{f_k}$'s normalize our maps $f_k$, $F_k$, and $\phi_{f_k}$ into
\[
    f_k^\natural : U_k^\natural \to V_k^\natural, \quad F_k^\natural := (f_{k,\pm}^\natural : U_{k,\pm}^\natural \to S_k^\natural), \quad \phi_{f_k}^\natural : V_k^\natural \to S_k^\natural
\]
respectively. Consider the linear map 
\[
A_*(z):= \mu_* z
\]
where $\mu_* := (\phi_{f_*}^\natural)'(0) \in \D$ is the self-similarity factor of $f_*$.

\begin{lemma}
    There exists an open disk neighborhood $D$ of $0$ such that for every backward orbit of renormalizations $f_0, f_{-1}, f_{-2},\ldots$ in $\unstloc$, the limit 
    \[
        h_{f_0}^\natural(z) := \lim_{k\to-\infty} A_*^k \circ \phi_{f_{k+1}}^\natural \circ \ldots \circ \phi_{f_{-1}}^\natural \circ \phi_{f_0}^\natural(z)
    \]
    exists and defines a univalent map on $D$.
\end{lemma}

\begin{proof}
    As $\phi_{f_k}^\natural \to \phi_{f_*}^\natural$ exponentially fast, so is the derivative $\mu_k := (\phi_{f_k}^\natural)'(0)$ towards $\mu_*$. There are positive constants $\varepsilon$ and $\delta$ such that $\varepsilon < 1 - |\mu_*| $ and for all $|z|<\delta$ and $k\leq 0$, we have $|\phi_{f_k}^\natural(z)| \leq (|\mu_*|+\varepsilon)|z|$. 
    Therefore, for all $|z|<\delta$ and $k <0$,
    \[
        | \phi_{f_{k+1}}^\natural \circ \ldots \circ \phi_{f_0}^\natural (z) | \leq (|\mu_*|+\varepsilon)^{-k} |z|.
    \]
    The sequence $h^{(k)}(z) := A_*^k \circ \phi_{f_{k+1}}^\natural \circ \ldots \circ \phi_{f_{-1}}^\natural \circ \phi_{f_0}^\natural(z)$ indeed converges to a univalent map on $\{ |z|<\delta\}$ since
    \[
        \frac{h^{(k-1)}(z)}{h^{(k)}(z)} = \frac{ \phi_{f_k}^\natural 
        \left( \phi_{f_{k+1}}^\natural \circ \ldots \circ \phi_{f_0}^\natural (z) \right) 
        }{ 
        \mu_* \, \phi_{f_{k+1}}^\natural \circ \ldots \circ \phi_{f_0}^\natural (z) } 
        = \frac{\mu_k}{\mu_*} + O(|\phi_{f_{k+1}}^\natural \circ \ldots \circ \phi_{f_0}^\natural (z)|) \to 1
    \]
    exponentially fast as $k \to -\infty$.
\end{proof}

From now on, for every $k \leq 0$, we will denote
\[
    h_{k}^\natural := h_{f_k}^\natural \quad \text{ and } \quad h_{k}^\# := A_*^k \circ h_{k}^\natural.
\]
The following properties are easy to verify.

\begin{lemma}
    For $k \leq 0$, 
    \[
    h_{k-1}^\natural \circ \phi_{f_k}^\natural = A_* \circ h_{k}^\natural \qquad \text{and} \qquad h_{f_0}^\natural = h_{k}^\# \circ \phi_{f_{k+1}}^\natural \circ \ldots \circ \phi_{f_0}^\natural.
    \]
    Moreover, $h_{0}^\natural$ extends to a univalent map on the interior of $V_0^\natural \backslash \gamma_1^\natural$.
\end{lemma}

The maps $h_k^\natural$ act as linearizing coordinates under which renormalization change of variables are simply linear maps. Objects in linearizing coordinates will be written in bold:
\[
    \Ubold_{k,\pm} := h_{f_k}^\natural (U_{k,\pm}^\natural), \quad \Sbold_k := h_k^\natural(S_k^\natural), \quad \Fbold_k := (\fbold_{k,\pm} : \Ubold_{k,\pm} \to \Sbold_{k}).
\]
Often, we will also work with the rescaled linearizing coordinates $h_{f_k}^\#$ in which we add the symbol ``$\#$`` as follows:
\[
    \Ubold_{k,\pm}^\# := h_k^\# (U_{k,\pm}^\natural), \quad \Sbold_k^\# := h_k^\#(S_k^\natural), \quad \Fbold_k^\# := \left( \fbold_{k,\pm}^\# : \Ubold_{k,\pm}^\# \to \Sbold_{k}^\# \right).
\]
By design, it is clear that for all $k\leq 0$,
\begin{equation}
\label{eqn:rescaling-fk's}
    \fbold_{k,\pm}^\# = A_*^k \circ \fbold_{k,\pm} \circ A_*^{-k}.
\end{equation}

\begin{lemma}
\label{lem:renorm-matrix}
    There is a matrix of positive integers $\Mbold = \begin{pmatrix}
        m_{11} & m_{12} \\
        m_{21} & m_{22}
    \end{pmatrix}$
    such that for every negative integer $k$,
    \[
        \fbold_{k+1,-}^\# = \left(\fbold_{k,-}^\#\right)^{m_{11}} \circ \left(\fbold_{k,+}^\#\right)^{m_{12}} 
        \quad \text{and} \quad 
        \fbold_{k+1,+}^\# = \left(\fbold_{k,-}^\#\right)^{m_{21}} \circ \left(\fbold_{k,+}^\#\right)^{m_{22}}.
    \]
\end{lemma}

\begin{proof}
    The action of renormalization restricted to the Herman quasicircle of $f_*$ is a sector renormalization, and in particular an iterate of prime renormalization. See \S\ref{ss:renorm-rotation}. The existence of such a matrix $\Mbold$ follows from \S\ref{ss:cascade}.
\end{proof}

\begin{theorem}[Maximal extension]
\label{thm:max-extension}
    Assume $\Ustar$ is a sufficiently small Banach neighborhood of $f_*$. For every $f \in \unstloc$ and every $k \leq 0$, the maps $\fbold_{k,\pm}^\#$ described above extend to $\sigma$-proper branched coverings $\mathbf{X}_{k,\pm}^\# \to \C$, where $\mathbf{X}_{k,\pm}^\#$ are simply connected domains in $\C$.
\end{theorem}

\begin{remark}
    Actually, $\mathbf{X}_{k,\pm}^\#$ are dense subsets of $\C$. For the renormalization fixed point $f_*$, this will follow from Corollary \ref{cor:triviality-of-comb-class} (2). For general $f \in \unstloc$, this property will be apparent after we establish Theorem \ref{thm:finite-esc-set} on the holomorphic motion of $\partial \mathbf{X}_{k,\pm}^\#$.
\end{remark}

\begin{proof}
    For every $k\leq 0$, the composition $\phi_{k+1} \circ \ldots \circ \phi_0$ embeds the pre-corona $F_0 = (f_{0,\pm}: U_{0,\pm} \to V\backslash \gamma_1)$ to the dynamical plane of $f_k$ as a pair of iterates
    \begin{equation}
    \label{eqn:pair-max-ext}
        \left( 
        f_k^{\abold_k}:  U_{0,-}^{(k)} \to V_0^{(k)}, \: f_k^{\bbold_k}:  U_{0,+}^{(k)} \to V_0^{(k)} 
        \right).
    \end{equation}
    Since $\phi_k$ is contracting at the critical value, the diameter of $U_{0,\pm}^{(k)} \to V_0^{(k)}$ shrinks to $0$ as $k \to -\infty$.

    To proceed, we need the following technical lemma.

\begin{lemma}
\label{key-lemma}
    Assume $\Ustar$ is a sufficiently small Banach neighborhood of $f_*$. There exists an open disk $D$ around the critical value $c_1(f_*)$ of $f_*$ such that for all sufficiently large $n \in \N$, $t \in \{\abold_n, \bbold_n\}$, and $f \in \Rstar^{-n}(\Ustar)$, the point $f^t(c_{1}(f))$ is contained in $D$ and $D$ can be pulled back by $f^t$ to a disk $D_0 \subset U_f \backslash \gamma_1$ containing $c_1(f)$ on which $f^t : D_0 \to D$ is a branched covering.
\end{lemma}

This lemma initially appears in \cite[Key Lemma 4.8]{DLS} in the context of quadratic Siegel pacmen. Due to its length, the proof will be supplied separately in \S\ref{ss:key-lemma}. The lemma tells us that for sufficiently large negative number $k \ll 0$, the disk $D$ contains the set $\{c_{1+\abold_k}(f_k), c_{1+\bbold_k}(f_k)\}$ and the pair in (\ref{eqn:pair-max-ext}) extends to a commuting pair of branched coverings
\begin{align}
    \label{eqn:from-w-to-d}
    \left( 
        f_k^{\abold_k}:  W_{-}^{(k)} \to D,  \: f_k^{\bbold_k}:  W_{+}^{(k)} \to D 
    \right),
\end{align}
where $W_\pm^{(k)} \cup D$ are disks in $V \backslash \gamma_1$ containing $c_1(f_k)$. By conjugating with $h_k^\# \circ T_k$, we transform this pair into the commuting pair of branched coverings
    \[
        \fbold_{0,\pm}: \mathbf{W}_\pm^{(k)} \to \mathbf{D}^{(k)}
    \]
    where 
    \[
    \mathbf{W}_{\pm}^{(k)} := h_k^\# \circ T_k \left( W_{\pm}^{(k)} \right) \qquad \text{and} \qquad \mathbf{D}^{(k)} := h_k^\# \circ T_k(D).
    \]
    For all sufficiently large $t$ and $m \leq 0$, $\mathbf{D}^{(tm)}$ is compactly contained in $\mathbf{D}^{(tm-t)}$, and
    \[
        \modu\left( \mathbf{D}^{(tm-t)} \backslash \overline{\mathbf{D}^{(tm)}} \right) \succ 1.
    \]
    As such,
    \[
        \bigcup_{k<0}^\infty \mathbf{D}^{(k)}= \C.
    \]
    Therefore, the maps $\fbold_{0,\pm}$ extend to $\sigma$-proper branched coverings from $\mathbf{X}_{0,\pm}:= \bigcup_{k<0} \mathbf{W}_\pm^{(k)}$ onto $\C$. It is clear from the construction that $\mathbf{X}_{0,\pm}$ is a simply connected domain.
\end{proof}

The proof of the theorem above actually gives us the following stronger property, which we will use later in \S\ref{ss:rigidity-finite-esc-set}.

\begin{lemma}[Stability of $\sigma$-branched structure]
\label{lem:stability-of-branched-covering}
    Assume $\Ustar$ is a sufficiently small Banach neighborhood of $f_*$. For every $f \in \unstloc$, there are sequences of nested disks 
    \[
        \mathbf{W}_\pm^{(-1)} \subset \mathbf{W}_\pm^{(-2)} \subset \mathbf{W}_\pm^{(-3)} \subset \ldots
        \qquad \text{and} \qquad 
        \mathbf{D}^{(-1)} \subset \mathbf{D}^{(-2)} \subset \mathbf{D}^{(-3)} \subset \ldots
    \]
    where
    \[
        \bigcup_{k<0} \mathbf{W}_\pm^{(k)} = \mathbf{X}_{0,\pm}
        \qquad \text{and} \qquad 
        \bigcup_{k<0} \mathbf{D}^{(k)} = \C, 
    \]
    such that for every $k<0$,
    \begin{enumerate}[label=\textnormal{(\arabic*)}]
        \item each of $\mathbf{W}_\pm^{(k)}$ and $\mathbf{D}^{(k)}$  depends continuously on $f$;
        \item the maps $\fbold_{0,\pm}: \mathbf{W}_\pm^{(k)} \to \mathbf{D}^{(k)}$ are proper branched coverings of fixed finite degree;
        \item critical points of $\fbold_{0,\pm}: \mathbf{W}_\pm^{(k)} \to \mathbf{D}^{(k)}$ move holomorphically over $f \in \Ustar$.
    \end{enumerate}
\end{lemma}

\begin{proof}
    The construction of such disks is similar to the proof of the previous theorem. We add the following modification. By Theorem \ref{thm:comb-pseudo-conj}, we can replace the disk $D$ with a slightly smaller disk $D(f_0,k)$ depending continuously on $f_0$ such that for all $i \leq \max\{\abold_k, \bbold_k\}$,
    \[
        c_i(f_*) \in D(f_*,k) \quad \text{if and only if} \quad c_i(f_k) \in D(f_0,k).
    \]
    Under this replacement, the domains of branched coverings $ \left( f_k^{\abold_k}, f_k^{\bbold_k} \right)$ from (\ref{eqn:from-w-to-d}) become 
    \[
        \fbold_{0,\pm}: W_\pm(f_0,k) \to D(f_0,k),
    \]
    which depend continuously on $f_0$. By conjugating with $h_k^\# \circ T_k$, we obtain the commuting pair $\fbold_{0,\pm} : \mathbf{W}_\pm^{(k)} \to \mathbf{D}^{(k)}$ with the desired property.
\end{proof}

\subsection{Key lemma for transcendental extension}
\label{ss:key-lemma}

The remainder of the section is dedicated to the proof of Lemma \ref{key-lemma}. 

Fix a small neighborhood $D$ of the critical value $c_1(f_*)$ of the renormalization fixed point $f_*$. Fix $n \in \N$ and a large constant $s \in \N$. We will denote by $\abold_n$ and $\bbold_n$ the $n$\textsuperscript{th} renormalization return times. Consider a corona $f$ that is $m:= n+s$ times renormalizable such that $f_i := \Rstar^i f$ is close to $f_*$ for all $i \in \{1,\ldots,m\}$. We will denote the critical orbit by 
\[
c_j(f):=f^j(c_0(f)).
\]
Our goal is to show that for $t \in \{\abold_n, \bbold_n\}$, $c_{1+t}(f)$ is contained in $D$ and there is a branched covering map $f^t: (D_0, c_1(f)) \to (D, c_{1+t}(f))$. The proof we present below is similar to the Key Lemma in \cite{DLS}, which is to ensure that pullbacks of $D$ must avoid the forbidden boundary.

Let $h$ be a level $m$ combinatorial pseudo-conjugacy between $f$ and $f_*$, and consider the renormalization tiling $\Deltabold_m(f):= h^{-1}(\Deltabold_m(f_*))$ defined in \S\ref{ss:renorm-tiling}. Recall that $f$ maps $\Delta_m(f,i)$ univalently onto $\Delta_m(f,i+\pbold_n)$ whenever $i \not\in \{-\pbold_m, -\pbold_m+1 \}$, and on $\Delta_m(f,-\pbold_m) \cup \Delta_m(f,-\pbold_m+1)$, $f$ is almost a degree $d$ covering map branched at its critical point $c_0(f)$ onto its image, which contains $\Delta_m(f,0) \cup \Delta_m(f,1)$. By Theorem \ref{thm:comb-pseudo-conj}, $h$ is close to the identity map and $\Deltabold_m(f)$ approximates the Herman quasicircle $\Hq_*$ of $f_*$. 
    
In the dynamical plane of $f_*$, for sufficiently large $n\gg 0$, both $c_{1+\abold_n}(f_*)$ and $c_{1+\bbold_n}(f_*)$ are contained in $D$ because it is sufficiently close to $c_1(f_*)$. Let us fix $t \in \{\abold_n, \bbold_n\}$. Since $s$ is picked to be large,
    \[
        t \leq \max\{\abold_n, \bbold_n\} < \min\{\abold_m,\bbold_m\} -1.
    \]
Therefore, the orbit $\{c_j(f_*)\}_{j=1,2,\ldots t+1}$ avoids both $\Delta_m(-\pbold_m, f_*)$ and $\Delta_m(-\pbold_m+1, f_*)$. Since $h$ is close to the identity, it follows that $c_{1+t}(f)$ is also contained in $D$.

Let 
    \[
    D_1, \enspace D_2, \enspace \ldots, \enspace D_{1+t}:= D
    \]
denote the lift of $D$ along the orbit $c_1(f)$, $c_2(f)$, \ldots, $c_{1+t}(f)$. We would like to show that for $i \in \{1,2,\ldots, t\}$, the disk $D_i$ does not intersect the forbidden boundary $\partial_\frb U_f$ so that $f:D_i \to D_{i+1}$ is a branched covering.

\subsubsection{A new tiling $\Lambdabold_m$}

We say that a subset $I$ of $\Z / \qbold_m \Z$ is an \emph{interval} if it is a sequence of consecutive elements of $\Z/\qbold_m\Z$ of cardinality less than $\pbold_m$. For any interval $I$ in $\Z/\qbold_m \Z$, we will use the notation 
    \[
    \Delta_m(I) := \bigcup_{i \in I} \Delta_m(i)
    \]
and
    \[
        f^{-1}I := \begin{cases}
            I-\pbold_m & \text{ if } I \cap \{\pbold_m, \pbold_m+1, 0, 1\} = \emptyset,\\
            (I-\pbold_m) \cup \{-\pbold_m, -\pbold_m+1\} & \text{ if } I \cap \{0, 1\} \neq \emptyset,\\
            (I-\pbold_m) \cup \{0, 1\} & \text{ if } I \cap \{\pbold_m, \pbold_m+1\} \neq \emptyset.
        \end{cases}
    \]
The following property holds.

\begin{claim1}
    For any interval $I$ in $\Z/\qbold_m \Z$, the lift of $\Delta_m(I)$ under $f|_{\Deltabold_m}$ is contained in $\Delta_m(f^{-1}I)$.
\end{claim1}

First, consider the dynamical plane of $f_m := \Rstar^m f: U_m \to V$. Let us define the tiling $\Lambdabold_0(f_m) := \{ \Lambda_0(i,f_m)\}_{i\in\{0,1\}}$, which is a skinnier version of $\Deltabold_0(f_m)$, as follows. For $i \in \{0,1\}$, we define $\Lambda_0(i,f_m)$ to be the closure of the connected component of $f_m^{-1}(U_m) \backslash \left(\gamma_0(f_m) \cup \gamma_1\right)$ contained in $\Delta_0(i,f_m)$. Let us embed it via $\Phi_m$ to the dynamical plane of $f$ and spread it around via iterates of $f$ to obtain the tiling $\Lambdabold_m =\Lambdabold_m(f)$.

Similar to Claim 1, we have:

\begin{claim2}
    For any interval $I$ in $\Z/\qbold_m \Z$, we have 
    \[
    \Lambda_m(I) = \Lambdabold_m \cap \Delta_m(I)
    \]
    and the lift of $\Lambda_m(I)$ under $f|_{\Lambdabold_m}$ is contained in $\Lambda_m(f^{-1}I)$.
\end{claim2}
    
The problem with the tiling $\Deltabold_m$ is that for $j \in \{1,\ldots,t\}$, even when $D_{j+1} \cap \Deltabold_m$ is contained in $\Delta_m(I)$ for some interval $I$, it is possible that $D_j \cap \Deltabold_m$ is not contained in $\Delta_m(f^{-1}I)$. However, this issue does not occur for the tiling $\Lambdabold_m$.

\begin{claim3}
    For any interval $I$ in $\Z/\qbold_m \Z$, any $j \in \left\{ 0,1,\ldots,\min\{\abold_m, \bbold_m\}-1\right\}$, and any subset $T \subset V$,
    \[
        T \cap \Deltabold_m \subset \Delta_m(I) \enspace \Longrightarrow \enspace f^{-j}(T) \cap \Lambdabold_m \subset \Lambda_m(f^{-j}I).
    \]
\end{claim3}

\begin{proof}
    By construction, the tiling $\Lambdabold$ has the property that $f^j (\Lambdabold_m) \subset \Deltabold_m$ for all $j < \min\{\abold_m, \bbold_m\}$. Let $I$, $j$, and $T$ be as in the hypothesis and suppose $T \cap \Deltabold_m$ is contained in $\Delta_m(I)$. Consider a point $z$ in $\Lambdabold_m$ such that $f^j(z)$ is contained in $T$. Clearly, $f^j(z)$ is in $\Delta_m(I)$, and by Claim 1, $z$ is contained in $\Deltabold_m(f^{-j}I)$. By Claim 2, the point $z$ is indeed contained in $\Lambda_m(f^{-j}I)$.
\end{proof}

Consider the smallest interval $I_{t+1}$ in $\Z / \qbold_m \Z$ such that 
    \[
    \{0,1\} \subset I_{t+1} \quad \text{and} \quad D_{t+1} \cap \Deltabold_m(f) \subset \Delta_m(I_{t+1}).
    \]
For $j\in \{1,\ldots t\}$, let $I_j:= f^{-(t+1-j)}I_t$. It is assumed that $D\cap \Hq_*$ is roughly a level less than $n$ combinatorial interval, so, since $m \succ n$, $|I_j|$ is large for all $j$. 

Let us fix some positive integer $\eta$ where $\eta \ll t$. This will be taken to be the maximum of the periods $\eta^\bullet_k$ introduced in the next subsection.

\begin{claim4}
    For $j \in \{1,2,\ldots t+1\}$,
        \begin{enumerate}[label=(\arabic*)]
            \item $|I_j|/\qbold_m$ is small and $\Delta_m(I_j,f_*) \cap \Hq_*$ has a small combinatorial length;
            \item if $j \leq t-2-\eta$, the intervals $I_j$, $I_{j+1}$, \ldots, $I_{j+\eta+3}$ are pairwise disjoint;
            \item if $1\leq j \leq \eta+2$, then $I_j$ is disjoint from $\{-\pbold_m, -\pbold_m+1\}$.
        \end{enumerate}
\end{claim4}

\begin{proof}
    Since the rotation number is of bounded type, the combinatorial length of the intersection of $\Hq_*$ with every tile of $\Deltabold_m(f_*)$ is comparable. Since $D$ is assumed to be small and $s$ is taken to be sufficiently large, (1) follows.

    Since (2) is combinatorial in nature, it suffices to prove (2) in the dynamical plane of $f_*$, which is obvious from the irrational rotational action of $f_*$ on $\Hq_*$. If (3) does not hold, then for some integer $j \in [2,\eta+3]$, the interval $I_j$ intersects $\{0,1\}$, but this contradicts (2) and the fact that $I_1$ must intersect $\{0,1\}$.
\end{proof}

\subsubsection{Spines and pseudo-spines}

    Let us first consider the dynamical plane of $f_*:U_* \to V_*$. Recall that the preimage of $f_*^{-1}(\gamma_1) \backslash \gamma_0$ consists of arcs
    \[
        \gamma^0_1, \ldots, \gamma^0_{2(d_0-1)} \subset \partial^{0} U_*, \quad \gamma^\infty_1, \ldots, \gamma^\infty_{2(d_\infty-1)} \subset \partial^{\infty} U_*.
    \]
    The strict preimage $f^{-1}(\Hq_*) \backslash \Hq_*$ is a bouquet of pairwise disjoint arcs
    \[
        \sigma^0_1, \enspace \ldots, \enspace \sigma^0_{2(d_0-1)}, \quad \sigma^\infty_1, \enspace \ldots, \enspace  \sigma^\infty_{2(d_\infty-1)}
    \]
    where each $\sigma^\bullet_i$ connects $c_0(f_*)$ to a point on $\gamma^\bullet_i$. We call each of $\sigma^\bullet_i$ a \emph{spine} of $f_*$ of generation one. In general, a \emph{spine} of generation $g\geq 1$ is a lift under $f_*^{g-1}$ of a spine of generation one, and its \emph{root} is the endpoint that is a critical point of $f_*^g$.

\begin{figure}
\vspace{-0.2in}

\begin{tikzpicture}
\coordinate (v4) at (1,1) {};
\coordinate (v3) at (1.5,0.5) {};
\coordinate (v2) at (1.5,-0.5) {};
\coordinate (v1) at (1,-1) {};
\coordinate (w1) at (2.7,-1.9) {};
\coordinate (w2) at (3.3,-1.7) {};
\coordinate (w3) at (3.8,-1.15) {};
\coordinate (w4) at (3.9,-0.5) {};
\coordinate (w5) at (3.9,0.5) {};
\coordinate (w6) at (3.8,1.15) {};
\coordinate (w7) at (3.3,1.7) {};
\coordinate (w8) at (2.7,1.9) {};

% fill
\draw[gray!40!white, fill=gray!10!white] (0.15,0) ellipse (5.25 and 3.25);
\draw[gray!40!white, fill=gray!20!white] (w1) .. controls (2.8,-1.56) .. (w2) .. controls (3.34,-1.28) .. (w3) .. controls (3.6,-0.76) .. (w4) .. controls (3.4,-0.4) and (3.4,0.4) .. (w5) .. controls (3.6,0.76) .. (w6) .. controls (3.34,1.28) .. (w7) .. controls (2.8,1.56) .. (w8) .. controls (-5.2,4.5) and (-5.2,-4.5) .. (w1); 
\draw[gray!40!white, fill=gray!10!white] (v1) .. controls (1.4,-1.2) and (1.7,-0.8) .. (v2) .. controls (1.8,-0.2) and (1.8,0.2) .. (v3) .. controls (1.7,0.8) and (1.4,1.2) .. (v4) .. controls (-2.5,3) and (-2.5,-3) .. (v1);
\draw[gray!40!white, fill=white] (0,0) ellipse (0.5 and 0.4);

% edges
\draw[ultra thick, red] (v1) .. controls (1.4,-1.2) and (1.7,-0.8) .. (v2); 
\draw[ultra thick, gray!40!white] (v2) .. controls (1.8,-0.2) and (1.8,0.2) .. (v3); 
\draw[ultra thick, red] (v3) .. controls (1.7,0.8) and (1.4,1.2) .. (v4); 
\draw[ultra thick, gray!40!white] (v4) .. controls (-2.5,3) and (-2.5,-3) .. (v1);
\draw[ultra thick, red] (w1) .. controls (2.8,-1.56) .. (w2);
\draw[ultra thick, gray!40!white] (w2) .. controls (3.34,-1.28) .. (w3); 
\draw[ultra thick, red] (w3) .. controls (3.6,-0.76) .. (w4); 
\draw[ultra thick, gray!40!white] (w4) .. controls (3.4,-0.4) and (3.4,0.4) .. (w5); 
\draw[ultra thick, red] (w5) .. controls (3.6,0.76) .. (w6); 
\draw[ultra thick, gray!40!white] (w6) .. controls (3.34,1.28) .. (w7); 
\draw[ultra thick, red] (w7) .. controls (2.8,1.56) .. (w8);
\draw[ultra thick, gray!40!white] (w8) .. controls (-5.2,4.5) and (-5.2,-4.5) .. (w1); 
\draw[ultra thick, green!90!black] (2.5,0) .. controls (2,2.5) and (-2.5,2.5) .. (-2.5,0) .. controls (-2.5,-2.5) and (2,-2.5) .. (2.5,0);
\draw[blue!80!black] (2.5,0) .. controls (2.2,0.13) .. (1.58,0.7);
\draw[blue!80!black] (2.5,0) .. controls (3,0.2) .. (3.65,0.75);
\draw[blue!80!black] (2.5,0) .. controls (2.8,0.6) .. (3.1,1.62);
\draw[blue!80!black] (2.5,0) .. controls (2.2,-0.13) .. (1.58,-0.7);
\draw[blue!80!black] (2.5,0) .. controls (3,-0.2) .. (3.65,-0.75);
\draw[blue!80!black] (2.5,0) .. controls (2.8,-0.6) .. (3.1,-1.62);

% labels
\node[blue!80!black] at (2.05,0.55) {\small $\sigma^0_1$};
\node[blue!80!black] at (1.9,-0.65) {\small $\sigma^0_2$};
\node[blue!80!black] at (2.75,1.3) {\small $\sigma^\infty_4$};
\node[blue!80!black] at (3.25,0.7) {\small $\sigma^\infty_3$};
\node[blue!80!black] at (3.4,-0.2) {\small $\sigma^\infty_2$};
\node[blue!80!black] at (2.7,-1.3) {\small 
 $\sigma^\infty_1$};
\node[green!80!black] at (-0.1,2.1) {$\Hq_*$};
\node[red] at (1.2,-0.7) {\small $\gamma^0_2$};
\node[red] at (1.2,0.7) {\small $\gamma^0_1$};
\node[red] at (3.05,-2) {\small $\gamma^\infty_1$};
\node[red] at (4.05,-0.8) {\small $\gamma^\infty_2$};
\node[red] at (4.05,0.8) {\small $\gamma^\infty_3$};
\node[red] at (3.05,2) {\small $\gamma^\infty_4$};
\node[red] at (-1.7,1.1) {\small $\gamma_0$};
\node[red] at (-2.95,-1.9) {\small $\gamma_1$};
\draw[red] (-1.27,1) -- (-2.2,1.8);
\draw[red] (-0.42,-0.2) -- (-3.8,-2.14);
\node at (2.5,0) {$\bullet$};
\node at (2.45,-0.3) {\small $c_0$};
\node[gray] at (-3.48,0) {\small $U_*$};
\node[gray] at (-5.33,0) {\small $V_*$};
\end{tikzpicture}
\vspace{-0.4in}

\caption{The spines of $f_*$ of generation one when $(d_0,d_\infty)=(2,3)$.}
\label{fig:strict-preimage-H}
\end{figure}
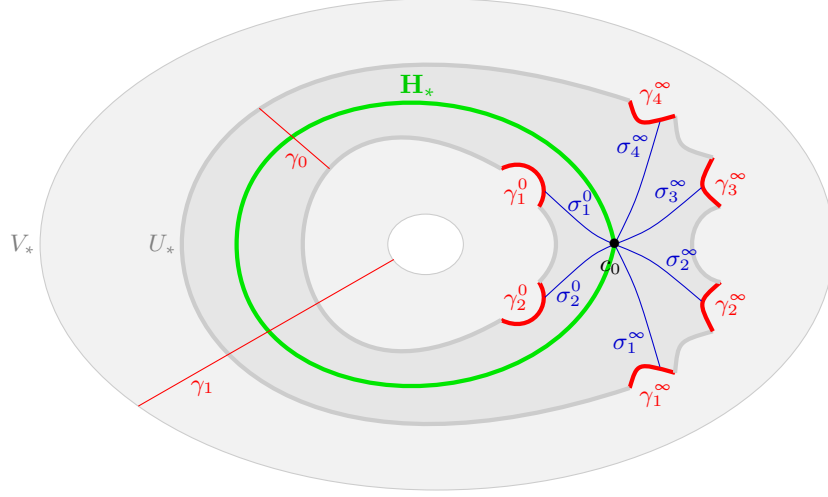

    A \emph{spine chain} of generation $g$ is an infinite sequence of spines 
    \[
    \Sigma = (S_1,S_2,S_3,\ldots)
    \]
    of increasing generation such that $S_1$ has generation $g$ and for all $i\geq 1$, the root of $S_{i+1}$ is contained in $S_i$. We say that a spine chain $\mathcal{S}$ is \emph{periodic} with period $p$ if for all $i\geq 1$, $f_*^p(S_{i+1})=S_i$.

    The following is a direct consequence of Lemma \ref{lem:small-limbs} and Theorem \ref{thm:qc-rigidity}. 

    \begin{proposition}
    \label{prop:local-connectivity}
        Every spine chain of $f_*$ lands at a unique point. Different spine chains admit different landing points. The landing point of a periodic spine chain of some period $p$ is a repelling periodic point of period $p$, and it is also the landing point of exactly one periodic external ray of period $p$.
    \end{proposition}

    When $f$ is rotational with bounded type rotation number, the notion of spines of $f$ can be formulated analogously and the proposition above holds. Below, we will formulate an analog of bubbles for arbitrary coronas $f$ which are sufficiently close to $f_*$. This is achieved by replacing $\Hq_*$ with $\Lambdabold_m(f)$.

    For $f$, $\bullet \in \{0,1\}$, and $i \in \{1,\ldots, 2(d_\bullet-1)\}$, we define the \emph{pseudo-spine} $\mathbb{S}^\bullet_i$ of generation one to be the closure of the connected component of $f^{-1}(\Lambdabold_m) \backslash \Lambdabold_m$ that intersects with the spine $\sigma^\bullet_i$ of $f_*$. Each $\mathbb{S}^\bullet_i$ is connected and
    \[
        \mathbb{S}^\bullet_i \cap \Lambdabold_m \subset \Lambda_m(\{-\pbold_m, -\pbold_m+1\}), \qquad f(\mathbb{S}^\bullet_i) \subset \Lambdabold_m.
    \]
    We say that every pseudo-spine of generation one is attached to $\Lambdabold(\{-\pbold_m,-\pbold_m+1\})$. In general, a \emph{pseudo-spine} of generation $g\geq 1$ is a lift under $f^{g-1}$ of a pseudo-spine of generation one.

    Let us fix a large integer $M \gg 1$. We will assume that $f$ is sufficiently close to $f_*$ depending on $M$. 
    
    \begin{claim5}
        Every spine $S$ of $f_*$ of generation up to $M$ is approximated by a pseudo-spine $\mathbb{S}$ of $f$ such that
        \begin{enumerate}
            \item $\mathbb{S}$ is close to $S$ and $f|_{\mathbb{S}}$ is close to $f_*|_{S}$,
            \item if $S$ is attached to another spine $S'$, then $\mathbb{S}$ is attached to the pseudo-spine corresponding to $S'$;
            \item if $S$ is attached to $\Hq_*$, then $\mathbb{S}$ is attached to $\Lambda_m(I)$ for some interval $I$ disjoint from $\{0,1\}$.
        \end{enumerate}
    \end{claim5}

    \begin{proof}
        This is because $\Lambdabold_m$ compactly contains and well approximates $\Hq_*$.
    \end{proof}

    Let us fix $\bullet \in \{0,1\}$ and $k \in \{1,\ldots,2(d_\bullet-1)\}$. Let us construct a periodic spine chain
    \[
        \Sigma^\bullet_k = (S_1, S_2, S_3, \ldots)
    \]
    for $f_*$ that is very close to $\gamma^\bullet_k \cup \sigma^\bullet_k$. First, we set $S_1 := \sigma^\bullet_k$. Let us pick $\eta^\bullet_k \geq 1$ such that the pre-critical point $c_{-\eta^\bullet_k+1}(f)$ is close to the critical arc $\gamma_1$. Let $c^\bullet_k$ be the preimage of $c_{-\eta^\bullet_k+1}(f)$ located on $\sigma^\bullet_k$ close to $\gamma^\bullet_k$. Then, we set $S_2$ to be the unique spine rooted at $c^\bullet_k$ that is the lift of $S_1$ under $f^{\eta^\bullet_k}$. The other spines are then defined by induction, forming a periodic spine chain of period $\eta^\bullet_k$.

    Let $x^\bullet_k(f_*)$ be the landing point of $\Sigma^\bullet_k$. It is a repelling periodic point of period $\eta^\bullet_k$ and it is also the landing point of a periodic external ray $R^\bullet_k(f_*)$. Since $f$ is close to $f_*$, periodic rays $R^\bullet_k(f)$ and repelling periodic points $x^\bullet_k(f)$ exist in the dynamical plane of $f$.

    Let us define a periodic pseudo-spine chain
    \begin{equation}
    \label{eqn:pseudo-spine-chain}
        \mathbb{\Sigma}^\bullet_k = (\mathbb{S}_1, \mathbb{S}_2, \mathbb{S}_3, \ldots)
    \end{equation}
    for $f$ landing at $x^\bullet_k(f)$ as follows. Assume $M \gg \eta^\bullet_k$ and let $M' \in \N$ satisfy $\eta^\bullet_k M' \leq M$. For $2\leq j \leq M'$, we set $\mathbb{S}_j$ to be the pseudo-spine approximating $S_j$. This can be arranged so that $\mathbb{S}_{M'}$ is within the linearization domain of the repelling periodic point $x^\bullet_k(f)$, and so inductively we define $\mathbb{S}_{M'+j+1}$ to be the unique lift of $\mathbb{S}_{M'+j}$ under $f^{\eta^\bullet_k}$ that is even closer to $x^\bullet_k(f)$.

\subsubsection{Enlargements of $D_j$}

Let us inductively define enlargements $\mathcal{D}_j$ and $\mathcal{D}'_j$ of $D_j$ as follows. First, we set $\mathcal{D}_{t+1} = \mathcal{D}'_{t+1} := D$. For $j\leq t$, we set 
\begin{list}{$\rhd$}{}
    \item $\mathcal{D}'_j =$ the connected component of $f^{-1}(\mathcal{D}_{j+1})$ containing $D_j$;
    \item $\mathcal{D}_j =$ the smallest topological disk containing $\mathcal{D}'_j$ and the interior of $\Lambda_m(I_j)$.
\end{list}
For all $j$, we have $D_j \subset \mathcal{D}'_j \subset \mathcal{D}_j$.

\begin{claim6}
    For $j \in \{1,2,\ldots,t+1\}$, $\mathcal{D}_j \cap \Lambdabold_m$ is connected and its closure is $\Lambda_m(I_j)$.
\end{claim6}

\begin{proof}
    This follows from Claim 2 and the observation that, due to Claim 3, $D_j \cap \Lambdabold_m \subset \Lambda_m(I_j)$ for all $j$.
\end{proof}

    We will assume $D$ to be small enough such that it is disjoint from the periodic rays $f^i(R^\bullet_k)$ for all $i \in \{0,\ldots,t\}$, $\bullet \in \{0,\infty\}$, and $k\in \{1,2,\ldots,2(d_\bullet-1)\}$.

\begin{claim7}
    For $j \in \{1,2,\ldots,t\}$, the disk $\mathcal{D}_j$ is disjoint from all the periodic rays of the form $f^i(R^\bullet_k)$ for all $i \in \{0,\ldots,j-1\}$.
\end{claim7}

\begin{proof}
    This claim follows from induction. If $\mathcal{D}_j$ intersects $f^i(R^\bullet_k)$, then $\mathcal{D}'_j$ intersects $f^i(R^\bullet_k)$ and so $\mathcal{D}_{j+1}$ intersects $f^{i+1}(R^\bullet_k)$.
\end{proof}

\subsubsection{Proof of Lemma \ref{key-lemma}}

Let $\Lambdabold'_m$ denote the union of all pseudo-spines of $f$ of generation one. Recall that the constant $\eta$ is set to be the maximum of the periods $\eta_k^\bullet$ of the pseudo-spine chains $\mathbb{\Sigma}^\bullet_k$. To finally show that $f^t: D_1 \to D$ is a branched covering, we will prove by induction the following statements for $j=1,\ldots, t+1$.
\begin{enumerate}[label=(\alph*)]
    \item $\mathcal{D}_j$ intersects $\Lambdabold'_m$ if and only if $I_j$ contains $\{-\pbold_m, -\pbold_{m+1}\}$;
    \item If $\mathcal{D}_j$ intersects $\Lambdabold'_m$, then the intersection is in a small neighborhood of $c_0$;
    \item If $\mathcal{D}_j$ intersects $\Lambdabold'_m$ for $j<t$, then $j<t-\eta$ and $\mathcal{D}_{j+1}$, \ldots $\mathcal{D}_{j+\eta+1}$ are all disjoint from $\Lambdabold'_m$;
    \item If $\mathcal{D}_j$ intersects a pseudo-spine chain $\mathbb{\Sigma}^\bullet_k$ from (\ref{eqn:pseudo-spine-chain}), then the intersection is within $\Lambdabold'_m$;
    \item $\mathcal{D}_j$ is an open disk disjoint from the forbidden boundary $\partial_\frb U_f$.
\end{enumerate}

Suppose (a)--(e) hold for $j+1$, $j+2$, \ldots $t+1$. Let us show that they hold for $j$.
    
Suppose $I_j$ contains $\{-\pbold_m, -\pbold_{m+1}\}$. Then, $\mathcal{D}_{j+1}$ contains either $\Lambda_m(\{-1,0,1\})$ or $\Lambda_m(\{0,1,2\})$, and so the lift $\mathcal{D}'_j$ of $\mathcal{D}_{j+1}$ contains the critical point $c_0(f)$ and intersects $\Lambdabold'_m$.

Suppose $I_j$ is disjoint from $\{-\pbold_m, -\pbold_{m+1}\}$. Then, $\mathcal{D}_{j+1}$ does not contain the critical value $c_1(f)$ and every point in $\mathcal{D}_{j+1}$ has at most one preimage under $f$ in $\mathcal{D}'_j$. By Claim 6, the preimage of $\mathcal{D}_{j+1} \cap \Lambdabold_m$ under $f|_{\mathcal{D}'_j}$ must be contained in $\Lambdabold_m$. It follows that $\mathcal{D}'_j$ is disjoint from $\Lambdabold'_m$. Since $\mathcal{D}'_j \cup \Lambda_m(I_j)$ does not surround $\Lambdabold'_m$, then $\mathcal{D}_j$ is also disjoint from $\Lambdabold'_m$.

We just proved (a). Item (b) follows from Claim 6 and the fact that $\Lambda_m(I_{j+1})$ is a small neighborhood of $c_1(f)$ as a result of Claim 4 (i). Item (c) then follows from Claim 4 (2). 

Item (e) follows from (b) and (d). Indeed, if $\mathcal{D}_j$ were to intersect $\partial_\frb U_f$, then by Claim 7, it must intersect some pseudo-spine chain $\mathbb{\Sigma}^\bullet_k$ from (\ref{eqn:pseudo-spine-chain}) and because of (d), its intersection is contained in $\Lambdabold'_m$. In particular, $\mathcal{D}_k$ can only intersect $\Lambdabold'_m$ in a small neighborhood of $c_0$, which implies that $\mathcal{D}_k$ cannot intersect $\partial_\frb U_f$.

It remains to prove (d). By continuity, we can assume that (d) holds whenever $j\geq t-\eta$. Let us assume that $j<t-\eta$ and suppose for a contradiction that (d) fails, that is, there is a pseudo-spine chain $\mathbb{\Sigma}^\bullet_k = (\mathbb{S}_1, \mathbb{S}_2,\ldots)$ such that $\mathcal{D}_j$ intersects $\mathbb{S}_i$ where $i \geq 2$. 

We claim that $\mathcal{D}_j$ intersects $\mathbb{S}_2$. Indeed, suppose otherwise that the smallest possible $i\geq 2$ such that $\mathcal{D}_j$ intersects $\mathbb{S}_i$ satisfies $i>2$. Since $\mathcal{D}'_j \cap \Lambda_m(I_j)$ is disjoint from the ray $R^\bullet_k$, then the subchain $\mathbb{\Sigma}^{(i)} = (\mathbb{S}_i, \mathbb{S}_{i+1},\ldots)$ intersects $\mathcal{D}'_j$ and its image $f(\mathbb{\Sigma}^{(i)})$ intersects $\mathcal{D}_{j+1}$. By periodicity of $\mathbb{\Sigma}^\bullet_k$, the chain $\mathbb{\Sigma}^{(i-1)}$ intersects $\mathcal{D}_{j+\eta^\bullet_k}$, which is a contradiction to (d) for index $j+\eta^\bullet_k$.

The argument from the previous paragraph results in the intersection of $\mathcal{D}_{j+\eta^\bullet_k}$ and $\mathbb{S}_1$ being non-empty. By (a), the interval $I_{j+\eta^\bullet_k}$ contains $\{-\pbold_m, -\pbold_{m+1}\}$, so for $l\in\{1,2,\ldots,\eta^\bullet_k\}$, $f^l(\mathbb{S}_2)$ is attached to $\Lambda_m(I_{j+l})$. Moreover, since the critical value $c_1(f)$ is not contained in $\mathcal{D}_{j+l} \cap \Lambdabold_m$, every point in $\mathcal{D}_{j+l}$ has at most one preimage in $\mathcal{D}'_{j+l-1}$.
    
Consider the lift $\mathbb{S}'_2$ of $f(\mathbb{S}_2)$ under $f$ that is attached to $\Lambda_m(I_j)$. Since $c_1(f)$ is not contained nor surrounded by $\mathcal{D}_{j+1} \cap f(\mathbb{S}_2)$, the lift $E$ of $f(\mathcal{D}_j \cap \mathbb{S}_2)$ under $f|_{\mathcal{D}'_j}$ agrees with the lift under $f|_{\mathbb{S}'_2}$. Therefore, $E$ would be contained in $\mathbb{S}'_2$, not $\mathbb{S}_2$, which is impossible. This concludes the proof of (d).

% \part{Cascades}

\section{Renormalization cascades}
\label{sec:cascades-intro}

In \S\ref{sec:trans-ext}, we established that the anti-renormalization tower associated to a corona $f$ on the local unstable manifold $\unstloc$ of the fixed point $f_*$ of $\Rstar: \Ustar \to \Bstar$ induces the sequence $\{\Fbold_n^\# = (\fbold_{n,\pm}^\#: \Xbold^\#_{n,\pm} \to \C)\}_{n \leq 0}$ of pairs of $\sigma$-proper branched coverings. We will reinterpret such a sequence as a global transcendental cascade.

\subsection{Cascades}
\label{ss:trans-cascades}

Consider the anti-renormalization matrix $\Mbold$ from Lemma \ref{lem:renorm-matrix}. Let us denote by $\tbold>1$ and $1/\tbold$ the eigenvalues of $\Mbold$.

Let us extend the backward sequence $\{\Fbold_n^\#\}_{n\leq 0}$ to a bi-infinite one as follows. For every positive integer $n$, we define $\Fbold_n^\# = (\fbold_{n,\pm}^\#)$ inductively by the relation
\begin{equation}
    \label{eqn:renorm-matrix-rel}
    \left( \fbold_{n,-}^\# \right)^a \circ \left( \fbold_{n,+}^\# \right)^b = \left( \fbold_{n-1,-}^\# \right)^{a'} \circ \left( \fbold_{n-1,+}^\# \right)^{b'}
\end{equation}
for any $a,b,a',b' \in \Z_{\geq 0}$ satisfying $(a' \ b') = (a \ b) \Mbold$. Then, $\left\{ \Fbold_n^\#\right\}_{n\in\Z}$ forms a sequence of commuting $\sigma$-proper holomorphic maps acting on the same dynamical plane.

Let us identify the local unstable manifold $\unstloc$ with the space $\manibold_{\textnormal{loc}}$ of pairs of $\sigma$-proper maps $\Fbold = (\fbold_{0,\pm})$ associated to each $f \in \unstloc$.
On $\manibold_{\textnormal{loc}}$, the renormalization operator will be denoted by $\Rboldstar$.
Assuming that $\unstloc$ lives in a sufficiently small neighborhood of $f_*$, the action of $\Rboldstar$ on $\manibold_{\textnormal{loc}}$ is topologically conjugate to an expanding linear map on some Euclidean space; in particular, it is injective.
We extend $\Rboldstar$ beyond $\manibold_{\textnormal{loc}}$ by setting
\[
    \Rboldstar^n \Fbold_0 = \Fbold_n := A_*^{-n} \Fbold_n^\# A_*^n,
\]
(compare with (\ref{eqn:rescaling-fk's})) and extend $\manibold_{\textnormal{loc}}$ to a ``global unstable manifold`` $\manibold$ by adding $\Fbold_n$ for all $n \geq 0$ and $\Fbold \in \manibold_{\textnormal{loc}}$. 
From this construction, $\manibold$ is just an abstract set.
The complex manifold structure of $\manibold_{\textnormal{loc}}$ naturally and uniquely extends to $\manibold$ such that the renormalization operator $\Rboldstar$ acts on $\manibold$ as a biholomorphism with a unique fixed point $\Fbold_*$ that is expanding.

\begin{definition}
    We define the space $\Tbold$ of \emph{power-triples} to be the quotient of the semigroup $\Z\times \Z_{\geq 0}^2$ under the equivalence relation $\sim$ where $(n,a,b) \sim (n-1,a',b')$ if and only if $(a' \ b') = (a \ b) \Mbold$.
\end{definition}

We will equip $\Tbold$ with the binary operation $+$ defined by
\[
    (n,a,b) + (n,a',b') = (n, a+a', b+b').
\]
With respect to $+$, $\Tbold$ has a unique identity element $0 := (n,0,0)$. For $P, Q \in \Tbold$, let us denote by $P \geq Q$ if for all sufficiently large $n \ll 0$, there exist $a,b,a',b' \in \N$ such that $P=(n,a,b)$, $Q=(n,a',b')$, $a\geq a'$, and $b\geq b'$.

By Lemma \ref{lem:power-triple-embedding}, $(\Tbold, +, \geq)$ can be identified with a sub-semigroup of $(\R_{\geq 0}, +, \geq)$. Moreover, $\Tbold$ inherits a well-defined scalar multiplication by powers of $\tbold$ as follows. For every $(n,a,b) \in \Tbold$ and integer $k$,
\[
    \tbold^k(n,a,b) = (n+k,a,b).
\]
For every $\Fbold \in \manibold$ and every power-triple $P = (n,a,b)$, we will use the notation
\[
    \Fbold^P := \left( \fbold_{n,-}^\# \right)^a \circ \left( \fbold_{n,+}^\# \right)^b.
\]
Each $\Fbold^P$ is a $\sigma$-proper map from its domain $\Dom\left(\Fbold^P\right)$ onto $\C$. We denote by $\Fbold^{\geq 0}$ the cascade $\left(\Fbold^P\right)_{P \in \Tbold}$ associated to $\Fbold$.

\begin{lemma}
    For every $\Fbold \in \manibold$, $P \in \Tbold$, and $n \in \Z$, 
    \[
        \Fbold_0^P = \left(\Fbold_{-n}^\#\right)^{\tbold^n P}.
    \]
    In particular, when $\Fbold = \Fbold_*$,
    \begin{equation}
    \label{eqn:self-sim-cascade}
        \Fbold_*^P = A_*^{-n} \circ \Fbold_*^{\tbold^n P} \circ A_*^n.
    \end{equation}
\end{lemma}

\subsection{Critical points and periodic points}
\label{ss:crit-periodic-points}

Let us pick
\[
\Fbold = [\fbold_\pm: \Ubold_\pm \to \Sbold] := [\fbold_{0,\pm}: \Ubold_{0,\pm} \to \Sbold_0] \in \manibold_{\textnormal{loc}}
\]
and let $\Fbold_n := \Rboldstar^n\Fbold$ for all $n\in\Z$. Within the cascade $\Fbold^{\geq 0}$, $\fbold_\pm$ is the first return map of points in $\Ubold_\pm$ back to $\Sbold$. In particular, $\Ubold_- \cup \Ubold_+$ is disjoint from $\Fbold^P(\Ubold_-)$ for all $P < (0,1,0)$ and $\Fbold^P(\Ubold_+)$ for all $P < (0,0,1)$.

\begin{definition}
    We define the \emph{zeroth renormalization tiling} $\Deltabold_0 = \Deltabold_0(\Fbold)$ associated to $\Fbold^{\geq 0}$ to be the tiling consisting of $\Delta_0(0) := \overline{\Ubold_+}$ and $\Delta_0(1) := \overline{\Ubold_-}$, as well as $\Fbold^P(\Delta_0(0))$ for all $P<(0,0,1)$ and $\Fbold^P(\Delta_0(1))$ for all $P<(0,1,0)$. We label the tiles in left-to-right order as $\Delta_0(i)$ for $i\in \Z$.
    For all $n \in \Z_{<0}$, we define the \emph{$n$\textsuperscript{th} renormalization tiling} to be the rescaling of the zeroth tiling for $\Fbold_n$, namely
    \[
        \Deltabold_n(\Fbold) := A_*^n \left(\Deltabold_0(\Fbold_n)\right).
    \]
\end{definition}

In $\manibold_{\textnormal{loc}}$, for $\Fbold$ sufficiently close to $\Fbold_*$, the tiling $\Deltabold_0(\Fbold)$ moves holomorphically in $\Fbold$. (Compare with Lemma \ref{lem:renormalization-tiling-01}.) In general, for $\Fbold \in \manibold$, the $n$\textsuperscript{th} tiling $\Deltabold_n(\Fbold)$ is well-defined for all sufficiently large $n\ll 0$. Each tile $\Delta_n(i)$ is a compact disk in $\C$, and the union $\bigcup_{i \in \Z} \Delta_n(i)$ is a closed bi-infinite strip in $\C$, that is, its complement consists of two unbounded disks each having a single access to infinity.

\begin{definition}
\label{def:full-lift}
    Consider $[f: U_f \to V] \in \unstloc$ and the associated pre-corona $\Fbold = [\fbold_\pm: \Ubold_\pm \to \Sbold] \in \manibold_{\textnormal{loc}}$ in linearizing coordinates. Given a subset $Z$ of $U_f$, the \emph{full lift} $\mathbf{Z}$ of $Z$ to the dynamical plane of $\Fbold$ is defined as
    \[
        \mathbf{Z} := \bigcup_{0\leq P< (0,0,1)} \Fbold^P \left(\mathbf{Z}_0\right) \cup \bigcup_{0\leq P< (0,1,0)} \Fbold^P \left(\mathbf{Z}_1\right),
    \]
    where $\mathbf{Z}_0$ and $\mathbf{Z}_1$ are the embedding of $Z \cap \Delta_0(0,f)$ and $Z \cap \Delta_0(1,f)$ to the dynamical plane of $\Fbold$ respectively.
\end{definition}

For the rest of the section, let us fix $\Fbold$ in $\manibold$. For every $x \in \C$ and $T \in \Tbold$, we denote the finite orbit of $x$ up to time $T$ by 
\[
    \text{orb}_x^T(\Fbold) := \left\{ \Fbold^P(x) \: : \: 0\leq P \leq T \right\}.
\]

\begin{definition}
    For $P \in \Tbold_{>0}$, let us denote by $\CP\left(\Fbold^P\right)$ the set of critical points of $\Fbold^P$ and by $\CV\left(\Fbold^P\right)$ the set of critical values of $\Fbold^P$. We say that a point $x$ is 
    \begin{list}{$\rhd$}{}
        \item a \emph{critical point} of $\Fbold^{\geq 0}$ if it is in $\CP\left(\Fbold^P\right)$ for some $P\in\Tbold_{>0}$, 
        \item a \emph{critical value} of $\Fbold^{\geq 0}$ if it is in $\CV\left(\Fbold^P\right)$ for some $P \in \Tbold_{>0}$, and
        \item a \emph{periodic point} of $\Fbold^{\geq 0}$ if there is some $P\in\Tbold_{>0}$ such that $\Fbold^P(x)=x$.
    \end{list}
\end{definition}

\begin{lemma}
\label{lem:critical-pts}
    Critical points of $\Fbold^{\geq 0}$ satisfy the following properties.
    \begin{enumerate}[label=\textnormal{(\arabic*)}]
        \item For $P \in \Tbold_{>0}$,
        \[
            \CP\left(\Fbold^P\right) = \bigcup_{0<S\leq P} \Fbold^{-S}\{0\} \quad \text{and} \quad \CV\left(\Fbold^P\right) = \{ \Fbold^S(0) \: : \: 0 \leq S<P\}.
        \]
        \item There is some $K_\Fbold \in \Tbold_{>0}$ such that for every power-triple $P < K_\Fbold$, every critical point of $\Fbold^P$ has local degree $d$. If $0$ is not periodic, this is still true for $P \geq K_\Fbold$. In general, for every $P \in \Tbold$, there is some $k \in \N$ such that the local degree of every critical point of $\Fbold^P$ is at most $k$.
    \end{enumerate}
    Let $T:= \min\{ (0,1,0), (0,0,1) \}$. If $\Fbold \in \manibold_{\textnormal{loc}}$, then for every $P<T$,
    \begin{enumerate}[label=\textnormal{(\arabic*)},start=4]
        \item $\CV\left(\Fbold^P\right)$ is a subset of $\Deltabold_0(\Fbold) \backslash \Sbold \cup \{0\}$ which moves holomorphically with $\Fbold$, and
        \item every critical point of $\Fbold^P$ has local degree $d$.
    \end{enumerate}
\end{lemma}

\begin{proof}
    Pick a bounded domain $\Dbold \Subset \C$ and select a connected component $\Dbold'$ of $\Fbold^{-P}(\Dbold)$. Suppose $\Fbold$ represents a corona $f$ in $\unstloc$, and recall that the anti-renormalizations of $f$ when the critical value is normalized to be at $0$ are denoted by $f_n^\natural$, $n\leq 0$. Recall that for sufficiently large negative number $n\ll 0$, the map $\Fbold^P: \Dbold' \to \Dbold$ can be identified via $h_n^\#$ with $(f_n^\natural)^{s_n} : D' \to D$ for some domains $D', D \Subset \C$ and some $s_n \geq 0$. Therefore, $x$ is a critical point of $\Fbold^P$ if and only if $(h_n^\#)^{-1}(x)$ is a critical point of $(f_n^\natural)^{s_n}$, which happens precisely when $\Fbold^S(x) = 0$ for some $S \leq P$. This leads to (1).

    Suppose $[\Fbold: \Ubold_\pm \to \Sbold]$ is in $ \manibold_{\textnormal{loc}}$ and fix $P \leq T$. For all $S < P$, $\Fbold^S(0)$ is contained in some tile $\Delta_0(i,\Fbold)$ that is disjoint from $\Sbold$. This implies (3). Also, (4) follows from the fact that for every critical point $x$ of $\Fbold^P$, $\textnormal{orb}_x^P(\Fbold)$ passes through the critical value $0$ exactly once.

    If $\Fbold$ is not close to $\Fbold_*$, then we can take some $n \ll 0$ such that $\Rboldstar^n \Fbold$ is in $\manibold_{\textnormal{loc}}$. Then, (2) follows from (3) and (4) by taking $K_\Fbold$ to be $\tbold^n T$ and $k$ to be such that $P < (k-1) K_{\Fbold}$.
\end{proof}

\begin{lemma}[Discreteness]
\label{lem:discreteness}
    For any bounded open subset $D$ of $\C$, there is some $Q \in \Tbold_{>0}$ such that for all $\Gbold \in \manibold$ sufficiently close to $\Fbold$ and whenever $P' < P < Q$,
    \begin{enumerate}[label=\textnormal{(\arabic*)}]
        \item $\Gbold^P$ is well-defined and univalent on $D$, and
        \item $\Gbold^P(D)$ is disjoint from $\Gbold^{P'}(D)$.
    \end{enumerate}
    For every $x \in \C$ and $T \in \Tbold$, $\textnormal{orb}_x^T(\Fbold)$ is discrete in $\C$.
\end{lemma}

\begin{proof}
    There exist some integers $m \leq 0$ and $j \in \{0,1\}$ such that $D$ is compactly contained in some level $m$ tile $\Delta_m(j, \Gbold)$ for all $\Gbold$ close to $\Fbold$. Let us set $Q := \tbold^m \min\{(0,1,0),(0,0,1)\}$. For $P < Q$, the tile $\Delta_m(j, \Gbold)$ is mapped by $\Gbold^P$ univalently onto to some other tile $\Delta_m(i,\Gbold)$ disjoint from $\Delta_m(0,\Gbold) \cup \Delta_m(1,\Gbold)$. This implies (1) and (2).

    Let us fix $x \in \C$ and $T \in \Tbold$. Let us pick any point $y$ in the closure of $\textnormal{orb}_x^T(\Fbold)$, and pick a small open neighborhood $D$ of $y$. From the first part, $\Fbold^P(D)$ is disjoint from $D$ for all sufficiently small $P \in \Tbold_{>0}$. This implies that only finitely many points in $\textnormal{orb}_x^T(\Fbold)$ are contained in $D$.
\end{proof}

By a straightforward compactness argument, the lemma above has the following consequence.

\begin{corollary}[Proper discontinuity]
\label{cor:discreteness}
    For any $P \in \Tbold$, any compact subset $\mathbf{Y}$ of $\Dom\left(\Fbold^P\right)$, and any bounded subset $\mathbf{X}$ of $\C$, there are at most finitely many power-triples $T \leq P$ such that $\Fbold^T(\mathbf{Y})$ intersects $\mathbf{X}$. 
\end{corollary}

\begin{corollary}
    Every critical point $x$ of $\Fbold^{\geq 0}$ admits a minimal $P \in \Tbold_{>0}$, called the generation of $x$, such that $\Fbold^P(x) = 0$.
\end{corollary}

\begin{proof}
    By definition, there is some $P \in \Tbold_{>0}$ such that $\Fbold^P(x) = 0$. By Lemma \ref{lem:discreteness}, $\textnormal{orb}_x^P(\Fbold)$ is discrete, so there are at most finitely many power-triples $S$ such that $S<P$ and $\Fbold^S(x) = 0$.
\end{proof}

\begin{corollary}
\label{cor:minimal-period}
    Every periodic point of $\Fbold^{\geq 0}$ has a minimal period.
\end{corollary}

\begin{proof}
    Suppose $x$ is a periodic point of $\Fbold^{\geq 0}$. The set $\Tbold_x := \{ P \in \Tbold \: : \: \Fbold^P(x) = x \}$ of periods of $x$ is a sub-semigroup of $\Tbold$. Pick a small neighborhood $D$ of $x$. By Lemma \ref{lem:discreteness}, there is some $Q \in \Tbold_{>0}$ such that for all $0<P<Q$, $\Fbold^P(D)$ is disjoint from $D$ and thus $P \not\in \Tbold_x$. This implies that $\Tbold_x$ is finitely generated, and in particular, of the form $\{nS\}_{n\in \N}$, where $S \in \Tbold_{>0}$ is the minimal period.
\end{proof}

\subsection{Dynamical sets}
\label{ss:esc-set}

Consider $\Fbold \in \manibold$. 

\begin{definition}
    Given $P \in \Tbold$, the \emph{$P$\textsuperscript{th} escaping set} of $\Fbold$ is
    \[
        \Esc_{\leq P}(\Fbold) := \C \backslash \Dom\left(\Fbold^P\right).
    \]
    The \emph{finite-time escaping set} of $\Fbold$ is the union 
    \[
        \Esc_{<\infty}(\Fbold) := \bigcup_{P \in \Tbold} \Esc_{\leq P}(\Fbold),
    \]
     the \emph{infinite-time escaping set} of $\Fbold$ is
    \[
        \Esc_{\infty}(\Fbold) := \{ z \in \C \backslash \Esc_{<\infty}(\Fbold) \: : \: \Fbold^P(z) \to \infty \text{ as } P \to \infty\},
    \]
    and the \emph{full escaping set} of $\Fbold$ is
    \[
        \Esc(\Fbold) := \Esc_{<\infty}(\Fbold) \cup \Esc_{\infty}(\Fbold).
    \]
\end{definition}

\begin{lemma}
\label{lem:unboundedness-of-esc}
    Every connected component of $\Esc_{\leq P}(\Fbold)$ is unbounded.
\end{lemma}

\begin{proof}
    There exists some $n \leq 0$ such that $\Fbold_n := \Rboldstar^n \Fbold$ is in $\manibold_{\textnormal{loc}}$. Since the domains of $\fbold_{n,\pm}$ are simply connected, then $\Dom\left(\Fbold_n^P\right)$ is simply connected for all $P \in \Tbold$ and so the claim is true for $\Fbold_n$. Since $\Fbold$ is just a rescaling of $\Fbold_n$, the claim is also true for $\Fbold$.
\end{proof}

In \S\ref{sec:external}, we will thoroughly study the structure of the finite-time escaping set of the fixed point $\Fbold_*$. In \S\ref{sec:escaping-sets}, we will show that when $\Fbold$ is hyperbolic, the finite and infinite-time escaping sets do not carry any invariant line field. This will imply that the unstable manifold indeed has codimension one.

It is clear from the definition that the boundary of $\Esc_{\leq P}(\Fbold)$ coincides with the boundary of $\Dom\left(\Fbold^P\right)$. Points on $\partial \Esc_{\leq P}(\Fbold)$ can be regarded as essential singularities of $\Fbold^P$. The following lemma is an analog of Picard's theorem.

\begin{lemma}
\label{lem:transitivity}
    For every $P \in \Tbold_{>0}$ and any sufficiently small Euclidean disk $D$ centered at a point in $\partial \Esc_{\leq P}(\Fbold)$, the image $\Fbold^P(D')$ of any connected component $D'$ of $D \cap \Dom\left(\Fbold^P\right)$ is dense in $\C$. 
\end{lemma}

This lemma is a direct consequence of $\sigma$-properness of $\Fbold^P$. The keen reader may refer to {\cite[Lemma 6.5]{DL23}} for a detailed proof.

\begin{corollary}
\label{cor:accumulation-to-bdry}
    For every $\Fbold \in \manibold$, $P \in \Tbold_{>0}$, and $x \in \C$, the boundary of $\Esc_{\leq P}(\Fbold)$ is the set of accumulation points of $\Fbold^{-P}(x)$.
\end{corollary}
    
We will later show that $\Esc_{<\infty}(\Fbold)$ is non-empty, has no interior, and its closure coincides with the Julia set, which we will define below. This corollary is an analog of the basic result in holomorphic dynamics which states that iterated preimages are dense in the Julia set. The proof below is similar to \cite[Corollary 6.7]{DL23}.

\begin{proof}
    By Lemma \ref{lem:discreteness}, there exists a disk neighborhood $B$ of $x$ such that $B\backslash\{x\}$ is disjoint from $\CV\left(\Fbold^P\right)$. Then, every connected component $B'$ of $\Fbold^{-P}(B)$ contains at most one critical point and the degree of $\Fbold^P : B' \to B$ is at most some uniform constant. Let $\Omega \subset B$ be an even smaller disk neighborhood of $x$ such that $\modu\left(B\backslash \overline{\Omega}\right) \asymp 1$. Any lift $\Omega' \subset B'$ of $\Omega$ under $\Fbold^P$ is also a disk with $\modu(B' \backslash \overline{\Omega'}) \asymp 1$.

    Consider a connected component $D$ of $\Dom\left(\Fbold^P\right)$, a point $y \in \partial D$, and a small $\varepsilon>0$. By Lemma \ref{lem:transitivity}, there is a connected component $\Omega' \subset D$ of $\Fbold^{-P}(\Omega)$ that is of distance at most $\varepsilon$ away from $y$. Since $\modu\left(B'\backslash\overline{\Omega'}\right)\asymp 1$, then $\Omega'$ has a small diameter depending on $\varepsilon$. Since $\Omega'$ contains a point in $\Fbold^{-P}(x)$, the assertion follows.
\end{proof}

\begin{definition}
    The \emph{Fatou set} $\Fcas(\Fbold)$ of $\Fbold$ is the set of points $z$ which admit a small open neighborhood $D$ disjoint from $\Esc_{<\infty}(\Fbold)$ such that $\{ \Fbold^P|_{D} \}_{P\in \Tbold}$ forms a normal family. The \emph{Julia set} $\Jcas(\Fbold)$ of $\Fbold$ is the complement $\C \backslash \Fcas(\Fbold)$.
\end{definition}

From the definition, it is clear that $\Jcas(\Fbold)$ contains the closure of $\Esc_{<\infty}(\Fbold)$. 

\begin{proposition}[Invariance]
    For every $P \in T$,
    \begin{enumerate}[label=\textnormal{(\arabic*)}]
        \item both $\Fcas(\Fbold)$ and $\Jcas(\Fbold)$ are completely invariant under $\Fbold^P$;
        \item the Fatou and Julia sets of $\Fbold^P$ are equal to $\Fcas(\Fbold)$ and $\Jcas(\Fbold)$ respectively.
    \end{enumerate} 
\end{proposition}

We say that a connected component $D$ of $\Fcas(\Fbold)$ is \emph{periodic} if there is some $P \in \Tbold_{>0}$ such that $\Fbold^P(D)=D$. The smallest such $P$ is called the \emph{period} of $D$. Moreover, we say that $D$ is \emph{pre-periodic} if there is some $Q \in \Tbold$ such that $\Fbold^Q(D)$ is periodic. The smallest such $Q$ is called the \emph{pre-period} of $D$. (These quantities exist due to Lemma \ref{lem:discreteness}. Compare with Corollary \ref{cor:minimal-period}.) We will show later in \S\ref{ss:no-wandering} that every Fatou component is pre-periodic.

\begin{proposition}
\label{prop:simply-conn}
    Every connected component of $\Fcas(\Fbold)$ is simply connected.
\end{proposition}

In particular, $\Fbold$ does not admit any Herman rings.

\begin{proof}
    This is a standard application of the maximum modulus principle. Pick any Jordan domain $D$ such that $\partial D$ is contained in $\Fcas(\Fbold)$. For all $P \in \Tbold$, $\Fbold^P|_D$ attains maximum on $\partial D$, thus $\{\Fbold^P|_D\}_P$ forms a normal family.
\end{proof}

In the case of $\Fbold = \Fbold_*$, equation (\ref{eqn:self-sim-cascade}) again implies self-similarity of the corresponding dynamical sets.

\begin{lemma}
\label{lem:self-similarity}
    The linear map $A_*$ preserves $\Esc_{<\infty}(\Fbold_*)$, $\Esc_{\infty}(\Fbold_*)$, $\Fcas(\Fbold_*)$, and $\Jcas(\Fbold_*)$. For all $P \in \Tbold_{>0}$, $A_*(\Esc_{\leq P}(\Fbold_*)) = \Esc_{\leq \tbold P}(\Fbold_*)$.
\end{lemma}

From now on, we will reserve the notation $\Hq$ to denote the \emph{Herman curve} of $\Fbold_*$, that is, the full lift of the Herman quasicircle of $f_*$ (see Definition \ref{def:full-lift}). Clearly, $\Hq$ is an $A_*$-invariant quasiarc.

\begin{lemma}
\label{lem:jcas-and-H}
    The set $\bigcup_{P \in \Tbold} \Fbold_*^{-P}(\Hq)$ of iterated preimages of $\Hq$ is dense in $\C$. In particular, the Julia set of $\Fbold_*$ is the whole plane: $\Jcas(\Fbold_*)=\C$.
\end{lemma}

\begin{proof}
    By \cite[Theorem 5.3]{Lim23b} and Theorem \ref{thm:qc-rigidity}, the critical value $c_1(f_*)$ of $f_*$ is a deep point of the non-escaping set of the corona $f_*$. 
    (Roughly speaking, magnifications of the iterated preimages of the Herman quasicircle of $f_*$ about $c_1(f_*)$ converge to the whole plane exponentially fast in the Hausdorff metric.) 
    As we pass to the corresponding dynamical plane of the transcendental extension, $0$ is a deep point of the iterated preimages of the Herman curve $\Hq$ of $\Fbold_*$ under $\Fbold_*^{(0,1,0)}$ and $\Fbold_*^{(0,0,1)}$. 
    By self-similarity, the iterated preimages of $\Hq$ under $\Fbold_*^{\geq 0}$ must be dense in $\C$, so its closure $\Jcas(\Fbold_*)$ is equal to $\C$.
\end{proof}

Further properties of the escaping sets, the Fatou set, and the Julia set will be discussed in Part 4.

\subsection{Invariant line field}
\label{ss:qc-deformation}

For $\Fbold \in \manibold$, let us denote by $\mathcal{B}(\Fbold)$ the set of $L^\infty$ Beltrami differentials $\mu(z) \frac{d \bar{z}}{dz}$ on $\C$ such that $\left(\Fbold^P\right)^*\boldsymbol{\mu}=\boldsymbol{\mu}$ almost everywhere on $\Dom\left(\Fbold^P\right)$ for all $P \in \Tbold$. An \emph{invariant line field} of $\Fbold$ is a non-zero element $\boldsymbol{\mu}$ of $\mathcal{B}\mathcal(\Fbold)$ such that $|\boldsymbol{\mu}(z)| = 1$ for all $z$ in the support of $\boldsymbol{\mu}$.

Given a corona $f: U \to V$, an \emph{invariant line field} of $f$ supported on a completely invariant set $E \subset V$ is a $L^\infty$ Beltrami differential $\mu(z) \frac{d \bar{z}}{dz}$ such that $f^*\mu=\mu$ almost everywhere on $U$, $|\mu|=1$ on a positive measure subset of $E$, and $\mu=0$ elsewhere. Given a corona $f$ in $\unstloc$ and its associated cascade $\Fbold$, a invariant line field $\boldsymbol{\mu}$ of $\Fbold$ induces a sequence of line fields $\mu_{-n}$ invariant under $\Rstar^{-n}f$ for all $n \geq 0$.

In classical holomorphic dynamics, the absence of invariant line fields is equivalent to the triviality of deformation space associated to a single holomorphic map. This principle remains valid for cascades in the unstable manifold.

For $r>0$, we again denote the open disk $\{ z \in \C \: : \: |z| < r \}$ by $\D_r$.

\begin{proposition}
\label{prop:integration-of-ILF}
    Consider $\Fbold \in \manibold$ and a non-zero $\boldsymbol{\mu} \in \mathcal{B}(\Fbold)$. 
    \begin{enumerate}[label=\textnormal{(\arabic*)}]
        \item There exist a constant $\varepsilon \in (0,\|\boldsymbol{\mu}\|_\infty^{-1})$, a unique holomorphic family $\{\Gbold_t\}_{t \in \D_\varepsilon}$ in $\manibold$, and a unique holomorphic family $\{\phi_t\}_{t \in \D_\varepsilon}$ of quasiconformal maps on $\C$ such that for every $t \in \D_\varepsilon$, 
        \begin{enumerate}[label=\textnormal{(\alph*)}]
            \item $\Gbold_0 = \Fbold$ and $\Fbold^{\geq 0}$ is quasiconformally conjugate to $\Gbold^{\geq 0}_t$ via $\phi_t$;
            \item $t\boldsymbol{\mu}$ is the Beltrami coefficient of $\phi_t$, i.e. $\bar{\partial} \phi_t = t\boldsymbol{\mu} \cdot \partial \phi_t$.
        \end{enumerate}
        \item Suppose $\Fbold$ is in $\manibold_{\textnormal{loc}}$. Let $f \in \unstloc$ be the associated corona, and let $f_{-n} := \Rstar^{-n}(f)$ for $n\geq 0$. For each $n\geq 0$, there exist a unique holomorphic family $\{g_{-n,t}\}_{t \in \D_\varepsilon}$ in $\unstloc$ and a holomorphic family of quasiconformal maps $\{\phi_{-n,t}: V \to V\}_{t \in \D_\varepsilon}$ such that 
        \begin{enumerate}[label=\textnormal{(\alph*)}]
            \item $g_{-n,0} = f_{-n}$ and $f_{-n}$ is quasiconformally conjugate to $g_{-n,t}$ via $\phi_{-n,t}$;
            \item $\Rstar g_{-n-1,t} = g_{-n,t}$ for all $t \in \D_\varepsilon$,
            \item $\Gbold_t \equiv \Fbold$ for all $t$ if and only if $g_{-n,t} \equiv f_{-n}$ for all $t$.
        \end{enumerate}
    \end{enumerate} 
\end{proposition}

\begin{proof}
    A standard application of the measurable Riemann mapping theorem gives us the appropriate holomorphic families $\Gbold_t$ and $\phi_t$, but a priori we do not know whether the cascade $\Gbold_t$ lives in $\manibold$. To deal with this issue, we shall descend to the realm of coronas. To simplify the exposition, let us assume $\mu$ is a line field.
    
    By anti-renormalizing, let us assume without loss of generality that $\Fbold$ is in $ \manibold_{\textnormal{loc}}$. Pick a pair of integers $m,n \geq 0$. Let us project $\boldsymbol{\mu}$ to the dynamical plane of $f_{-m-n}$ and obtain an invariant line field $\mu_{-m-n}$ of $f_{-m-n}$. Then, we integrate $t\mu_{-m-n}$ for every $t \in \D$ to obtain a Beltrami path $\{f_{-m-n,t}\}_{t \in \D}$ of coronas in a neighborhood of $f_*$. Let us renormalize $m$ times to obtain a new path $f^{(m)}_{-n,t} := \Rstar^{m}f_{-m-n,t}$ about $f^{(m)}_{-n,0} \equiv f_{-n}$. When $|t|<\frac{1}{2}$, $f^{(m)}_{-n,t}$ is quasiconformally conjugate to $f_{-n}$ with uniformly bounded dilatation. Therefore, we can take a limit as $m \to \infty$ and obtain a holomorphic path $g_{-n,t}$ of infinitely anti-renormalizable coronas. 
    
    For sufficiently small $\varepsilon>0$, the limiting path $\{g_{-n,t}\}_{|t|<\varepsilon}$ lies in the local unstable manifold $\mani^u_{\textnormal{loc}}$, and it satisfies the relation $\Rstar g_{-n-1,t} = g_{-n,t}$ for all $n\geq 0$ and $t$. In particular, $\{g_{0,t}\}$ corresponds to the desired holomorphic path $\{\Gbold_t\}$ in $\manibold_{\textnormal{loc}}$. Let us elaborate on the property (2) (c). Recall from \S\ref{ss:maximal-extension} that $\Gbold_t$ is constructed as the ``union`` of analytic extensions of rescalings of $g_{-n,t}$ across all $n$. Thus, $\Gbold_t$ is a trivial family if and only if $g_{-n,t}$ is trivial for all $n$, but by (2) (b), this occurs if and only if $g_{-n,t}$ is trivial for some $n$.

    It remains for us to prove uniqueness. Suppose there are two families of quasiconformal maps $\phi_t$ and $\tilde{\phi_t}$ integrating $t \boldsymbol{\mu}$ and giving two holomorphic curves $\Gbold_t = \phi_t \circ \Fbold \circ \phi_t^{-1}$ and $\tilde{\Gbold}_t = \tilde{\phi}_t \circ \Fbold \circ \tilde{\phi}_t^{-1}$ about $\Fbold$ in $\manibold$. Clearly, for all $t$, both $\phi_t$ and $\tilde{\phi}_t$ fix $\infty$ as well as the critical value $0$. By Weyl's Lemma, $\tilde{\phi}_t \circ \phi_t^{-1}$ defines a non-trivial family of linear maps conjugating $\Gbold_t$ and $\tilde{\Gbold}_t$. By Lemma \ref{lem:linear-conjugacy} below, $\tilde{\phi}_t = \phi_t$  and $\tilde{\Gbold}_t = \Gbold_t$ for all $t$.
\end{proof}

\begin{lemma}
\label{lem:linear-conjugacy}
    Consider two continuous families $\{\Fbold_t\}_{-1\leq t \leq 1}$ and $\{\Gbold_t\}_{-1 \leq t \leq 1}$ of cascades in $\manibold$. Suppose there exists a continuous family of linear maps $\{L_t\}_{-1 \leq t \leq 1}$ of the plane such that $L_0$ is the identity map and for all $t$, $\Gbold_t = L_t \circ \Fbold_t \circ L_t^{-1}$. Then, $L_t$ must be the identity map for all $t$.
\end{lemma}

\begin{proof}
    Let $\Fbold_{t,-n}$ and $\Gbold_{t,-n}$ denote the $n$\textsuperscript{th} anti-renormalization of $\Fbold_{t}$ and $\Gbold_{t}$ respectively. Then, for all $n$ and $t \in [-1,1]$,
    \[
        \Gbold_{t,-n} = L_t \circ \Fbold_{t,-n} \circ L_t^{-1}.
    \]
    As $n \to \infty$,
    \[
        \Fbold_* = \lim_{n \to \infty} \Gbold_{t,-n} = \lim_{n\to\infty} L_t \circ \Gbold_{t,-n} \circ L_t^{-1} = L_t \circ \Fbold_* \circ L_t^{-1}.
    \]
    Now, suppose for a contradiction that $\{L_t\}_t$ is a non-trivial family. Pick a critical point $C$ of $\Fbold_*$. From the equation above, $L_t(C)$ must also be a critical point of $\Fbold_*$ of the same generation for all $t$, but this is impossible because critical points of a fixed generation are discrete.
\end{proof}

\begin{corollary}
\label{cor:no-ilf-fixed-pt}
    The renormalization fixed point $\Fbold_*$ admits no invariant line field.
\end{corollary}

\begin{proof}
    Suppose for a contradiction that $\Fbold_*$ admits an invariant line field $\boldsymbol{\mu}$. By Proposition \ref{prop:integration-of-ILF}, $\boldsymbol{\mu}$ induces a family $\{\Gbold_t\}_t$ in $\manibold$ where $\Gbold_0 \equiv \Fbold_*$, together with quasiconformal maps $\phi_t: \C \to \C$ conjugating $\Fbold_*$ with $\Gbold_t$ for all $t$. Each of $\Gbold_t$ induces a rotational corona $g_t$ with rotation number $\theta$, which, by Theorem \ref{thm:weak-hyperbolicity}, implies that $g_t$ must also be on the local stable manifold. Therefore, $g_t \equiv f_*$ and the family $\Gbold_t$ is trivial. For every $t$, $\phi_t$ commutes with $\Fbold_*$ along the Herman curve $\Hq$ of $\Fbold_*$. As such, $\phi_t$ is the identity on $\Hq$, and so on the iterated preimages $\bigcup_{P} \Fbold_*^{-P} (\Hq)$ of $\Hq$ as well. By Lemma \ref{lem:jcas-and-H}, $\phi_t$ must be the identity map on the whole plane. This contradicts the assumption that the support of $\boldsymbol{\mu}$ has positive measure.
\end{proof}

%%%%%%%%%%%%%%%%%%%%%%%%%%%%%%%%%%%%%%%%%%%%%%%%%%%%%%%%%%%%%%%%%%%%%%%%%%%%%%%%%%%%%%%%%%%%%%%%%%%%%%%%%%%%%%%%%%%%%%%%%%%%%%%%%%%%%%%%%%%%%%%%%%%%%%%%%%%%%%%%%%%%%%%%%%%%%%%%%%%%

\part{Dynamical properties of cascades}

\section{The external structure of \texorpdfstring{$\Fbold_*$}{F*}}
\label{sec:external}

Consider the dynamics of the cascade $\Fbold = \Fbold_*$ corresponding to the fixed point $f_*$ of the renormalization operator. We denote by $\Hq$ the Herman curve of $\Fbold$, which is defined to be the full lift of the Herman curve of $f_*$. The action of $\Fbold$ along $\Hq$ can be described as follows. For $a \in \C$, we denote the translation map by $a$ by $T_a(z):= z+a$.

\begin{lemma}
\label{lem:dynamics-on-H}
    There is a quasisymmetric map $h: (\Hq,0) \to (\R,0)$ that conjugates $\left(\Fbold^P|_\Hq\right)_{P \in \Tbold}$ with the cascade of translations $\left(T^P\right)_{P \in \Tbold}$ defined by $T^{(n,a,b)} := T_{\tbold^{-n}(b\vbold-a\ubold)}$, where $\ubold,\vbold>0$ and $\theta= \frac{\ubold}{\ubold+\vbold}$.
\end{lemma}

\begin{proof}
    The pre-corona $F_*$ associated to $f_*$ admits an invariant quasiarc which projects to the Herman curve of $f_*$. In linearizing coordinates, this corresponds to an invariant quasiarc $\Hq_0$ of $\fbold_{0,\pm} : \Ubold_\pm \to \Sbold)$ which passes through $0$ and connects $\fbold_{0,+}(0)$ and $\fbold_{0,-}(0)$. The dynamics of $\fbold_{0,\pm}$ along $\Hq_0$ is quasisymetrically conjugate to a pair of translations $\left(T_{-\theta}|_{[0,1-\theta]}, T_{1-\theta}|_{[-\theta,0]} \right)$ on the real interval $[-\theta,1-\theta]$. Set $\ubold= - \theta$ and $\vbold=1-\theta$. As we extend $\fbold_{0,\pm}$ to its maximal $\sigma$-proper extension via $A_*$, the quasisymmetric conjugacy $h$ between $(\fbold_{0,-},\fbold_{0,+})$ and $(T_{-\ubold}, T_{-\vbold})$ extends to the whole lift $\Hq$ of $\Hq_0$. The claim holds because the pairs $(\fbold_{0,-},\fbold_{0,+})$ and $(T_{-\ubold}, T_{-\vbold})$ generate the cascades $\Fbold^{\geq 0}|_\Hq$ and $T^{(n,a,b)} := T_{\tbold^{-n}(b\vbold-a\ubold)}$ via iteration and rescaling according to (\ref{eqn:self-sim-cascade}) and \S\ref{ss:cascade}.
\end{proof}

In this section, we will comprehensively describe the dynamics of $\Fbold$ beyond $\Hq$. We study the structure of preimages of $\Hq$ in \S\ref{ss:lakes}--\ref{ss:limbs}, then the structure of the finite-time escaping set $\Esc_{<\infty}:=\Esc_{<\infty}(\Fbold)$ in \S\ref{ss:alpha}--\ref{ss:external-chains}, and lastly the dynamical puzzles cut out by subsets of $\Esc_{<\infty}$ in \S\ref{ss:wakes}.

\subsection{Lakes}
\label{ss:lakes}

Let us label the components of $\C \backslash \Hq$ by $\Obold^0$ and $\Obold^\infty$, which we will refer to as the \emph{oceans} of $\Fbold$. The two oceans will be distinguished as follows. For $\bullet \in \{0,\infty\}$ and for any point $x$ in $\Sbold \cap \Obold^\bullet$ close to $0$, we assume that there are $d_\bullet$ preimages of $x$ under $\fbold_{0,\pm}: \Ubold_{\pm} \to \Sbold$ that are located near the critical point and inside of $\Obold^\bullet$.

\begin{definition}
\label{def:lakes}
    A \emph{lake} $\Obold$ \emph{of generation} $P \in \Tbold$ is a connected component of $\Fbold^{-P}(\Obold^\bullet)$ for some $\bullet \in \{0,\infty\}$. Its \emph{coast} is defined by $\partial^c\Obold := \partial \Obold \cap \Dom\left(\Fbold^P\right)$.
\end{definition}

\begin{lemma}[Chessboard rule]
\label{chessboard-rule}
    For every $P \in \Tbold_{>0}$ and $\bullet \in \{0,\infty\}$, the preimage $\Fbold^{-P}(\Hq)$ is a tree in $\Dom\left(\Fbold^P\right)$ and $\Fbold^{-P}(\Obold^\bullet)$ is disjoint union of lakes $\bigcup_{i \in \N} \Obold_i$ of generation $P$ such that 
    \begin{enumerate}[label=\textnormal{(\arabic*)}]
        \item each lake $\Obold_i$ is an unbounded non-separating disk in $\Dom\left(\Fbold^P\right)$;
        \item for $j \neq i$, the intersection $\partial^c \Obold_i \cap \partial^c \Obold_j$ is either empty or a singleton consisting of a critical point of $\Fbold^P$.
    \end{enumerate}
\end{lemma}

\begin{proof}
    The whole lemma follows immediately from $\sigma$-properness of the cascade, e.g. \cite[Lemma 5.1]{DL23}, and the fact that $\CV(\Fbold)$ is contained in $\Hq$.
\end{proof}

Given any lake $\Obold$ of some generation $P\in \Tbold_{>0}$, the map $\Fbold^P$ sends $\Obold$ univalently onto an ocean, and its coast homeomorphically onto $\Hq$. In general, when $0<P<Q$, a lake of generation $Q$ is contained in a lake of generation $P$, and $\Fbold^{Q-P}$ conformally sends any lake of generation $Q$ onto a lake of generation $P$.

\begin{lemma}
    For every $P \in \Tbold_{>0}$, there is a unique critical point $C_P \in \Hq$ of $\Fbold^{\geq 0}$ of generation $P$ and a pairwise disjoint collection of lakes
    \begin{equation}
    \label{list-of-lakes-at-CP}
        \prescript{}{1}{\Obold}^0_P, \enspace \ldots, \enspace \prescript{}{2d_0-3}{\Obold}^0_P, \enspace 
        \prescript{}{1}{\Obold}^\infty_P, \enspace \ldots, \enspace \prescript{}{2d_\infty-3}{\Obold}^\infty_P, 
    \end{equation}
    of generation $P$ together with a bouquet of pairwise-disjoint open quasiarcs
    \begin{equation}
    \label{eqn:spines}
        \prescript{}{1}{\Hq}^0_P, \enspace \ldots, \enspace
        \prescript{}{2d_0-2}{\Hq}^0_P, \enspace
        \prescript{}{1}{\Hq}^\infty_P, \enspace \ldots, \enspace \prescript{}{2d_\infty-2}{\Hq}^\infty_P, 
    \end{equation}
    rooted at $C_P$ such that for each $\bullet \in \{0,\infty\}$ and $j \in \{1,\ldots, 2d_\bullet-3\}$,
    \begin{enumerate}[label=\textnormal{(\arabic*)}]
        \item the coast of $\prescript{}{j}{\Obold}^\bullet_P$ is $\prescript{}{j}{\Hq}^\bullet_P \cup \{C_P\} \cup \prescript{}{j+1}{\Hq}^\bullet_P$;
        \item $\prescript{}{j}{\Obold}^\bullet_P$ is contained in $\Obold^\bullet$;
        \item $\prescript{}{j}{\Obold}^\bullet_P$ is mapped conformally by $\Fbold^P$ onto $\Obold^\bullet$ if $j$ is even, and onto $\C \backslash \overline{\Obold^\bullet}$ if $j$ is odd.
    \end{enumerate}
\end{lemma}

\begin{proof}
    The existence and uniqueness of $C_P$ is due to the fact that $\Fbold^P$ restricts to a homeomorphism on $\Hq$. From the previous lemma, $\Fbold^{-P}(\Hq)$ is a tree. The quasiarcs $\prescript{}{j}{\Hq}^\bullet_P$'s are precisely the components of $\Fbold^{-P}(\Hq) \backslash \{C_P\}$, and the lakes $\prescript{}{j}{\Obold}^\bullet_P$'s in (\ref{eqn:spines}) are precisely the connected components of $\Dom\left(\Fbold^P\right) \backslash \Fbold^{-P}(\Hq)$ which touch $\Hq$ at exactly one point, which is $C_P$. For all $S<P$, the image of each quasiarc $\prescript{}{j}{\Hq}^\bullet_P$ under $\Fbold^S$ is disjoint from $0$. Therefore, $\Fbold^P$ maps each of $\prescript{}{j}{\Hq}^\bullet_P$ onto a component of $\Hq \backslash \{0\}$ homeomorphically. They can be enumerated such that the three claims above hold because $C_P$ has inner and outer criticalities $d_0$ and $d_\infty$ respectively. 
\end{proof}

Each quasiarc in (\ref{eqn:spines}) is called a \emph{spine} of $C_P$. The spines in (\ref{eqn:spines}) are labelled in counterclockwise order about $C_P$.

\begin{figure}
    \centering
    \begin{tikzpicture}[scale=1.25]
        % for Hq
        \draw[ultra thick] (-2.4,-1) -- (0,0) -- (3,-1.25) -- (5.4,-0.25) -- (7.2,-1);
        \node [black, font=\bfseries] at (-2.2,-1.2) {$\Hq$};

        % for the big one
        \filldraw[gray, fill=yellow!05!white] (0,0) .. controls (1.5,0.2) and (2,1.5) .. (2,2) .. controls (1.8,3) and (0.5,3.8) .. (0,3.5) .. controls (-0.5,3.8) and (-1.8,3) .. (-2,2) .. controls (-2,1.5) and (-1.5,0.2) .. (0,0); 
        \filldraw[gray, fill=green!05!white] (0,0) .. controls (1, 2.2) .. (0,3.5) .. controls (-1,2.2) .. (0,0); 
        \filldraw[gray, fill=green!05!white] (0,0) .. controls (0.9,-1.8) .. (0,-2.5) .. controls (-0.9,-1.8) .. (0,0);
        \node [red, font=\bfseries] at (0,-0.02) {\large $\bullet$};
        \node [red, font=\bfseries] at (-1,0) {\small $C_P$};
        \draw[red] (-0.7,0) -- (-0.15,0);
        \node [white, font=\bfseries] at (0, -2.5) {$\bullet$};
        \node [white, font=\bfseries] at (0, 3.5) {$\bullet$};
        \node [yellow!50!black, font=\bfseries] at (1.1,1.1) {\small $\prescript{}{1}{\Obold}^0_P$};
        \node [green!50!black, font=\bfseries] at (0,2) {\small $\prescript{}{2}{\Obold}^0_P$};
        \node [yellow!50!black, font=\bfseries] at (-1.1,1.1) {\small $\prescript{}{3}{\Obold}^0_P$};
        \node [green!50!black, font=\bfseries] at (0,-1.5) {\small $\prescript{}{1}{\Obold}^\infty_P$};

        % for the small one
        \filldraw[gray, fill=yellow!10!white] (5.4,-0.25) .. controls (6.3,-0.13) and (6.6,0.65) .. (6.6,0.95) .. controls (6.48,1.55) and (5.7,2.03) .. (5.4,1.85) .. controls (5.1,2.03) and (4.32,1.55) .. (4.2,0.95) .. controls (4.2,0.65) and (4.5,-0.13) .. (5.4,-0.25);
        \filldraw[gray, fill=green!10!white] (5.4,-0.25) .. controls (6, 1.07) .. (5.4,1.85) .. controls (4.8,1.07) .. (5.4,-0.25);
        \filldraw[gray, fill=green!10!white] (5.4,-0.25) .. controls (5.94,-1.33) .. (5.4,-1.75) .. controls (4.86,-1.33) .. (5.4,-0.25);
        \node [red, font=\bfseries] at (5.4,-0.26) {\large $\bullet$};
        \node [red, font=\bfseries] at (6.4,-0.25) {\small $C_Q$};
        \draw[red] (6.2,-0.25) -- (5.55,-0.25);
        \node [white, font=\bfseries] at (5.4, 1.85) {$\bullet$};
        \node [white, font=\bfseries] at (5.4, -1.75) {$\bullet$};
        \node [yellow!50!black, font=\bfseries] at (6.2,0.9) {\small $\prescript{}{1}{\Obold}^0_Q$};
        \node [green!50!black, font=\bfseries] at (5.4,1) {\small $\prescript{}{2}{\Obold}^0_Q$};
        \node [yellow!50!black, font=\bfseries] at (4.6,0.9) {\small $\prescript{}{3}{\Obold}^0_Q$};
        \node [green!50!black, font=\bfseries] at (5.4,-1.1) {\small $\prescript{}{1}{\Obold}^\infty_Q$};

        % for the baby
        \filldraw[gray, fill=yellow!10!white] (2,2) .. controls (2.05,2.375) and (2.375,2.5) .. (2.5,2.5) .. controls (2.75,2.45) and (2.95,2.125) .. (2.875,2) .. controls (2.95,1.875) and (2.75,1.55) .. (2.5,1.5) .. controls (2.375,1.5) and (2.05,1.625) .. (2,2); 
        \filldraw[gray, fill=green!10!white] (2,2) .. controls (2.55, 2.25) .. (2.875, 2) .. controls (2.55, 1.75) .. (2,2); 
        \filldraw[gray, fill=green!10!white] (2,2) .. controls (1.55, 2.225) .. (1.375,2) .. controls (1.55, 1.785) .. (2,2);
        \node [red, font=\bfseries] at (2,2) {$\bullet$};
        \node [white, font=\bfseries] at (2.875,2) {$\bullet$};
        \node [white, font=\bfseries] at (1.375,2) {$\bullet$};
        \node [red, font=\bfseries] at (2.1,1) {\small $\prescript{}{1}C^0_{P,Q}$};
        \draw[red] (2.1,1.25) -- (2.02,1.85);
        \draw[line width=0.5pt,-latex] (2.8,2.4) .. controls (3.5,2.6) and (4.5,2.3) .. (5,2);
        \node [black, font=\bfseries] at (3.8,2.7) {$\Fbold^P$};
\end{tikzpicture}

    \caption{The structure of lakes attached to critical points $C_P$, $C_Q$, and $\prescript{}{1}C^0_{P,Q}$ when $(d_0,d_\infty)=(3,2)$.}
    \label{fig:lakes}
\end{figure}
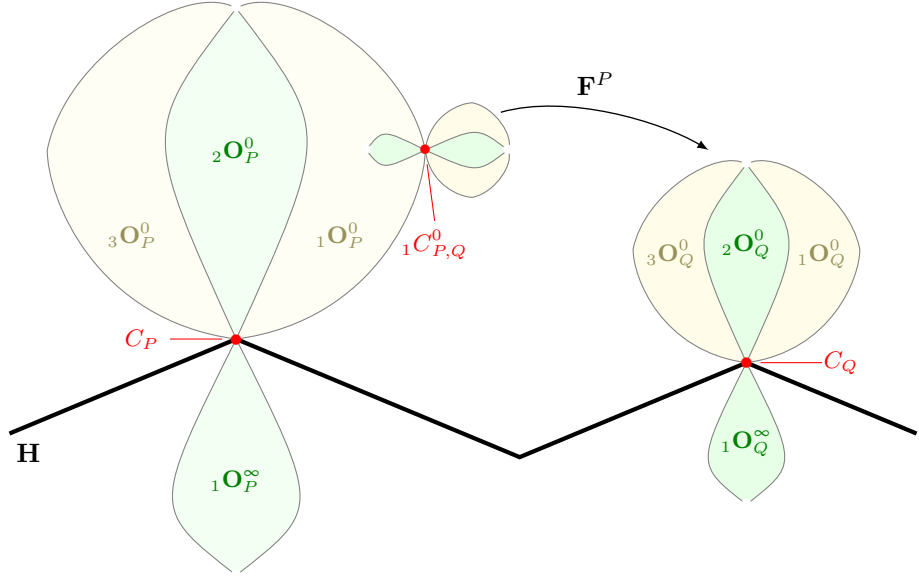

Let us pick a pair of power-triples $P,Q \in \Tbold_{>0}$. For any $\bullet \in \{0,\infty\}$ and any $j \in \{1,\ldots, d_\bullet-1\}$, the union of two consecutive spines $\prescript{}{2j-1}{\Hq}^\bullet_P \cup \prescript{}{2j}{\Hq}^\bullet_P$ are mapped homeomorphically by $\Fbold^P$ onto $\Hq \backslash \{0\}$ and so it contains a unique critical point $\prescript{}{j}{C}^\bullet_{P,Q}$ of generation $P+Q$. Attached to this critical point is a bouquet of lakes
\[
    \prescript{}{j,1}{\Obold}^{\bullet,0}_{P,Q}, \enspace \ldots, \enspace
    \prescript{}{j,2d_0-3}{\Obold}^{\bullet,0}_{P,Q}, \enspace
    \prescript{}{j,1}{\Obold}^{\bullet,\infty}_{P,Q}, \enspace \ldots, \enspace \prescript{}{j,2d_\infty-3}{\Obold}^{\bullet,\infty}_{P,Q}, 
\]
of generation $P+Q$ together with spines
\[
    \prescript{}{j,1}{\Hq}^{\bullet,0}_{P,Q}, \enspace \ldots, \enspace \prescript{}{j,2d_0-2}{\Hq}^{\bullet,0}_{P,Q}, \; 
    \prescript{}{j,1}{\Hq}^{\bullet,\infty}_{P,Q}, \enspace \ldots, \enspace \prescript{}{j,2d_\infty-2}{\Hq}^{\bullet,\infty}_{P,Q}, 
\]
meeting at $\prescript{}{j}{C}^\bullet_{P,Q}$ such that $\prescript{}{j,k}{\Obold}^{\bullet,\circ}_{P,Q}$ has coast
\[
\partial^c \prescript{}{j,k}{\Obold}^{\bullet,\circ}_{P,Q} = \prescript{}{j,k}{\Hq}^{\bullet,\circ}_{P,Q} \cup \left\{\prescript{}{j}{C}^\bullet_{P,Q}\right\} \cup \prescript{}{j,k+1}{\Hq}^{\bullet,\circ}_{P,Q}
\]
and is mapped univalently by $\Fbold^P$ onto $\prescript{}{k}{\Obold}^\circ_{Q}$.

Consider a tuple $S = (P_1,\ldots, P_{m+1}) \in \Tbold^{m+1}_{>0}$ of $m+1$ power-triples for some $m \in \N$. We denote the sum by
\[
    |S| := \sum_{i=1}^{m+1} P_i.
\]
Given $\blacksquare = (\bullet_1,\ldots, \bullet_m) \in \{0,\infty\}^m$ and $J = (j_1,\ldots j_m)$ where $j_i \in \{1,\ldots, d_{\bullet_i}-1\}$ for all $i$, we inductively define a critical point $\prescript{}{J}{C}^\blacksquare_S$ of generation $|S|$. Attached to this critical point are lakes $\prescript{}{J,i}{\Obold}^{\blacksquare,\bullet}_S$ for $\bullet \in \{0,\infty\}$ and $i \in \{1,\ldots,2d_\bullet-3\}$, and spines $\prescript{}{J,j}{\Hq}^{\blacksquare,\bullet}_S$ for $\bullet \in \{0,\infty\}$ and $j \in \{1,\ldots,2d_\bullet-2\}$.

\begin{definition}
\label{def:middle-lake}
    We say that a lake $\Obold$ is a \emph{middle lake} if it is of the form $\prescript{}{J,j}{\Obold}^{\blacksquare,\bullet}_S$ described above. The finite tuple $S$ is called the \emph{itinerary} of $\Obold$.
\end{definition}

Consider a lake $\Obold$ of generation $P\in \Tbold_{>0}$. Let $Q \in \Tbold$ be the smallest power-triple such that the coast of $\Obold$ touches $\Fbold^{-Q}(\Hq)$. 

\begin{lemma}[Left and right coasts]
\label{lem:empty-coast}
    The intersection between $\partial^c \Obold$ and $\Fbold^{-Q}(\Hq)$ is either a singleton or a closed quasiarc, and the complement $\partial^c \Obold \backslash \Fbold^{-Q}(\Hq)$ consists of two non-empty open quasiarcs $\partial^c_l \Obold$ and $\partial^c_r \Obold$.
\end{lemma}

\begin{proof}
    It is sufficient to consider the case when $Q=0$. The intersection between $\partial^c \Obold$ and $\Fbold^{-Q}(\Hq)$ is connected because of the tree structure of $\Fbold^{-P}(\Hq)$. For any point $z$ in $\partial^c \Obold \cap \Hq$, every component of $\Hq \backslash \{z \}$ contains infinitely many critical points of generation at most $P$, and each of these points is a branch point of the tree $\Fbold^{-P}(\Hq)$. Since $\partial^c \Obold \cap \Hq$ does not contain such branch points, the claim follows.
\end{proof}

We call $\partial^c_l \Obold$ and $\partial^c_r \Obold$ the \emph{left and right coasts} of $\Obold$ respectively, and we always assume that $\partial^c_l \Obold$, $\partial^c \Obold \cap \Fbold^{-Q}(\Hq)$, and $\partial^c_r \Obold$ are oriented counterclockwise relative to $\Obold$.

The closure of the left coast of $\Obold$ admits a maximal sequence of critical points $c_{l,1}$, $c_{l,2}$, $c_{l,3}$, $\ldots$ of $\Fbold^P$, labelled in increasing order of generation. We define the \emph{left itinerary} of $\Obold$ to be the sequence $I_l := (P_{l,1}, P_{l,2},\ldots)$ where each $P_{l,i}$ is the generation of $c_{l,i}$. Similarly, we define the \emph{right itinerary} $I_r$ of $\Obold$.

\begin{lemma}
\label{lem:finite-or-infinite}
    Consider a lake $\Obold$ of generation $P \in \Tbold_{>0}$ with left and right itineraries $I_l = (P_{l,1}, P_{l,2},\ldots)$ and $I_r = (P_{r,1}, P_{r,2},\ldots)$ respectively. Then,
    \begin{enumerate}[label=\textnormal{(\arabic*)}]
        \item $\sup_i P_{l,i} = \sup_j P_{r,j} = P$;
        \item if $I_l$ (resp. $I_r$) is finite, then the left (resp. right) coast of $\Obold$ contains a spine attached a critical point $\prescript{}{J}C^\blacksquare_S$ of generation $|S|=P$;
        \item if both $I_l$ and $I_r$ are finite, then $\Obold$ is a middle lake attached to the critical point $\prescript{}{J}C^\blacksquare_S$;
        \item either $I_l$ or $I_r$ is a finite sequence.
    \end{enumerate}
\end{lemma}

\begin{proof}
    Suppose for a contradiction that $\sup_i P_{l,i}< P$, so then there is some $P' \in \Tbold$ such that $\sup_i P_{l,i} < P' < P$. Then, $\Fbold^{P'}(\Obold)$ is a lake of positive generation with an empty left coast, which is impossible due to Lemma \ref{lem:empty-coast}. Therefore, (1) holds.

    Suppose $I_l$ is finite. By (1), there exists a critical point $c_l$ of generation $P$ on $\overline{\partial^c_l \Obold}$. Removing $c_l$ splits the coast into two open quasiarcs, one of which contains no critical points of $\Fbold^P$ and is thus a spine attached to $c_l$. This implies (2). Suppose $I_r$ is also finite, so there also exists a critical point $c_r$ of generation $P$ on $\overline{\partial^c_r \Obold}$. The complement of the interval $[c_l,c_r] \subset \partial^c \Obold$ is now a pair of spines of generation $P$ attached to $c_l$ and $c_r$ respectively. Recall that $\Fbold^P$ sends each of these spines to a component of $\Hq \backslash \{0\}$. However, since $\Fbold^P : \partial^c \Obold \to \Hq$ is a homeomorphism, we see that $c_l=c_r$ and $\Obold$ is a middle lake. Hence, (3) holds.

    Let us now prove (4). We will again assume without loss of generality that the coast of $\Obold$ touches $\Hq$. Let us pick a point $y$ in $\partial^c \Obold \cap \Hq$. If the open interval $(y,C_P) \subset \Hq$ does not contain any critical point of generation $\leq P$, then either $\partial^c_l \Obold$ or $\partial^c_r \Obold$ is rooted at $C_P$ and contains no other critical points of generation $\leq P$. Otherwise, by Lemma \ref{lem:critical-pts}, there are only finitely many critical points of generation $\leq P$ within $(y, C_P)$, and they have some maximum generation $R<P$. We then apply the previous argument to the lake $\Fbold^R(\Obold)$ and the interval $\left(\Fbold^R(y), C_{P-R}\right) \subset \Hq$.
\end{proof}

Consider a critical point $\prescript{}{J}C^\blacksquare_S$ of $\Fbold^{\geq 0}$. There exist lakes
\begin{equation}
\label{eqn:extra-lakes}
    \prescript{}{J,l}\Obold^{\blacksquare,0}_S, \enspace \prescript{}{J,r}\Obold^{\blacksquare,0}_S, \enspace
    \prescript{}{J,l}\Obold^{\blacksquare,\infty}_S, \enspace
    \prescript{}{J,r}\Obold^{\blacksquare,\infty}_S
\end{equation}
of generation $|S|$ such that 

\begin{enumerate}[label=\textnormal{(\roman*)}]
    \item they are disjoint from all the middle lakes rooted at $\prescript{}{J}C_S^\blacksquare$;
    \item for $\bullet \in \{0,\infty\}$, the right coast of $\prescript{}{J,l}\Obold^{\blacksquare,\bullet}_S$ contains the spine $\prescript{}{J,2d_\bullet-2}\Hq^{\blacksquare,\bullet}_S$ and the left coast of $\prescript{}{J,r}\Obold^{\blacksquare,\bullet}_S$ contains the spine $\prescript{}{J,1}\Hq^{\blacksquare,\bullet}_S$;
    \item if $j,j' \in \{l,r\}$ and $j \neq j'$, the coasts of $\prescript{}{J,j}\Obold^{\blacksquare,0}_S$ and $\prescript{}{J,j'}\Obold^{\blacksquare,\infty}_S$ intersect on a non-degenerate closed interval in $\Fbold^{-|S|}(\Hq)$ with endpoint $\prescript{}{J}C^\blacksquare_S$.
\end{enumerate} 
We will call the lakes in (\ref{eqn:extra-lakes}) the \emph{left/right side lakes} of $\prescript{}{J}C^\blacksquare_S$. 

Observe that by (ii),
\[
    \prescript{}{J,r}\Obold^{\blacksquare,0}_S, \:
    \prescript{}{J,1}\Obold^{\blacksquare,0}_S, \:
    \ldots, \:
    \prescript{}{J, 2d_0-3}\Obold^{\blacksquare,0}_S, \:
    \prescript{}{J,l}\Obold^{\blacksquare,0}_S, 
    \prescript{}{J,r}\Obold^{\blacksquare,\infty}_S, \:
    \prescript{}{J,1}\Obold^{\blacksquare,\infty}_S, \:
    \ldots, \:
    \prescript{}{J, 2d_\infty-3}\Obold^{\blacksquare,\infty}_S, \:
    \prescript{}{J,l}\Obold^{\blacksquare,\infty}_S
\]
are in counterclockwise order about $\prescript{\blacksquare}{J}C_S$ and the closure of their union is a neighborhood of $\prescript{\blacksquare}{J}C_S$. By Lemma \ref{lem:finite-or-infinite} (4), the left itinerary of $\prescript{}{J,l}\Obold^{\blacksquare,\bullet}_S$ and the right itinerary of $\prescript{}{J,r}\Obold^{\blacksquare,\bullet}_S$ are infinite. The following is a consequence of Lemma \ref{lem:finite-or-infinite} (2)--(4).

\begin{corollary}
\label{cor:middle-or-side}
    Every lake $\Obold$ is either a middle lake or a side lake of a critical point $\prescript{}{J}C^{\blacksquare}_S$. In other words, $\Obold$ is of the form $\prescript{}{J,j}\Obold^{\blacksquare,\bullet}_S$ where $j \in \{l,1,\ldots,2d_\bullet-3,r\}$.
\end{corollary}

Given some tuple $S = (P_1,\ldots,P_k) \in \Tbold_{>0}^k$, we can perform scalar multiplication by $\tbold$ and denote $\tbold S := (\tbold P_1,\ldots, \tbold P_k)$. The following is a direct consequence of (\ref{eqn:self-sim-cascade}).

\begin{lemma}
\label{lem:self-sim-lakes}
    For any middle or side lake $\prescript{}{J,j}\Obold^{\blacksquare,\bullet}_S$ rooted at a critical point $\prescript{}{J}C^{\blacksquare}_S$,
    \[ 
    A_* \left( \prescript{}{J}C^{\blacksquare}_S \right) = \prescript{}{J}C^{\blacksquare}_{\tbold S}
    \quad \text{and} \quad
    A_* \left(\prescript{}{J,j}\Obold^{\blacksquare,\bullet}_S \right) = \prescript{}{J,j}\Obold^{\blacksquare,\bullet}_{\tbold S}.
    \]
\end{lemma}

\begin{proof}
    Recall from (\ref{eqn:self-sim-cascade}) that $A_*$ conjugates $\Fbold^P$ and $\Fbold^{\tbold P}$ for any $P \in \Tbold_{>0}$. Since $A_*$ preserves $\Hq$, then $A_*(C_P) = C_{\tbold P}$ and thus $A_*\left(\prescript{}{j}\Obold^{\bullet}_{P} \right) = \prescript{}{j}\Obold^{\bullet}_{\tbold P} $ for all $\bullet \in \{0,\infty\}$ and $j \in \{l,1,\ldots,2 d_\bullet-3,r\}$.

    Suppose a spine $\prescript{}{j}\Hq^\bullet_P$ attached to $C_P$ contains some critical point $\prescript{}{i}C^\bullet_{P,Q}$ where $i = \lceil \frac{j}{2}\rceil$. Since $A_*(\prescript{}{i}C^\bullet_{P,Q})$ is contained in $\prescript{}{j}\Hq^\bullet_{\tbold P}$ and is a critical point of generation $\tbold(P+Q)$, then it is equal to $\prescript{}{i}C^\bullet_{\tbold P,\tbold Q}$. The rest follows by induction.
\end{proof}

\subsection{Limbs}
\label{ss:limbs}

\begin{definition}
    A \emph{limb} $\prescript{}{J}{\Lbold}^\blacksquare_S$ is the union of the spine $\prescript{}{J}{\Hq}^\blacksquare_S$ together with all spines of the form $\prescript{}{J, j_1,\ldots, j_k}{\Hq}^{\blacksquare,\bullet_1,\ldots, \bullet_k}_{S, P_1,\ldots, P_k}$. The \emph{generation} of $\prescript{}{J}{\Lbold}^\blacksquare_S$ is $|S|$.
\end{definition}

By Lemma \ref{lem:self-sim-lakes}, the linear map $A_*$ sends each limb $\prescript{}{J}{\Lbold}^{\blacksquare}_S$ onto another limb $\prescript{}{J}{\Lbold}^{\blacksquare}_{\tbold S}$.

\begin{lemma}
\label{lem:bdd-limbs}
    Every limb is bounded in $\C$.
\end{lemma}

The proof we present below is identical to \cite[Lemma 5.10]{DL23}.

\begin{proof}
    Recall the rescaled pre-corona $\Fbold^\#_n = \left( \fbold_{n,\pm}^\# : \Ubold_{n,\pm}^\# \to \Sbold_n^\# \right)$ where $\Sbold^\#_n := A_*^n(\Sbold)$ for all $n \in \Z$. 
    Each $\Sbold_n^\#$ is compactly contained in $\Sbold_{n-1}^\#$ and the union of $ \Sbold_n^\#$ across all $n \in \Z$ is the whole complex plane. 
    For every $n \in \Z$, there is a gluing map 
    \[
    \rho_n : \Sbold_n^\# \to V
    \]
    projecting $\Fbold_n^\#$ to the corona $f: U \to V$. 
    
    Let us fix a large $n \ll 0$. 
%   Consider open rectangles $X_0 \subset X_1$ defined by
%   \[
%   X_0 := \rho_n\left(\Sbold_0^\#\right) \quad \text{and} \quad X_1 := \rho_n\left(\Sbold_{-1}^\#\right)
%   \]
%   in the dynamical plane of $f$. 
    Denote by $\Hq_*$ the Herman curve of $f$, and consider the interval $I := \rho_n(\Sbold_0^\#) \cap \Hq_*$ and pick a slightly smaller interval $J \subset I$.

    \begin{claim1}
        There is some $M \in \N$ such that the following holds. For any connected component $W$ of $\rho_n(\Sbold_{-1}^\#) \backslash \Hq_*$, any $m \geq M$, and any point $x \in J$ with $f^m(x) \in \partial W$, the univalent lift $W_{-m}$ of $W$ under $f^m$ along the orbit $x, \ldots, f^m(x)$ is contained in $\rho_n(\Sbold_0^\#)$.
    \end{claim1}

    \begin{proof}
        Let us assume without loss of generality that $W$ is contained in the outer component $Y^\infty$ of $\C \backslash \Hq_*$. Since $f^i(x) \in \Hq_*$ for all $i\geq 0$, then the lift $W_{-m}$ is also contained in $Y^\infty$. We will first claim that $W_{-m}$ is well-defined and $f^m:W_{-m} \to W$ is univalent by ensuring that $W_{-k}$ is disjoint from $\partial_\frb U$ for all $k\geq 0$.

        Let us pick two outer external rays $R_l$ and $R_r$ landing at a pair of points on $\Hq_*$ such that $R_l$ is slightly on the left of $W$ and $R_r$ is slightly on the right of $W$. Since $n\ll 0$, the difference $\delta$ between the external angles of $R_l$ and $R_r$ is small. For $k=1,\ldots,m$, let $R_{l,-k}$ and $R_{r,-k}$ be the preimages of $R_l$ and $R_r$ under $f^{k}$ such that they are slightly on the left and right of $W_{-k}$ respectively.
        
        By definition, for every arc $\gamma^\infty_j$ on the forbidden boundary $\partial_\frb U$ of $U$, the part that gets mapped to $\gamma_1 \cap Y^\infty$ is an external ray of some definite distance from $\Hq_*$. The difference between the external angles of $R_{l,-k}$ and $R_{r,-k}$ is $\delta/d_\infty^k$, which is much smaller than $\delta$. Therefore, $W_{-k}$ is disjoint from $\partial_\frb U$ for all $k$ and so $f^m : W_{-m} \to W$ is univalent.

        For sufficiently large $m$, $W_{-m}$ is within a small neighborhood of $\Hq_*$ and it is sandwiched between the rays $R_{l,-m}$ and $R_{r,-m}$, whose external angles differ by a small constant. By local connectivity (Lemma \ref{lem:small-limbs}), $W_{-m}$ must be contained in a small neighborhood of $J$, and thus $W_{-m}$ is contained in $\rho_n(\Sbold_0^\#)$.
    \end{proof}

    The composition $\rho_n \circ A_*^{-n}$ identifies $\Sbold_n^\#$ with $\rho_n(\Sbold_0^\#)$. Let $\Jbold_n := A_*^n \circ \rho_n^{-1}(J)$.
    
    \begin{claim2}
        There is a power-triple $R \in \Tbold_{>0}$ such that $\Fbold^R(\Jbold_0) \subset \Jbold_{-1}$ and for every point $x$ on $\Jbold_0$, if $\Fbold^P(x) \in \Sbold_{-1}^\#$ for some $P \geq R$, then there is an open subset $\mathbf{\Omega}_P$ of $\Sbold_0^\# \backslash \Hq$ such that $x \in \partial \mathbf{\Omega}_P$ and $\Fbold^P$ maps $\mathbf{\Omega}_P$ conformally to $\Sbold_{-1}^\# \backslash \Hq$.
    \end{claim2}

    \begin{proof}
        Since the action of $\Fbold^{\geq 0}$ on $\Hq$ is combinatorially modelled by the cascade of translations $(T^P)_{P \in \Tbold}$ on $\R$, there is an arbitrarily large $R \in \Tbold$ such that $\Fbold^R(\Jbold_0) \subset \Jbold_{-1}$. Suppose $x \in \Jbold_0$ and $\Fbold^P(x) \in \Sbold_{-1}^\#$ for some $P \geq R$. Since $\fbold_{-1,\pm}^\#$ is the first return map of the cascade $\Fbold^{\geq 0}$ back to $\Sbold_{-1}^\#$, then $\Fbold^P$ is the $m$\textsuperscript{th} iterate of the pair $\fbold_{-1,\pm}^\#$ for some $m \in \N$. If $R$ is chosen to be large enough, then $m \geq M$ and the claim now follows from Claim 1.
    \end{proof}

    By self-similarity, Claim 2 also holds if we replace $\Jbold_0$, $\Jbold_{-1}$, $P$, and $R$ with $\Jbold_n$, $\Jbold_{n-1}$, $\tbold^n P$, and $\tbold^n R$ respectively.  

    \begin{claim3}
        There is a power-triple $Q \in \Tbold_{>0}$ such that for every $n \ll 0$ and every point $x \in \Jbold_0$, if $\Fbold^P(x) \in \Sbold_n^\#$ for some $P\geq Q$, then there is an open subset $\mathbf{\Omega}_0$ of $\Sbold_0^\# \backslash \Hq$ such that $x \in \partial \mathbf{\Omega}_0$ and $\Fbold^P$ maps $\mathbf{\Omega}_0$ conformally to $\Sbold_n^\# \backslash \Hq$.
    \end{claim3}

    \begin{proof}
        Let us fix a large negative integer $n$ and choose $Q \in \Tbold_{>0}$ such that 
        \[
        Q > R + R/\tbold + R/\tbold^2 + R/\tbold^3 + \ldots.
        \] 
        Consider a point $x_0:= x \in \Jbold_0$ such that $\Fbold^P(x) \in \Sbold_n^\#$ for some $P \geq Q$. For $j \in \{0,-1,-2,\ldots, n+2\}$, set $P_j:= \tbold^j R$ and $x_{j-1}:= \Fbold^{P_j}(x_j)$ inductively. Then, we set
    \[
        P_{n+1}:= P - P_0 - P_{-1} - \ldots - P_{n+2} \quad \text{ and } x_n:= \Fbold^{P_{n+1}}(x_{n+1}).
    \]
        Clearly, $P_{n+1} \geq \tbold^{n+1} R$. By Claim 2, there exists an open set $\mathbf{\Omega}_{n+1} \subset \Sbold_{n+1}^\# \backslash \Hq$ such that $x_{n+1} \in \partial \mathbf{\Omega}_{n+1}$ and $\Fbold^{P_{n+1}}$ maps $\mathbf{\Omega}_{n+1}$ conformally to $\Sbold_n^\# \backslash \Hq$. Inductively, for $j \in \{0,-1,\ldots,n+2\}$, we construct open sets $\mathbf{\Omega}_j \subset \Sbold_j^\# \backslash \Hq$ such that $x_j \in \partial \mathbf{\Omega}_j$ and $\Fbold^{P_j}$ maps $\mathbf{\Omega}_j$ conformally to $\mathbf{\Omega}_{j-1}$. Therefore, $\Fbold^P$ maps $\mathbf{\Omega}_0$ conformally to $\Sbold_n^\# \backslash \Hq$.
    \end{proof}

    Now, we are ready to prove the lemma. By $\sigma$-properness, it suffices to prove the boundedness of limbs rooted at critical points on $\Hq$. Hence, let us consider a limb $L$ of some generation $K \in \Tbold_{>0}$ rooted at the critical point $C_K$ on $\Hq$. 
    
    Consider the power-triple $Q$ from Claim 3. Choose a large $m \in \N$ such that $T := \tbold^m K $ is larger than $Q+K$ and that the critical point $C_T$ is on $\Jbold_0$. By Lemma \ref{lem:self-sim-lakes}, the limb $L' :=A_*^m(L)$ must be rooted at $C_T$ and $\Fbold^{T-K}(L')=L$. Since $T-K \geq Q$, for $n \ll 0$, the connected component of $\Sbold_n^\# \cap \overline{L}$ containing $C_K$ can be lifted by $\Fbold^{T-K}$ into $\Sbold_0^\#$. Since the lifts of $\Sbold_n^\# \cap \overline{L}$ exhaust $L'$, then $L'$ must be contained in $\Sbold_0^\#$. Therefore, $L$ must also be bounded.
\end{proof}

\subsection{Alpha-points}
\label{ss:alpha}

For $P \in \Tbold_{>0}$, let $\Esc_{\leq P} := \Esc_{\leq P}(\Fbold)$ be the $P$\textsuperscript{th} escaping set of $\Fbold$.

\begin{lemma}
\label{alpha-points}
    Every critical point $\prescript{}{J}C^\blacksquare_S$ admits a pair of points $\prescript{}{J}\alpha^{\blacksquare,0}_S$ and $\prescript{}{J}\alpha^{\blacksquare,\infty}_S$ with the following properties. For any $\bullet \in \{0,\infty\}$ and $j \in \{l,1,\ldots,2d_\bullet-3,r\}$, both the left and the right coasts of $\prescript{}{J,j} \Obold^{\blacksquare,\bullet}_S$ land at $\prescript{}{J}\alpha^{\blacksquare,\bullet}_S$ and
    \[
        \partial \prescript{}{J,j} \Obold^{\blacksquare,\bullet}_S \, \big\backslash \, \partial^c \prescript{}{J,j} \Obold^{\blacksquare,\bullet}_S = \left\{ \prescript{}{J}\alpha^{\blacksquare,\bullet}_S \right\}.
    \]
\end{lemma}

In particular, every lake is a disk and each of the spines $\prescript{}{J,j} \Hq^{\blacksquare,\bullet}_S$ attached to $\prescript{}{J}C^{\blacksquare}_S$ is a quasiarc connecting its common root $\prescript{}{J}C^{\blacksquare}_S$ to a common landing point $\prescript{}{J}\alpha^{\blacksquare,\bullet}_S$. See Figure \ref{fig:alpha-points} for an illustration. We call $\prescript{}{J}\alpha^{\blacksquare,0}_S$ and $\prescript{}{J}\alpha^{\blacksquare,\infty}_S$ the \emph{inner} and \emph{outer alpha-points} corresponding to $\prescript{}{J}C^{\blacksquare}_S$. Moreover, we say that $\prescript{}{J}\alpha^{\blacksquare,\bullet}_S$ is \emph{the alpha-point} of any lake of the form $\prescript{}{J,j}\Obold^{\blacksquare,\bullet}_S$.

\begin{remark}
    In the pacman case, alpha-points can be interpreted as limiting preimages of the $\alpha$-fixed point of an infinitely anti-renormalizable pacman. For the classical Feigenbaum renormalization fixed point, the existence of alpha points dates back to early 1990's in the work of H. Epstein \cite{Ep92}.
\end{remark}

\begin{figure}
    \centering
    \begin{tikzpicture}[scale=1.04]
        % for the side coasts
        \filldraw[gray, fill=green!05!white] (-4.8,0) .. controls (-3.75, 1.75) .. (-4,3.5) .. controls (-3,4) .. (-2.5,4.5) .. controls (-1.75, 4.25) .. (-1,4.3) .. controls (-0.7,3.9) .. (-0.3,3.8) .. controls (-0.2,3.63) .. (0,3.5) .. controls (0.2,3.63) .. (0.3,3.8) .. controls (0.7,3.9) .. (1,4.3) .. controls (1.75, 4.25) .. (2.5,4.5) .. controls (3,4) ..  (4,3.5) .. controls (3.75, 1.75) .. (4.8,0) -- (2.4,-0.8) -- (0,0) -- (-2.4,-0.8) -- (-4.8,0); 
        \filldraw[gray, fill=yellow!05!white] (-4.8,0) .. controls (-3.9, -1.5) .. (-4,-3) .. controls (-3,-3.2) .. (-2.5,-3.5) .. controls (-1.75,-3.25) .. (-1,-3.3) .. controls (-0.7,-2.9) .. (-0.3,-2.8) .. controls (-0.2,-2.63) .. (0,-2.5) .. controls (0.2,-2.63) .. (0.3,-2.8) .. controls (0.7,-2.9) .. (1,-3.3) .. controls (1.75,-3.25) .. (2.5,-3.5) .. controls (3,-3.2) ..  (4,-3) .. controls (3.9, -1.5) .. (4.8,0) -- (2.4,-0.8) -- (0,0) -- (-2.4,-0.8) -- (-4.8,0); 
        
        % for Hq
        \draw[ultra thick]  (-6,-0.4) -- (-4.8,0) -- (-2.4,-0.8) -- (0,0) -- (2.4,-0.8) -- (4.8,0) -- (6,-0.4);
        \node [black, font=\bfseries] at (-5.6,0) {$\Hq$};

        % for the big one
        \filldraw[gray, fill=yellow!05!white] (0,0) .. controls (1.5,0.2) and (2,1.5) .. (2,2) .. controls (1.8,3) and (0.5,3.8) .. (0,3.5) .. controls (-0.5,3.8) and (-1.8,3) .. (-2,2) .. controls (-2,1.5) and (-1.5,0.2) .. (0,0); 
        \filldraw[gray, fill=green!05!white] (0,0) .. controls (1, 2.2) .. (0,3.5) .. controls (-1,2.2) .. (0,0); 
        \filldraw[gray, fill=green!05!white] (0,0) .. controls (0.9,-1.8) .. (0,-2.5) .. controls (-0.9,-1.8) .. (0,0);

        % labels
        \node [red, font=\bfseries] at (0,-0.02) {\large $\bullet$};
        \node [red, font=\bfseries] at (-1,0) {\small $C_P$};
        \draw[red] (-0.7,0) -- (-0.15,0);
        \node [red, font=\bfseries] at (0, -2.5) {$\bullet$};
        \node [red, font=\bfseries] at (0,-2.8) {\small $\alpha_P^\infty$};
        \node [red, font=\bfseries] at (0,3.5) {$\bullet$};
        \node [red, font=\bfseries] at (0,3.8) {\small $\alpha_P^0$};
        \node [yellow!50!black, font=\bfseries] at (1.1,1.1) {\small $\prescript{}{1}{\Obold}^0_P$};
        \node [green!50!black, font=\bfseries] at (0,2) {\small $\prescript{}{2}{\Obold}^0_P$};
        \node [yellow!50!black, font=\bfseries] at (-1.1,1.1) {\small $\prescript{}{3}{\Obold}^0_P$};
        \node [green!50!black, font=\bfseries] at (0,-1.5) {\small $\prescript{}{1}{\Obold}^\infty_P$};
        \node [yellow!50!black, font=\bfseries] at (2.5,-2.2) {\small $\prescript{}{l}{\Obold}^\infty_P$};
        \node [yellow!50!black, font=\bfseries] at (-2.5,-2.2) {\small $\prescript{}{r}{\Obold}^\infty_P$};
        \node [green!50!black, font=\bfseries] at (-2.8,1.5) {\small $\prescript{}{l}{\Obold}^0_P$};
        \node [green!50!black, font=\bfseries] at (2.8,1.5) {\small $\prescript{}{r}{\Obold}^0_P$};
\end{tikzpicture}

    \caption{The configuration of middle and side lakes rooted at $C_P$ when $(d_0,d_\infty)=(3,2)$. Their coasts land at $\alpha_P^0$ and $\alpha_P^\infty$.}
    \label{fig:alpha-points}
\end{figure}
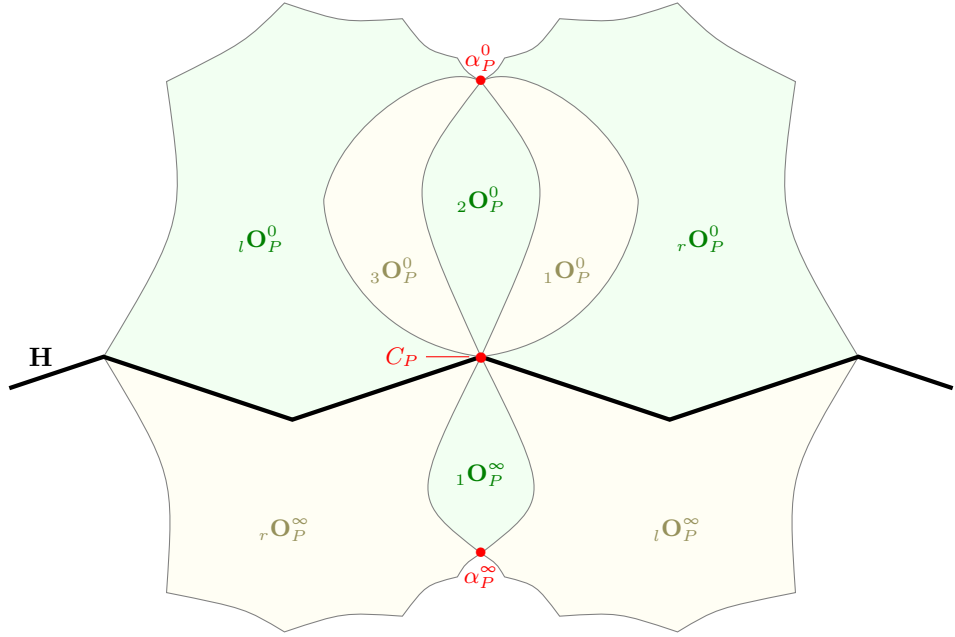

\begin{proof}
    By Corollary \ref{cor:middle-or-side}, for every lake $\Obold$, there is some $Q \in \Tbold$ such that $\Fbold^Q(\Obold)$ is either a side lake or a middle lake attached to some critical point on $\Hq$. Therefore, it is sufficient to prove the lemma for lakes of the form $\prescript{}{j}\Obold^\bullet_P$ where $\bullet \in \{0,\infty\}$ and $j \in \{l,1,\ldots,2d_\bullet-3,r\}$. 
    
    Suppose $\prescript{}{j}\Obold_P^\bullet$ is a middle lake. Then, it is contained in some side lake $\prescript{}{k}\Obold^{\bullet}_{P-P/\tbold}$ of generation $P-P/\tbold$ where $k \in \{l,r\}$. Consider the conformal map
    \[
        \Gbold := \Fbold^{\tbold P-P} \circ A_* : \prescript{}{k}\Obold^{\bullet}_{P-P/\tbold} \to \Obold^\bullet.
    \]
    Observe that $\Gbold$ expands the hyperbolic metric of the ocean $\Obold^\bullet$, and $\Gbold$ sends $\prescript{}{j}\Obold_P^\bullet$ onto itself. Since $\overline{\prescript{}{j}\Obold_P^\bullet} \cap \Esc_{\leq P}$ is a $\Gbold$-invariant compact subset of $\Obold^\bullet$, then it must be a singleton $\left\{\alpha_P^\bullet\right\}$ consisting of the unique repelling fixed point of $\Gbold$.

    It remains to show that for $j \in \{l,r\}$, the intersection $\overline{ \partial^c_j \prescript{}{j}\Obold^\bullet_P} \cap \Esc_{\leq P}$ is also a compact subset of $\Obold^\bullet$. By invariance under $\Gbold$, this will again imply that $\overline{ \partial^c_j \prescript{}{j}\Obold^\bullet_P} \cap \Esc_{\leq P}$ is the same singleton $\left\{\alpha_P^\bullet\right\}$, and we are done.
    
    Let us assume without loss of generality that $j=l$. Denote the left itinerary of $\prescript{}{l}\Obold^\bullet_P$ by $(Q_1,Q_2,Q_3,\ldots)$. The left coast of $\prescript{}{l}\Obold^\bullet_P$ starts with a segment of the spine $\prescript{}{1}\Hq^\bullet_{Q_1}$ connecting $C_{Q_1}$ and $\prescript{}{1}C^\bullet_{Q_1,Q_2}$. Let us pick a pair of power-triples $R_-, R_+ \in \Tbold_{>0}$ such that the critical points $\prescript{}{1}C^\bullet_{Q_1,R_-}$ and $\prescript{}{1}C^\bullet_{Q_1,R_+}$ form a small open interval neighborhood $J \subset \prescript{}{1}\Hq^\bullet_{Q_1}$ of $\prescript{}{1}C^\bullet_{Q_1,Q_2}$. Let $B_\pm$ be the spines of generation $Q_1 + R_\pm$ attached to $\prescript{}{1}C^\bullet_{Q_1,R_\pm}$ that are combinatorially closest to $\prescript{}{1,1}\Hq^{\bullet,\bullet}_{Q_1,Q_2}$. Let $R := Q_1+\max\{R_+,R_-\}$. By Lemma \ref{lem:unboundedness-of-esc}, every connected component of $\Esc_{\leq R}$ is unbounded and thus the union $J \cup B_+ \cup B_- \cup \Esc_{\leq R}$ separates $\partial^c_l \prescript{}{l}\Obold^\bullet_P \backslash \prescript{}{i}\Hq^\bullet_{Q_1}$ from $\Hq$. This observation implies that $\overline{ \partial^c_l \prescript{}{l}\Obold^\bullet_P} \cap \Esc_{\leq P}$ is indeed compactly contained in $\Obold^\bullet$.
\end{proof}

The alpha-points $\prescript{}{J}\alpha^{\blacksquare,\bullet}_S$ can be viewed as preimages of infinity under the map $\Fbold^{|S|}$. They are unique in the following sense.

\begin{lemma}
    Two alpha-points $\prescript{}{J}\alpha^{\blacksquare,\bullet}_S$ and $\prescript{}{J'}\alpha^{\square,\circ}_{S'}$ coincide if and only if $J=J'$, $\blacksquare = \square$, $\bullet = \circ$, and $S=S'$.
\end{lemma}

\begin{proof}
    Suppose $\prescript{}{J}\alpha^{\blacksquare,\bullet}_S = \prescript{}{J'}\alpha^{\square,\circ}_{S'}$. Clearly, $|S|=|S'|$. Let us write $S=(P_1,\ldots,P_m)$ and $S'=(Q_1,\ldots, Q_k)$, and pick a power-triple $R \in \Tbold$ such that 
    \[
    \max\{P_1+\ldots+P_{m-1}, Q_1+\ldots+Q_{k-1}\}<R<|S|.
    \]
    Pushing forward by $\Fbold^R$ yields a pair of alpha-points $\alpha^\bullet_{|S|-R}$ and $\alpha^\circ_{|S'|-R}$ where, since they are equal, $\bullet = \circ$. If $(J,\blacksquare,S) \neq (J',\square,S')$, then this would imply that $\alpha^\bullet_{|S|-R}$ is a critical point of $\Fbold^R$, which is not the case.
\end{proof}

Consequently, if two disjoint spines touch at a common alpha-point, then they are rooted at a common critical point. This guarantees a more precise tree structure of $\Fbold^{-P}(\Hq)$ in terms of spines. For convenience, we will call $\Hq$ the unique spine of generation $0$.

\begin{corollary}
\label{cor:spine-structure}
    Consider two distinct spines $\prescript{}{J,j}\Hq_S^{\blacksquare,\bullet}$ and $\prescript{}{J',j'}\Hq_{S'}^{\square,\circ}$ with $|S|\geq |S'|$.
    \begin{enumerate}[label=\textnormal{(\arabic*)}]
        \item If the intersection $\overline{\prescript{}{J,j}\Hq_S^{\blacksquare,\bullet}} \cap \overline{\prescript{}{J',j'}\Hq_{S'}^{\square,\circ}}$ is non-empty, then it is either the singleton $\left\{\prescript{}{J}C_S^\blacksquare \right\}$ or the set $\left\{\prescript{}{J}C_S^\blacksquare, \prescript{}{J}\alpha_S^{\blacksquare,\bullet} \right\}$. The former case happens if and only if $\prescript{}{J',j'}\Hq_{S'}^{\square,\circ}$ contains $\prescript{}{J}C_S^\blacksquare$, and the latter case happens if and only if $(J,\blacksquare,\bullet,S) = (J',\square,\circ,S')$.
        \item There is a unique sequence of pairwise different spines 
        \[
        B_1 = \prescript{}{J,j}\Hq_S^{\blacksquare,\bullet}, \enspace B_2, \enspace \ldots, \enspace B_{n-1}, \enspace B_n= \prescript{}{J',j'}\Hq_{S'}^{\square,\circ}
        \]
        such that $\overline{B_{i}}$ intersects $\overline{B_{i'}}$ if and only if $|i-i'|\leq 1$.
    \end{enumerate}
\end{corollary}

Given an alpha-point $\alpha = \prescript{}{J}\alpha^{\blacksquare,\bullet}_S$, we define
\begin{list}{$\rhd$}{}
    \item a \emph{finite skeleton landing at} $\alpha$ to be the union of a spine $\prescript{}{J,j}\Hq^{\blacksquare,\bullet}_S$ together with the unique closed quasiarc in $\Fbold^{-|S|}(\Hq)$ connecting $\prescript{}{J}C^{\blacksquare}_S$ to $0$;
    \item an \emph{infinite skeleton landing at} $\alpha$ to be the union of $\partial^c_k \prescript{}{J,k}\Obold^{\blacksquare,\bullet}_S$ for some $k \in \{l,r\}$ together with the unique closed quasiarc in $\Fbold^{-|S|}(\Hq)$ connecting the root of $\partial^c_k \prescript{}{J,k}\Obold^{\blacksquare,\bullet}_S$ to $0$.
\end{list}
In short, skeletons landing at $\alpha$ are the shortest paths from $0$ to $\alpha$ within the tree of preimages of $\Hq$. There are exactly $2 d_\bullet$ skeletons landing at $\alpha$, and precisely two of them are infinite. 

The set of skeletons admit a total order ``$<$`` defined as follows. Let us fix a ray $\gamma$ in $\Hq$ connecting $0$ to $\infty$. Given two distinct skeletons $\mathfrak{S}$ and $\mathfrak{S}'$, 
\begin{list}{$\rhd$}{}
    \item we write $\mathfrak{S} < \mathfrak{S}'$ if $\gamma$, $\mathfrak{S}$, and $\mathfrak{S}'$ have a counterclockwise orientation around the quasiarc $\mathfrak{S} \cap \mathfrak{S}'$, and
    \item we say that $\mathfrak{S}$ and $\mathfrak{S}'$ are \emph{$<$-separated} if there is another skeleton $\mathfrak{S}''$ such that either $\mathfrak{S} < \mathfrak{S}'' < \mathfrak{S}'$ or $\mathfrak{S}' < \mathfrak{S}'' < \mathfrak{S}$.
\end{list}
We say that two alpha-points $\alpha$ and $ \alpha'$ in the same ocean $\Obold^\bullet$ are \emph{$<$-separated} if 
\begin{list}{$\rhd$}{}
    \item there exists an alpha point $\alpha'' \in \Obold^\bullet$ with generation lower than that of $\alpha$ and $\alpha'$, and
    \item there exist skeletons $\mathfrak{S}$, $\mathfrak{S}'$ , $\mathfrak{S}''$ landing at $\alpha$, $\alpha'$, $\alpha''$ respectively such that $\mathfrak{S}$ and $\mathfrak{S}'$ are $<$-separated by $\mathfrak{S}''$.
\end{list}

Let us introduce another partial order on the set of alpha-points. Given two alpha-points $\alpha$ and $\alpha'$ in the same ocean, 
\begin{list}{$\rhd$}{}
    \item we write $\alpha \prec \alpha'$ if $\alpha'$ is contained in the closure of a lake attached to $\alpha$, and
    \item we say that $\alpha$ and $\alpha'$ are \emph{$\prec$-separated} if $\alpha$ and $\alpha'$ are contained in two distinct lakes with a common alpha-point.
\end{list}

The following proposition describes the relation between ``$\prec$`` and ``$<$``.

\begin{proposition}
\label{comparing-alphas}
    Consider two distinct alpha-points $\alpha$ and $\alpha'$ of generations $P$ and $P'$ inside of the ocean $\Obold^\bullet$ for some $\bullet \in \{0,\infty\}$. Assume $P\leq P'$. The following are equivalent.
    \begin{enumerate}[label=\textnormal{(\arabic*)}]
        \item $\alpha \prec \alpha'$;
        \item $\alpha$ and $\alpha'$ are not $<$-separated by another alpha-point $\alpha''$ of generation $<P$;
        \item $\alpha$ and $\alpha'$ are not $\prec$-separated.
    \end{enumerate}
\end{proposition}

\begin{proof}
    Suppose (1) holds. Then, $\alpha$ is the alpha-point of a lake $\Obold$ containing $\alpha'$, which implies (3). Meanwhile, (2) follows from the observation that any alpha-point $\alpha''$ $<$-separating $\alpha$ and $\alpha'$ must be contained in a proper sub-lake of $\Obold$, which necessarily has generation higher than $P$.
    
    Suppose (1) does not hold, so $\alpha'$ is located outside of every lake with alpha-point $\alpha$. Let us pick any skeleton $\mathfrak{S}'$ landing at $\alpha'$, and let $\mathfrak{S}_l$ and $\mathfrak{S}_r$ denote the left and right infinite skeletons landing at $\alpha$ respectively. The assumption implies that either $\mathfrak{S}_l$ $<$-separates $\mathfrak{S}_r$ and $\mathfrak{S}'$, or $\mathfrak{S}_r$ $<$-separates $\mathfrak{S}_l$ and $\mathfrak{S}'$. Without loss of generality, let us assume the latter.
    
    Denote by $(c_{r,1},c_{r,2},\ldots)$ the infinite sequence of critical points of $\Fbold^P$ of increasing generation that is found along $\mathfrak{S}_r$. Let $\alpha_{r,i}$ denote the alpha-point that is the landing point of the unique spine attached to $c_{r,i}$ that intersects $\mathfrak{S}_k$. It has generation $P_{r,i}$ where $P_{r,i}<P$ and $P_{r,i} \to P$ as $i \to \infty$. Since the intersection $\mathfrak{S}_r \cap \mathfrak{S}'$ is a compact subset of $\Dom\left(\Fbold^P\right)$, then for any sufficiently large $i \gg 0$ and any skeleton $\mathfrak{S}_{r,i}$ landing at $\alpha_{r,i}$, $\mathfrak{S}_r \cap \mathfrak{S}'$ is a proper subset of $\mathfrak{S}_r \cap \mathfrak{S}_{r,i}$. Therefore, $\mathfrak{S}_{r,i}$ $<$-separates $\mathfrak{S}_r$ and $\mathfrak{S}'$, and so $\alpha$ and $\alpha'$ are $<$-separated by $\alpha_{r,i}$.
    
    We have just shown that (1) and (2) are equivalent. Suppose (1) and (2) do not hold. We will now prove that (3) also does not hold.

    Let us consider the unique spine $B$ such that $\alpha$ and $\alpha'$ are contained in the closure of different components of $(\mathfrak{S} \cup \mathfrak{S}') \backslash B$. We claim that the generation $Q$ of $B$ is less than $P$. Indeed, if $Q=P$, then $\alpha$ is the landing point of $B$ and so there exists a lake with alpha-point $\alpha$ which contains both $\mathfrak{S}' \backslash \mathfrak{S}$ and $\alpha'$. However, this would instead imply (1).

    Let $\hat{\Obold}$ and $\hat{\Obold}'$ denote the pair of lakes of generation $Q$ such that their coast contains $B$ and $\mathfrak{S}\backslash \mathfrak{S}' \subset \hat{\Obold}$ and $\mathfrak{S}' \backslash \mathfrak{S} \subset \hat{\Obold}'$. If $\hat{\Obold}$ and $\hat{\Obold}'$ are distinct, they lie on different sides of $B$ and so $\alpha$ and $\alpha'$ are $\prec$-separated by the landing point of $B$.
    
    Now, suppose instead that $\hat{\Obold} = \hat{\Obold}'$. Consider the roots $c$ and $c'$ of $\mathfrak{S}\backslash B$ and $\mathfrak{S}'\backslash B$ respectively. Within the closed interval $[c,c'] \subset B$ (possibly degenerate if $c=c'$), we can find a unique critical point $c''$ of the smallest generation $P''$ such that $Q < P'' \leq P$. In fact, $P'' \not\equal P$ because if otherwise, $\mathfrak{S}'\backslash\mathfrak{S}$ would have been contained in a lake attached to $c$, and so $\alpha \prec \alpha'$ instead. Since $[c,c']$ does not contain any critical point of generation lower than $P''$, then $\mathfrak{S}\backslash \mathfrak{S}'$ and $\mathfrak{S}' \backslash \mathfrak{S}$ are contained in distinct lakes of generation $P''$ attached to $c''$. Thus, the alpha-point $\alpha'' \in \hat{\Obold}$ corresponding to $c''$ must $\prec$-separate $\alpha$ and $\alpha'$.
\end{proof}

\subsection{External chains}
\label{ss:external-chains}

Let us pick a power-triple $P \in \Tbold_{>0}$ and $\bullet \in \{0,\infty\}$. Let $\Obold^\bullet(P)$ denote the unique lake of generation $P$ inside of the ocean $\Obold^\bullet$ that contains $0$ on its boundary. Then, the coast of $\Obold^\bullet(P)$ intersects $\Hq$ on some interval $J \subset \Hq$ containing $0$ on its interior. (In fact, $J$ is independent of $\bullet$.) Let us denote by $\alpha^\bullet(P)$ the unique alpha-point in $\partial \Obold^\bullet(P)$. By self-similarity, 
\[
\Obold^\bullet(\tbold^n P) = A_*^n (\Obold^\bullet(P)) \qquad \text{ for all } n \in \Z
\]
and
\begin{equation}
\label{eqn:lakes-exhaust-ocean}
    \bigcup_{n<0} \Obold^\bullet(\tbold^n P) = \Obold^\bullet.
\end{equation}

Let us denote by $\Esc^\bullet_{\leq P}$ the intersection $\Esc_{\leq P} \cap \Obold^\bullet$ for $\bullet \in \{0,\infty\}$.

\begin{lemma}
    For every $\bullet \in \{0,\infty\}$ and $P, Q\in \Tbold_{>0}$ with $P<Q$, 
    \begin{enumerate}[label=\textnormal{(\arabic*)}]
        \item $\Esc^\bullet_{\leq P}$ is connected;
        \item $\Esc^\bullet_{\leq Q} \backslash \Esc^\bullet_{\leq P}$ is bounded;
        \item every connected component of $\Esc^\bullet_{\leq Q} \backslash \Esc^\bullet_{\leq P}$ is a lift of a component of $\Esc_{\leq Q-P}$ under $\Fbold^P$; it is contained in a unique lake $\Obold$ of generation $P$ and its boundary contains the alpha-point of $\Obold$.
    \end{enumerate}
\end{lemma}

\begin{proof}
    Consider a component $I$ of $\Esc^\bullet_{\leq P}$. 
    Let $k = k(I) \in \Z$ be the largest integer such that $I$ intersects the lake $\Obold^\bullet(\tbold^k P)$. 
    For every integer $n$, the only escaping point on the boundary of $\Obold^\bullet(\tbold^n P)$ is the alpha-point $\alpha^\bullet(\tbold^n P)$.
    Since $I$ is unbounded (Lemma \ref{lem:unboundedness-of-esc}), 
    $I$ must contain $\alpha^\bullet(\tbold^n P)$ for all $n \leq k$. 
    This implies that $\Esc^\bullet_{\leq P}$ is connected.

    Next, consider a connected component $X$ of $\Esc^\bullet_{\leq Q} \backslash \Esc^\bullet_{\leq P}$. 
    From the previous paragraph, there exists some integer $k$ such that $X$ avoids $\alpha^\bullet(\tbold^n P)$ for all $n \leq k$. 
    Therefore, for $n \leq k$, $\Esc^\bullet_{\leq Q} \backslash \Esc^\bullet_{\leq P}$ must be contained inside of the lake $\Obold^\bullet(\tbold^n P)$, which is bounded. Since $X$ avoids $\Fbold^{-P}(\Hq)$ and alpha-points of generation $P$, $X$ is contained in a unique lake $\Obold$ of generation $P$. 
    The map $\Fbold^P$ sends $\Obold$ conformally onto an ocean $\Obold^\circ$ for some $\circ \in \{0,\infty\}$, hence $\Fbold^P (X) = \Esc^\circ_{\leq Q-P}$. By unboundedness, $X$ must be attached to the alpha-point of $\Obold$. 
\end{proof}

\begin{definition}
    Consider two alpha-points $\alpha$ and $\alpha'$ in the same ocean $\Obold^\bullet$ with generation $P$ and $P'$ respectively, and suppose $P < P'$ and $\alpha \prec \alpha'$. We define the \emph{external chain} $[\alpha,\alpha']$ to be the set of points in $\Esc_{\leq P'}^\bullet$ that are inside the closure of the lakes attached to $\alpha$ and outside of any lake that does not contain $\alpha'$.
\end{definition}

\begin{lemma}
\label{lem:ext-chain}
    For any triplet of alpha-points $\alpha, \alpha',\alpha''$ with $\alpha \prec \alpha' \prec \alpha''$,
    \[
        [\alpha, \alpha'] \cap [\alpha', \alpha''] = \{\alpha'\}\quad \text{and} \quad [\alpha, \alpha'] \cup [\alpha', \alpha''] = [\alpha, \alpha''].
    \]
\end{lemma}

\begin{proof}
    The first equation follows from the fact that $\alpha'$ is a cut point with respect to the ``$\prec$`` ordering. The inclusion $[\alpha, \alpha'] \cup [\alpha', \alpha''] \subset [\alpha, \alpha'']$ is obvious. Consider a point $x$ in $[\alpha,\alpha''] \backslash [\alpha,\alpha']$. We know that $x$ is within a lake attached to $\alpha$. If $x$ is inside of a lake that does not contain $\alpha'$, then this lake avoids all lakes attached to $\alpha'$ and in particular does not contain $\alpha''$ as well, which is a contradiction. Therefore, $x \in [\alpha',\alpha'']$.
\end{proof}

For $P \in \Tbold_{>0}$, we say that the critical point $C_P$ on $\Hq$ is \emph{dominant} if the interval $[0,C_P] \subset \Hq$ does not contain any critical point of generation less than $P$. We will enumerate dominant critical points by $\{C_{P_n}\}_{n \in \Z}$ where $\{P_n\}_{n \in \Z}$ is monotonically increasing in $n$.

\begin{lemma}
\label{dominant-alphas}
    For $\bullet \in \{0,\infty\}$, $\ldots \prec \alpha^\bullet_{P_{-2}} \prec \alpha^\bullet_{P_{-1}} \prec \alpha^\bullet_{P_0} \prec \alpha^\bullet_{P_1} \prec \alpha^\bullet_{P_2} \prec\ldots$.
\end{lemma}

\begin{proof}
    Suppose for a contradiction that $\alpha^\bullet_{P_n} \not\prec \alpha^\bullet_{P_{n+1}}$ for some $\bullet \in \{0,\infty\}$ and $n \in \Z$. By Propositon \ref{comparing-alphas}, there is an alpha-point $\alpha \in \Obold^\bullet$ of some generation $P$ less than $P_n$ which $<$-separates $\alpha^\bullet_{P_n}$ and $\alpha^\bullet_{P_{n+1}}$. Then, $\alpha$ is contained in the closure of a lake attached to a critical point $C_Q \in \Hq$ of some generation $Q\leq P$. By $<$-separation, $C_Q$ is contained in the interval $(C_{P_n}, C_{P_{n+1}}) \subset \Hq$. However, this would contradict the assumption that $C_{P_n}$ and $C_{P_{n+1}}$ are dominant.
\end{proof}

Consider the concatenations of external chains
\[
    \Rbold^0 = \bigcup_{n \in \Z} \left[ \alpha^0_{P_n}, \alpha^0_{P_{n+1}} \right]
    \quad \text{ and } \quad
    \Rbold^\infty = \bigcup_{n \in \Z} \left[ \alpha^\infty_{P_n}, \alpha^\infty_{P_{n+1}} \right],
\]
which we will refer to as the \emph{inner} and \emph{outer zero chains} respectively.

\begin{proposition}
\label{prop:zero-chains}
    For $\bullet \in \{0,\infty\}$,
    \begin{enumerate}[label=\textnormal{(\arabic*)}]
        \item $\Rbold^\bullet$ is $A_*$-invariant;
        \item $\Rbold^\bullet$ is an infinite arc in $\Obold^\bullet$ landing at $0$;
        \item alpha-points are dense on $\Rbold^\bullet$;
        \item points on $\Rbold^\bullet$ are continuously parametrized by their escaping time ranging from $0$ (near $\infty$) to $+\infty$ (near $0$).
    \end{enumerate}
\end{proposition}

Let us clarify the last statement. For $P \in \R_{>0} \backslash \Tbold$, we can define the $P$\textsuperscript{th} escaping set to be 
\[
\displaystyle{\Esc_{\leq P} := \bigcap_{Q\in \Tbold, Q>P} \Esc_{\leq Q}}.
\]
The \emph{escaping time} of a point $x$ in $\Esc_{<\infty}$ is the minimum time $P \in \R_{>0}$ such that $x$ is in $\Esc_{\leq P}$.

\begin{proof}
    To lighten the notation, we will denote $\alpha_n^\bullet := \alpha_{P_n}^\bullet$ and $J_n^\bullet := [\alpha_n^\bullet, \alpha_{n+1}^\bullet]$ for all $\bullet \in \{0,\infty\}$ and $n \in \Z$.
 
    By definition, $C_P$ is dominant if and only if $C_{\tbold P} = A_*(C_P)$ is dominant, so there is some integer $k \geq 1$ such that $\tbold P_n = P_{n+k}$ for all $n \in \Z$. Therefore, $A_*$ maps each of $[\alpha^\bullet_{(n-1)k}, \alpha^\bullet_{nk}]$ onto $[\alpha^\bullet_{nk}, \alpha^\bullet_{(n+1)k}]$. This immediately implies items (1) and (2). 
    
    Due to self-similarity, it remains for us to show that the external chain $J^\bullet := [\alpha^\bullet_0, \alpha^\bullet_k]$ is an arc that can be continuously parametrized by their escaping time, and that alpha-points are dense on $J^\bullet$. We will do so by constructing nested Markov tilings $\mathcal{P}_r$ for $r \geq 0$ on $J^\bullet$.

    Firstly, we set the tiling $\mathcal{P}_0$ of level 0 to be $\{ J^\bullet_i\}_{0\leq i \leq k-1}$. The tiling $\mathcal{P}_1$ of level 1 is constructed as follows. 
    By Lemma \ref{lem:dominant-push}, for every chain $J^\bullet_i \in \mathcal{P}_0$, there exist some $Q_i \in \Tbold_{>0}$ and a pair of integers $l_i$ and $r_i$ such that $0<l_i<r_i\leq i$ and $\Fbold^{Q_i}$ maps $J^\bullet_i$ homeomorphically onto the chain $[\alpha_{l_i}^\bullet, \alpha_{r_i}^\bullet]$. 
    A tile of level 1 in $\mathcal{P}_1$ is the lift of a chain $J_j^\bullet \subset [\alpha_{l_i}^\bullet, \alpha_{r_i}^\bullet]$ under the map $\Fbold^{Q_i} : J_i^\bullet \to [\alpha_{l_i}^\bullet, \alpha_{r_i}^\bullet]$.
    
    For each tile $I \in \mathcal{P}_1$ in $J_i^\bullet$, there exists some $m_I \in \N$ such that $A_*^{m_I}$ sends $\Fbold^{Q_i}(I)$ back to a tile of level $0$. Let $\Obold_i$ denote the lake of generation $Q_i$ which contains $[\alpha_{l_i}^\bullet, \alpha_{r_i}^\bullet]$. The composition
    \begin{equation}
        \label{eqn:tiling-map}
        \chi_I := A_*^{m_I} \circ \Fbold^{Q_i} : \Obold_i \to \Obold^\bullet
    \end{equation}
    expands the hyperbolic metric of $\Obold^\bullet$.

    Inductively, we define tiles in $\mathcal{P}_{n+1}$ of level $n+1$ to be the preimages of tiles of level $n$ under maps of the form (\ref{eqn:tiling-map}). Since each map $\chi_I$ is expanding, the diameter of every tile of level $n$ uniformly exponentially shrinks to zero. Since each tile in $\mathcal{P}_n$ is an external chain containing alpha-points, alpha-points are dense on $J$.

    By Lemma \ref{lem:ext-chain}, we can enumerate our level $n$ tiles by $I_1^n, I_2^n, \ldots, I_{s_n}^n \in \mathcal{P}_n$ in increasing order of generation such that $I_i^n$ and $I_l^n$ touch if and only if $|l-i|\leq 1$. As tiles shrink, we can extend the ``$\prec$`` order to a total order on $J^\bullet$ by defining $x \prec y$ when $x \in I_i^n$ and $y \in I_j^n$ for sufficiently high $n$ and some indices $i$, $j$ with $i<j$.

    Consider a tile $I_i^n$ in $\mathcal{P}_n$ of some high level $n$, and a composition $\chi:= \chi_1 \circ \chi_2 \circ \ldots \circ \chi_n$ of $n$ maps of the form (\ref{eqn:tiling-map}) sending $I_i^n$ onto a tile in $\mathcal{P}_n$. By (\ref{eqn:self-sim-cascade}), we can write $\chi$ as $A_*^{m(n,i)} \circ \Fbold^{Q(n,i)}$ for some $m(n,i) \in \N$ and $Q(n,i) \in \Tbold_{>0}$. Therefore, the difference in the escaping time between the endpoints of $I_i^n$ is at most
    \begin{equation}
    \label{eqn:difference-esc-times}
        \tbold^{-m(r,i)}(P_k - P_0).
    \end{equation}
    Since $Q_i >0$ for all $i \in \{0,\ldots,k-1\}$, there exists some integer $M \geq 1$ independent of $n$ such that every sequence of $M$ consecutive integers between $1$ and $n$ contains an element $j_*$ such that $\chi_{j_*}$ has the scaling factor $A_*$ in (\ref{eqn:tiling-map}). Consequently, as $n \to \infty$, $\displaystyle{\min_{1\leq i \leq s_n} m(n,i)} \to \infty$ and thus the quantity in (\ref{eqn:difference-esc-times}) tends to zero. Therefore, the escaping time continuously parametrizes points on $J$.
\end{proof}

In general, for every alpha-point $\alpha$, there is an infinite sequence of alpha-points $\alpha_0 = \alpha$, $\alpha_{-1}$, $\alpha_{-2}$, $\ldots$ of generation decreasing to $0$ such that $\ldots \prec \alpha_{-2} \prec \alpha_{-1} \prec \alpha_0$. This allows us to generate the chain
\[
    (\infty,\alpha] := \bigcup_{n \leq 0} [\alpha_{n-1}, \alpha_n].
\]

\begin{corollary}
\label{cor:description-of-chains}
    Consider any alpha-point $\alpha$ of some generation $P>0$. The chain $(\infty,\alpha]$ is an infinite arc continuously parametrized by the escape time from $|P|$ to $0$. Moreover, alpha-points are dense in $(\infty, \alpha]$.
\end{corollary}

\begin{proof}
    Suppose first that $\alpha$ is of the form $\alpha^\bullet_P$ for some $P \in \Tbold_{>0}$ and $\bullet \in \{0,\infty\}$. Let us pick a dominant $\alpha^\bullet_{P_n}$ for some $n \in \Z$ such that $P_n \geq P$. There is a unique point $x \in \left(\infty,\alpha^\bullet_{P_n}\right]$ of generation $P_n - P$. Then, $\Fbold^{P_n - P}$ maps the arc $\left(x,\alpha^\bullet_{P_n}\right]$ onto $\left(\infty, \alpha^\bullet_P\right]$, which implies the claim.

    In general, let $\alpha = \prescript{}{J}\alpha^{\blacksquare,\bullet}_{S}$ where $S=(P_1,P_2,\ldots, P_k)$ is the corresponding itinerary. There exist alpha-points $\alpha_1$, $\alpha_2$, $\ldots$, $\alpha_k = \alpha$ such that $\alpha_1 \prec \alpha_2 \prec \ldots \prec \alpha_k$ and for every $i$, $\alpha_i$ has itinerary $S_i:=(P_1,\ldots,P_i)$. Therefore, we can split $(\infty,\alpha]$ into 
    \[
    J_1 = (\infty,\alpha_1], \quad J_2 = (\alpha_1,\alpha_2], \quad \ldots \quad J_k = (\alpha_{k-1},\alpha_k].
    \]
    When $2\leq i\leq k$, the map $\Fbold^{P_1 + \ldots + P_{i-1}}$ sends $J_i$ homeomorphically onto the chain $\left(\infty,\alpha_{P_i}^\circ\right]$ for some $\circ \in \{0,\infty\}$. By the previous paragraph, each $J_i$ is an arc continuously parametrized by the landing time.
\end{proof}

As a consequence, whenever $\alpha \prec \alpha'$, the chain $[\alpha, \alpha']$ is a simple arc.

\begin{definition}
    An \emph{external ray} is an infinite arc of the form $\Rbold = \bigcup_{n\in\Z} [\alpha_{n},\alpha_{n+1}]$ for some sequence of alpha-points $\{\alpha_n\}_{n\in\Z}$ where
    \begin{list}{$\rhd$}{}
        \item $\alpha_{n} \prec \alpha_{n+1}$ for all $n$;
        \item the generation of $\alpha_n$ decreases to $0$ as $n \to -\infty$;
        \item there is no alpha-point $\alpha$ such that $\alpha_n \prec \alpha$ for all $n \in \Z$.
    \end{list}
    The \emph{generation} of $\Rbold$ is the limit of the generation of $\alpha_n$ as $n \to +\infty$. For any $P \in \Tbold_{>0}$, we define the image of an external ray $\Rbold$ under $\Fbold^P$ by
\[
    \Fbold^P(\Rbold) := \Fbold^P \left( \Rbold \cap \Dom\left(\Fbold^P\right) \right).
\]
    We say that $\Rbold$ is \emph{periodic} if $\Fbold^P(\Rbold) = \Rbold$ for some $P \in \Tbold_{>0}$.
\end{definition}

The zero chains $\Rbold^0$ and $\Rbold^\infty$ are indeed external rays, which from now on will be referred to as \emph{zero rays}.

The following corollary is an immediate consequence of Proposition \ref{comparing-alphas}.

\begin{corollary}
\label{cor:two-external-rays}
    The intersection of any two external rays in the same ocean is non-empty and of the form $(\infty,\alpha]$ for some alpha-point $\alpha$.
\end{corollary}

\subsection{Wakes}
\label{ss:wakes}

Consider a critical point $\prescript{}{J}C^\blacksquare_S$. For every lake of the form $\prescript{}{J,j}\Obold^{\blacksquare,\bullet}_S$ where $j$ is either in $\{l,r\}$ or an even number, the map $\Fbold^{|S|}$ sends such a lake conformally onto $\Obold^\bullet$. The zero ray $\Rbold^\bullet$ lifts under $\Fbold^{|S|}: \prescript{}{J,j}\Obold^{\blacksquare,\bullet}_S \to \Obold^\bullet$ to a ray segment, which we will label as $\prescript{}{J,k}\Rbold^{\blacksquare,\bullet}_S$ where
\[
k = \begin{cases}
    1 & \text{ if } j=r,\\
    \frac{j}{2}+1 & \text{ if } j \text{ is even},\\
    d_\bullet & \text{ if } j=l.
\end{cases}
\]
Therefore, we obtain $d_\bullet$ ray segments
\begin{equation}
    \label{rays-for-wakes}
    \prescript{}{J,1}\Rbold^{\blacksquare,\bullet}_S, \enspace \prescript{}{J,2}\Rbold^{\blacksquare,\bullet}_S, \enspace \ldots, \prescript{}{J,d_\bullet}\Rbold^{\blacksquare,\bullet}_S
\end{equation}
starting from the alpha-point $\prescript{}{J}\alpha^{\blacksquare,\bullet}_S$ and landing at the critical point $\prescript{}{J}\Rbold^{\blacksquare}_S$, labelled in an anticlockwise order about $\prescript{}{J}C^{\blacksquare}_S$. See Figure \ref{fig:rays-and-wakes}. For $k \in \{1,\ldots,d_\bullet-1\}$, the closure of $\prescript{}{J,k}\Rbold^{\blacksquare,\bullet}_S \cup \prescript{}{J,k+1}\Rbold^{\blacksquare,\bullet}_S$ bounds a Jordan domain which we denote by $\prescript{}{J,k}\Wbold^{\blacksquare,\bullet}_S$.

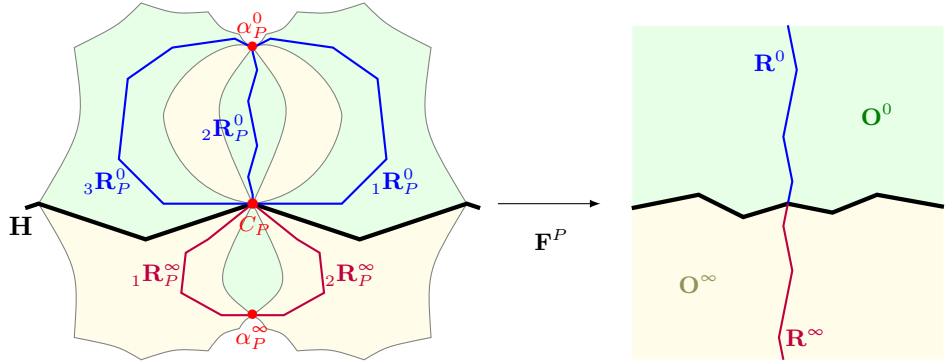
\begin{figure}
    \centering
    \begin{tikzpicture}[scale=0.59]
        % for the side coasts
        \filldraw[gray, thin, fill=green!10!white] (-4.8,0) .. controls (-3.75, 1.75) .. (-4,3.5) .. controls (-3,4) .. (-2.5,4.5) .. controls (-1.75, 4.25) .. (-1,4.3) .. controls (-0.7,3.9) .. (-0.3,3.8) .. controls (-0.2,3.63) .. (0,3.5) .. controls (0.2,3.63) .. (0.3,3.8) .. controls (0.7,3.9) .. (1,4.3) .. controls (1.75, 4.25) .. (2.5,4.5) .. controls (3,4) ..  (4,3.5) .. controls (3.75, 1.75) .. (4.8,0) -- (2.4,-0.8) -- (0,0) -- (-2.4,-0.8) -- (-4.8,0); 
        \filldraw[gray, thin, fill=yellow!10!white] (-4.8,0) .. controls (-3.9, -1.5) .. (-4,-3) .. controls (-3,-3.2) .. (-2.5,-3.5) .. controls (-1.75,-3.25) .. (-1,-3.3) .. controls (-0.7,-2.9) .. (-0.3,-2.8) .. controls (-0.2,-2.63) .. (0,-2.5) .. controls (0.2,-2.63) .. (0.3,-2.8) .. controls (0.7,-2.9) .. (1,-3.3) .. controls (1.75,-3.25) .. (2.5,-3.5) .. controls (3,-3.2) ..  (4,-3) .. controls (3.9, -1.5) .. (4.8,0) -- (2.4,-0.8) -- (0,0) -- (-2.4,-0.8) -- (-4.8,0); 
        
        % for Hq
        \draw[ultra thick]  (-5.1,-0.1) -- (-4.8,0) -- (-2.4,-0.8) -- (0,0) -- (2.4,-0.8) -- (4.8,0) -- (5.1,-0.1);
        \node [black, font=\bfseries] at (-5.2,-0.5) {$\Hq$};

        % for the big one
        \filldraw[gray, thin, fill=yellow!10!white] (0,0) .. controls (1.5,0.2) and (2,1.5) .. (2,2) .. controls (1.8,3) and (0.5,3.8) .. (0,3.5) .. controls (-0.5,3.8) and (-1.8,3) .. (-2,2) .. controls (-2,1.5) and (-1.5,0.2) .. (0,0); 
        \filldraw[gray, thin, fill=green!10!white] (0,0) .. controls (1, 2.2) .. (0,3.5) .. controls (-1,2.2) .. (0,0); 
        \filldraw[gray, thin, fill=green!10!white] (0,0) .. controls (0.9,-1.8) .. (0,-2.5) .. controls (-0.9,-1.8) .. (0,0);

        \draw[thick,blue] (0,0) -- (2,0) -- (3,1) -- (2.8,2.8) -- (1.8,3.5) -- (0.4,3.7) -- (0.05,3.52);
        \draw[thick,blue] (0,0) -- (-2,0) -- (-3,1) -- (-2.8,2.8) -- (-1.8,3.5) -- (-0.4,3.7) -- (-0.05,3.52);
        \draw[thick,blue] (0,0) -- (0,0.2) -- (-0.1,0.6) -- (0.1,1.3) -- (-0.1,2.3) -- (0.1,3) -- (0,3.4) -- (0,3.5);
        \draw[thick,purple] (0,0) -- (1,-0.8) -- (1.5,-1.1) -- (1.6,-2) -- (0.7,-2.5) -- (0.05,-2.5);
        \draw[thick,purple] (0,0) -- (-1,-0.8) -- (-1.5,-1.1) -- (-1.6,-2) -- (-0.7,-2.5) -- (-0.05,-2.5);

        % for the right
        \filldraw[white,fill=green!10!white] (8.5,-0.1) -- (10,0.2) -- (11,-0.3) -- (12,0) -- (13,-0.2) -- (14,0.2) -- (15.5,-0.07) -- (15.5,4) -- (8.5,4) -- (8.5,-0.1);
        \filldraw[white,fill=yellow!10!white] (8.5,-0.1) -- (10,0.2) -- (11,-0.3) -- (12,0) -- (13,-0.2) -- (14,0.2) -- (15.5,-0.07) -- (15.5,-3.5) -- (8.5,-3.5) -- (8.5,-0.1);
        \draw[ultra thick]  (8.5,-0.1) -- (10,0.2) -- (11,-0.3) -- (12,0) -- (13,-0.2) -- (14,0.2) -- (15.5,-0.07);
        \draw[blue, thick]  (12,0) -- (12.1,0.5) -- (11.9,1.5) -- (12.2,3) -- (12,4);
        \draw[purple, thick]  (12,0) -- (11.9,-0.5) -- (12.1,-1.5) -- (11.8,-3) -- (11.9,-3.5);
        
        % labels
        \node [red, font=\bfseries] at (0,-0.02) {\large $\bullet$};
        \node [red, font=\bfseries] at (0.05,-0.5) {\footnotesize $C_P$};
        \node [red, font=\bfseries] at (0, -2.5) {$\bullet$};
        \node [red, font=\bfseries] at (0,-3) {\small $\alpha_P^\infty$};
        \node [red, font=\bfseries] at (0,3.5) {$\bullet$};
        \node [red, font=\bfseries] at (0,4) {\small $\alpha_P^0$};
        \node [blue, font=\bfseries] at (3.2,0.5) {\small $\prescript{}{1}\Rbold^{0}_P$};
        \node [blue, font=\bfseries] at (-0.6,1.7) {\small $\prescript{}{2}\Rbold^{0}_P$};
        \node [blue, font=\bfseries] at (-3.25,0.5) {\small $\prescript{}{3}\Rbold^{0}_P$};
        \node [purple, font=\bfseries] at (-2.15,-1.6) {\small $\prescript{}{1}\Rbold^{\infty}_P$};
        \node [purple, font=\bfseries] at (2.2,-1.6) {\small $\prescript{}{2}\Rbold^{\infty}_P$};
        \draw[-latex] (5.5,0) -- (7.8,0);
        \node[blue, font=\bfseries] at (11.6,3.2) {\small $\Rbold^0$};
        \node[purple, font=\bfseries] at (12.4,-3) {\small $\Rbold^\infty$};
        \node[green!50!black, font=\bfseries] at (14,2) {\small $\Obold^0$};
        \node[yellow!50!black, font=\bfseries] at (10,-1.8) {\small $\Obold^\infty$};
        \node[font=\bfseries] at (6.7,-0.8) {\small $\Fbold^P$};
\end{tikzpicture}

    \caption{The construction of wakes rooted at $C_P$ when $(d_0,d_\infty)=(3,2)$. }
    \label{fig:rays-and-wakes}
\end{figure}

\begin{definition}
    A wake $\Wbold$ is a Jordan domain of the form $\prescript{}{J,k}\Wbold^{\blacksquare,\bullet}_S$. We call $\prescript{}{J}C^{\blacksquare}_S$ the \emph{root} of $\Wbold$ and $\prescript{}{J}\alpha^{\blacksquare,\bullet}_S$ the \emph{alpha-point} of $\Wbold$. The \emph{generation} of $\Wbold$ is $|S|$. If $S$ is a tuple of length $m$, we say that $m$ is the \emph{level} of $\Wbold$. If $m=1$, we call $\Wbold$ a \emph{primary} wake.
\end{definition}

Due to the tree structure of $\Esc_{<\infty}$, primary wakes are always pairwise disjoint.

\begin{lemma}
    Consider a wake $\prescript{}{J,j}\Wbold^{\blacksquare,\bullet}_S$ rooted at a critical point $\prescript{}{J}C^\blacksquare_S$.
    \begin{enumerate}[label=\textnormal{(\arabic*)}]
        \item If $\Fbold^Q$ sends $\prescript{}{J}C^\blacksquare_S $ to another critical point $\prescript{}{J'}C^\square_{S'}$, then $\Fbold^Q : \overline{\prescript{}{J,j}\Wbold^{\blacksquare,\bullet}_S} \to \overline{\prescript{}{J',j}\Wbold^{\square,\bullet}_{S'}}$ is a homeomorphism.
        \item The map $\Fbold^{|S|}$ conformally sends $\prescript{}{J,j}\Wbold^{\blacksquare,\bullet}_S$ onto $\C \backslash \overline{ \Rbold^\bullet }$.
    \end{enumerate}    
\end{lemma}

\begin{proof}
    (1) follows from the fact that $\Fbold^Q$ maps $\prescript{}{J,j}\Rbold^{\blacksquare,\bullet}_S \cup \prescript{}{J,j+1}\Rbold^{\blacksquare,\bullet}_S$ homeomorphically onto $\prescript{}{J',j}\Rbold^{\square,\bullet}_{S'} \cup \prescript{}{J',j+1}\Rbold^{\square,\bullet}_{S'}$, whereas (2) follows from the fact that $\Fbold^{|S|}$ maps $\prescript{}{J,j}\Rbold^{\blacksquare,\bullet}_S$ for each $j \in \{1,\ldots,d_\bullet\}$ homeomorphically onto the zero ray $\Rbold^\bullet$.
\end{proof}

To reduce notation, let us consider the \emph{full wake} 
\[
    \prescript{}{J}\Wbold^{\blacksquare,\bullet}_S := \bigcup_{j=1}^{d_\bullet-1}\prescript{}{J,j}\Wbold^{\blacksquare,\bullet}_S(j)
\]
which is the union of wakes attached to the critical point $\prescript{}{J}C^{\blacksquare,\bullet}_S$ on the same side.

\begin{lemma}[Primary wakes shrink]
\label{lem:primary-wakes-shrink}
    For every $n \in \Z$ and every $\varepsilon>0$, there are at most finitely many primary wakes of diameter at most $\varepsilon$ rooted at a point on $\Hq \cap \Sbold_n^\#$.
\end{lemma}

\begin{proof}
    The proof we present below is similar to \cite[Lemma 5.29]{DL23}. By self-similarity, it is sufficient to prove the lemma for $n=0$. Let 
    \[
    \Jbold_- := \Ubold_- \cap \Hq, \qquad \Jbold_+ := \Ubold_+ \cap \Hq, \quad \text{and} \quad \Jbold := \Jbold_- \cup \Jbold_+.
    \]
    The maps $\fbold_-=\Fbold^{(0,1,0)}: \Jbold_- \to \Jbold$ and $\fbold_+=\Fbold^{(0,0,1)}: \Jbold_- \to \Jbold$ are precisely the first return maps of $\Fbold$ back to $\Jbold$.

    Consider the semigroup generated by $(0,1,0)$ and $(0,0,1)$ and let us label its elements by $0$, $Q_0$, $Q_1$, $Q_2,\ldots$ written in increasing order. Then, every critical point on $\Jbold$ is of the form $C_{Q_n}$ for some $n \geq 0$. Let us fix $\bullet \in \{0,\infty\}$ and consider the full primary wake $\Wbold_n := \Wbold^\bullet_{Q_n}$ attached to $C_{Q_n}$. For all $n > 0$, $\Wbold_n$ is a preimage under $\Fbold^{Q_n-Q_0}$ of the full wake $\Wbold_0$ with the smallest generation. 
    
    Let us pick a curve $\Gamma_0$ in $\Wbold_0$ connecting a point in $\Wbold_0$ to the critical point $C_{Q_0}$. Consider the lift $\Gamma_{-n}$ of $\Gamma_0$ under $\Fbold^{Q_n-Q_0}: \Wbold_n \to \Wbold_0$, which connects a point in $\Wbold_n$ to the critical point $C_{Q_n}$.

    \begin{claim}
        There is a sequence $\varepsilon_0$, $\varepsilon_{-1}$, $\varepsilon_{-2}$, $\ldots$ of positive numbers decreasing to $0$ such that the following holds. If the (Euclidean) diameter of $\Gamma_0$ is less than $\varepsilon_0$, then the diameter of $\Gamma_{-n}$ is less than $\varepsilon_n$ for all $n \geq 0$.
    \end{claim}

    \begin{proof}
        It is sufficient to prove the claim in the dynamical plane of the corona $f_*$. Consider the rational map $g=F_c$ from Theorem \ref{thm:comb-rigidity} which admits a $(d_0,d_\infty)$-critical Herman curve $\Hq_g$ with rotation number equal to that of $f_*$. By Theorem \ref{thm:naive-rigidity}, $g$ is quasiconformally conjugate to $f_*$ on a neighborhood of $\Hq_g$, so it suffices to prove the claim in the dynamical plane of $g$. We shall do so by applying the local connectivity of the boundary of the immediate basin of attraction of $\bullet$ of $g$.

        Recall that the critical point of $g$ is normalized at $1$. For $k \in \Z$, denote $c_k := (g|_{\Hq_g})^{k}(1)$. Within the immediate basin of $\bullet$, let us pick two external rays $R_l$ and $R_r$ landing at points on $\Hq_g$ that are slightly on the left and right of $c_0$ respectively. Let us pick a disk $D_0$ of small diameter bounded by $\Hq_g$, $R_l$, $R_r$, and an equipotential within the immediate basin of $\bullet$. Let $D_{-k}$ be the unique lift of $D_0$ under $g^k$ whose boundary contains $c_{-k}$. The disk $D_{-k}$ is bounded by $g^{-k}(\Hq)$, a pair of external rays which are preimages of $R_l$ and $R_r$, and an equipotential of an even smaller level. By local connectivity, the Euclidean diameter of $D_{-k}$ shrinks to zero as $k \to \infty$.
    \end{proof}
    
    Let $\Obold_- \subset \Obold^\bullet$ be the union of all lakes of generation $(0,1,0)$ whose closure intersects $\Jbold_-$, and let $\Obold_+ \subset \Obold^\bullet$ be the union of all lakes of generation $(0,0,1)$ whose closure intersects $\Jbold_+$. The maps $\fbold_\pm : \Obold_\pm \to \Obold^\bullet$ expand the hyperbolic metric of $\Obold^\bullet$. Note that for all $n\geq 0$, $\Fbold^{Q_{n+1}-Q_n}: \Wbold_{n+1}\to \Wbold_n$ is a restriction of $\fbold_\pm : \Obold_\pm \to \Obold^\bullet$. Then, due to the claim, the Euclidean diameter of $\Wbold_n$ shrinks as $n \to \infty$.
\end{proof}

The outer boundary of each of the full wakes attached to $\prescript{}{J}C^{\blacksquare}_S$ consists of two ray segments, which we will relabel as
\[
    \prescript{+}{J}\Rbold^{\blacksquare,0}_S := \prescript{}{J,1}\Rbold^{\blacksquare,0}_S,
    \quad
    \prescript{-}{J}\Rbold^{\blacksquare,0}_S := \prescript{}{J,d_\bullet}\Rbold^{\blacksquare,0}_S,
    \quad
    \prescript{-}{J}\Rbold^{\blacksquare,\infty}_S := \prescript{}{J,1}\Rbold^{\blacksquare,\infty}_S,
    \quad
    \prescript{+}{J}\Rbold^{\blacksquare,\infty}_S := \prescript{}{J,d_\bullet}\Rbold^{\blacksquare,\infty}_S.
\]

For every $P \in \Tbold_{>0}$, let $P^-$ (resp. $P^+$) be the first entry of the left (resp. right) itinerary of the side lake $\prescript{}{l}\Obold^0_P$ (resp. $\prescript{}{r}\Obold^0_P$). Both $P^-$ and $P^+$ are characterized by the property that $(C_{P^-}, C_{P^+}) \subset \Hq$ is the maximal open interval in which the only critical point of generation $\leq P$ is $C_P$.

\begin{lemma}[Combinatorics of primary wakes]
\label{lem:comb-of-wakes}
    Given $P \in \Tbold_{>0}$ and $\bullet \in \{0,\infty\}$,
    \begin{enumerate}[label=\textnormal{(\arabic*)}]
        \item both $\prescript{+}{}\Rbold^\bullet_{P^-}$ and $\prescript{-}{}\Rbold^\bullet_{P^+}$ contain $\alpha_P^\bullet$;
        \item the closure of $\Wbold^\bullet_{P} \cup \Wbold^\bullet_{P^-} \cup \Wbold^\bullet_{P^+}$ is a neighborhood of $\alpha^\bullet_P$;
        \item the ray segments $\prescript{+}{}\Rbold^{\bullet}_P$ and $\prescript{-}{}\Rbold^{\bullet}_P$ can be presented as infinite concatenations of ray segments
        \begin{align*}
        \prescript{\pm}{}\Rbold^{\bullet}_P &= 
        \left[\alpha^\bullet_P, \alpha^\bullet_{Q^\pm_1}\right]
        \cup 
        \left[\alpha^\bullet_{Q^\pm_1}, \alpha^\bullet_{Q^\pm_2}\right]
        \cup 
        \left[\alpha^\bullet_{Q^\pm_2}, \alpha^\bullet_{Q^\pm_3}\right]
        \cup \ldots, 
        \end{align*}
    where
        \[
        \prescript{\pm}{}\Rbold^{\bullet}_P \cap \prescript{\mp}{}\Rbold^{\bullet}_{P^\pm} = \left[\alpha^\bullet_P, \alpha^\bullet_{Q^\pm_1}\right], 
        \quad \text{and for } i\geq 1, \enspace
        \prescript{\pm}{}\Rbold^{\bullet}_P \cap \prescript{\mp}{}\Rbold^{\bullet}_{Q^\pm_i} = \left[\alpha^\bullet_{Q^\pm_i}, \alpha^\bullet_{Q^\pm_{i+1}}\right];
        \]
        \item the sequences of alpha-points $\left\{\alpha_{Q_i^+}^\bullet\right\}_{i \geq 1}$ and $\left\{\alpha_{Q_i^-}^\bullet\right\}_{i \geq 1}$ tend to $C_P$ as $i \to \infty$.
    \end{enumerate}
\end{lemma}

See Figures \ref{fig:wake-combinatorics} and \ref{fig:wake-real-picture}.

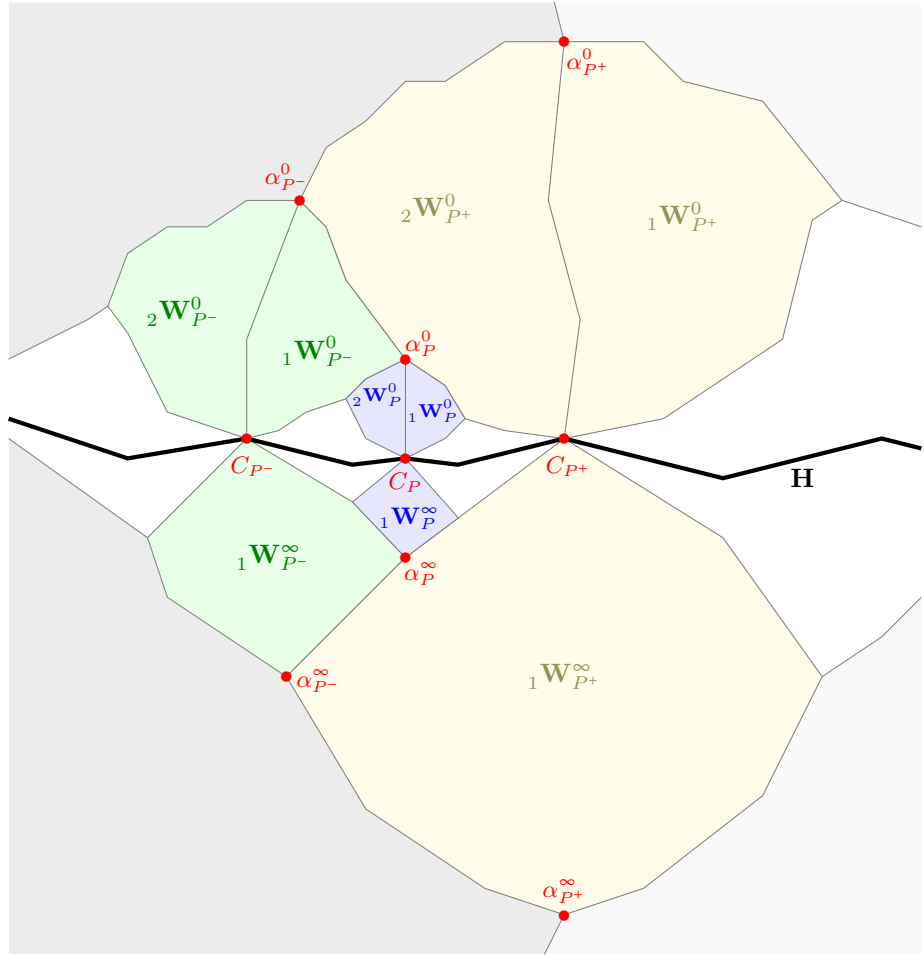
\begin{figure}
    \centering
    \begin{tikzpicture}[scale=1.05]
        % for below
        \filldraw[white, fill=gray!15!white] (-3,0) -- (-1.25,-1.25) -- (-1,-2) -- (0.5,-3) -- (1.5,-4.667) -- (3,-5.667) -- (4,-6) -- (3.75,-6.5) -- (-3,-6.5) -- (-3,0);
        \filldraw[white, fill=gray!05!white] (3.75,-6.5) -- (4,-6) -- (5,-5.667) -- (6.5,-4.5) -- (7.25,-3) -- (8,-2.5) -- (8.5,-2) -- (8.5,-6.5) -- (3.75,-6.5);
        \filldraw[gray, fill=green!10!white] (0,0) -- (-1.25,-1.25) -- (-1,-2) -- (0.5,-3) -- (2,-1.5) -- (1.333,-0.8) -- (0,0);
        \filldraw[gray, fill=yellow!10!white] (4,0) -- (2,-1.5) -- (0.5,-3) -- (1.5,-4.667) -- (3,-5.667) -- (4,-6) -- (5,-5.667) -- (6.5,-4.5) -- (7.25,-3) -- (6,-1.25) -- (4,0);
        \filldraw[gray, fill=blue!10!white] (2,-0.25) -- (1.333,-0.8) -- (2,-1.5) -- (2.667,-1) -- (2,-0.25);
        \draw[gray] (-1.25,-1.25) -- (-3,-0);
        \draw[gray] (4,-6) -- (3.75,-6.5);
        \draw[gray] (7.25,-3) -- (8,-2.5) -- (8.5,-2);

        % for above
        \filldraw[white,fill=gray!15!white] (-3,1) -- (-2,1.5) -- (-1.75,1.667) -- (-1.5,2.333) -- (-1,2.667) -- (-0.5,2.667) -- (0,3) -- (0.667, 3) -- (1,3.667) -- (1.5,4) -- (2,4.5) -- (2.5,4.5) -- (3.25,5) -- (4,5) -- (3.875,5.5) -- (-3,5.5) -- (-3,1);
        \filldraw[white,fill=gray!05!white] (3.875,5.5) -- (4,5) -- (5,5) -- (5.5,4.5) -- (6.5,4.25) -- (7.5,3) -- (8.5,2.667) -- (8.5,5.5) -- (3.875,5.5);
        \filldraw[gray, fill=green!10!white] (0,0) -- (-1,0.333) -- (-1.5,1.33) -- (-1.75,1.667) -- (-1.5,2.333) -- (-1,2.667) -- (-0.5,2.667) -- (0,3) -- (0.667, 3) -- (1,2.667) -- (1.25,2) -- (1.5,1.667) -- (2,1) -- (1.5,0.75) -- (1.25, 0.5) -- (0.75,0.333) -- (0.4,0.1) -- (0,0);
        \filldraw[gray, fill=yellow!10!white] (4,0) -- (5.25,0.25) -- (6.75, 1.25) -- (7.125,2.75) -- (7.5,3) -- (6.5,4.25) -- (5.5,4.5) -- (5,5) -- (3.25,5) -- (2.5,4.5) -- (2,4.5) -- (1.5,4) -- (1,3.667) -- (0.667, 3) -- (1,2.667) -- (1.25,2) -- (1.5,1.667) -- (2,1) -- (2.5,0.667) -- (2.75,0.25) -- (3.25,0.1) -- (4,0);
        \filldraw[gray, fill=blue!10!white] (2,-0.25) -- (2.5,0) -- (2.75,0.25) -- (2.5,0.667) -- (2,1) -- (1.5,0.75) -- (1.25, 0.5) -- (1.5,0) -- (2,-0.25);
        \draw[gray] (-1.75,1.667) -- (-2,1.5) -- (-3,1);
        \draw[gray] (4,5) -- (3.875,5.5);
        \draw[gray] (7.5,3) -- (8.5,2.667);
        \draw[gray] (0,0) -- (0,1.25) -- (0.667,3);
        \draw[gray] (4,0) -- (4.2,1.5) -- (3.8,3) -- (4,5);
        \draw[gray] (2,-0.25) -- (2,1);
        
        % for Hq
        \draw[ultra thick] (-3,0.25) -- (-1.5,-0.25) -- (0,0) -- (1.33,-0.33) -- (2,-0.25) -- (2.66,-0.33) -- (4,0) -- (6,-0.5) -- (8,0) -- (8.5,-0.125);
        \node [black, font=\bfseries] at (7,-0.5) {$\Hq$};
        
        % for the labels
        \node [red, font=\bfseries] at (0,-0.02) {\large $\bullet$};
        \node [red, font=\bfseries] at (0.08,-0.35) {\small $C_{P^-}$};
        \node [red, font=\bfseries] at (0.667,2.98) {\large $\bullet$};
        \node [red, font=\bfseries] at (0.5,3.3) {\small $\alpha^0_{P^-}$};
        \node [red, font=\bfseries] at (0.5,-3.02) {\large $\bullet$};
        \node [red, font=\bfseries] at (0.9,-3.05) {\small $\alpha^\infty_{P^-}$};
        
        \node [red, font=\bfseries] at (2,-0.27) {\large $\bullet$};
        \node [red, font=\bfseries] at (2,-0.55) {\small $C_{P}$};
        \node [red, font=\bfseries] at (2,0.98) {\large $\bullet$};
        \node [red, font=\bfseries] at (2.2,1.2) {\small $\alpha^0_{P}$};
        \node [red, font=\bfseries] at (2,-1.52) {\large $\bullet$};
        \node [red, font=\bfseries] at (2.2,-1.7) {\small $\alpha^\infty_{P}$};
        
        \node [red, font=\bfseries] at (4,-0.02) {\large $\bullet$};
        \node [red, font=\bfseries] at (4.05,-0.35) {\small $C_{P^+}$};
        \node [red, font=\bfseries] at (4,4.98) {\large $\bullet$};
        \node [red, font=\bfseries] at (4.3,4.75) {\small $\alpha^0_{P^+}$};
        \node [red, font=\bfseries] at (4,-6.02) {\large $\bullet$};
        \node [red, font=\bfseries] at (4,-5.7) {\small $\alpha^\infty_{P^+}$};
                
        \node [blue, font=\bfseries] at (1.65,0.55) {\scriptsize$\prescript{}{2}\Wbold^0_{P}$};
        \node [green!50!black, font=\bfseries] at (-0.8,1.6) {$\prescript{}{2}\Wbold^0_{P^-}$};
        \node [yellow!50!black, font=\bfseries] at (2.4,2.9) {$\prescript{}{2}\Wbold^0_{P^+}$};
        
        \node [blue, font=\bfseries] at (2.35,0.35) {\scriptsize$\prescript{}{1}\Wbold^0_{P}$};
        \node [green!50!black, font=\bfseries] at (0.9,1.1) {$\prescript{}{1}\Wbold^0_{P^-}$};
        \node [yellow!50!black, font=\bfseries] at (5.5,2.8) {$\prescript{}{1}\Wbold^0_{P^+}$};
        
        \node [blue, font=\bfseries] at (2.05,-1) {\small$\prescript{}{1}\Wbold^\infty_{P}$};
        \node [green!50!black, font=\bfseries] at (0.33,-1.5) {$\prescript{}{1}\Wbold^\infty_{P^-}$};
        \node [yellow!50!black, font=\bfseries] at (4,-3) {$\prescript{}{1}\Wbold^\infty_{P^+}$};
\end{tikzpicture}

    \caption{A cartoon picture of the structure of wakes when $(d_0,d_\infty)=(3,2)$. See Figure \ref{fig:wake-real-picture} for a more realistic picture.}
    \label{fig:wake-combinatorics}
\end{figure}

\begin{figure}
    \centering
    \includegraphics[width=\linewidth]{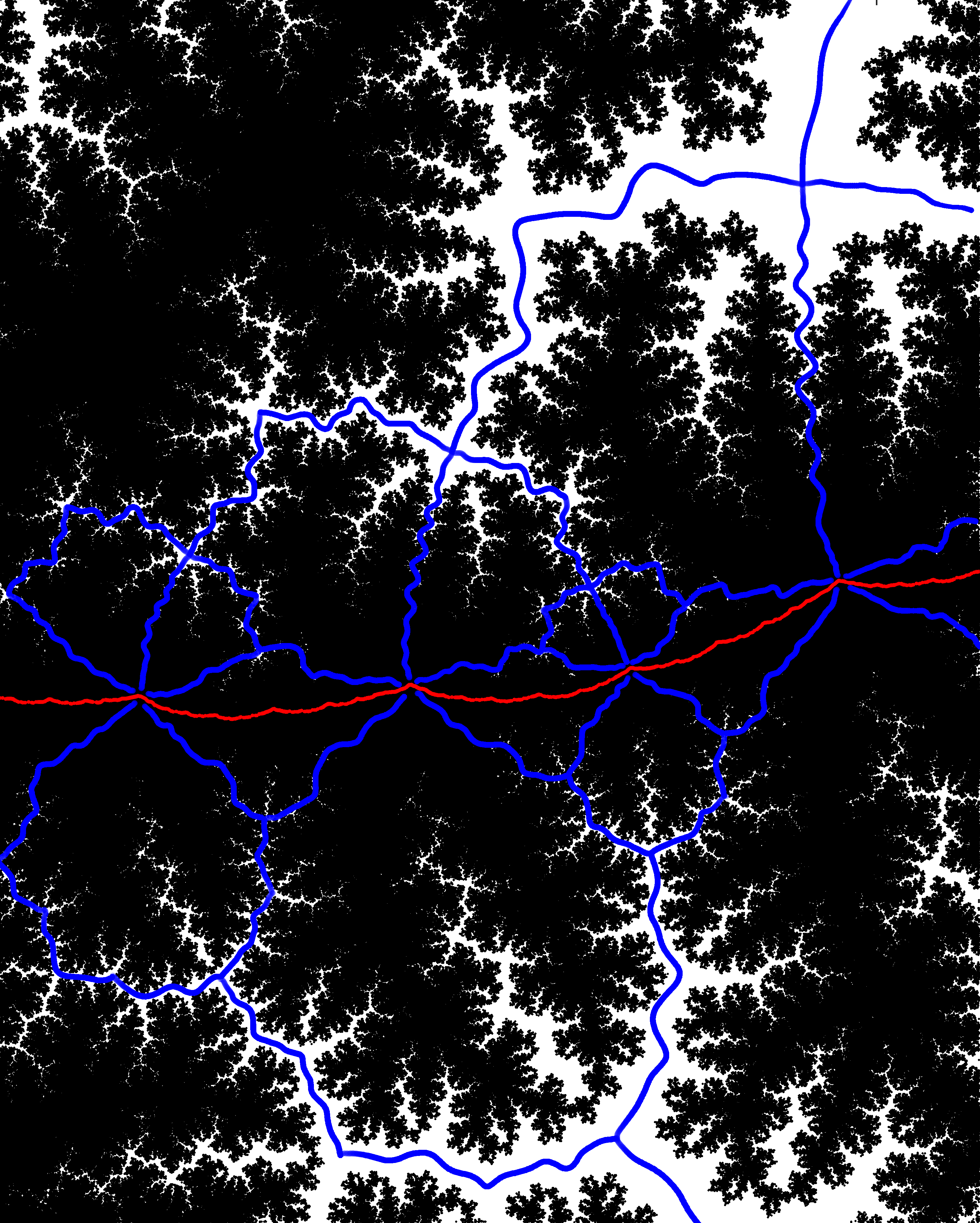}

    \caption{An approximate picture of the dynamical plane of $\Fbold_*$ when $(d_0,d_\infty)=(3,2)$ and $\theta$ is the golden mean irrational. This figure is obtained from the magnification of the Julia set of the rational map $f_{3,2}$ in Figure \ref{fig:cqc-comparison} around a point on its Herman curve. The Herman curve $\Hq$ of $\Fbold_*$ is colored red and some external ray segments are displayed in blue. These external rays are the boundaries of the primary wakes attached to four critical points on $\Hq$.}
    \label{fig:wake-real-picture}
\end{figure}

\begin{proof}
    The left coast of $\prescript{}{l}\Obold^0_P$ is contained in $\prescript{}{1}\Wbold^0_{P^-}$ because it starts with a segment of the spine $\prescript{}{1}\Hq^0_{P-}$ rooted at $C_{P^-}$ and is disjoint from the external rays landing at $C_{P^-}$. Since the left coast of $\prescript{}{l}\Obold^0_P$ lands at the alpha-point $\alpha^0_P$, the boundary of the wake $\prescript{}{1}\Wbold^0_{P^-}$ must contain $\alpha^0_{P}$. The treatment for the other side lakes of $C_P$ is analogous, and this implies (1).
    
    We have established that $\alpha^\bullet_P$ is in the boundary of each of $\Wbold^\bullet_P$, $\Wbold^\bullet_{P^-}$, and $\Wbold^\bullet_{P^+}$. By Corollary \ref{cor:two-external-rays}, there exist alpha-points $\alpha'$, $\alpha_-$, and $\alpha_+$ such that
    \[
        \alpha' \prec \alpha^\bullet_P, \quad \alpha^\bullet_P \prec \alpha_-, \quad  \alpha^\bullet_P \prec \alpha_+
    \]
    and
    \[
        \prescript{+}{}\Rbold^\bullet_{P^-} \cap \prescript{-}{}\Rbold^\bullet_{P^+} = [\alpha' ,\alpha_P^\bullet], \quad \prescript{+}{}\Rbold^\bullet_{P^-} \cap \prescript{-}{}\Rbold^\bullet_{P} = [\alpha_P^\bullet, \alpha_-], \quad \prescript{+}{}\Rbold^\bullet_{P} \cap \prescript{-}{}\Rbold^\bullet_{P^+} = [\alpha_P^\bullet, \alpha_+].
    \]
    Therefore, the union of $\Wbold^\bullet_{P^-}$, $\Wbold^\bullet_{P}$, and $\Wbold^\bullet_{P^+}$ form a neighborhood of $\alpha^\bullet_P$, thus proving (2). More generally, we have just shown that every primary alpha-point is the meeting point of exactly three distinct primary full wakes.

    Let us prove (3) and (4) for $\prescript{-}{}\Rbold^\bullet_P$. The treatment for $\prescript{+}{}\Rbold^\bullet_P$ is analogous. Let us define $Q^-_1 \in \Tbold$ to be the unique smallest moment greater than $P$ such that $C_{Q^-_1}$ is contained on the interval $(C_{P^-}, C_P) \subset \Hq$. Then, based on the previous paragraph, the alpha-point $\alpha_-$ must be equal to $\alpha^\bullet_{Q^-_1}$ because it is the meeting point of $\prescript{-}{}\Rbold^\bullet_{P}$, $\prescript{+}{}\Rbold^\bullet_{P^-}$, and the boundary of a primary full wake, which is $\Wbold_{Q^-_1}^\bullet$. Similarly, $\prescript{+}{}\Rbold^\bullet_{Q^-_1}$ and $\prescript{-}{}\Rbold^\bullet_{P}$ meet along a ray segment $\left[\alpha_{Q^-_1}^\bullet, \alpha_{Q^-_2}^\bullet \right]$ for some $Q^-_2 > Q^-_1$. Inductively, we obtain the desired increasing sequence $\left\{Q^-_i\right\}_{i\in\N}$ of power-triples. It remains to show that the corresponding sequence of alpha-points $\alpha^\bullet_{Q^-_i}$ indeed converges to $C_P$.

    By Proposition \ref{prop:zero-chains}, there exists an alpha-point $\alpha$ on $\prescript{-}{}\Rbold^\bullet_{P}$ close to $C_P$, which is the alpha-point of some primary full wake $\Wbold^\bullet_Q$ where $Q > P$. Since there are at most finitely many critical points on $\Hq$ of generation less than $Q$ between $C_Q$ and $C_{Q^-_1}$, the ray segment $[\alpha^\bullet_{Q^-_1}, \alpha^\bullet_Q]$ intersects the boundaries of at most finitely many primary wakes. Therefore, $Q = Q^-_i$ for some $i \in \N$. Since $\alpha$ can be picked to be arbitrarily close to $C_P$, then $\alpha^\bullet_{Q^-_i}$ indeed converges to $C_P$.
\end{proof}

\begin{corollary}[Tiling of wakes] 
\label{cor:tiling-of-wakes}
\leavevmode
    \begin{enumerate}[label=\textnormal{(\arabic*)}]
        \item Primary wakes fill up the ocean: for $\bullet \in \{0,\infty\}$, 
        \[
            \Obold^\bullet \subset \bigcup_{P\in\Tbold_{>0}}\overline{\Wbold^\bullet_P}.
        \]
        \item The closure of a wake $\prescript{}{J,j}\Wbold^{\blacksquare,\bullet}_S$ is the union of spines $\prescript{}{J,2j-1}\Hq^{\blacksquare,\bullet}_{S}$ and $\prescript{}{J,2j}\Hq^{\blacksquare,\bullet}_{S}$ and the closure of all full wakes rooted at critical points on any of these two spines. 
        \item For every $z \in \Esc_{<\infty}$ and $m \in \N_{\geq 1}$, there are at most three disjoint full wakes of level $\geq m$ containing $z$ on their boundaries. The union of these full wakes forms a neighborhood of $z$.
    \end{enumerate} 
\end{corollary}

\begin{proof}
    To prove (1), let us assume for a contradiction that there is a non-empty connected component $X$ of $\Obold^\bullet \backslash \bigcup_{P}\overline{\Wbold^\bullet_P}$. By Lemma \ref{lem:comb-of-wakes}, $\overline{X}$ intersects some point $x$ on $\Hq$. There exists two distinct sequences of power-triples $\{Q_n\}$ and $\{T_n\}$ such that for all $n \in \N$, the primary wakes $\Wbold^\bullet_{Q_n}$ and $\Wbold^\bullet_{T_n}$ touch, the union $\Hq \cup \overline{\Wbold^\bullet_{Q_n}} \cup \overline{\Wbold^\bullet_{T_n}}$ encloses a unique disk $D_n$ containing $X$, and the corresponding roots $C_{Q_n}$ and $C_{T_n}$ converge to $x$ as $n\to \infty$. By Lemma \ref{lem:primary-wakes-shrink}, the diameter of $D_n$ tends to $0$ as $n \to \infty$, which implies that such $X$ cannot exist.

    Item (2) follows from pulling back the tiling of wakes in (1) by the map $\Fbold^{|S|}$ on $\overline{\prescript{}{J,j}\Wbold^{\blacksquare,\bullet}_S}$. We have thus shown that wakes of a fixed level tile each of the two oceans, and every point in the ocean is contained in the closure of at most three wakes of the same level. This implies (3).
\end{proof}

\begin{lemma}
\label{lem:uniform-expansion}
    Let us equip $\C \backslash \Hq$ with the hyperbolic metric $\rho_0$. For every $P \in \Tbold_{>0}$, 
    \begin{enumerate}[label=\textnormal{(\arabic*)}]
        \item the map $\Fbold^P: \Wbold^\bullet_P \backslash \Fbold^{-P}(\Hq) \to \C \backslash \Hq$ is uniformly expanding (with respect to $\rho_0$) with a factor independent of $P$;
        \item the hyperbolic diameter of every wake of level two is at most some uniform constant independent of $P$.
    \end{enumerate}
\end{lemma}

\begin{proof}
    For all $P \in \Tbold$, let $\rho_P$ be the hyperbolic metric of $\C \backslash \Fbold^{-P}(\Hq)$. To prove (1), it suffices to show that the inclusion map 
    \[
    \iota: \left(\C \backslash \Fbold^{-P}(\Hq), \rho_P\right) \to \left(\C \backslash \Hq, \rho_0\right)
    \]
    is uniformly contracting on $\Wbold^\bullet_P \backslash \Fbold^{-P}(\Hq)$.

    Clearly, $\iota$ is uniformly contracting on $\Wbold^\bullet_P$ minus a small neighborhood of $C_P$ because this region is a compact subset of $\Obold^\bullet$. The uniform contraction of $\iota$ on a neighborhood of $C_P$ follows from the asymptotic self-similarity of $\Hq$ and $\partial \Wbold^\bullet_P$ near $C_P$ induced by pulling back $A_*$-invariance near $0$ by $\Fbold^P: C_P \mapsto 0$. One may refer to \cite[Lemma 5.33]{DL23} for further details.
    
    Item (2) follows from essentially the same argument. By compactness, every subwake of $\Wbold^\bullet_P$ of level two has uniformly bounded diameter away from a neighborhood of $C_P$. Near $C_P$, the claim again follows from the asymptotic self-similarity at $C_P$. 
    
    Lastly, the bounds in both claims are independent of $P$ because every full wake can be mapped to a full wake $\Wbold^\bullet_Q$ for some $\bullet \in \{0,\infty\}$ for all sufficiently small $Q \in \Tbold_{>0}$.
\end{proof}

\begin{lemma}
\label{lem:nested-wakes-shrink}
    Any infinite sequence of nested wakes shrinks to a point.
\end{lemma}

\begin{proof}
    Let us define a holomorphic map $\chi$ sending level two wakes to level one wakes as follows. Given a critical point $c$ of $\Fbold^{\geq 0}$, let $W(c)$ be the union of all wakes rooted at $c$. Consider a secondary critical point $\prescript{}{j}C^\bullet_{P,Q}$, which is contained in $W(C_P)$. The map $\Fbold^P$ sends $W(\prescript{}{j}C^\bullet_{P,Q})$ univalently onto $W(C_Q)$. 
    Let $T \in \Tbold$ be the smallest power-triple such that $Q-T=\tbold^n P$ for some $n \in \Z$. Then, $\chi := A^{-n} \circ \Fbold^{P+T}$ sends $W(\prescript{}{j}C^\bullet_{P,Q})$ univalently back onto $W(C_P)$. 
    By Lemma \ref{lem:uniform-expansion} (1), $\chi$ must be uniformly expanding on $W(\prescript{}{j}C^\bullet_{P,Q})$ with expansion factor independent of $P$.

    Now, consider an infinite sequence of nested wakes $W_1 \supset W_2 \supset W_3 \supset \ldots$ where each $W_n$ is of level $n$. By Lemma \ref{lem:uniform-expansion} (2), there is a uniform constant $C>0$ such that for all $n \geq 3$,
    \[
        \diam_{\rho_0}\left(\chi^{n-2}(W_n)\right) \leq C.
    \]
    Since $\chi$ is uniformly expanding, the hyperbolic diameter of $W_n$ tends to $0$ exponentially fast as $n \to \infty$.
\end{proof}

\subsection{The structure of $\Esc_{<\infty}$ and $\mathfrak{X}$}
\label{ss:structure-esc-set-and-X}

Using wakes, we will show in this final subsection that the finite-time escaping set consists of topologically tame external rays.

\begin{corollary}
    Every external ray lands at a unique point.
\end{corollary}

\begin{proof}
    Let $X$ be the accumulation set of an external ray. Since the boundary of every wake is made of ray segments, then for every wake $W$, either $X \subset \overline{W}$ or $X \subset \C \backslash W$. 
    
    If $X$ intersects $\Hq$, then by Corollary \ref{cor:tiling-of-wakes}, $X$ must be contained in $\Hq$. In general, if $X$ intersects $\Fbold^{-P}(\Hq)$ for some $P \in \Tbold$, then $X \subset \Fbold^{-P}(\Hq)$. Since the roots of wakes are dense in $\Fbold^{-P}(\Hq)$, $X$ must be a singleton.

    Suppose $X$ is disjoint from $\Fbold^{-P}(\Hq)$ for all $P$. Then, $X$ is contained in an infinite sequence of nested wakes which, by Lemma \ref{lem:nested-wakes-shrink}, implies that $X$ is a singleton.
\end{proof}

We say that two points $x$ and $y$ in $\Esc_{\leq P}$ are \emph{combinatorially equivalent} if there is no alpha-point $\alpha$ such that $x$ and $y$ belong in distinct connected components of $\Esc_{\leq P} \backslash \{\alpha\}$. Combinatorial equivalence is an equivalence relation in $\Esc_{<\infty}$.

\begin{corollary}\leavevmode
\label{cor:triviality-of-comb-class}
    \begin{enumerate}[label=\textnormal{(\arabic*)}]
        \item Every combinatorial equivalence class in $\Esc_{<\infty}$ is a singleton.
        \item $\Esc_{<\infty}$ is dense in $\C$ and has empty interior.
        \item For every $P \in \R_{>0}$,
    \[
        \Esc_{\leq P} = \overline{\bigcup_{Q<P} \Esc_{\leq Q}}.
    \]
    \end{enumerate}
\end{corollary}

\begin{proof}
    Consider a point $x$ in $\Esc_{\leq P}$. There are two cases. 
    If $x$ is on an external ray segment, then the triviality of the combinatorial class follows from Corollary \ref{cor:description-of-chains}. 
    Otherwise, by Corollary \ref{cor:tiling-of-wakes}, $x$ is contained in an infinite sequence of nested wakes $W_1 \supset W_2 \supset W_3 \supset \ldots$. In this case, since the boundary of every wake is made of ray segments, the whole combinatorial equivalence class of $x$ must be contained in the same nest $W_n$'s. Then, the triviality of combinatorial class of $x$ follows from Lemma \ref{lem:nested-wakes-shrink}. 

    Suppose for a contradiction that the interior of $\Esc_{<\infty}$ is non-empty. Any connected component of the interior would be contained in a single combinatorial equivalence class, and this would contradict item (1). By Corollary \ref{cor:tiling-of-wakes} and Lemma \ref{lem:nested-wakes-shrink}, wakes of any fixed level tile the plane and any nested wakes shrink to points. Since $\Esc_{<\infty}$ intersects the closure of every wake of every level, then
    $\Esc_{<\infty}$ is dense in $\C$. This proves item (2). Lastly, item (3) follows directly from item (1).
\end{proof}

%\begin{corollary}
%\label{cor:jcas-and-H}
%    The Julia set of $\Fbold$ is the whole plane: $\Jcas(\Fbold)=\C$.
%\end{corollary}

%\begin{remark}
%    Let us discuss an alternative proof of Corollary \ref{cor:jcas-and-H} that is independent of the entirety of this section. By \cite[Theorem 5.3]{Lim23b}) and Theorem \ref{thm:qc-rigidity}, the critical value $c_1(f_*)$ of $f_*$ is a deep point of the non-escaping set of the corona $f_*$. (Roughly speaking, magnifications of the iterated preimages of the Herman quasicircle of $f_*$ about $c_1(f_*)$ converge to the whole plane exponentially fast in the Hausdorff metric.) As we pass to the corresponding dynamical plane of the transcendental extension, $0$ is a deep point of the iterated preimages of $\Hq$ under $\Fbold=(\fbold_\pm)$. By self-similarity, the iterated preimages of $\Hq$ must be dense in $\C$, so its closure $\Jcas(\Fbold)$ is equal to $\C$.
%\end{remark}

Since the finite-time escaping set is a subset of the Julia set, this corollary gives an alternative proof of Lemma \ref{lem:jcas-and-H}. 

Let us end this section with a discussion on the dynamics of $\Fbold$ outside of $\Esc_{<\infty}$ and the grand orbit of $\Hq$. Consider the set
\[
    \mathfrak{X} := \C \backslash \left( \Esc_{<\infty} \cup \bigcup_{P \in \Tbold} \Fbold^{-P} (\Hq) \right).
\]
Note that $\mathfrak{X}$ contains the infinite-time escaping set $\Esc_\infty = \Esc_{\infty}(\Fbold_*)$ of $\Fbold_*$.

Every point $x$ in $\mathfrak{X}$ is characterized by the property that for all $P \in \Tbold$, $\Fbold^P(x)$ is contained in a unique primary wake. Consider the holomorphic map
\begin{equation}
\label{eqn:F-tilde}
\hat{\Fbold}: \mathfrak{X} \to \mathfrak{X}, \quad \hat{\Fbold}(z) = \Fbold^P(z) \enspace \text{ if } z \in \Wbold^0_P \cup \Wbold^\infty_P.
\end{equation}
Thus, every point in $\mathfrak{X}$ is subject to infinite iteration of the map $\hat{\Fbold}$.

\begin{definition}
    For every point $z$ in $\mathfrak{X}$, the \emph{complete address} of $z$ is an infinite tuple $(\Wbold_0, \Wbold_1, \Wbold_2,\ldots)$ where for every $n\geq 0$, $\Wbold_n$ is the primary wake containing the unique point $\hat{\Fbold}^n(z)$. The (\emph{incomplete}) \emph{address} of $z$ is the infinite tuple $(P_0,P_1,P_2,\ldots) \in \Tbold^\N_{>0}$ where $P_n$ is the generation of $\Wbold_n$.
\end{definition}

We say that an element $(P_0,P_1,\ldots)$ of $\Tbold^\N_{>0}$ is \emph{admissible} if 
\[
\displaystyle{\sum_{n=0}^\infty P_n = \infty}.
\]
Moreover, we say that an infinite tuple of primary wakes is \emph{admissible} if the corresponding tuple of generations is admissible.

\begin{proposition}
\label{prop:address}
    \leavevmode
    \begin{enumerate}[label=\textnormal{(\arabic*)}]
        \item An infinite tuple of primary wakes is admissible if and only if it is the complete address of a point in $\mathfrak{X}$. 
        \item Two different points in $\mathfrak{X}$ always have distinct complete addresses.
    \end{enumerate}
\end{proposition}

\begin{proof}
    Given a point $z \in \mathfrak{X}$, if the sum of its incomplete address were finite, say $Q \in \R_{>0}$, then $z$ would have escape time $Q$ instead. Conversely, consider any admissible tuple of primary wakes $(\Wbold_0,\Wbold_1,\Wbold_2,\ldots)$. Consider the sequence of nested wakes $ \Wbold'_0:=\Wbold_0 \supset \Wbold'_1 \supset \Wbold'_2 \supset \ldots$ where for $n\geq 0$, $\Wbold'_{n+1}$ is defined inductively by the lift of $\Wbold_n$ under $\hat{\Fbold}^{n+1}|_{\Wbold'_n}$. The intersection of such nested wakes is precisely the set of points admitting the complete address $(\Wbold_0,\Wbold_1,\Wbold_2,\ldots)$, and according to Lemma \ref{lem:nested-wakes-shrink}, it is a singleton.
\end{proof}

\begin{corollary}
    $\mathfrak{X}$ is a dense, totally disconnected subset of $\C$.
\end{corollary}

\begin{proof}
    $\mathfrak{X}$ is dense because its complement has no interior. By Proposition \ref{prop:address} (2), two distinct points in $\C \backslash \Esc_{<\infty}$ have different complete itineraries and thus belong in disjoint wakes of sufficiently high generation. This implies the total disconnectivity of $\C \backslash \Esc_{<\infty}$.
\end{proof}

For $R>0$, define the large radius ``non-escaping`` set of $\Fbold$ by
\[
\mathfrak{K}_R := \left\{ z \in \C \backslash \Esc_{<\infty} \: : \: \left|\Fbold^P(z)\right| \geq R \text{ for all } P \in \Tbold \right\}.
\]
Whenever $R'>R>0$, we have
\[
\mathfrak{K}_{R'} \subset \mathfrak{K}_R \subset \mathfrak{X}.
\]

We say that $(P_0,P_1,P_2,\ldots) \in \Tbold^\N_{>0}$ is \emph{bounded by} $T \in \Tbold$ if $P_n \leq T$ for all $n$.

\begin{lemma}
    \leavevmode
    \begin{enumerate}[label=\textnormal{(\arabic*)}]
        \item For any high $R>0$, there exists some $Q_R \in \Tbold_{>0}$ such that $Q_R \to 0$ as $R\to \infty$ and that every point $z$ in $\mathfrak{K}_R$ has address bounded by $Q_R$.
        \item For any $Q \in \Tbold_{>0}$, there is some $R_Q>0$ such that every point in $\mathfrak{X}$ with address bounded by $Q$ is contained in $\mathfrak{K}_{R_Q}$.
    \end{enumerate}
\end{lemma}

\begin{proof}
    Let us fix $R>0$, and let $Q_R \in \Tbold_{>0}$ be the smallest power-triple such that all primary wakes of generation $Q_R$ are contained in the disk $\D_R:=\{|z|<R\}$. (This quantity exists due to Lemma \ref{lem:discreteness}.) Consider a point $z$ in $\mathfrak{X}$ and let $(P_0,P_1,P_2,\ldots)$ be its address. If $P_n \geq Q_R$ for some $n \in \N$, then $z$ is eventually mapped into a wake of generation $Q_R$, which is contained inside of $\D_R$. This implies (1).

    Next, let us fix $Q \in \Tbold_{>0}$, and let $R_Q>0$ be such that all primary wakes of generation $\leq Q$ are disjoint from $\D_{R_Q}$. Suppose that $\Fbold^P(z)$ is in $\D_{R_Q}$ for some $P \in \Tbold$. Then, $\Fbold^P(z)$ is contained in a wake of generation greater than $Q$. This implies (2).
\end{proof}

In the next section, we are interested in the infinite-time escaping set as well. For $\Fbold=\Fbold_*$, this set can be described as follows.

\begin{corollary}
    The infinite-time escaping set $\Esc_{\infty}$ of $\Fbold$ is the set of points in $\mathfrak{X}$ whose address $(P_0,P_1,P_2,\ldots)$ satisfies $P_n \to 0$ as $n \to \infty$.
\end{corollary}

\begin{proof}
    Consider a point $z$ in $\mathfrak{X}$ with some address $(P_0,P_1,P_2,\ldots)$. If $z \in \Esc_\infty$, then given any $R>0$, $\hat{\Fbold}^n(z)$ must be in $\mathfrak{K}_R$ for all sufficiently high $n$. By the previous lemma, $(P_n, P_{n+1},\ldots)$ is bounded by $Q_R$ where $Q_R \to 0$ as $R \to \infty$. Conversely, if $P_n \to 0$ as $n \to \infty$, then for all $n$, $(P_n,P_{n+1},\ldots)$ is bounded by some $Q_n \in \Tbold_{>0}$ where $Q_n \to 0$ as $n \to \infty$. By the previous lemma, $\hat{\Fbold}^n(z)$ is contained in $\mathfrak{K}_{R_{Q_n}}$ where $R_{Q_n} \to \infty$ as $n \to \infty$, thus $z$ is contained in $\Esc_\infty$.
\end{proof}

%%%%%%%%%%%%%%%%%%%%%%%%%%%%%%%%%%%%%%%%%%%%%%%%%%%%%%%%%%%%%%%%%%%%%%%%%%%%%%%%%%%%%%%%%%%%%%%%%%%%%%%%%%%%%%%%%%%%%%%%%%%%%%%%%%%%%%%%%%%%%%%%%%%%%%%%%%%%%%%%%%%%%%%%%%%%%%%%%%%%

\section{The escaping set \texorpdfstring{$\Esc(\Fbold)$}{I(F)}}
\label{sec:escaping-sets}

We will now discuss the topology and rigidity of both the finite-time and the infinite-time escaping sets of a cascade in $\manibold$. Our aim is to prove the first half of Theorem \ref{main-theorem-rigidity} by applying the external structure of $\Fbold_*$ and adapting an argument by Rempe \cite{R09} to show that the set of points in the full escaping set that remain sufficiently close to $\infty$ under iteration must move holomorphically with quasiconformal dilatation arbitrarily close to one.

\subsection{The finite-time escaping set}
\label{ss:rigidity-finite-esc-set}

Let us fix
\[
    T:= \min\{(0,1,0), (0,0,1)\}.
\]

\begin{lemma}
\label{lem:hol-mot-finite-esc}
    There is a unique equivariant holomorphic motion of $\Esc_{\leq T}(\Fbold)$ over some neighborhood $\Ucalbold$ of $\Fbold_*$.
\end{lemma}

\begin{proof}
    By Lemma \ref{lem:critical-pts}, the set of critical values $\CV\left(\Fbold^T\right)$ of $\Fbold^T$ moves holomorphically within a small neighborhood of $\Fbold_*$. By Lemma \ref{lem:stability-of-branched-covering}, there is a small neighborhood $\Ucalbold$ of $\Fbold_*$ and some point $x \in \C$ such that $x$ belongs in the interior of $\Ubold_-(\Fbold)$ and does not collide with $\CV\left(\Fbold^T\right)$ for all $\Fbold \in \Ucalbold$. Moreover, $\Fbold^{-S}(x)$ moves holomorphically with $\Fbold \in \Ucalbold$ for all $S \leq T$. 
    
    Given $Q,S \in \Tbold_{>0}$, if $Q<S \leq T$, then $\Fbold^{-S}(x)$ is disjoint from $\Fbold^{-Q}(x)$ because every point is mapped by $\Fbold^S$ and $\Fbold^Q$ to different tiles of the zeroth renormalization tiling of $\Fbold$. Hence, $\bigcup_{S \leq T} \Fbold^{-S}(x)$ moves holomorphically and equivariantly with $\Fbold \in \Ucalbold$. By the $\lambda$-lemma, this holomorphic motion extends to the closure. Then, by Corollaries \ref{cor:triviality-of-comb-class} (2) and \ref{cor:accumulation-to-bdry}, $\Esc_{\leq T}(\Fbold)$ has no interior and moves holomorphically and equivariantly over $\Ucalbold$.

    Let us show that the motion $\tau$ of $\Esc_{\leq T}(\Fbold)$ obtained above is independent of $x$. Let us pick another point $y=y(\Fbold) \in \C \backslash \CV(\Fbold)$ depending holomorphically on $\Fbold \in \Ucalbold$. By shrinking $\Ucalbold$, we can connect $x$ and $y$ by a simple arc $l=l(\Fbold)$ which is surrounded by an annulus $A = A(\Fbold) \subset \C \backslash \CV(\Fbold)$. Every preimage of $l$ under $\Fbold^T$ is separated from $\Esc_{\leq T}(\Fbold)$ by a conformal preimage of $A$. Therefore, any sequence of preimages of $l$ under $\Fbold^T$ which accumulates at a point in $\Esc_{\leq T}(\Fbold)$ necessarily shrinks in diameter. As a result, the holomorphic motion $\tau$ coincides with the motion of $\overline{\Fbold^{-T}(y(\Fbold))}$.

    Finally, let us show that the equivariant holomorphic motion $\tau$ of $\Esc_{\leq T}(\Fbold)$ over $\Ucalbold$ is unique. Suppose there is another equivariant holomorphic motion $\tau'$ of $\Esc_{\leq T}(\Fbold)$. Pick any $S \in \Tbold_{>0}$ where $S<T$ and consider the motion $y(\Fbold)$ of a point in $\Esc_{\leq S}(\Fbold)$ induced by $\tau'$. By equivariance, $\Fbold^{-(T-S)}(y(\Fbold))$ moves holomorphically by $\tau'$. However, since $\Esc_{\leq T-S} (\Fbold)$ is contained in the closure of $\Fbold^{-(T-S)}(y(\Fbold))$, then $\tau$ and $\tau'$ coincide on $\Esc_{\leq T-S}(\Fbold)$ for all $S \in \Tbold_{>0}$. By Corollary \ref{cor:triviality-of-comb-class} (3), $\tau \equiv \tau'$. 
\end{proof}

\begin{definition}
\label{def:conformal-motion}
    Consider a holomorphic motion $\{ \phi_{\lambda}: E \to \C\}_{\lambda \in X}$ of a set $E \subset \C$ parametrized by some complex manifold $X$. We say that $\phi_\lambda$ is a \emph{conformal motion} of $E$ if for all $\lambda$, there exists a global quasiconformal extension $\phi_\lambda: \C \to \C$ that is conformal ($\bar{\partial} \phi_\lambda = 0$) almost everywhere on $E$.
\end{definition}

\begin{theorem}
\label{thm:finite-esc-set}
    For every $\Fbold \in \manibold$, $\Esc_{<\infty}(\Fbold)$ has empty interior and supports no invariant line field. For every $P \in \Tbold_{>0}$, on every connected component of the open set $\{ \Fbold \in \manibold \: : \: 0 \not\in \Esc_{\leq P}(\Fbold) \}$, there is a unique equivariant holomorphic motion of $\Esc_{\leq P}$, and this motion is conformal.
\end{theorem}

\begin{proof}
    Fix $P \in \Tbold_{>0}$ and consider the set $\boldsymbol{\Omega}_P := \{\Fbold \in \manibold \: : \: 0 \not\in \Esc_{\leq P}(\Fbold)\}$. If $P<T$, then the neighborhood $\Ucalbold$ of $\Fbold_*$ from Lemma \ref{lem:hol-mot-finite-esc} is contained in $\boldsymbol{\Omega}_P$. Else, if $P \geq T$, then $\Fbold \in \boldsymbol{\Omega}_P \cap \Ucalbold$ if and only if $\Fbold^{P-T}(0) \not\in \Esc_{\leq T}(\Fbold)$, which is an open condition because $\Esc_{\leq T}$ moves holomorphically over $\Ucalbold$. Therefore, $\boldsymbol{\Omega}_P \cap \Ucalbold$ is open for all $P$.

    If $\Fbold \in \boldsymbol{\Omega}_P \cap \Ucalbold$, we can obtain the unique equivariant holomorphic motion of $\Esc_{\leq P}$ by pulling back the holomorphic motion of $\Esc_{\leq T}$ via $\Fbold^{P-T}$. 
    Otherwise, we can pick a sufficiently large $n \in \N$ such that $\Fbold_{-n}$ is in $\Ucalbold$. 
    Clearly, $\Fbold \in \boldsymbol{\Omega}_P$ if and only if $\Fbold_{-n} \in \boldsymbol{\Omega}_{\tbold^{n} P}$, so $\boldsymbol{\Omega}_P$ is always an open subset of $\manibold$ on which $\Esc_{\leq P}(\Fbold)$ moves holomorphically and equivariantly. 
    The dilatation of the motion of $\Esc_{\leq T}(\Fbold)$ over $\Ucalbold$ goes to one as $\Fbold \to \Fbold_*$. Therefore, for every $\Fbold \in \boldsymbol{\Omega}_P$, we can take an arbitrarily high $n$ to ensure that the dilatation of the motion of $\Esc_{\leq T}$ at $\Fbold_{-n}$, and hence that of the motion of $\Esc_{\leq \tbold^{-n}T}$ at $\Fbold$ as well, are arbitrarily small. Pulling back via $\Fbold^{P-\tbold^{-n}T}$ does not affect the dilatation, so the motion of $\Esc_{\leq P}$ is indeed conformal over $\boldsymbol{\Omega}_P$.

    By Corollaries \ref{cor:triviality-of-comb-class} (2) and \ref{cor:no-ilf-fixed-pt}, $\Esc_{\leq \tbold^{n}P}(\Fbold_{-n})$ has empty interior and supports no invariant line field of $\Fbold_{-n}$. Therefore, $\Esc_{\leq P}(\Fbold)$ also has empty interior and supports no invariant line field of $\Fbold$.
\end{proof}

\subsection{The infinite-time escaping set}

For $R>0$ and $\Fbold \in \manibold$, define 
\[
\Jcas_R(\Fbold) := \left\{ z \in \C \backslash \Esc_{<\infty}(\Fbold) \: : \: \left|\Fbold^P(z)\right| \geq R \text{ for all } P \in \Tbold \right\}.
\]
The forward orbit of every point in $\Jcas(\Fbold) \cap \Esc_{\infty}(\Fbold)$ is eventually contained in $\Jcas_R(\Fbold)$. The following lemma is inspired by \cite{R09}.

\begin{lemma}
\label{lem:hol-mot-inf}
    For every $\Fbold$ on a neighborhood $\Ucalbold \subset \manibold_{\textnormal{loc}}$ of $\Fbold_*$, there exists a totally disconnected subset $\Lambda(\Fbold)$ of $\C \backslash \Esc_{<\infty}(\Fbold)$ with the following properties.
    \begin{enumerate}[label=\textnormal{(\arabic*)}]
        \item $\Lambda(\Fbold)$ is forward invariant under $\Fbold^{\geq 0}$.
        \item There is a unique equivariant holomorphic motion of $\Lambda$ over $\Ucalbold$.
        \item There exists some $R>1$ such that $\Lambda(\Fbold)$ contains $\Jcas_R(\Fbold)$.
    \end{enumerate}
\end{lemma}

\begin{proof}
    In the dynamical plane of $\Fbold_*$, every point in the forward orbit of a point in $\Jcas_R(\Fbold_*)$ must be contained in a wake of sufficiently low generation in order to avoid the disk $\D_R:= \{|z|<R\}$. We consider all such points and define $\Lambda(\Fbold_*)$. In the proof below, we apply the motion of the finite-time escaping set from the previous subsection to show that $\Lambda(\Fbold)$ can be defined naturally via a unique holomorphic motion. The proof will be broken down into four steps.
    \vspace{0.1in}
    
    \noindent \textbf{Step 1:} Construct truncated wakes which move holomorphically.

    Let $T:= \min\{(0,1,0),(0,0,1)\}$. Let us enumerate primary wakes of $\Fbold_*$ of generation at most $T$ by $\{ \Wbold_i \}_{i \in I}$ for some countable index set $I$. Denote the generation of each wake $\Wbold_i$ by $P_i$. 

    Let us pick a sufficiently small $r>0$ such that $\cup_{i\in I} \Wbold_i$ is contained in the domain $\Vbold := \C \backslash \overline{\D_r}$. 
    For every $i \in I$, consider the truncated wake 
    \[
    \hat{\Wbold}_i := \Wbold_i \cap \Fbold_*^{-P_i}(\Vbold)
    \]
    obtained by removing from $\Wbold_i$ a small neighborhood of the critical point $C_{P_i}$ that gets mapped to $\overline{\D_r}$.

    For each $\bullet \in \{0,\infty\}$, there exists a unique point $z^\bullet$ on the intersection of $\partial \Vbold$ and the zero ray $\Rbold^\bullet$ such that the external ray segment 
    \[
    \hat{\Rbold}^\bullet := (\infty,z^\bullet) \subset \Rbold^\bullet
    \]
    is contained in $\Vbold$. The ray segments $\hat{\Rbold}^0$ and $\hat{\Rbold}^\infty$ are contained in $\Esc_{\leq Q}(\Fbold_*)$ where $Q$ is the maximum of the escaping times of $z^0$ and $z^\infty$. By Theorem \ref{thm:finite-esc-set}, the $Q$\textsuperscript{th} escaping set $\Esc_{\leq Q}$ moves holomorphically and equivariantly on a small neighborhood $\Ucalbold$ of $\Fbold_*$. By the $\lambda$-lemma, such a motion induces a holomorphic motion of $\hat{\Rbold}^0(\Fbold) \cup \hat{\Rbold}^\infty(\Fbold) \cup \partial \Vbold(\Fbold)$, which, by shrinking $\Ucalbold$ if necessary, can be assumed to not collide with $\CV\left(\Fbold^T\right)$. This allows us to pull back via $\Fbold^P$ for all $P \leq T$ and further extend this motion to a holomorphic motion of
    \[
        \hat{\Rbold}^0(\Fbold) \cup \hat{\Rbold}^\infty(\Fbold) \cup \partial \Vbold(\Fbold) \cup \bigcup_{i \in I} \partial \hat{\Wbold}_i(\Fbold)
    \]
    that is equivariant on $\partial \hat{\Wbold}_i(\Fbold)$ with respect to $\Fbold^{P_i}$ for every $i \in I$. By $\lambda$-lemma, this motion can again be extended to a holomorphic motion $\Phi_0$ on the whole plane that is equivariant with respect to $\Fbold^{P_i}$ on $\partial \hat{\Wbold}_i(\Fbold)$ for every $i \in I$.
    \vspace{0.1in}
    
    \noindent \textbf{Step 2:} Construct $\Lambda$ which moves holomorphically and equivariantly.

    Consider $\Vbold_0(\Fbold):= \bigcup_{i\in I} \hat{\Wbold}_i(\Fbold)$ and define the holomorphic map 
    \[
        \hat{\Fbold}: \Vbold_0(\Fbold) \to \Vbold(\Fbold), \quad \hat{\Fbold}(z) = \Fbold^{P_i}(z) \text{ for } z\in \hat{\Wbold}_i(\Fbold).
    \]
    This map satisfies a Markov-like property that $\Vbold_0(\Fbold) \subset \Vbold(\Fbold)$ and $\hat{\Fbold}$ sends every connected component of $\Vbold_0(\Fbold)$ univalently onto a dense subset of $\Vbold(\Fbold)$. Note that $\hat{\Fbold}_*$ coincides with the map defined in (\ref{eqn:F-tilde}).
    
    Consider the non-escaping set $\Lambda(\Fbold)$ of $\hat{\Fbold}$ which is defined by
    \[
        \Lambda(\Fbold) := \bigcap_{n \geq 0} \Vbold_{-n}(\Fbold) \quad \text{where} \quad \Vbold_{-n}(\Fbold):= \hat{\Fbold}^{-n}(\Vbold_0(\Fbold)). 
    \]
    Clearly, $\Lambda(\Fbold)$ is non-empty and forward invariant under $\Fbold^{\geq 0}$. For $\Fbold=\Fbold_*$, the set $\Lambda(\Fbold_*)$ is a subset of the set $\mathfrak{X}$ defined in \S\ref{ss:structure-esc-set-and-X} and is totally disconnected.

    Let us treat the holomorphic motion $\Phi_0 = \Phi_0(\Fbold)$ discussed in Step 1 as a map from the dynamical plane of $\Fbold_*$ to the dynamical plane of $\Fbold$. We will apply the pullback argument to $\Phi_0$ as follows. For $n \geq 0$, let us inductively define the lift of $\Phi_n$ to be 
    \[
        \Phi_{n+1} := \begin{cases}
            \Phi_n & \text{ on } \C \backslash \Vbold_{-n}(\Fbold_*), \\
            \left( \hat{\Fbold}|_{\hat{\Wbold}_i(\Fbold)} \right)^{-1} \circ \Phi_n \circ \hat{\Fbold}_* & \text{ on } \Vbold_{-n}(\Fbold_*) \cap \hat{\Wbold}_i(\Fbold_*) \text{ for each } i \in I.
        \end{cases}
    \]
    By equivariance, for all $n$, $\Phi_n$ is quasiconformal on $\C$ with uniformly bounded dilatation and it eventually stabilizes at every point outside of $\Lambda(\Fbold_*)$. Since $\Lambda(\Fbold_*)$ has no interior, $\Phi_n$ converges in subsequence to a limiting holomorphic motion $\Phi$ which is equivariant on $\Lambda(\Fbold)$.
    \vspace{0.1in}
    
    \noindent \textbf{Step 3:} Show that the equivariant holomorphic motion of $\Lambda$ is unique.

    Suppose $\Psi$ is another holomorphic motion of $\Lambda(\Fbold)$ on some small neighborhood $\Ucalbold \subset \manibold_{\textnormal{loc}}$ of $\Fbold_*$. We will use the notation $\Psi_\Fbold(x)$ to highlight the dependence of $\Fbold$. Let us pick any point $x \in \Lambda(\Fbold_*)$. By Proposition \ref{prop:address}, there is some $(i_0,i_1,\ldots) \in I^\N$ such that $x$ is the unique point with address $(i_0,i_1,\ldots)$, that is, $\hat{\Fbold}_*^n(x)$ lies in the truncated wake $\hat{\Wbold}_{i_n}(\Fbold_*)$ for all $n$. 
    
    Suppose for a contradiction that $\Psi_\Fbold(x)$ and $\Phi_\Fbold(x)$ are distinct. Then, the address of $\Psi_\Fbold(x)$ is not equal to $(i_0,i_1,\ldots)$ and, in particular, there is some $n \in \N$ such that $\hat{\Fbold}^n (\Psi_\Fbold(x))$ lies in a truncated wake other than $\hat{\Wbold}_{i_n}(\Fbold)$. Since the boundary of $\hat{\Wbold}_{i_n}(\Fbold)$ moves holomorphically and equivariantly, there exists some $\Gbold \in \manibold_{\textnormal{loc}}$ sufficiently close to $\Fbold_*$ such that $x'_n:= \hat{\Gbold}^n \left( \Psi_\Gbold(x) \right)$ is on the boundary of $\hat{\Wbold}_{i_n}(\Gbold)$. Then, the image $y'_n := \Gbold^{P_{i_n}}(x'_n)$ would lie on $\hat{\Rbold}^0(\Gbold) \cup \hat{\Rbold}^\infty(\Gbold) \cup \partial \Vbold(\Gbold)$, which is disjoint from $\Lambda(\Gbold)$. However, due to forward invariance, $y'_n$ must be contained in $\Lambda(\Gbold)$, hence a contradiction.
    \vspace{0.1in}
    
    \noindent \textbf{Step 4:} Show that $\Lambda(\Fbold)$ contains $\Jcas_R(\Fbold)$ for some $R> 0$ independent of $\Fbold \in \Ucalbold$.
    
    It suffices to find $R$ such that for all $\Fbold \in \Ucalbold$, every point outside of $\Esc_{<\infty}(\Fbold) \cup \Lambda(\Fbold)$ will be sent into the disk $\D_{R}$ by $\Fbold^P$ for some $P \in \Tbold$.

    Let us recall the renormalization tiling $\Deltabold_n(\Fbold)$ defined in \S\ref{ss:crit-periodic-points}. In the dynamical plane of $\Fbold_*$, there exists some sufficiently large $N \in \N$ such that all primary wakes rooted at critical points located in $\Delta_0(0,\Fbold_*)\cup\Delta_0(1,\Fbold_*)$ are contained in the tile $\Delta_{-N}(i, \Fbold_*)$ for some $i \in \{0,1\}$. Then, every wake of generation greater than $T$ is contained in the tiling $\Deltabold_{-N}(\Fbold_*)$. In particular, $\overline{\Vbold_0(\Fbold_*)}$ is disjoint from $\Deltabold_{-N}(\Fbold_*)$.
    
    By shrinking $\Ucalbold$ if needed, the tiling $\Deltabold_{-N}(\Fbold)$ moves holomorphically and equivariantly over $\Ucalbold$ and always contains $\C \backslash \overline{\Vbold_0(\Fbold)}$. Therefore, for all $\Fbold \in \Ucalbold$, every point outside of $\Esc_{<\infty}(\Fbold) \cup \Lambda(\Fbold)$ is eventually mapped to a point in $\C \backslash \overline{\Vbold_0(\Fbold)}$, which is eventually mapped to another point in $\Fbold^{(-N,0,1)}(\Delta_{-N}(0,\Fbold)) \cup \Fbold^{(-N,1,0)}(\Delta_{-N}(1, \Fbold))$, which is contained in the disk $\D_R$ for some large $R>0$ independent of $\Fbold$.
\end{proof}

\begin{theorem}
\label{thm:inf-esc-set}
    For every $\Fbold \in \manibold$, $\Esc_\infty(\Fbold)$ is a totally disconnected subset of $\Jcas(\Fbold)$ and supports no invariant line field. Moreover, on every connected component of the interior of $\{ \Fbold \in \manibold \: : \: 0 \not\in \Esc_\infty(\Fbold) \}$, there is a unique equivariant holomorphic motion of $\Esc_\infty(\Fbold)$, and this motion is conformal.
\end{theorem}

\begin{proof}
    Let $\Ucalbold$, $\Lambda$, and $R$ be from the previous lemma. For every $\Fbold \in \manibold$, there is some sufficiently large $n \in \N$ such that the $n$\textsuperscript{th} anti-renormalization $\Fbold_{-n}$ lies in $\Ucalbold$. Since $\Fbold^P = A_*^{-n} \circ \Fbold_{-n}^{P/\tbold^{-n}} \circ A_*^{n}$ for all $P \in \Tbold$, the set 
    \[
    \Lambda_{-n}(\Fbold) := A_*^{-n}(\Lambda(\Fbold_{-n}))
    \]
    is forward invariant, contains $\Jcas_{|\mu_*|^{-n} R}(\Fbold)$, and admits a unique equivariant holomorphic motion $\Phi_{-n}$ over $\Rboldstar^{n}(\Ucalbold)$. The dilatation of $\Phi_{-n}$ near $\Fbold$ can be made arbitrarily close to one by choosing $\Fbold_{-n}$ arbitrarily close to $\Fbold_*$, or equivalently, $n$ to be an arbitraily large. In particular, there is a unique equivariant holomorphic motion of $\Esc_\infty(\Fbold) \cap \Lambda_{-n}(\Fbold)$ and its dilatation near $\Fbold$ tends to one as $n \to \infty$.

    Every point in $\Esc_\infty(\Fbold)$ is eventually mapped to $\Esc_\infty(\Fbold) \cap \Lambda_{-n}(\Fbold)$. Since $\Lambda_{-n}(\Fbold)$ is totally disconnected, then so is $\Esc_\infty(\Fbold)$. To show that $\Esc_{\infty}(\Fbold)$ is in the Julia set, suppose for a contradiction that $\Esc_{\infty}(\Fbold)$ contains a point $x$ in the Fatou set. By normality, points sufficiently close to $x$ are also attracted to $\infty$, which contradicts the total disconnectivity of $\Esc_{\infty}(\Fbold)$.
    
    On a component $\boldsymbol{\Omega}$ of the interior of $\{ \Fbold \in \manibold \: : \: 0 \not\in \Esc_\infty(\Fbold) \}$, for $\Fbold \in \boldsymbol{\Omega}$, we can extend the motion $\Phi_{-n}$ by iteratively pulling back the holomorphic motion of $\Esc_\infty(\Fbold) \cap \Lambda_{-n}(\Fbold)$, yielding a unique equivariant holomorphic motion $\widetilde{\Phi}_{-n}$ of $\Esc_\infty(\Fbold)$. Since we are pulling back by a holomorphic map, the dilatation of $\widetilde{\Phi}_{-n}$ is equal to that of $\Phi_{-n}$. By the uniqueness of the motion, $\widetilde{\Phi}=\widetilde{\Phi}_{-n}$ is independent of $n$. Moreover, since the dilatation tends to one as $n \to \infty$, then $\widetilde{\Phi}$ is a conformal motion of $\Esc_\infty$.

    Lastly, suppose for a contradiction that $\Esc_\infty(\Gbold)$ supports an invariant line field $\boldsymbol{\mu}$ of some $\Gbold \in \manibold$. Since $\Esc_\infty(\Fbold) \cap \Lambda_{-n}(\Fbold)$ moves holomorphically over a neighborhood of $\Fbold_*$ containing $\Gbold$ for some sufficiently high $n$, then there is a quasiconformal map $\phi: \C \to \C$ which conformal almost everywhere on $\Esc_\infty(\Fbold_*) \cap \Lambda_{-n}(\Fbold)$ and conjugates $\Fbold_*|_{\Esc_\infty(\Fbold_*) \cap  \Lambda_{-n}(\Fbold_*)}$ to $\Gbold|_{\Esc_\infty(\Gbold) \cap \Lambda_{-n}(\Gbold)}$. Consider $\phi^* \boldsymbol{\mu}$ on $\Esc_\infty(\Fbold_*) \cap \Lambda_{-n}(\Fbold)$ and pull it back via $\Fbold_*$ to obtain an $\Fbold_*$-invariant Beltrami differential $\boldsymbol{\mu}'$ supported on $\Esc_\infty(\Fbold_*)$. Then, $\boldsymbol{\mu}'$ would be an invariant line field of $\Fbold_*$ supported on $\Esc_\infty(\Fbold_*)$, which is impossible due to Corollary \ref{cor:no-ilf-fixed-pt}.
\end{proof}

\section{Fatou-Julia theory}
\label{sec:fatou-julia}

Throughout this section, we will consider a single cascade $\Fbold$ in $\manibold$ and describe the properties of the Fatou set $\Fcas(\Fbold)$ and the Julia set $\Jcas(\Fbold)$, mirroring the classical Fatou-Julia theory for rational maps.

\subsection{The Julia set}

In the study of dynamics of transcendental entire functions $g: \C \to \C$, a fundamental yet highly non-trivial result is the non-emptiness of the escaping set $I(g)$ of $g$ \cite{E89,D98}, which implies that the Julia set of $g$ is the closure of the boundary of $I(g)$. Similarly, we have the following.

\begin{proposition}
\label{prop:closure-of-esc}
    The finite-time escaping set $\Esc_{<\infty}(\Fbold)$ is non-empty and 
    \[
    \Jcas(\Fbold) = \overline{\Esc_{<\infty}(\Fbold)}.
    \]
\end{proposition}

\begin{proof}
    By Theorem \ref{thm:finite-esc-set}, there exist some small $P \in \Tbold_{>0}$ and some open neighborhood $\Ucalbold \subset \manibold$ of $\Fbold_*$ containing $\Fbold$ in which the $P$\textsuperscript{th} escaping set moves holomorphically. Therefore, $\Esc_{\leq P}(\Fbold)$ is clearly non-empty. 

    Consider any open disk $D \subset \C$ disjoint from $\Esc_{<\infty}(\Fbold)$. By Montel's theorem, since $\Esc_{<\infty}(\Fbold)$ contains more than two points, then $\{\Fbold^P: D \to \C \backslash \Esc_{<\infty}(\Fbold)\}_P$ forms a normal family. Thus, such a disk $D$ is necessarily contained in $\Fcas(\Fbold)$. In particular, any open disk centered at a point in $\Jcas(\Fbold)$ must intersect $\Esc_{<\infty}(\Fbold)$.
\end{proof}

This implies many classical properties of the Julia set.

\begin{corollary}
\label{cor:accumulation-julia}
    For any $x \in \C$, the set of accumulation points of $\bigcup_{P \in \Tbold} \Fbold^{-P}(x)$ is contained in $\Jcas(\Fbold)$.
\end{corollary}

\begin{proof}
    This follows from Corollary \ref{cor:accumulation-to-bdry} and Proposition \ref{prop:closure-of-esc}.
\end{proof}

\begin{corollary}
\label{cor:no-interior}
    Either $\Jcas(\Fbold) = \C$ or $\Jcas(\Fbold)$ has no interior. 
\end{corollary}

\begin{proof}
    Suppose $\Jcas(\Fbold)$ contains an open disk $B$. Proposition \ref{prop:closure-of-esc} tells us that $\Esc_{<\infty}(\Fbold) \cap B$ is dense in $B$. By Lemma \ref{lem:transitivity}, there is some $P \in \Tbold_{>0}$ such that $\Fbold^P(B \backslash \Esc_{\leq P}(\Fbold))$ is dense in $\C$. Thus, $\Jcas(\Fbold)$ is the whole plane.
\end{proof} 

Given a periodic point $x$ of period $P$ of some $\Fbold$, we say that $x$ is \emph{superattracting / attracting / parabolic / Siegel / Cremer / repelling} if $x$ is a \emph{superattracting / attracting / parabolic / Siegel / Cremer / repelling} fixed point of $\Fbold^P$.

\begin{theorem}
\label{thm:density-of-repelling}
    Repelling periodic points of $\Fbold$ are dense in $\Jcas(\Fbold)$.
\end{theorem}

\begin{proof}
    Pick a small disk neighborhood $D=\D(y,\varepsilon)$ of a point $y$ in $\Jcas(\Fbold)$. We will show that there exists a repelling periodic point in $D$.
    
    Pick a repelling periodic point $x$ of $\Fbold$, and let $P \in \Tbold_{>0}$ be its period. There exists a conformal map $\phi: (K,x) \to (\D,0)$ on a small disk neighborhood $K$ of $x$ such that $\phi \circ \Fbold^P \circ \phi^{-1}$ is an expanding linear map. By Corollary \ref{cor:accumulation-julia}, there exists a point $x_{-R}$ in $\Fbold^{-R}(x) \cap D$ for some $R \in \Tbold_{>0}$. We then apply Corollary \ref{cor:accumulation-julia} again to obtain a point $z_0$ in $\Fbold^{-Q}(x_{-R}) \cap K$ for some $Q\in \Tbold_{>0}$. We can assume $z_0 \neq x$ without loss of generality.
    
    Pick a disk $V_0 \subset K$ around $z_0$ small enough such that $\Fbold^Q(V_0)$ is contained in $D$ and that $\Fbold^{Q+R}|_{V_0}$ is either univalent or a covering map branched only at $z_0$. Let $V := \Fbold^{Q+R}(V_0)$ be the image, which can be assumed to be disjoint from $V_0$. Consider the following backward orbit under $\Fbold^P|_K$:
    \[
        \ldots \to (V_{-3}, z_{-3}) \to (V_{-2}, z_{-2}) \to (V_{-1}, z_{-1}) \to (V_0, z_0).
    \]
    The pointed disks $(V_{-n}, z_{-n})$ converge to $x$ as $n \to \infty$, so there exists some high $N \in \N$ such that $V_{-N}$ is compactly contained in $V$.

    We have constructed a map $g := \Fbold^{NP+Q+R}: V_{-N} \to V$ that is either conformal or a covering map branched at the point $z_{-N}$, in which case the image $g(z_{-N})$ is outside of $V_{-N}$. Select a connected component $W$ of the preimage of $V_{-N}$ under this map, then $g^2: W \to V$ is a conformal map and $W$ is compactly contained in $V$. By Schwarz Lemma, $W$ contains a unique repelling fixed point $z$ of $g^2$. Then, $\Fbold^{NP+Q}(z)$ is the desired repelling periodic point of $\Fbold$ contained in $D$.
\end{proof}

\subsection{The postcritical set}

The \emph{postcritical set} of $\Fbold$ is
    \[
        \Pcas(\Fbold) := \overline{ \{ \Fbold^P(0) \: : \: P \in \Tbold \} }.
    \]
This is the smallest forward invariant closed subset of $\C$ that satisfies the property that for all $P \in \Tbold_{>0}$,
    \[
        \Fbold^P : \Dom \left( \Fbold^P \right) \backslash \Fbold^{-P}\left(\Pcas(\Fbold)\right) \to \C \backslash \Pcas(\Fbold)
    \]
is an unbranched covering map which is a local isometry with respect to the hyperbolic metrics.

\begin{proposition}
\label{prop:per-pt-poset}
    Suppose $\Fbold$ admits a periodic point $x$ of some period $P$.
    \begin{enumerate}[label=\textnormal{(\arabic*)}]
        \item If $x$ is attracting or parabolic, then the critical orbit $\{\Fbold^T(0)\}_{T \in \Tbold}$ converges to the periodic orbit $\textnormal{orb}_x^P(\Fbold)$.
        \item If $x$ is Cremer, then $x \in \Pcas(\Fbold)$.
        \item If $x$ is Siegel, then the boundary of the Siegel disk of $\Fbold^P$ centered at $x$ is contained in $\Pcas(\Fbold)$.
    \end{enumerate}
\end{proposition}

The proof is an adaptation of classical proofs for rational maps.

\begin{proof}
    Suppose $x$ is attracting. Let $A$ be the component of the Fatou set containing $x$, the immediate attracting basin of $x$.
    By Proposition \ref{prop:simply-conn}, $A$ is conformally equivalent to the unit disk $\D$. 
    Since $\Fbold^P$ is $\sigma$-proper, the map $\Fbold^P: A \to A$ is a branched covering map of the disk $A$ onto itself.
    By Denjoy-Wolff Theorem \cite[Theorem 5.2]{M06}, all orbits under $\Fbold^P|_A$ converge towards $x$.
    The basin $A$ has to contain a critical of $\Fbold^P$ because if otherwise, by Schwarz Lemma, $\Fbold^P: A \to A$ would have to be a conformal isomorphism without any attracting fixed point.
    Therefore, by Lemma \ref{lem:critical-pts}, there exists some $T \in \Tbold$ such that $\Fbold^T(0)$ is in $A$ and $\Fbold^{T+nP}(0)\to x$ as $n\to\infty$. Item (1) for the attracting case follows.

    Suppose $x$ is parabolic. The argument is similar to the previous paragraph. Denote by $A$ an attracting parabolic basin containing $x$ on its boundary. 
    Standard Fatou coordinate analysis \cite{M06,Shi98} tells us that, near $x$, the quotient of $A$ by $\Fbold^P$ is conformally isomorphic to the cylinder $\C/\Z$. 
    Suppose for a contradiction that $A$ does not contain any points in the critical orbit.
    Then, $\Fbold^P : A \to A$ would be a conformal isomorphism of a conformal disk and there would be a conformal conjugacy $A \to \D$ between $\Fbold^P$ and a non-elliptic M\"obius transformation $M: \D \to \D$. 
    However, the quotient of $\D$ by $M$ is conformally isomorphic to $\{r< |z|<1\}$ for some $0\leq r < 1$, which is not conformally equivalent to the cylinder. 
    This gives us (1) for the parabolic case.
    
    Suppose $x$ is not in $\Pcas(\Fbold)$ and is not repelling. From (1), $x$ is either Cremer or Siegel. Let us first prove (2) by showing that $x$ must be Siegel. For all $T \in \Tbold$, let us denote by $D_T$ the connected component of $\Dom \left( \Fbold^T \right) \backslash \Fbold^{-T}\left(\Pcas(\Fbold)\right)$ containing $x$. Suppose first that $D_P$ is properly contained in $D_0$. Then, $\Fbold^P : D_P \to D_0$ is strictly expanding with respect to the hyperbolic metric of $D_0$, which implies that $x$ must be repelling. Suppose instead $D_P = D_0$. Then, $\left\{\Fbold^{nP}|_{D_0}\right\}_{n \in \N}$ is a normal family of automorphisms of a hyperbolic Riemann surface. By Denjoy-Wolff, the fixed point $x$ must be Siegel.

    Denote by $Z$ the Siegel disk centered at $x$. If there exists some minimal $T \in \Tbold$ where $\Fbold^T(0)$ intersects $Z$, then the intersection $\Pcas(\Fbold) \cap Z$ is a single $\Fbold^P$-invariant curve on $Z$. Suppose for a contradiction that $\C \backslash \Pcas(\Fbold)$ intersects the boundary $\partial Z$. Then, a component $E_0$ of $\C \backslash \Pcas(\Fbold)$ contains some neighborhood of $\partial Z$. For $n \in \N$, let $E_{nP}$ be the connected component of $\Dom \left(\Fbold^{nP}\right) \backslash \Fbold^{-nP}(\Pcas(\Fbold))$ containing $E_0 \cap Z_0$. There are again two cases. If $E_{P} = E_{0}$, then $\left\{\Fbold^{nP}|_{E_0}\right\}_{n \in \N}$ forms a normal family and $E_0$ must be contained in the Fatou set, which is a contradiction. If $E_P$ is a proper subset of $E_0$, then $\Fbold^P : E_P \to E_0$ is strictly expanding with respect to the hyperbolic metric of $E_0$, which would contradict the fact that $\Fbold^P$ restricts to a self diffeomorphism of any invariant curve in $Z \cap E_0$.
\end{proof}

For any tangent vector $v$ at a point $z$ in $\C \backslash \Pcas(\Fbold)$, denote by $\|v\|$ the norm of $v$ with respect to the hyperbolic metric of $\C \backslash \Pcas(\Fbold)$. If $z \in \Pcas$, we set $\|v\| = \infty$.

\begin{lemma}[Julia expansion]
\label{lem:julia-expansion}
    For every point $z$ in $\Jcas(\Fbold) \backslash \Esc_{<\infty}(\Fbold)$,
    \[
    \left\|\left(\Fbold^P\right)'(z)\right\| \to \infty \quad \text{as} \quad P \to \infty.
    \]
\end{lemma}

\begin{proof}
    Let us fix a point $z \in \Jcas(\Fbold) \backslash \Esc_{<\infty}(\Fbold)$. Without loss of generality, assume that $z$ does not eventually land on $\Pcas(\Fbold)$.
    
    For any $P \in \Tbold_{>0}$, let 
    \[
    \Pcas_P := \Esc_{\leq P}(\Fbold) \cup \Fbold^{-P}(\Pcas(\Fbold)).
    \]
    The map $\Fbold^P: \C \backslash \Pcas_P \to \C \backslash \Pcas$ is a local isometry with respect to their hyperbolic metrics. By Corollary \ref{cor:accumulation-julia}, the closure of the union $\bigcup_{P \in \Tbold} \Pcas_P$ contains the Julia set, so the distance between $\Pcas_P$ and $z$ shrinks to $0$ as $P \to \infty$. Consequently, the distance $r_P$ between $z$ and $\Pcas_P$ with respect to the hyperbolic metric of $\C \backslash \Pcas$ tends to $0$ as $P \to \infty$. The inclusion map $\iota: \C \backslash \Pcas_P \to \C \backslash \Pcas$ is contracting by some factor $C(r_P)$ where $C(r) \to 0$ as $r \to 0$. Therefore, as $P \to \infty$, 
    \[
    \left\|\left(\Fbold^P\right)'(z)\right\| \geq C(r_P)^{-1} \to \infty. \qedhere
    \]
\end{proof}

Denote by $\dist_{\RS}(\cdot,\cdot)$ the spherical distance between two subsets of $\RS$.

\begin{theorem}[Measure-theoretic attractor]
\label{thm:poset-attractor}
    If $\Jcas(\Fbold)$ has no interior, then for almost every point $z$ in $\Jcas(\Fbold) \backslash \Esc_{<\infty}(\Fbold)$,
    \[
        \dist_{\RS}\left(\Fbold^P(z), \Pcas(\Fbold) \cup \{\infty\} \right) \to 0 \qquad \text{ as } P \to \infty.
    \]
\end{theorem}

In other words, as long as the Fatou set is non-empty, almost every non-escaping point in the Julia set is attracted to the postcritical set.

\begin{proof}
    Suppose for a contradiction that there exist a positive number $\varepsilon>0$ and a positive area subset $E$ of $\Jcas(\Fbold) \backslash \Esc_{<\infty}(\Fbold)$ such that for all $z \in E$, 
    \[
    \limsup_{P\to \infty} \dist_{\RS} \left(\Fbold^P(z), \Pcas(\Fbold) \cup \{\infty\} \right) \geq \varepsilon.
    \]
    Let $z$ be a Lebesgue density point of $E$. There is a sequence of power-triples $P_n$ such that $P_n \to \infty$ and $y_n := \Fbold^{P_n}(z)$ lies in the compact subset
    \[
    K := \{ z \in \C \: : \: \dist_{\RS} (z, \Pcas(\Fbold) \cup \{\infty\} ) \geq \varepsilon \}.
    \]
    For each $n \in \N$, consider the spherical ball $B_n$ of radius $\varepsilon/2$ centered at $y_n$, and let $B'_n$ be the lift of $B_n$ under $\Fbold^{P_n}$ containing $z$.

    By Lemma \ref{lem:julia-expansion}, $\left\| \left(\Fbold^{P_n}\right)'(z)\right\| \to \infty$. Since $K$ is compact and $\Fbold^{P_n}|_{B'_n}$ has bounded distortion, the disks $B'_n$ must shrink to a point. Since $z$ is a density point of $E$,
    \[
        \lim_{n\to\infty} \frac{ \text{area}(B'_n \cap E)}{\text{area}(B'_n)} = 1.
    \]
    Therefore, we also have
    \[
        \lim_{n\to\infty} \frac{ \text{area}(B_n \cap \Jcas(\Fbold) \backslash \Esc_{<\infty}(\Fbold) )}{\text{area}(B_n)} = 1.
    \]
    Since $K$ is compact, $y_n$ converges in subsequence to some point $y \in K$. Then, the ball $B$ of radius $\varepsilon/2$ centered at $y$ must have the same area as $B \cap \Jcas(\Fbold) \backslash \Esc_{<\infty}(\Fbold)$. Since $\Jcas(\Fbold)$ is closed, then the ball $B$ has to be contained in $\Jcas(\Fbold)$. This contradicts the assumption that $\Jcas(\Fbold)$ has no interior.
\end{proof}

\subsection{No wandering domains}
\label{ss:no-wandering}

We say that a connected component $U$ of the Fatou set $\Fcas(\Fbold)$ is \emph{wandering} if for all $P,Q \in \Tbold$, $\Fbold^P(U)$ is disjoint from $\Fbold^Q(U)$ whenever $P\neq Q$. If $U$ is not wandering, then it is pre-periodic, that is, there exists $P, Q_0 \in \Tbold$ such that for all $Q \geq Q_0$, $\Fbold^Q(U) = \Fbold^{P+Q}(U)$.

The absence of wandering domains for rational maps was proven by Sullivan \cite{S85}, in which he pioneered the use of quasiconformal mappings in the field of complex dynamics. The idea is to show that the presence of wandering domains induces an infinite-dimensional deformation space, which would contradict the finite-dimensionality of the space of rational maps of any fixed degree. The original proof was simplified in a short note by Zakeri \cite{Z02}.

\begin{theorem}
\label{thm:NWD}
    Every component of $\Fcas(\Fbold)$ is pre-periodic.
\end{theorem}

The proof below mimics the one for rational maps presented by Zakeri. It relies on the finite-dimensionality of the unstable manifold $\manibold$, which automatically comes from the compactness of the renormalization operator. We will first need a little bit of preparation.

Let us extend the unstable manifold $\manibold$ to
\[
    \widetilde{\manibold} := \{L \circ \Gbold \circ L^{-1} \: : \: \Gbold \in \manibold \text{ and } L \text{ is an invertible linear map} \}.
\]
Clearly, $\widetilde{\manibold}$ is still a finite-dimensional complex manifold.

Let us pick a sufficiently small power-triple $T \in \Tbold_{>0}$ such that $\zeta:=\Fbold^T(0)$ is non-zero and well-defined. Denote by $\mathcal{Q}_\zeta$ the vector space of quasiconformal vector fields $v=v(z)/dz$ on $\RS$ vanishing on $0$, $\zeta$, and $\infty$. We will associate any invariant Beltrami differential $\boldsymbol{\mu} \in \mathcal{B}(\Fbold)$ to a quasiconformal vector field $v \in \mathcal{Q}_\zeta$ and a tangent vector $\mathbf{w} \in T_\Fbold \widetilde{\manibold}$ as follows. 

By Proposition \ref{prop:integration-of-ILF}, there exists some sufficiently small $\varepsilon>0$ such that the $1$-parameter family $\{t \boldsymbol{\mu} \}_{t \in \D_{\varepsilon}}$ can be integrated to a unique analytic $1$-parameter family $\{ \phi_t: \C \to \C \}_{t \in \D_{\varepsilon}}$ of quasiconformal maps, where $\phi_0$ is the identity map, and that $\{\Gbold_t := \phi_t \circ \Fbold \circ \phi_t^{-1}\}_{t \in \D_{\varepsilon}}$ is a $1$-parameter holomorphic family of cascades in $\manibold$. For $t \in \D_\varepsilon$, let $L_t$ be the unique linear map sending $\Gbold^T_t(0)$ to $\zeta$, and then let $\psi_t := L_t \circ \phi_t$ and $\Fbold_t := L_t \circ \Gbold_t \circ L_t^{-1} \in \widetilde{\manibold}$. Then, $\psi_t$ is a quasiconformal conjugacy between $\Fbold$ and $\Fbold_t$ that fixes $0$, $\zeta$, and $\infty$.
    
Consider the quasiconformal vector field 
\[
v := \frac{d}{dt}\psi_t \; \bigg|_{t=0}
\]
and consider the vector $\mathbf{w} = \left\{\mathbf{w}^P\right\}_{P \in \Tbold} \in T_\Fbold \widetilde{\manibold}$ tangent to the analytic curve $\left\{\Fbold_t\right\}_{t \in \D_\varepsilon}$, which is given by
    \[
    \mathbf{w}^P := \frac{d}{dt}\Fbold_t^P \, \bigg|_{t=0} \qquad \text{for } P \in \Tbold. 
    \]
Observe that $v$ is the unique solution in $\mathcal{Q}_\zeta$ to the $\bar{\partial}$-equation $\bar{\partial} v = \boldsymbol{\mu}$. Moreover, $v$ and $\mathbf{w}$ are related by the cohomological equation
\begin{equation}
\label{eqn:cohomological-equation}
    \mathbf{w}^P(z) = v\left(\Fbold^P(z)\right) - \left(\Fbold^P\right)'(z) \cdot v(z) \qquad \text{for } z \in \Dom\left(\Fbold^P\right) \text{ and } P \in \Tbold.
\end{equation}
These two observations immediately imply that
\begin{equation}
\label{eqn:D-maps}
    D_1: \mathcal{B}(\Fbold) \to \mathcal{Q}_\zeta, \: \boldsymbol{\mu} \mapsto v \qquad \text{and} \qquad D_2: \mathcal{B}(\Fbold) \to T_\Fbold \widetilde{\manibold}, \: \boldsymbol{\mu} \mapsto \mathbf{w}
\end{equation}
are linear maps between complex vector spaces.

\begin{lemma}
\label{lem:v-and-w}
    For $\boldsymbol{\mu} \in \mathcal{B}(\Fbold)$, if $D_2(\boldsymbol{\mu}) \equiv 0$, then $D_1(\boldsymbol{\mu})(z) = 0$ for all $z \in \Jcas(\Fbold)$.
\end{lemma}

\begin{proof}
    Let $D_1(\boldsymbol{\mu})=v$ and $D_2(\boldsymbol{\mu})=\mathbf{w}$, and suppose $\mathbf{w}^P \equiv 0$ for all $P \in \Tbold$. By continuity of $v$ and Theorem \ref{thm:density-of-repelling}, it suffices to show that $v(x)=0$ for every repelling periodic point $x$ of $\Fbold$. Let $x$ be a repelling periodic point of some period $Q \in \Tbold_{>0}$. Since $\mathbf{w}^Q(x)=0$, then by (\ref{eqn:cohomological-equation}),
        \[
            v(x) = v\left(\Fbold^Q(x)\right) = \left(\Fbold^Q\right)'(z) \cdot v(x).
        \]
        Since $x$ is repelling in nature, then $v(x)=0$.
\end{proof}

\begin{proof}[Proof of Theorem \ref{thm:NWD}]
    Suppose for a contradiction that a cascade $\Fbold \in \manibold$ admits a wandering domain $U$. Denote by $\textnormal{Bel}(U)$ the vector space of $L^\infty$ Beltrami differentials $\mu$ on $U$. Let us define an inclusion map $\iota : \textnormal{Bel}(U) \hookrightarrow \mathcal{B}(\Fbold)$ by spreading. Given $\mu \in \textnormal{Bel}(U)$, $\iota(\mu)$ is supported on a subset of the grand orbit of $U$, and it is equal to $\left(\Fbold^{-P}\right)^*\mu$ on $\Fbold^P(U)$ and $\left(\Fbold^P\right)^*\mu$ on $\Fbold^{-P}(U)$ for all $P \in \Tbold$. Since $U$ is wandering, $\iota$ is a well-defined injective linear map. 

    To get a contradiction, we need the following key observation. The space $\textnormal{Bel}(U)$ contains an infinite-dimensional subspace $\textnormal{Bel}^*(U)$ of compactly supported Beltrami differentials satisfying the following property: if $\mu \in \textnormal{Bel}^*(U)$ satisfies $\mu = \bar{\partial} v$ for some quasiconformal vector field $v$ on $U$ with $v|_{\partial U} =0$, then $\mu=0$. A short constructive proof of this fact can be found in \cite[pg. 3--4]{Z02}.

    Consider the linear maps $D_1$ and $D_2$ from (\ref{eqn:D-maps}). We will now show that the linear map $D_2 \circ \iota |_{\textnormal{Bel}^*(U)}$ is injective; this will immediately contradict the finite-dimensionality of $\widetilde{\manibold}$. Suppose $\mu \in \textnormal{Bel}^*(U)$ is in the kernel of $D_2 \circ \iota$. By Lemma \ref{lem:v-and-w}, the solution $v = D_1 \circ \iota(\mu)$ to the equation $\bar{\partial} v = \iota(\mu)$ vanishes on the Julia set of $\Fbold$ and in particular on $\partial U$. Since $\mu$ is in $\textnormal{Bel}^*(U)$, then $\mu$ must be zero.
\end{proof}

\section{Hyperbolic cascades}
\label{sec:hyperbolic}

\begin{definition}
    We say that a cascade $\Fbold \in \manibold$ is \emph{hyperbolic} if $\Fbold$ admits an attracting cycle of periodic points.
\end{definition}

In this section, we will discuss the properties of hyperbolic cascades, including their existence and rigidity properties. These will be essential in the proof of the second half of Theorem \ref{main-theorem-rigidity} as well as Theorem \ref{main-theorem} (4).

\subsection{Hyperbolic cascades}

If $\Fbold$ is hyperbolic, the critical orbit $\Fbold^P(0)$ automatically converges to an attracting periodic cycle (by Proposition \ref{prop:per-pt-poset}) and so $\Fbold$ has a unique attracting periodic cycle. By Proposition \ref{prop:per-pt-poset}, the postcritical set $\Pcas(\Fbold)$ is contained in $\Fcas(\Fbold)$. By Corollary \ref{cor:no-interior} and Theorem \ref{thm:poset-attractor}, we have the following.

\begin{corollary}
\label{cor:no-leb-meas}
    If $\Fbold$ is hyperbolic, then $\Jcas(\Fbold) \backslash \Esc(\Fbold)$ has zero Lebesgue measure.
\end{corollary}

Hyperbolicity is clearly an open condition. An open subset $\boldsymbol{\Omega}$ of $\manibold$ is called a \emph{hyperbolic component} if it is a connected component of the set of hyperbolic cascades in $\manibold$.

\begin{corollary}
\label{cor:conformal-motion-julia-set}
    Consider a hyperbolic component $\boldsymbol{\Omega} \subset \manibold$. There is a unique equivariant holomorphic motion of $\Jcas(\Fbold)$ over $\Fbold \in \boldsymbol{\Omega}$, and such a motion is a conformal motion. For $\Fbold \in \boldsymbol{\Omega}$, $\Jcas(\Fbold)$ supports no invariant line field of $\Fbold$.
\end{corollary}

\begin{proof}
    For $\Fbold \in \boldsymbol{\Omega}$, the critical value $0$ is not contained in $\Esc(\Fbold)$, and so the assertion follows from Theorems \ref{thm:finite-esc-set} and \ref{thm:inf-esc-set}, Corollary \ref{cor:no-leb-meas}, and the $\lambda$-lemma.
\end{proof}

Together with Theorems \ref{thm:finite-esc-set} and \ref{thm:inf-esc-set}, this corollary completes the proof of Theorem \ref{main-theorem-rigidity}. In order to prove Theorem \ref{main-theorem}, we first need to also understand the dynamics on the Fatou set.

One feature of transcendental dynamics that distinguishes itself from rational dynamics is the emergence of wandering domains and Baker domains. We have previously shown in \S\ref{ss:no-wandering} that $\Fbold$ never admits any wandering domains. Examples of Baker domains arising from lifts of parabolic basins can be found amongst pacman renormalization cascades \cite[\S7]{DL23}. Below, we will show that Baker domains do not exist when $\Fbold$ is hyperbolic.

\begin{proposition}
\label{prop:no-baker-wandering}
    If $\Fbold \in \manibold$ is hyperbolic, the Fatou set of $\Fbold$ is equal to the basin of the unique attracting periodic cycle of $\Fbold$.
\end{proposition}

\begin{proof}
    Let $\mathcal{A}$ denote the basin of attraction of the unique attracting cycle of $\Fbold$, and suppose for a contradiction that $\Fcas(\Fbold) \backslash \mathcal{A}$ is non-empty. Let us pick a connected component $\Omega$ of $\Fcas(\Fbold) \backslash \mathcal{A}$. 

    Let us pick any point $x$ in $\Omega$. By Theorem \ref{thm:inf-esc-set}, $\Esc_\infty(\Fbold)$ is contained in the Julia set and so it is disjoint from $\Omega$. Hence, there exist some $R>1$ and some increasing sequence of times $P_0:=0$, $P_1$, $P_2$, $\ldots$ in $\Tbold$ such that $P_n \to \infty$ as $n \to \infty$ and that each of $x_n := \Fbold^{P_n}(x)$ is contained in
    \[
        K := \{z \in \C \: : \: |z|\leq R \text{ and } z \not\in \mathcal{A} \}.
    \]
    
    Let $P$ denote the period of the attracting cycle of $\Fbold$. Since $K$ is a compact subset of $\C \backslash \Pcas(\Fbold)$, the hyperbolic metric $\rho(z) dz$ of $\C \backslash \Pcas(\Fbold)$ satisfies
    \begin{equation}
        \label{eqn:pcas-rho}
        \rho(z) \asymp 1 \qquad \text{ for all } z \in K,
    \end{equation}
        and the hyperbolic distance between any point in $K$ and $\Fbold^{-P}\left(\Pcas(\Fbold)\right) \cup \Esc_{\leq P}(\Fbold)$ is uniformly bounded from above. As such, since each of $x_n$ is in $K$, then there is some constant $C>1$ such that for all $n \geq 1$, $\left\| \left( \Fbold^P \right)'(x_n) \right\| \geq C$. Let us pass to a subsequence and assume that $P_{n+1}-P_n \geq P$ for all $n\geq 1$. By chain rule,
    \begin{equation}
    \label{eqn:pcas-blowup}
        \left\| \left( \Fbold^{P_n} \right)'(x) \right\|
        \geq \prod_{k=0}^{n-1} \left\| \left( \Fbold^{P_{k+1}-P_{k}} \right)'(x_k) \right\|
        \geq C^n \to \infty \quad \text{as } n \to \infty.
    \end{equation}

    Since $\Omega$ is simply connected (Proposition \ref{prop:simply-conn}) and does not contain any critical point of $\Fbold^{\geq 0}$, then $\Fbold^{P_n}$ is univalent on $\Omega$ for all $n$. Suppose $\Omega$ contains the Euclidean disk $D:=\D(x,\varepsilon)$ for some $\varepsilon>0$. By Koebe quarter, $\Fbold^{P_n}(D)$ contains the Euclidean disk $\D(x_n, r_n)$ where
    \[
        r_n := \frac{\varepsilon}{4} \left| \left( \Fbold^{P_n} \right)'(x) \right|.
    \]
    By (\ref{eqn:pcas-rho}) and (\ref{eqn:pcas-blowup}), we have
    \[
        r_n \asymp \left| \left( \Fbold^{P_n} \right)'(x) \right| = \frac{\rho(x)}{\rho(x_n)} \left\| \left( \Fbold^{P_n} \right)'(x) \right\| \asymp \left\| \left( \Fbold^{P_n} \right)'(x) \right\| \to \infty.
    \]
    
    We have just established that $\Fbold^{P_n}(\Omega)$ contains the disk $\D(x_n,r_n)$ where $|x_n|\leq R$ and $r_n \to \infty$ as $n \to \infty$. Therefore, $\Fbold^{P_n}(\Omega)$ converges to the whole plane in the Hausdorff metric, which is absurd since the Fatou set cannot be the whole plane.
\end{proof}

\subsection{Superattracting cascades}
\label{ss:superattracting}

 We say that $\Fbold \in \manibold$ is \emph{superattracting} if $0$ is a periodic point of $\Fbold^{\geq 0}$. Superattracting cascades are clearly hyperbolic.

\begin{lemma}[Density of hyperbolicity at $\Fbold_*$]
\label{lem:superattr-reduction}
    Every neighborhood $\Ucalbold \subset \manibold_{\textnormal{loc}}$ of the renormalization fixed point $\Fbold_*$ contains a superattracting cascade.
\end{lemma}

\begin{proof}
    Suppose for a contradiction that there is a neighborhood $\Ucalbold$ of $\Fbold_*$ in which for all $\Fbold \in \Ucalbold$, we have $\Fbold^{P+Q}(0) \neq \Fbold^Q(0)$ for all $P \in \Tbold_{>0}$, $Q \in \Tbold$. By $\lambda$-lemma, this implies that the postcritical set of $\Fbold$ moves holomorphically over $\Ucalbold$. In the realm of coronas, the corresponding neighborhood $\mathcal{V} \subset \unstloc$ of $f_*$ consists of rotational coronas. By Theorem \ref{thm:weak-hyperbolicity}, $\mathcal{V}$ must lie in the stable manifold, which is impossible.

    Therefore, every neighborhood $\Ucalbold$ of $\Fbold_*$ contains some $\Gbold$ such that $\Gbold^{P+Q}(0) = \Gbold^Q(0)$ for some $P \in \Tbold_{>0}$ and $Q \in \Tbold$. If $Q =0$, then $\Gbold$ is superattracting and we are done. Hence, let us assume that $Q>0$. In this case, $\Gbold^Q(0)$ is a periodic point of period $P$, and by Proposition \ref{prop:per-pt-poset}, it must be repelling in nature.

    Consider any sufficiently small one-dimensional disk $\Ucalbold'$ about $\Gbold$ embedded in $\Ucalbold$. By implicit function theorem, every $\Fbold \in \Ucalbold'$ admits a repelling periodic point $x_\Fbold$ of period $P$ such that $x_{\Gbold} = \Gbold^Q(0)$ and $x_\Fbold$ depends holomorphically on $\Fbold$. By Corollary \ref{cor:accumulation-julia}, there exists a sequence of critical points $x_{\Fbold,n}$ of some generation $P_n$ depending holomorphically on $\Fbold \in \Ucalbold'$ such that $P_n \to \infty$ and $x_{\Fbold,n} \to x_\Fbold$ as $n \to \infty$. By Rouch\'e's theorem, for sufficiently large $n$, the number of zeros of $\Fbold^{Q+P_n}(x_{\Fbold,n}) - x_{\Fbold,n}$ as a function of $\Fbold \in \Ucalbold'$ is equal to that of $\Fbold^{Q+P_n}(x_{\Fbold,n}) - x_\Fbold$, which is at least one (e.g. $\Gbold$). Therefore, there exist some large $n \in \N$ and some $\Fbold \in \Ucalbold'$ such that $\Fbold^{Q+P_n}( x_{\Fbold,n} ) = x_{\Fbold,n}$ and so $\Fbold^{Q+P_n}(0)=0$.
\end{proof}

\begin{remark}
    With a little more effort, one can improve the proof of Lemma \ref{lem:superattr-reduction} and adapt McMullen's approach \cite{McM00} to show the existence of infinitely many disjoint baby Multibrot copies in $\manibold$ accumulating at $\Fbold_*$.
\end{remark}

\begin{lemma}
\label{lem:superattracting-variety}
    For $P \in \Tbold$, the set $\{ \Fbold \in \manibold \: : \: \Fbold^P(0)=0\}$ is a zero-dimensional analytic variety in $\manibold$.
\end{lemma}

\begin{proof}
    The equation ``$\Fbold^P(0)=0$`` surely cuts out an analytic variety in $\manibold$. Suppose for a contradiction that it has a component of dimension at least one. Then, there exists an embedded holomorphic curve $\D \to \manibold, t \mapsto \Fbold_t$ such that each $\Fbold_t$ is superattracting of period $P$. Below, we will run the pullback argument.
    
    Let $D_t$ be the immediate basin of attraction of $0$ for the cascade $\Fbold_t$. By Proposition \ref{prop:simply-conn} (or alternatively, Riemann-Hurwitz formula), $D_t$ is simply connected. Let $b_t: (D_t,0) \to (\D,0)$ be a B\"ottcher conjugacy, i.e. a Riemann mapping which conjugates $\Fbold_t^P$ with the power map $z \mapsto z^{d}$ where $d=d_0+d_\infty-1$. Observe that 
    \[
    B_t:= b_t^{-1} \circ b_0: (D_0,0) \to (D_t,0)
    \]
    conjugates $\Fbold_0^P$ with $\Fbold_t^P$. The B\"ottcher conjugacy is unique up to multiplication by some roots of unity. We can select them such that $b_t$ depends holomorphically on $t$ and so $B_0$ is the identity map on $D_0$.
    
    By Corollary \ref{cor:conformal-motion-julia-set}, the Julia set $\Jcas(\Fbold_t)$ moves conformally and equivariantly in $t$. More precisely, there exists a holomorphic family of quasiconformal maps $\phi_t : \C \to \C$ such that $\phi_t$ is conformal almost everywhere on $\Jcas(\Fbold_0)$ and conjugates $\Fbold_0|_{\Jcas(\Fbold_0)}$ and $\Fbold_t|_{\Jcas(\Fbold_t)}$. 
    
    For $r \in (0,1)$, let $E_{t}(r) := b_t^{-1}(\D_r)$ be a disk neighborhood of $0$ cut out by an equipotential. Let $\varepsilon = \frac{1}{2}$ and $\varepsilon' = \varepsilon^{d}$. Define the global quasiconformal map
    \begin{align*}
        \psi_{t,0}(z) := \begin{cases}
            \phi_t(z) & \text{ if } z \in \C \backslash \bigcup_{0\leq T < P} \Fbold_0^T(E_{t}(\varepsilon)) \\
            \Fbold_{t}^T \circ B_{t} \circ
            \left( \Fbold_0^T|_{E_0(\varepsilon')}
            \right)^{-1} & 
            \text{ if } z \in \Fbold_0^T(E_0(\varepsilon')) 
            \text{ for some } T<P \\
            \text{quasiconformal interpolation} & \text{ if otherwise}.
        \end{cases}
    \end{align*}
    On $\Jcas(\Fbold_0)$ and a neighborhood of the periodic cycle $\left\{\Fbold_0^T(0)\right\}_T$, $\psi_{t,0}$ conjugates $\Fbold_0^P$ and $\Fbold_t^P$. Inductively, we define for all $n \geq 1$ the quasiconformal map $\psi_{t,n}: \C \to \C$ by lifting $\psi_{t,n-1}$ such that 
    \[
    \Fbold_t^P \circ \psi_{t,n} =  \psi_{t,n-1} \circ \Fbold_0^P.
    \]
    The map $\psi_{t,n}$ has dilatation equal to that of $\psi_{t,0}$ and it agrees with $\psi_{t,n-1}$ on a neighborhood of $\Jcas(\Fbold_0)$ and on increasingly large part of the Fatou set $\Fcas(\Fbold_0)$. Moreover $\psi_{t,n}$ is a conformal conjugacy between $\Fbold_0^P$ and $\Fbold_t^P$ on $\Fbold^{-nP} \left( \Fbold^T (E_0(\varepsilon)) \right)$ for all $0\leq T < P$. As $n \to \infty$, $\psi_{t,n}$ stabilizes and converges to a quasiconformal map $\psi_t$ conjugating $\Fbold_0^P$ to $\Fbold_t^P$ everywhere. By Proposition \ref{prop:no-baker-wandering}, $\psi_t$ is conformal on the whole Fatou set and almost everywhere on the Julia set. By Weyl's lemma, $\psi_t$ is a linear conjugacy between $\Fbold_0$ and $\Fbold_t$. By Lemma \ref{lem:linear-conjugacy}, $\psi_t$ must be the identity map for all $t$.
\end{proof}

%%%%%%%%%%%%%%%%%%%%%%%%%%%%%%%%%%%%%%%%%%%%%%%%%%%%%%%%%%%%%%%%%%%%%%%%%%%%%%%%%%%%%%%%%%%%%%%%%%%%%%%%%%%%%%%%%%%%%%%%%%%%%%%%%%%%%%%%%%%%%%%%%%%%%%%%%%%%%%%%%%%%%%%%%%%%%%%%%%%%

\part{Conclusion}

\section{Proof of the main theorem}

In the previous section, we have proven Theorem \ref{main-theorem-rigidity}. Here, we will complete the rest of the proof of our main results.

\begin{theorem}
\label{thm:codim-one}
    Consider the corona renormalization operator $\Rstar: (\Ustar,f_*) \to (\Bstar,f_*)$ from \S\ref{sec:stab-mani} and let $\manibold$ be the associated global unstable manifold constructed in \S\ref{ss:trans-cascades}.
    Then, $\manibold$ is biholomorphic to $\C$.
\end{theorem}

\begin{proof}
    By Lemma \ref{lem:superattr-reduction}, there exists a superattracting cascade in $\manibold$ of some period $P>0$. The equation ``$\Fbold^P(0)=0$`` defines a non-empty analytic hypersurface in $\manibold$. By Lemma \ref{lem:superattracting-variety}, the dimension of $\manibold$ must be equal to one. Since $\Rboldstar$ is an automorphism of $\manibold$ admitting a unique repelling fixed point $\Fbold_*$, then clearly $\manibold$ must be isomorphic to $\C$ and $\Rboldstar: \manibold \to \manibold$ must be conformally conjugate to an expanding linear map $\C \to \C, z \mapsto \lambda z$.
\end{proof}

This theorem clearly implies Theorem \ref{main-theorem} (4). Let us also prove Corollary \ref{cor:continuity-submanifold}.

\begin{corollary}
    Consider a sufficiently small Banach neighborhood $\mathcal{N}$ of a $(d_0,d_\infty)$-critical quasicircle map $f$ with a preperiodic type rotation number $\tau$. The space $\mathcal{S}$ of maps in $\mathcal{N}$ that admit a $(d_0,d_\infty)$-critical Herman quasicircle with rotation number $\tau$ forms an analytic submanifold of $\mathcal{N}$ of codimension at most one. The Herman quasicircles of maps in $\mathcal{S}$ move holomorphically.
\end{corollary}

\begin{proof}
    Let $G$ be the Gauss map. There exists some $k \in \N$ such that $\theta := G^k(\tau)$ is a periodic type irrational. Consider the corona renormalization operator $\Rstar: (\Ustar,f_*) \to (\Bstar,f_*)$ from \S\ref{sec:stab-mani} associated to the data $(d_0,d_\infty,\theta)$. 
    
    By Lemma \ref{lem:renorm-of-any-critical rotational}, there is a compact analytic corona renormalization operator $\mathcal{R}$ on a Banach neighborhood of $f$ such that $\mathcal{R} f$ is sufficiently close to the fixed point $f_*$ of $\Rstar$, and thus it lies in the stable manifold of $f_*$. Then, the preimage $\mathcal{S}:= \renorm^{-1}(\mani^s_{\textnormal{loc}})$ is an analytic submanifold of the Banach neighborhood of $f$ consisting of perturbations of $f$ which admit a $(d_0,d_\infty)$-critical Herman quasicircle with rotation number $\tau$. By $\lambda$-lemma, the Herman quasicircles of coronas in $\mani^s_{\textnormal{loc}}$ moves holomorphically. Thus, the Herman quasicircles of maps in $\mathcal{S}$ also move holomorphically over $\mathcal{S}$.
    
    By Theorem \ref{main-theorem}, the codimension of $\mani^s_{\textnormal{loc}}$ is one, so there is an analytic function $\phi: \mathcal{U}' \to \C$ on a Banach neighborhood $\mathcal{U}'$ of $f_*$ such that $\mani^s_{\textnormal{loc}} = \phi^{-1}(0)$. Therefore, $\mathcal{S}$ is the zero set of $\phi \circ \renorm$ and so the codimension of $\mathcal{S}$ is at most one. 
\end{proof}

\section{Some remarks and conjectures}
\label{sec:conjectures}

\subsection{Structural instability}
\label{ss:structural-instability}

Consider a critical quasicircle map $f: \Hq \to \Hq$ with irrational rotation number. Recall that a Banach neighborhood $\mathcal{N}$ of $f$ is defined to be the space of unicritical holomorphic maps $g:U \to \C$ such that $g$ extends continuously to the boundary of $U$, where $U$ is a fixed small neighborhood of $\Hq$, and $g$ is close to $f$ in sup norm over $U$.

When the rotation number is the pre-periodic, Corollary \ref{cor:continuity-submanifold} and Theorem \ref{thm:qc-rigidity} tell us that the local quasiconformal conjugacy class $\mathcal{S} \subset \mathcal{N}$ of $f$ is an analytic submanifold of codimension at most one. We believe in the following conjecture.

\begin{conj}
    The local conjugacy class $\mathcal{S}$ of $f$ described above is a codimension one submanifold of the Banach neighborhood $\mathcal{N}$ of $f$. In particular, critical quasicircle maps are structurally unstable.
\end{conj}

So far, this conjecture is known to be true for:
\begin{itemize}
    \item periodic type critical quasicircle maps that are close to the associated renormalization fixed point $f_*$ (due to Theorem \ref{main-theorem}), and
    \item critical circle maps with arbitrary irrational rotation number (due to standard monotonicity properties of the rotation number).
\end{itemize}
To solve the conjecture, one needs to show the existence of a vertical direction, that is, a holomorphic vector field that is not tangent to $\mathcal{S}$. It is likely that the answer follows from an infinitesimal argument similar to unimodal maps \cite{ALM}.

\begin{figure}
    \centering
    \includegraphics[width=0.67\linewidth]{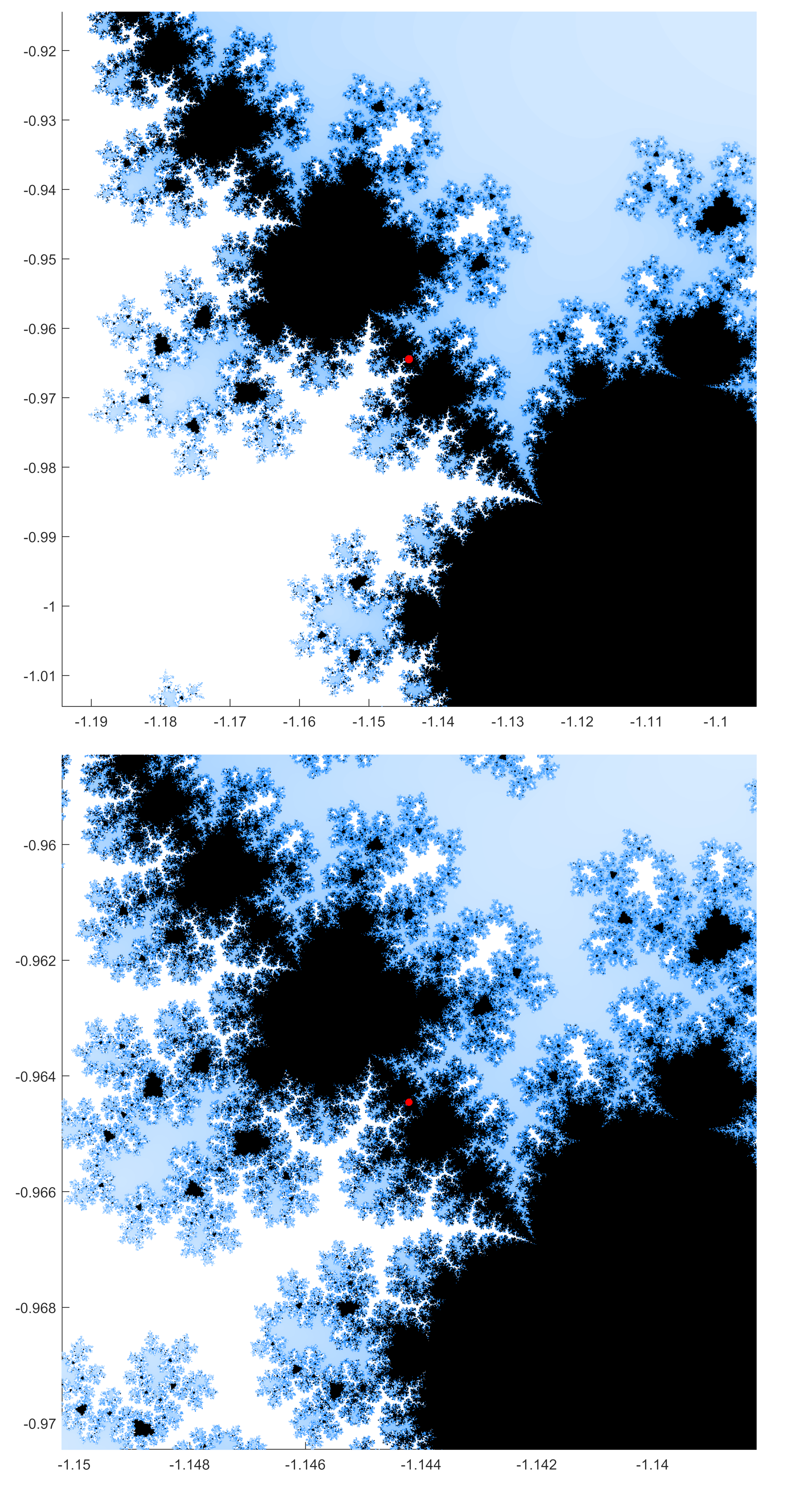}
    \caption{Magnifications of the bifurcation locus of the parameter space
    \leavevmode\\
    \vspace{-0.1in}
    \\
    \begin{minipage}{\linewidth} 
    \begin{align*}
        \left\{F_c(z) = cz^3 \dfrac{4-z}{1-4z+6z^2}\right\}_{c \in \C^*}
    \end{align*}
    by different scales about the parameter $c_* \approx -1.144208-0.964454i$ marked in red. This family is characterized by critical points $0$, $\infty$, and $1$ of local degrees $2$, $3$, and $4$ respectively, where both $0$ and $\infty$ are fixed and $F_c(1)=c$. The point $c_*$ is the unique parameter such that $F_{c_*}$ has a golden mean Herman curve. The Julia set of $F_{c_*}$ can be found in Figure \ref{fig:cqc-comparison}.
  \end{minipage}}
    \label{fig:parspace2}
\end{figure}

\subsection{Parameter self-similarity}
\label{ss:self-similarity}

Consider the one-parameter family of degree $d_0+d_\infty-1$ rational maps $\{F_c\}_{c \in \C^*}$ characterized by critical fixed points at $0$ and $\infty$ with local degrees $d_0$ and $d_\infty$ respectively, and a full degree critical point at $1$ with image $c=F_c(1)$. (See (\ref{eqn:rat-map-formula}).) Based on \cite{Lim23a, Lim23b}, the first examples of $(d_0,d_\infty)$-critical quasicircle maps can be found in this family. As stated in Theorem \ref{thm:comb-rigidity}, for every $\theta \in \IrratBdd$, there is a unique parameter $c(\theta)$ such that $F_{c(\theta)}$ admits an invariant quasicircle passing through $1$ with rotation number $\theta$. Computer pictures (Figure \ref{fig:parspace2}) suggest the following conjecture.

\begin{conj}
    For every $\theta \in \IrratPer$, the bifurcation locus of the family $\{F_c\}_{c \in \C^*}$ is asymptotically self-similar at $c(\theta)$ with a universal self-similarity factor depending only on $(d_0,d_\infty,\theta)$.
\end{conj}

Here, the universality of the factor means that self-similarity with the same scaling factor holds for arbitrary families of unicritical holomorphic maps with similar properties. It can be viewed as an extension of the golden mean universality of critical circle maps \cite{Y03a}. A version of this conjecture initially appears in \cite{Lim23a}. Theorem \ref{main-theorem} is a major step forward towards solving this conjecture, but a complete solution remains out of reach due to two subtle challenges.

The first challenge is the lack of control of the regularity of the renormalization operator sending $\{F_c\}$ to a holomorphic curve near $f_*$. Once hyperbolicity of the renormalization horseshoe for bounded type rotation numbers is established (see the discussion on \S\ref{ss:general-horseshoe}), we will have a lamination of stable manifolds to help us address this regularity issue. 

The second challenge is our limited understanding of the bifurcation locus within the unstable manifold $\unstloc$. To resolve this, we likely need to study the parameter rays within $\unstloc$ (see \cite[\S6.5]{DL23}), but it is currently unclear how to do that.

\subsection{Multiple critical points}
\label{ss:multicritical}

Many of the methods used in proving Theorem \ref{main-theorem}, particularly item \ref{main-4}, are fairly soft. For example, our hyperbolicity result holds even when we allow perturbations of $f_*$ which are no longer unicritical (i.e. when item (4) in Definition \ref{def:corona} no longer holds). Let us elaborate in more detail.

The corona structure persists under small perturbation even when the critical point splits. Therefore, the corona renormalization operator $\Rstar : (\Ustar, f_*) \to (\Bstar,f_*)$ from \S\ref{sec:stab-mani} naturally extends to a compact analytic operator $\renorm_{\text{full}} : (\mathcal{U}_{\text{full}},f_*) \to (\mathcal{B}_{\text{full}}, f_*)$ where $\mathcal{U}_{\text{full}}$ is a small open ball in $\mathcal{B}_{\text{full}}$ in which the coronas generically have $d_0+d_\infty-2$ distinct simple critical points. 

As a fixed point of $\renorm_{\text{full}}$, $f_*$ is still hyperbolic with the same local stable manifold. However, its local unstable manifold becomes larger; it is $(d_0+d_\infty-2)$-dimensional. Indeed, the proof presented in this paper still applies. At the final step, we can perform a quasiconformal surgery (along with a technique similar to that in the proof of Proposition \ref{prop:integration-of-ILF}) to perturb the unicritical superattracting cascade described in Lemma \ref{lem:superattr-reduction} to a cascade $\Fbold$ on the unstable manifold with $d_0+d_\infty-2$ distinct critical orbits. This shows that the hyperbolic component in the unstable manifold containing $\Fbold$ is $(d_0+d_\infty-2)$-dimensional.

Analogously, our result and a similar remark hold for pacman renormalization for higher degree Siegel maps. The general philosophy can be formulated as follows.

\begin{conj}
    Consider a compact analytic renormalization operator with a hyperbolic fixed point such that every map on the unstable manifold $\unstloc$ admits a global transcendental extension. Then, 
    \[
    \dime{\unstloc} \leq \textnormal{number of critical orbits}.
    \]
\end{conj}

\subsection{The general horseshoe case}
\label{ss:general-horseshoe}

For $N \geq 1$, denote by $\Theta_N$ the set of irrationals whose continued fraction expansion terms are bounded above by $N$. 
In \cite{Lim23b}, we demonstrated for any integers $d_0 \geq 2$, $d_\infty \geq 2$, and $N \geq 1$, the renormalizations of $(d_0,d_\infty)$-critical quasicircle maps with rotation numbers in $\Theta_N$ converge to a renormalization horseshoe on which $\Rcp$ is conjugate to a shift map on $N$ symbols. 
Mimicking the arguments in \S\ref{sec:stab-mani}, we should have a corona renormalization operator $\renorm$ with a renormalization attractor $\mathcal{A}_N$ on which $\renorm: \mathcal{A}_N \to \mathcal{A}_N$ is conjugate to a shift map on $N$ symbols.
In the future, we hope to investigate and prove a generalization of our main theorem: that $\mathcal{A}_N$ is uniformly hyperbolic with exactly one unstable direction.
However, at this moment, we recognize that there are many subtleties that need to be addressed in order to prove such a generalization.
For example, in this paper, we rely on the spectral theory of compact operators to construct stable and unstable manifolds.
In the general case, the construction of invariant manifolds requires a lot more work. We highly anticipate that techniques from non-hyperbolic dynamics, in particular Oseledets-Pesin theory, will be needed to rule out neutral eigenvalues and construct the local unstable manifold.

The keen reader may have also noticed the lack of discussion on $(d_0,d_\infty)$-critical quasicircle maps with unbounded irrational rotation number. 
The reason is simple. 
Currently, the only known examples of such maps are circle maps when $d_0=d_\infty$; for example, the formula from (\ref{eqn:rat-map-formula}) becomes a Blaschke product and methods from real dynamics become amenable. 
When either of the criticalities is one, we do not anticipate any generalization of Siegel maps with unbounded irrational rotation number due to Cremer phenomenon. Recently, Yang Fei \cite{Y24} proved the existence of cubic rational maps admitting Herman curves with a high type irrational number; however, in his examples, the critical points are simple and located away from the Herman curves.
When $d_0\neq d_\infty$ and both are at least two, I would expect that examples of critical quasicircle maps with arbitrary unbounded irrational rotation number may exist within the family in (\ref{eqn:rat-map-formula}).
Solving the realization problem would likely require some analysis in the near-degenerate regime that is more complicated than the work presented in \cite{Lim23a}.

Another natural question would be on the renormalization theory for bounded type multicritical quasicircle maps. In recent years, there has been great progress in the renormalization theory of multicritical circle maps \cite{EdFG18,Y19,ESY22,EY23,GY25}. The existence of multicritical quasicircle maps for arbitrary combinatorics was solved in \cite{Lim23a}, but there has yet been any study of the renormalization theory beyond the circle case. We believe that once complex a priori bounds is attained for the multicritical case, we should be able to understand the properties of the corresponding renormalization horseshoe allowing multicriticality.

\subsection*{Acknowledgements}

I would like to thank Dzmitry Dudko for numerous discussions and valuable suggestions on this project. I cannot thank him enough for his kindness and constant encouragement. I would also like to thank James Waterman, Lasse Rempe, Tanya Firsova, and Jonathan Galv\'{a}n Bermudez for helpful discussions, and Timothy Alland for sharing the results of his PhD thesis which ultimately inspired me to study transcendental dynamics associated to renormalizations. This project has been partially supported by the NSF grant DMS 2055532 and by Simons Foundation International, LTD. 

\subsection*{Data availability}

This manuscript has no associated data.

\subsection*{Conflict of interest}

The author has no conflict of interest relating to the contents of this paper.

%%%%%%%%%%%%%%%%%%%%%%%%%%%%%%%%%%%%%%%%%%%%%%%%%%%%%%%%%%%%%%%%%%%%%%%%%%%%%%%%%%%%%%%%%%%%%%%%%%%%%%%%%%%%%%%%%%%%%%%%%%%%%%%%%%%%%%%%%%%%%%%%%%%%%%%%%%

\bibliographystyle{alpha}
 
{\small \bibliography{bibliography}}

\end{document}